\documentclass[10pt,reqno]{amsart}

\usepackage[pass]{geometry}
\newlength\DX
\paperwidth=\dimexpr\paperwidth-\DX\relax
\hoffset=\dimexpr\hoffset-.5\DX\relax
\newlength\DY
\paperheight=\dimexpr\paperheight-\DY\relax
\voffset=\dimexpr\voffset-.5\DY-.5\footskip\relax

\usepackage{mathtools}

\usepackage{lmodern}
\usepackage{mathabx}
\usepackage[T1]{fontenc}

\usepackage{amsmath,amsfonts,amsbsy,amsgen,amscd,mathrsfs,amssymb,amsthm}
\usepackage{enumerate}
\usepackage{bm}

\usepackage{stmaryrd}
\SetSymbolFont{stmry}{bold}{U}{stmry}{m}{n} 

\usepackage[usenames,dvipsnames]{xcolor}
\usepackage[colorlinks=true,citecolor=blue,linkcolor=blue]{hyperref}

\usepackage{tikz}
\usetikzlibrary{shadings}

\newtheorem{thm}{Theorem}[section]
\newtheorem{lem}[thm]{Lemma}

\newtheorem{prop}[thm]{Proposition}
\newtheorem{cor}[thm]{Corollary}

\theoremstyle{definition}

\newtheorem{defn}[thm]{Definition}

\newtheorem*{claim}{Claim}

\theoremstyle{remark}

\newtheorem{example}[thm]{Example}
\newtheorem{rem}[thm]{Remark}

\numberwithin{equation}{section} 
\numberwithin{figure}{section}
\numberwithin{table}{section}

\let\oldtocsection=\tocsection
\let\oldtocsubsection=\tocsubsection
\let\oldtocsubsubsection=\tocsubsubsection
\renewcommand{\tocsection}[2]{\hspace{-1.2em}\oldtocsection{#1}{#2}}
\renewcommand{\tocsubsection}[2]{\hspace{-.2em}\oldtocsubsection{#1}{#2}}
\renewcommand{\tocsubsubsection}[2]{\hspace{0.8em}\oldtocsubsubsection{#1}{#2}}

\DeclareRobustCommand{\gobblefive}[5]{}
\newcommand*{\SkipTocEntry}{\addtocontents{toc}{\gobblefive}}

\makeatletter
\renewcommand\subsubsection{\@startsection{subsubsection}{3}%
  \z@{.5\linespacing\@plus.7\linespacing}{-.5em}%
  {\normalfont\bfseries}}
\makeatother

\newcommand{\Vol}{\mathrm{Vol}}
\newcommand{\V}{\mathsf{V}}
\newcommand{\intr}{\mathop{\mathrm{int}}}
\newcommand{\supp}{\mathop{\mathrm{supp}}}
\newcommand{\sspan}{\mathop{\mathrm{span}}}
\newcommand{\proj}{\bm{\mathsf{P}}}
\newcommand{\wto}{\stackrel{w}{\to}}
\newcommand{\relint}{\mathop{\mathrm{relint}}}
\newcommand{\aff}{\mathop{\mathrm{aff}}}
\newcommand{\A}{\mathrm{A}}
\newcommand{\h}{\mathrm{h}}

\begin{document}

\title[Extremals of the Alexandrov-Fenchel Inequality]{The Extremals
of the Alexandrov-Fenchel Inequality for Convex Polytopes}

\author{Yair Shenfeld}
\address{Department of Mathematics, Massachusetts Institute of Technology, 
Cambridge, MA, USA}
\email{shenfeld@mit.edu}

\author{Ramon van Handel}
\address{Fine Hall 207, Princeton University, Princeton, NJ 
08544, USA}
\email{rvan@math.princeton.edu}

\begin{abstract}
The Alexandrov-Fenchel inequality, a far-reaching generalization of the 
classical isoperimetric inequality to arbitrary mixed volumes, lies at the 
heart of convex geometry. The characterization of its extremal bodies is a 
long-standing open problem that dates back to Alexandrov's original 1937 
paper. The known extremals already form a very rich family, and even 
the fundamental conjectures on their general structure, due to Schneider, 
are incomplete. In this paper, we completely settle the extremals of the 
Alexandrov-Fenchel inequality for convex polytopes. In particular, we show 
that the extremals arise from the combination of three distinct 
mechanisms: translation, support, and dimensionality. The characterization 
of these mechanisms requires the development of a diverse range of 
techniques that shed new light on the geometry of mixed volumes of 
nonsmooth convex bodies. Our main result extends further beyond polytopes 
in a number of ways, including to the setting of quermassintegrals of 
arbitrary convex bodies. As an application of our main result, we settle 
a question of Stanley on the extremal behavior of certain log-concave 
sequences that arise in the combinatorics of partially ordered sets.
\end{abstract}

\subjclass[2010]{52A39; 
                 52A40; 
		 52B05; 
		 05B25} 

\keywords{Mixed volumes; Alexandrov-Fenchel inequality; convex polytopes; 
extremum problems in geometry and combinatorics}

\vspace*{-.01cm}
\maketitle

\thispagestyle{empty}

\vspace*{-.2cm}
\setcounter{tocdepth}{1}
{\small\tableofcontents}

\section{Introduction}
\label{sec:intro}

\subsection{The Alexandrov-Fenchel inequality and the extremal problem}

Let $C_1,\ldots,C_m$ be convex bodies (that is, nonempty compact convex 
sets) in $\mathbb{R}^n$. One of the most basic facts of convex geometry, 
due to Minkowski \cite{Min03}, is that the volume of convex bodies behaves 
as a homogeneous polynomial under addition
$\lambda C+\mu C':=\{\lambda x+\mu y:x\in C,y\in C'\}$: that is,
for all $\lambda_1,\ldots,\lambda_m\ge 0$
\begin{equation}
\label{eq:volpoly}
	\Vol_n(\lambda_1C_1+\cdots+\lambda_mC_m)
	= \sum_{i_1,\ldots,i_n=1}^m
	\V_n(C_{i_1},\ldots,C_{i_n})\,
	\lambda_{i_1}\cdots\lambda_{i_n}.
\end{equation}
The coefficients $\V_n(C_{i_1},\ldots,C_{i_n})$ of this polynomial, called 
\emph{mixed volumes}, form a large family of natural geometric parameters 
associated to convex bodies. For example, the special cases 
$\V_n(C,\ldots,C,B,\ldots,B)$, called quermassintegrals, already capture 
familiar notions such as the volume, surface area, and mean width of $C$, 
and the average volume of the projections of $C$ onto a random 
$k$-dimensional subspace.\footnote{%
	Throughout this paper $B$ denotes the 
	Euclidean unit ball in $\mathbb{R}^n$.}
In view of these and numerous other important examples, mixed volumes play 
a central role in convex geometry \cite{BF87,BZ88,San93,Sch14}.

When the convex bodies are polytopes, mixed volumes may also be viewed as 
belonging to combinatorial geometry. In this setting, striking connections 
arise between the theory of mixed volumes and other areas of mathematics. 
For example, in algebraic geometry, mixed volumes compute the number of 
solutions to systems of polynomial equations \cite[\S 27]{BZ88} and 
intersection numbers of divisors on toric varieties \cite{Ful93,Ew96}; and 
in combinatorics, mixed volumes compute quantities associated to objects 
such as matroids, partial orders, and permanents \cite{Sta81,Huh18}.

Given the central nature of mixed volumes, it is natural to expect that 
inequalities between mixed volumes capture important mathematical 
phenomena. The most fundamental result of this kind is the 
Alexandrov-Fenchel inequality, which expresses the fact that mixed volumes 
are log-concave.

\begin{thm}[Alexandrov-Fenchel inequality]
\label{thm:af}
We have
$$
	\V_n(K,L,C_1,\ldots,C_{n-2})^2 \ge
	\V_n(K,K,C_1,\ldots,C_{n-2}) \, \V_n(L,L,C_1,\ldots,C_{n-2})
$$
for any convex bodies $K,L,C_1,\ldots,C_{n-2}$ in $\mathbb{R}^n$.
\end{thm}

Theorem \ref{thm:af} was first proved by Minkowski in 1903 in dimension 
$n=3$ \cite{Min03}, and in full generality by Alexandrov in 1937 
\cite{Ale37,Ale96}. (Fenchel independently announced a proof \cite{Fen36}, 
but it was never published.) It lies at the heart of many applications of 
mixed volumes in convexity and in other areas of mathematics. 
This paper is concerned with a classical open problem surrounding 
the Alexandrov-Fenchel inequality that dates back to Alexandrov's original 
paper \cite[p.\ 80]{Ale96}.

To provide context for the problem studied in this paper, let us 
recall the original setting of Minkowski \cite{Min03}. 
Minkowski viewed Theorem \ref{thm:af} as a far-reaching generalization of 
the isoperimetric inequality between volume and surface area, 
which are merely two special cases of mixed volumes. For example, the 
special case
$$
	\V_3(B,C,C)^2 \ge \V_3(C,C,C)\,\V_3(B,B,C)
$$
states that the surface area of a three-dimensional convex body $C$ is 
lower bounded by the product of its volume and mean width, a kind of 
isoperimetric inequality involving three geometric parameters. From this 
viewpoint, a complete understanding of Theorem \ref{thm:af} should capture 
not only the inequality but also the associated extremum problem: which 
bodies minimize surface area when the volume and mean width are fixed? 
This question is equivalent to the study of the cases of equality in the 
above inequality. Remarkably, it turns out that the extremals in this 
example possess highly unusual properties: they consist of cap bodies 
(``spiky balls'') which are both non-unique and non-smooth, in sharp 
contrast with the situation in the classical isoperimetric problem (cf.\ 
\cite{SvH19} and the references therein).

The above example suggests that the extremum problems associated to more 
general cases of the Alexandrov-Fenchel inequality are likely to possess a 
rich and intricate structure. The problem of characterizing these 
extremals was raised in the original papers of Minkowski \cite{Min03} and 
Alexandrov \cite{Ale37}, but progress toward the resolution of this 
problem has proved to be elusive. None of the known proofs of the 
Alexandrov-Fenchel inequality provides information on its cases of 
equality. The geometric proofs (cf.\ \cite{Ale96,SvH18}) impose 
restrictions, such as smooth bodies or polytopes with identical face 
directions, under which only trivial extremals arise, and deduce the 
general result by approximation; nontrivial extremals arise only in the 
limit, and are thus invisible in the proofs of the inequality. The 
algebraic proofs (cf.\ \cite{BZ88,Ful93}) perform a reduction to a 
certain (non-toric)
algebraic surface, which causes the convex geometric structure of 
the problem to be lost.

It was long believed that the extremals of the Alexandrov-Fenchel 
inequality are too numerous to admit a meaningful geometric 
characterization, cf.\ \cite[\S 20.5]{BZ88} or \cite[p.\ 248]{Fav33}. 
However, detailed conjectures on the structure of the extremals 
(attributed in part to Loritz) were published in 1985 by Schneider 
\cite{Sch85}, breathing new life into the problem. Schneider's conjectures 
need not hold when some of the bodies have empty interior \cite{Ew88}, and 
no conjectures have been formulated to date about this setting (which, as 
we will see, is of special importance in applications). However, the 
validity of Schneider's conjectures for full-dimensional bodies has 
remained open, except in a few special cases that are reviewed in \cite[\S 
7.6]{Sch14}, \cite{San93}. Very recently, significant new progress was 
made in \cite{SvH19}, which enabled the proof of Schneider's conjectures 
in the case that dates back to Minkowski \cite{Min03}. The general case is 
however much richer, and entirely new ideas are needed.

\subsection{Main results}

In this paper, we completely settle the extremal problem \emph{in the 
combinatorial setting}. Our main result characterizes all equality cases
$$
	\V_n(K,L,P_1,\ldots,P_{n-2})^2 =
	\V_n(K,K,P_1,\ldots,P_{n-2})\,
	\V_n(L,L,P_1,\ldots,P_{n-2})
$$
when $P_1,\ldots,P_{n-2}$ are arbitrary convex polytopes in 
$\mathbb{R}^n$ and $K,L$ are convex bodies. The characterization of the 
extremal bodies is described in section \ref{sec:three}. In particular, we 
will show that the extremals of the Alexandrov-Fenchel inequality arise 
from the combination of three distinct mechanisms: translation, support, 
and dimensionality. The first two mechanisms were anticipated by 
Schneider's conjectures, while the third is responsible for the new 
extremals that arise when the polytopes $P_i$ may have empty interior. 
The proof of our main result (Theorem \ref{thm:main}), which is contained 
in sections \ref{sec:matrix}--\ref{sec:proofmain}, will in fact give 
considerably more precise information on the structure of the extremals 
than is provided by the characterization in section \ref{sec:three}; 
the most detailed form of our main result
will be formulated in section \ref{sec:unique}.

Aside from its intrinsic place in the foundation of convex geometry, the 
problem of characterizing the extremals of the Alexandrov-Fenchel 
inequality may be thought of in a broader context: the limited progress on 
this problem to date stems from major gaps in the understanding of the 
geometry of mixed volumes of non-smooth convex bodies. The fundamental 
issues that arise are both of a combinatorial and of an analytic nature, 
as we will explain presently.

As will become clear in section \ref{sec:three}, the extremals of the 
Alexandrov-Fenchel inequality are controlled by the boundary structure of 
the bodies $C_1,\ldots,C_{n-2}$ in Theorem \ref{thm:af}. In the case that 
was settled in \cite{SvH19}, only the boundary structure of a single body 
plays a role. In general, however, each of the bodies $C_1,\ldots,C_{n-2}$ 
has an arbitrary boundary structure, and the interactions between the 
different bodies conspire to give rise to the extremals. This interaction 
already arises in its full complexity in the combinatorial setting 
considered in this paper. In settling the problem, we develop a theory 
that explains these interactions: this includes, among other ingredients, 
a local Alexandrov-Fenchel inequality for mixed area measures, strong 
gluing principles for projections from limited data, and new geometric 
structures (``propellers'') of mixed area measures of bodies with empty 
interior. An overview of the proof of our main result will be 
given in section \ref{sec:overview}.

The main contribution of this paper is the complete solution of these 
combinatorial aspects of the problem. In contrast, the obstacle to going 
beyond polytopes stems from unresolved analytic problems in the theory of 
mixed volumes, which are largely independent of the problems studied in 
this paper. These analytic problems arise because the boundary of a 
general convex body may be almost arbitrarily irregular (for example, 
consider the convex hull of an arbitrary closed subset of the unit 
sphere), so that mixed volumes of general convex bodies give rise to 
analytic objects that live on highly irregular sets. The treatment of 
general bodies therefore requires the development of an appropriate 
functional-analytic framework, which has only been partially accomplished 
to date \cite{SvH19} (see section~\ref{sec:discussion} for discussion). 
The main ideas of this paper are not specific to polytopes, however, and 
may be expected to apply more generally when placed in an suitable 
analytic framework.


\subsection{Extensions and applications}

While this paper is primarily concerned with the combinatorial setting, 
our methods already admit a number of extensions beyond the setting of 
convex polytopes. In particular, we will show in section 
\ref{sec:extensions} that our main result extends to the setting where the 
convex bodies $C_1,\ldots,C_{n-2}$ in Theorem~\ref{thm:af} are a 
combination of polytopes, zonoids, and smooth bodies. By combining the 
present methods with \cite{SvH19}, we will also fully characterize the 
extremals of the Alexandrov-Fenchel inequality for quermassintegrals of 
arbitrary convex bodies, a special case that arises frequently in 
applications.

Section \ref{sec:stanley} develops an application in combinatorics. It was 
noticed long ago that various combinatorially defined sequences $(N_i)$ 
appear to be log-concave, that is, they satisfy $N_i^2\ge N_{i-1}N_{i+1}$. 
Such phenomena have received much attention in recent years \cite{Huh18}. 
One of the earliest advances in this area is due to Stanley \cite{Sta81}, 
who observed that if one can represent the relevant combinatorial 
quantities in terms of mixed volumes, log-concavity is explained by the 
Alexandrov-Fenchel inequality. Stanley further raises the following 
question: in cases where $(N_i)$ is log-concave, can one characterize the 
associated extremum problem, that is, explain what combinatorial objects 
achieve equality $N_i^2=N_{i-1}N_{i+1}$? As an illustration of our main 
result, we will settle this problem in one of the settings considered by 
Stanley, where $N_i$ is the number of linear extensions of a partially 
ordered set for which a distinguished element has rank $i$. Such extremal 
problems appear to be inaccessible by currently known methods of 
enumerative or algebraic combinatorics. This example highlights the 
significance of the questions considered in this paper to extremal 
problems in other areas of mathematics, and hints at the 
possibility that the structures developed here might have analogues 
outside convexity; a brief discussion of algebraic analogues of our 
results is given in section \ref{sec:discussion}.

Let us note that, far from being esoteric, it is precisely the case of 
convex bodies with empty interior (which is not covered by previous 
conjectures) that arises in combinatorial applications \cite{Sta81}. This 
reinforces the importance of a complete characterization of the extremals, 
whose formulation we turn to presently.

\section{Three extremal mechanisms}
\label{sec:three}

The aim of this section is to formulate and explain the main result of 
this paper. We first recall some key facts on mixed volumes and mixed 
area measures. We will subsequently describe three distinct mechanisms 
that give rise to extremals of the Alexandrov-Fenchel inequality, and 
state our main result. Here and throughout the paper, our standard 
reference on convexity is the monograph \cite{Sch14}.

\subsection{Basic facts}
\label{sec:basic}

\subsubsection{Convex bodies, mixed volumes, mixed area measures}

Fix $n\ge 3$. A \emph{convex body} is a nonempty compact 
convex set in $\mathbb{R}^n$. A (convex) \emph{polytope} is the 
convex hull of a finite number of points.

To each convex body $K$, we associate its 
\emph{support function}
$$
	h_K(u) := \sup_{y\in K} \langle y,u\rangle.
$$
We think of $h_K$ either as a function on $S^{n-1}$ or as a 
$1$-homogeneous function on $\mathbb{R}^n$.
Geometrically, if $u\in S^{n-1}$, then $h_K(u)$ is the (signed) distance 
to the origin of the supporting hyperplane of $K$ with outer normal $u$; 
thus $h_K:S^{n-1}\to\mathbb{R}$ uniquely determines $K$, as any convex 
body is the intersection of its supporting halfspaces. The key property of 
support functions is that they behave naturally under addition, that is, 
$h_{\lambda K+\mu L}=\lambda h_K+\mu h_L$ for any bodies $K,L$ and 
$\lambda,\mu\ge 0$.

The \emph{mixed volume} $\V_n(C_1,\ldots,C_n)$ of $n$ convex bodies 
$C_1,\ldots,C_n$ in $\mathbb{R}^n$ is defined by \eqref{eq:volpoly}. Mixed 
volumes are nonnegative, and are symmetric and multilinear in 
their arguments. Moreover, there exists a nonnegative measure 
$S_{C_1,\ldots,C_{n-1}}$ on $S^{n-1}$, called the \emph{mixed area 
measure} of $C_1,\ldots,C_{n-1}$, such that
\begin{equation}
\label{eq:mixvolarea}
	\V_n(K,C_1,\ldots,C_{n-1}) =
	\frac{1}{n}\int h_K(u) \,S_{C_1,\ldots,C_{n-1}}(du).
\end{equation}
Like mixed volume, $S_{C_1,\ldots,C_{n-1}}$ is symmetric and multilinear
in $C_1,\ldots,C_{n-1}$.

Consider a function $f=h_K-h_L$ that is a difference of support functions.
As mixed volumes and mixed area measures are multilinear as functions of 
the underlying bodies (and hence of their support functions), we may 
uniquely extend their definitions to differences of support functions
\cite[\S 5.2]{Sch14}. That is, we will write
\begin{align*}
	\V_n(f,C_1,\ldots,C_{n-1}) & :=
	\V_n(K,C_1,\ldots,C_{n-1}) - 
	\V_n(L,C_1,\ldots,C_{n-1}), \\
	S_{f,C_1,\ldots,C_{n-2}} & :=
	S_{K,C_1,\ldots,C_{n-2}} -
	S_{L,C_1,\ldots,C_{n-2}}.
\end{align*}
We may analogously define $\V_n(f,g,C_1,\ldots,C_{n-2})$ when $f,g$ are 
differences of support functions, etc. The extended definitions are still 
symmetric and multilinear, but are not necessarily nonnegative. 
Differences of support functions form a large class of functions 
on $S^{n-1}$: in particular, we have the following \cite[Lemma 1.7.8]{Sch14}.

\begin{lem}
\label{lem:c2}
Any $f\in C^2(S^{n-1})$ is a difference of support functions.
\end{lem}

\subsubsection{Positivity}

While mixed volumes and mixed area measures of convex bodies are always 
nonnegative, they need not be strictly positive. Positivity of mixed 
volumes and mixed area measures will play an important role throughout 
this paper. We presently state two key facts in this direction. First, we 
recall that positivity of mixed volumes is characterized by 
dimensionality conditions \cite[Theorem 5.1.8]{Sch14}.
Throughout this paper, we denote by $[n]:=\{1,\ldots,n\}$.

\begin{lem}
\label{lem:dim}
For convex bodies $C_1,\ldots,C_n$ in $\mathbb{R}^n$, the following are
equivalent:
\begin{enumerate}[a.]
\itemsep\abovedisplayskip
\item $\V_n(C_1,\ldots,C_n)>0$.
\item There are segments $I_i\subseteq C_i$, $i\in[n]$ with
linearly independent directions.
\item $\dim(C_{i_1}+\cdots+C_{i_k})\ge k$ for all 
$k\in[n]$, $1\le i_1<\cdots<i_k\le n$.
\end{enumerate}
\end{lem}

Similarly, the mixed area measure $S_{C_1,\ldots,C_{n-1}}$ need not be 
supported on the entire sphere $S^{n-1}$. Unlike the positivity of mixed 
volumes, the problem of characterizing geometrically the support of mixed 
area measures of arbitrary convex bodies is not yet fully settled, cf.\ 
\cite[Conjecture 7.6.14]{Sch14}. However, for the present purposes we 
require only the following special case. For 
any vector $u\in\mathbb{R}^n$, let 
\begin{equation}
\label{eq:facedef}
	F(K,u) := \{x\in K: \langle u,x\rangle = h_K(u)\}
\end{equation}
be the unique face of $K$ with outer normal direction $u$. The following 
result states that when $P_1,\ldots,P_{n-2}$ are polytopes, the support of 
the mixed area measure $S_{B,P_1,\ldots,P_{n-2}}$ is characterized by 
dimensionality conditions on faces of $P_1,\ldots,P_{n-2}$.
This result is essentially known; we will provide a proof in section 
\ref{sec:qgraph}.

\begin{lem}
\label{lem:supp}
Let $P_1,\ldots,P_{n-2}$ be any convex polytopes in $\mathbb{R}^n$, and 
let $u\in S^{n-1}$. Then the following conditions are equivalent:
\vspace{.5\abovedisplayskip}
\begin{enumerate}[a.]
\itemsep\abovedisplayskip
\item $u\in\supp S_{B,P_1,\ldots,P_{n-2}}$.
\item There are segments $I_i\subseteq F(P_i,u)$, $i\in [n-2]$
with linearly independent directions.
\item $\dim(F(P_{i_1},u)+\cdots+F(P_{i_k},u))\ge k$ for all
$k\in[n-2]$, $1\le i_1<\cdots<i_k\le n-2$.
\end{enumerate}
\vspace{.5\abovedisplayskip}
When a--c hold, $u\in S^{n-1}$ is called a 
\emph{$(B,P_1,\ldots,P_{n-2})$-extreme} 
normal direction.
\end{lem}

The appearance the Euclidean ball $B$ in Lemma \ref{lem:supp} may appear 
rather arbitrary: we did not assume $B$ appears as one of the bodies in 
Theorem \ref{thm:af}. Its significance is that the associated mixed area 
measure has maximal support \cite[Lemma 7.6.15]{Sch14} (an alternative 
proof may be given along the lines of Lemma \ref{lem:lowmaxsupp} below).

\begin{lem}
\label{lem:maxsupp}
For any convex bodies $M,C_1,\ldots,C_{n-2}$, we have
$$
	\supp S_{M,C_1,\ldots,C_{n-2}}\subseteq
	\supp S_{B,C_1,\ldots,C_{n-2}}.
$$
\end{lem}

Let us note that Lemma \ref{lem:maxsupp} remains valid if $B$ is replaced 
by any sufficiently smooth convex body; there is nothing uniquely
special about $B$. However, the choice of Euclidean ball will prove to be 
particularly convenient in our proofs.

\subsubsection{Equality}

We finally recall a basic fact about equality in the 
Alexandrov-Fenchel 
inequality. It is evident that there is equality in Theorem \ref{thm:af} 
if and only if the difference between the left- and right-hand sides of 
the inequality is minimized. The first-order optimality condition 
associated to this minimum problem gives rise to an equivalent formulation 
of the equality cases of the Alexandrov-Fenchel inequality, due to 
Alexandrov \cite[p.\ 80]{Ale96} (cf.\ section \ref{sec:3af} 
or \cite[Theorem 7.4.2]{Sch14}).

\begin{lem}
\label{lem:simpleeq}
Let $K,L,C_1,\ldots,C_{n-2}$ be convex bodies in $\mathbb{R}^n$ such that
$$
	\V_n(K,L,C_1,\ldots,C_{n-2})>0.
$$
Then the following are equivalent:
\begin{enumerate}[a.]
\itemsep\abovedisplayskip
\item $\V_n(K,L,C_1,\ldots,C_{n-2})^2=
\V_n(K,K,C_1,\ldots,C_{n-2})\,
\V_n(L,L,C_1,\ldots,C_{n-2})$.
\item $S_{h_K-ah_L,C_1,\ldots,C_{n-2}}=0$ for some $a>0$.
\end{enumerate}
\end{lem}

Let us emphasize that this result provides essentially no information on 
the geometry of the extremal bodies $K,L,C_1,\ldots,C_{n-2}$: it is merely 
a reformulation of the equality condition. The main problem that will be 
addressed in this paper is to develop a geometric characterization of the 
extremals.

\begin{rem}
\label{rem:triv}
When $\V_n(K,L,C_1,\ldots,C_{n-2})=0$, there is automatically
equality in Theorem \ref{thm:af}. These \emph{trivial} equality cases
are fully characterized by Lemma \ref{lem:dim}. Nontrivial equality 
cases arise only when $\V_n(K,L,C_1,\ldots,C_{n-2})>0$, as is assumed in
Lemma \ref{lem:simpleeq}. This is the setting that will 
concern us in the rest of this paper.
\end{rem}

\subsection{Extremal mechanisms}

What convex bodies yield equality in Theorem \ref{thm:af}? We will now 
describe three mechanisms that yield extremals of the Alexandrov-Fenchel 
inequality, each capturing a different geometric phenomenon: translation 
(section \ref{sec:mechtrans}), support (section \ref{sec:mechsupp}), and 
dimensionality (section \ref{sec:mechdim}).

It is important to note that the bodies $K,L$ and $C_1,\ldots,C_{n-2}$ 
play very different roles in Theorem \ref{thm:af}: $K,L$ vary, while 
$C_1,\ldots,C_{n-2}$ are the same in each term. We therefore consider the 
reference bodies $C_1,\ldots,C_{n-2}$ as fixed, and aim to characterize 
which $K,L$ yield equality in Theorem \ref{thm:af}. By 
Lemma~\ref{lem:simpleeq}, the problem can be formulated equivalently as 
follows: \emph{given $C_1,\ldots,C_{n-2}$, we aim to characterize what 
differences of support functions $f$ satisfy 
$S_{f,C_1,\ldots,C_{n-2}}=0$.}

\subsubsection{Translation}
\label{sec:mechtrans}

The simplest mechanism for equality in Theorem \ref{thm:af}
stems from the most basic invariance property of mixed volumes:
as volume is translation-invariant, \eqref{eq:volpoly} implies that 
mixed volumes are as well, that is,
$$
	\V_n(K,C_1,\ldots,C_{n-1}) =
	\V_n(K+v,C_1,\ldots,C_{n-1})
$$
for all $v\in\mathbb{R}^n$. In terms of support functions, we have
$h_{K+v}(u)=h_K(u)+\langle v,u\rangle$, that is, the support function of a 
convex body and its translate differ by a linear function. This gives rise 
to the following equality case.

\begin{lem}
\label{lem:lineq}
$S_{f,C_1,\ldots,C_{n-2}}=0$ whenever $f=\langle v,\cdot\rangle$
is a linear function.
\end{lem}

\begin{proof}
Let $f=\langle v,\cdot\rangle$ be any linear function. Then 
$f=h_{K+v}-h_K$ for any convex body $K$. Therefore, by 
translation-invariance of mixed volumes,
$$
	\frac{1}{n}\int g \,dS_{f,C_1,\ldots,C_{n-2}} =
	\V_n(g,f,C_1,\ldots,C_{n-2}) = 0
$$
for any difference of support 
functions $g$, and thus \emph{a fortiori} for any $g\in C^2(S^{n-1})$
by Lemma \ref{lem:c2}.
It follows immediately that $S_{f,C_1,\ldots,C_{n-2}}=0$.
\end{proof}

Lemma \ref{lem:lineq} and Lemma \ref{lem:simpleeq} imply, for example, 
that equality occurs in the Alexandrov-Fenchel inequality whenever 
$h_K-ah_L=\langle v,\cdot\rangle$ for some $a>0$ and $v\in\mathbb{R}^n$, 
which simply means that $K=aL+v$ (that is, $K$ and $L$ are homothetic). Of 
course, this also follows immediately from Theorem \ref{thm:af}.

\subsubsection{Support}
\label{sec:mechsupp}

A much more subtle invariance property of mixed volumes stems from the 
fact that mixed area measures need not be supported on the entire sphere 
$S^{n-1}$. Indeed, it follows immediately from \eqref{eq:mixvolarea} that
$$
	\V_n(K,C_1,\ldots,C_{n-1})=
	\V_n(L,C_1,\ldots,C_{n-1})
$$
whenever
$$
	h_K(u)=h_L(u)\mbox{ for all }u\in\supp S_{C_1,\ldots,C_{n-1}}.
$$
That this phenomenon gives rise to new extremals of the 
Alexandrov-Fenchel inequality dates back essentially to the work of 
Minkowski, and has been put forward systematically by Schneider. Let us 
give a precise formulation \cite[p.\ 430]{Sch14}.

\begin{lem}
\label{lem:suppeq}
$S_{f,C_1,\ldots,C_{n-2}}=0$ whenever $f(u)=0$ for all
$u\in\supp S_{B,C_1,\ldots,C_{n-2}}$.
\end{lem}

\begin{proof}
Suppose $f$ vanishes on $\supp S_{B,C_1,\ldots,C_{n-2}}$.
Then
$$
	\frac{1}{n}\int g\,dS_{f,C_1,\ldots,C_{n-2}} =
	\V_n(g,f,C_1,\ldots,C_{n-2})=
	\frac{1}{n}\int f\,dS_{g,C_1,\ldots,C_{n-2}} = 0,
$$
for any difference of support functions $g$,
where we used the symmetry of mixed volumes 
and that $\supp S_{g,C_1,\ldots,C_{n-2}} \subseteq
\supp S_{B,C_1,\ldots,C_{n-2}}$ by Lemma 
\ref{lem:maxsupp}.
The conclusion follows as we may choose any $g\in C^2(S^{n-1})$
by Lemma \ref{lem:c2}.
\end{proof}

In the case that $C_1,\ldots,C_{n-2}$ are polytopes, we have given a 
geometric characterization of the support of $S_{B,C_1,\ldots,C_{n-2}}$ in 
Lemma \ref{lem:supp}. This yields a fully geometric interpretation of the 
situation described by Lemma \ref{lem:suppeq}: that $f=h_K-h_L$ vanishes 
on $\supp S_{B,C_1,\ldots,C_{n-2}}$ means precisely that the convex bodies 
$K$ and $L$ have the same supporting hyperplanes in all 
$(B,C_1,\ldots,C_{n-2})$-extreme normal directions.

\begin{example}
\label{ex:cap}
Let $C=[0,1]^3$ be a cube in $\mathbb{R}^3$, and let the bodies $K$ 
and $L$ be 
derived from $C$ by slicing off some of its corners. This construction is 
illustrated in Figure \ref{fig:cap}. We claim that $h_K-h_L$ vanishes on 
$\supp S_{B,C}$, so that in particular
$$
	\V_3(K,L,C)^2 = \V_3(K,K,C)\,\V_3(L,L,C)
$$
in this example by Lemmas \ref{lem:suppeq} and \ref{lem:simpleeq}.
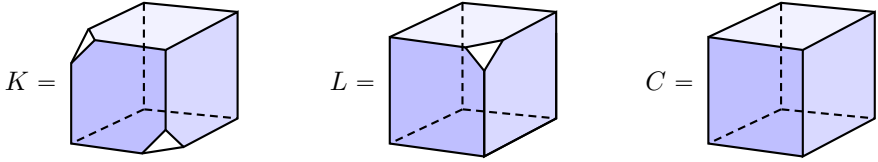
\begin{figure}
\centering
\begin{tikzpicture}[scale=.5]
\begin{scope}[xshift=-17cm]
\draw (-8.6,-.25) node[left] {$K=$};
\fill[blue!15] (-5.95,.55) -- (-4.03,1.55) -- 
(-4.03,-1.25) -- (-5.47,-2) -- (-5.95,-1.55) -- (-5.95,.55);
\fill[blue!25] (-5.95,.55) -- (-7.84,.81) -- (-8.47,.20) -- 
(-8.47,-1.91) -- (-6.58,-2.17) -- (-5.95,-1.55) -- (-5.95,.55);
\fill[blue!5] (-7.84,.81) -- (-8,1.1) -- (-6.53,1.79) -- (-4.03,1.55) --
(-5.95,.55) -- (-7.84,.81);
\draw[thick] (-5.95,.55) -- (-4.03,1.55) -- (-4.03,-1.25) --
(-5.47,-2) -- (-5.95,-1.55) -- (-5.95,.55);
\draw[thick] (-5.95,.55) -- (-7.84,.81) -- (-8.47,.20) -- (-8.47,-1.91) 
-- (-6.58,-2.17) -- (-5.95,-1.55);
\draw[thick] (-6.58,-2.17) -- (-5.47,-2);
\draw[thick] (-8.01,1.1) -- (-8.47,.20);
\draw[thick] (-7.84,.81) -- (-8,1.1) -- (-6.53,1.79) -- (-4.03,1.55);
\draw[thick,densely dashed] (-6.53,1.79) -- (-6.53,-1.01) --
(-4.03,-1.25);
\draw[thick,densely dashed] (-6.53,-1.01) -- (-8.47,-1.91);
\end{scope}
\begin{scope}[xshift=-8.5cm]
\draw (-8.6,-.25) node[left] {$L=$};
\fill[blue!15] (-5.95,0) -- (-5.45,.81) -- (-4.03,1.55) -- 
(-4.03,-1.25) -- (-5.95,-2.25) -- (-5.95,0);
\fill[blue!25] (-6.454,.618) -- (-8.47,.89) -- 
(-8.47,-1.91) -- (-5.95,-2.25) -- (-5.95,0);
\fill[blue!5] (-8.47,.89) -- (-6.53,1.79) -- 
(-4.03,1.55) -- (-5.45,.81) -- (-6.454,.618) -- (-8.47,.89);
\draw[thick] (-4.03,1) -- (-4.03,-1.25) --
(-5.95,-2.25) -- (-5.95,0);
\draw[thick] (-5.95,0) -- (-5.45,.81) -- (-4.03,1.55) -- (-4.03,-1.25) -- 
(-5.95,-2.25) -- (-5.95,0) -- (-6.454,.618) -- (-5.45,.81);
\draw[thick] (-6.454,.618) -- (-8.47,.89) -- (-8.47,-1.91) --
(-5.95,-2.25);
\draw[thick] (-8.47,.89) -- (-6.53,1.79) -- (-4.03,1.55);
\draw[thick,densely dashed] (-6.53,1.79) -- (-6.53,-1.01) --
(-4.03,-1.25);
\draw[thick,densely dashed] (-6.53,-1.01) -- (-8.47,-1.91);
\end{scope}
\draw (-8.6,-.25) node[left] {$C=$};
\fill[blue!15] (-5.95,.55) -- (-4.03,1.55) -- 
(-4.03,-1.25) -- (-5.95,-2.25) -- (-5.95,.55);
\fill[blue!25] (-5.95,.55) -- (-8.47,.89) -- 
(-8.47,-1.91) -- (-5.95,-2.25) -- (-5.95,.55);
\fill[blue!5] (-8.47,.89) -- (-6.53,1.79) -- 
(-4.03,1.55) -- (-5.95,.55) -- (-8.47,.89);
\draw[thick] (-5.95,.55) -- (-4.03,1.55) -- (-4.03,-1.25) --
(-5.95,-2.25) -- (-5.95,.55);
\draw[thick] (-5.95,.55) -- (-8.47,.89) -- (-8.47,-1.91) --
(-5.95,-2.25);
\draw[thick] (-8.47,.89) -- (-6.53,1.79) -- (-4.03,1.55);
\draw[thick,densely dashed] (-6.53,1.79) -- (-6.53,-1.01) --
(-4.03,-1.25);
\draw[thick,densely dashed] (-6.53,-1.01) -- (-8.47,-1.91);
\end{tikzpicture}
\caption{Example of an equality case described by Lemma \ref{lem:suppeq}.
\label{fig:cap}}
\end{figure}

To verify the claim, note that by part \emph{c} of Lemma \ref{lem:supp}, 
we have $u\in\supp S_{B,C}$ if and only if $u$ is a normal direction of a 
face of $C$ of dimension at least one, that is, if $u$ is the outer normal 
of a supporting hyperplane of one of the edges of the unit cube. But it is 
readily seen in Figure \ref{fig:cap} that any such hyperplane also 
supports both $K$ and $L$, so that $h_K(u)=h_L(u)$ for every $u\in\supp 
S_{B,C}$. There are of course many other directions in which the 
supporting hyperplanes of $K,L$ differ, but these are all normal to a 
corner of the cube $C$ and are therefore not in $\supp S_{B,C}$.
\end{example}

\subsubsection{Dimensionality}
\label{sec:mechdim}

We now describe yet another mechanism that gives rise to extremals of the 
Alexandrov-Fenchel inequality, which arises from the fact that mixed 
volumes may vanish for dimensionality reasons (Lemma \ref{lem:dim}). 
To make this idea precise, we introduce the following definition;
recall that we are interested in extremals for given reference bodies 
$\mathcal{C}:=(C_1,\ldots,C_{n-2})$.

\begin{defn}
\label{defn:deg}
Let $(M,N)$ be a pair of convex bodies, and let $f:S^{n-1}\to\mathbb{R}$.
\begin{enumerate}[a.]
\itemsep\abovedisplayskip
\item $(M,N)$ is called a \emph{$\mathcal{C}$-degenerate pair} if 
$M$ is not a translate of $N$,
\begin{align}
\label{eq:degp1}
	&\V_n(M,N,C_1,\ldots,C_{n-2})=0, \\
\label{eq:degp2}
\mbox{and}\quad
	&\V_n(M,B,C_1,\ldots,C_{n-2})=\V_n(N,B,C_1,\ldots,C_{n-2}).
\end{align}
\item $f$ is a \emph{$\mathcal{C}$-degenerate function} if
$f=h_M-h_N$ for some $\mathcal{C}$-degenerate pair $(M,N)$.
\end{enumerate}
\end{defn}

By Lemma \ref{lem:dim}, condition \eqref{eq:degp1} is of a purely 
geometric nature: it is characterized by the dimensions of the 
relevant bodies. Condition \eqref{eq:degp2} should be viewed 
merely as a normalization; for any pair $(M,N)$ satisfying the first 
condition, the second condition can always be made to hold by rescaling 
$M$ or $N$. We assume $M$ is not a translate of $N$ to exclude
the trivial case that $f=h_M-h_N$ is a linear function.

\begin{lem}
\label{lem:degeq}
$S_{f,C_1,\ldots,C_{n-2}}=0$ whenever $f$ is a $\mathcal{C}$-degenerate
function.
\end{lem}

\begin{proof}
Let $(M,N)$ be a $\mathcal{C}$-degenerate pair. The main observation is 
that we obtain equality in Theorem \ref{thm:af} for $K=B+M$ and 
$L=B+N$. Indeed, as
\begin{align*}
	\V_n(K,K,C_1,\ldots,C_{n-2}) &=
	\V_n(K,L,C_1,\ldots,C_{n-2}) +
	\V_n(M,M,C_1,\ldots,C_{n-2}),\\
	\V_n(L,L,C_1,\ldots,C_{n-2}) &=
	\V_n(K,L,C_1,\ldots,C_{n-2}) +
	\V_n(N,N,C_1,\ldots,C_{n-2})
\end{align*}
by \eqref{eq:degp1} and \eqref{eq:degp2}, we obtain
$$
	\V_n(K,L,C_1,\ldots,C_{n-2})^2 \le
	\V_n(K,K,C_1,\ldots,C_{n-2})\,
	\V_n(L,L,C_1,\ldots,C_{n-2}).
$$
As the reverse inequality
holds by Theorem \ref{thm:af}, we must in fact have equality.

Now note that if $\V_n(B,B,C_1,\ldots,C_{n-2})>0$, then
$S_{h_K-ah_L,C_1,\ldots,C_{n-2}}=0$ for some $a$ by
Lemma \ref{lem:simpleeq}. Integrating 
against $h_B$ and applying \eqref{eq:mixvolarea} and \eqref{eq:degp2} 
yields $a=1$. Thus $S_{f,C_1,\ldots,C_{n-2}}=0$ for 
$f=h_K-h_L=h_M-h_N$.

If $\V_n(B,B,C_1,\ldots,C_{n-2})=0$, however, then we have
$S_{B,C_1,\ldots,C_{n-2}}=0$ by \eqref{eq:mixvolarea} as $h_B=1$
on $S^{n-1}$. Thus in this case $S_{f,C_1,\ldots,C_{n-2}}=0$
for any $f$ by Lemma \ref{lem:maxsupp}.
\end{proof}

The geometric phenomena captured by Lemmas \ref{lem:suppeq} and 
\ref{lem:degeq} are quite different: the former captures the 
facial structure of the bodies in $\mathcal{C}$, while the latter 
captures the dimensions of the bodies. Let us illustrate 
the distinction in a concrete example.

\begin{example}
\label{ex:deg}
Let $C_1=[0,1]^4$ be a cube in $\mathbb{R}^4$, and let
$C_2=[0,e_1]+[0,e_2]$ be a two-dimensional square in the plane spanned by 
the first two coordinate directions $e_1,e_2$. Let
$M=[0,e_1]$ and $N=[0,e_2]$ be segments in the same plane.

We claim that $(M,N)$ is a degenerate pair. Indeed, as $\dim(M+N+C_2)=2$, 
Lemma \ref{lem:dim} verifies \eqref{eq:degp1}. On the other hand, it is 
clear that \eqref{eq:degp2} must hold, as this example is symmetric 
under exchanging the $e_1$ and $e_2$ directions.
This gives rise, for example, to the following equality case  
of the Alexandrov-Fenchel inequality: if we choose $K=C_1+M$ and 
$L=C_1+N$, then
Lemmas \ref{lem:degeq} and \ref{lem:simpleeq} yield
$$
        \V_4(K,L,C_1,C_2)^2 = \V_4(K,K,C_1,C_2)\,\V_4(L,L,C_1,C_2).
$$

We now aim to show that the present example cannot be explained by a 
combination of Lemmas \ref{lem:lineq} and \ref{lem:suppeq}, confirming that 
Lemma \ref{lem:degeq} captures a genuinely distinct phenomenon. That is, 
we aim to show that $f=h_M-h_N$ does not coincide with a linear 
function on the support of $S_{B,C_1,C_2}$. To this end, note that
$$
	f(u) = \max(u_1,0) - \max(u_2,0).
$$
On the other hand, for any $u\in S^3\cap\mathrm{span}\{e_1,e_3\}$ we have 
$\dim F(C_1,u)\ge 2$ and $\dim F(C_2,u)\ge 1$, so that 
$S^3\cap\mathrm{span}\{e_1,e_3\}\subset \supp S_{B,C_1,C_2}$ by Lemma 
\ref{lem:supp}. Thus $f$ cannot coincide with any linear 
function on $\supp S_{B,C_1,C_2}$, as the restriction of $f$ to the unit 
circle in $\mathrm{span}\{e_1,e_3\}$ is not a smooth function.
\end{example}

\subsection{Main result}
\label{sec:main}

In the previous section, we described three distinct mechanisms for 
equality in the Alexandrov-Fenchel inequality in Lemmas \ref{lem:lineq}, 
\ref{lem:suppeq}, and \ref{lem:degeq}. However, these three mechanisms may 
all appear simultaneously by linearity: if $S_{f,C_1,\ldots,C_{n-2}}=0$ 
and $S_{g,C_1,\ldots,C_{n-2}}=0$, then $S_{f+g,C_1,\ldots,C_{n-2}}=0$ as 
well. Thus any linear combination of the functions that appear in Lemmas 
\ref{lem:lineq}, \ref{lem:suppeq}, and \ref{lem:degeq} gives rise to an 
extremal case of the Alexandrov-Fenchel inequality.

As no other mechanism for equality is known, one may conjecture that these 
are the \emph{only} extremal cases of the Alexandrov-Fenchel inequality. 
The main result of this paper is a complete proof of this conjecture in 
the combinatorial setting. In geometric terms, we prove the following. 
(Recall that $\mathcal{P}$-degenerate pairs and $(B,\mathcal{P})$-extreme 
directions are defined in Definition \ref{defn:deg} and 
Lemma~\ref{lem:supp}.)

\begin{thm}
\label{thm:main}
Let $\mathcal{P}:=(P_1,\ldots,P_{n-2})$ be polytopes in 
$\mathbb{R}^n$, and let $K,L$ be convex bodies such that
$\V_n(K,L,P_1,\ldots,P_{n-2})>0$.\footnote{%
	As was noted in Remark \ref{rem:triv}, the trivial extremals
	$\V_n(K,L,P_1,\ldots,P_{n-2})=0$ are already fully characterized 
	geometrically by Lemma \ref{lem:dim}, so we do not consider them
	further.}
Then 
$$
	\V_n(K,L,P_1,\ldots,P_{n-2})^2=
	\V_n(K,K,P_1,\ldots,P_{n-2})\,
	\V_n(L,L,P_1,\ldots,P_{n-2})
$$
if and only if there exist $a>0$, $v\in\mathbb{R}^n$, and 
a number $0\le m<\infty$
of $\mathcal{P}$-degenerate pairs $(M_1,N_1),\ldots,(M_m,N_m)$, so that
$K+N_1+\cdots+N_m$ and $aL+v+M_1+\cdots+M_m$ have the same supporting 
hyperplanes in all $(B,\mathcal{P})$-extreme normal directions.
\end{thm}

The \emph{if} direction of Theorem \ref{thm:main} follows from Lemmas 
\ref{lem:simpleeq}, \ref{lem:lineq}, \ref{lem:suppeq}, and 
\ref{lem:degeq}, so it is the \emph{only if} part that requires proof. 
Some key ideas in the proof are described in section \ref{sec:overview}; 
the proof itself is contained in sections 
\ref{sec:matrix}--\ref{sec:proofmain}.

Schneider has conjectured \cite{Sch85} that equality in the 
Alexandrov-Fenchel inequality holds if and only if $K$ and $aL+v$ have the 
same supporting hyperplanes in all $(B,\mathcal{P})$-extreme normal 
directions. That this is not always the case was illustrated in Example 
\ref{ex:deg} (the existence of counterexamples was first noted in 
\cite{Ew88}). Nonetheless, no counterexample has been found to Schneider's 
conjecture in the case where all bodies in $\mathcal{P}$ are 
full-dimensional. This suggests that in the full-dimensional situation, 
degenerate pairs may not exist. Not only does this turn out to be
the case, but in fact a much weaker condition suffices.

\begin{defn}
\label{defn:supercrit}
A collection of convex bodies $\mathcal{C}=(C_1,\ldots,C_{n-2})$ is 
\emph{supercritical} if $\dim(C_{i_1}+\cdots+C_{i_k})\ge k+2$
for all $k\in[n-2]$, $1\le i_1<\cdots<i_k\le n-2$.
\end{defn}

\begin{lem}
\label{lem:supercrit}
If $\mathcal{C}$ is supercritical, $\mathcal{C}$-degenerate 
functions do not exist.
\end{lem}

\begin{proof}
Suppose $(M,N)$ is a $\mathcal{C}$-degenerate pair. By Lemma \ref{lem:dim} 
and the supercriticality assumption, \eqref{eq:degp1} implies that 
$\dim(M)=0$, $\dim(N)=0$, or $\dim(M+N)\le 1$.

Assume first that $\dim(M)=0$. Then $\V_n(N,B,C_1,\ldots,C_{n-2})=0$ by 
\eqref{eq:degp2}. But then by Lemma \ref{lem:dim} and the 
supercriticality assumption, $\dim(N)=0$ as well.

Thus there are two possibilities: $\dim(M)=\dim(N)=0$, or 
$\dim(M)=\dim(N)=\dim(N+M)=1$. In the first case $M$ and $N$ are 
singletons, while in the second case $M$ and $N$ are segments with 
parallel directions. Moreover, in the latter case 
$\V_n(N,B,C_1,\ldots,C_{n-2})>0$ by Lemma \ref{lem:dim} and the 
supercriticality assumption, so \eqref{eq:degp2} implies that $M$ and $N$ 
have the same length. Thus in either case $M$ and $N$ are translates of 
one another, which violates the definition of a degenerate pair.
\end{proof}

In other words, Lemma \ref{lem:supercrit} yields:

\begin{cor}
\label{cor:schneider}
Let $\mathcal{P}:=(P_1,\ldots,P_{n-2})$ be a supercritical collection of 
polytopes in $\mathbb{R}^n$, and let $K,L$ be convex bodies such that
$\V_n(K,L,P_1,\ldots,P_{n-2})>0$.
Then
$$
	\V_n(K,L,P_1,\ldots,P_{n-2})^2=
	\V_n(K,K,P_1,\ldots,P_{n-2})\,
	\V_n(L,L,P_1,\ldots,P_{n-2})
$$
if and only if there exist $a>0$ and $v\in\mathbb{R}^n$ so that
$K$ and $aL+v$ have the same supporting 
hyperplanes in all $(B,\mathcal{P})$-extreme normal directions.
\end{cor}

Corollary \ref{cor:schneider} highlights that even though Theorem 
\ref{thm:main} provides a complete characterization of the extremals of 
the Alexandrov-Fenchel inequality for arbitrary polytopes $\mathcal{P}$, 
its formulation leaves key questions open: it does not explain how many 
degenerate pairs can appear, what they look like, or whether the 
decomposition into degenerate pairs is unique. A complete understanding of 
these questions will emerge from the proof of Theorem \ref{thm:main}. As 
the requisite notions will only be introduced as we progress 
through the proof, we postpone formulating the definitive form of our main 
result until section \ref{sec:unique}.

While we have presented Corollary \ref{cor:schneider} as a special case of 
Theorem \ref{thm:main}, the supercritical case will prove to be 
fundamental to the proof. We will first give a self-contained proof of 
Corollary \ref{cor:schneider} in sections 
\ref{sec:matrix}--\ref{sec:supercrit}, and then characterize the 
degenerate equality cases in sections \ref{sec:panov}--\ref{sec:proofmain} 
by a separate argument that requires the introduction of additional 
techniques. In particular, the proof of Corollary \ref{cor:schneider} may 
be read independently from the rest of the paper.

\subsection{Prior work}

Let us briefly review what was known prior to this paper. Three cases of 
Corollary \ref{cor:schneider} were previously verified: when $\mathcal{P}$ 
consists of strongly isomorphic simple polytopes \cite[Theorem 
7.6.21]{Sch14}, when $P_1=\cdots=P_{n-2}$ \cite{Sch94b,SvH19} (in this 
case $P_i$ need not be simple), or when $\mathcal{P}$ consists 
of full-dimensional zonotopes and $K,L$ are symmetric \cite{Sch88}. All 
these results make crucial use of the special features that appear in 
these settings. In addition, one very special example of a degenerate 
equality case was previously known, when all the bodies $\mathcal{P}$ lie 
in a hyperplane \cite{ET94,Sch94}. This example sheds little light on more 
general cases, however, as it is essentially amenable to explicit 
computation, cf.\ \cite[\S 8]{SvH19}.

The characterization of lower-dimensional extremals in terms of degenerate 
pairs was conjectured by the authors during initial work on this paper. We 
subsequently realized, however, that an analogous phenomenon appears in 
work of Panov \cite{Pan85} on Alexandrov's mixed discriminant inequality, 
which may be viewed as an analogue of the Alexandrov-Fenchel inequality in 
linear algebra. Despite tantalizing similarities between these 
inequalities, the main feature the Alexandrov-Fenchel inequality does not 
arise here: dimensionality is the only extremal mechanism in the mixed 
discriminant inequality, while the central difficulty in the analysis of 
the Alexandrov-Fenchel inequality stems from degeneration of the support 
of mixed area measures. While most of our analysis has little in common 
with \cite{Pan85}, we will use a basic lemma of \cite{Pan85} to organize 
the collection of degenerate pairs (Lemma~\ref{lem:panov}).

\section{Preliminaries}
\label{sec:prelim}

The aim of this section is to recall some general background from convex 
geometry that will be needed in the remainder of the paper.

The following conventions will be in force throughout the paper.
We always denote by $B$ the Euclidean unit ball 
in $\mathbb{R}^n$.
For any collection $\mathcal{C}:=(C_1,\ldots,C_{n-2})$ of convex bodies in 
$\mathbb{R}^n$, we will often use the abbreviated notation
$$
	\V_n(K,L,\mathcal{C}) :=
	\V_n(K,L,C_1,\ldots,C_{n-2}),\qquad
	S_{L,\mathcal{C}}:=S_{L,C_1,\ldots,C_{n-2}}.
$$
For $I\subseteq[n-2]$, we denote by $\mathcal{C}_I:=(C_i)_{i\in 
I}$ and by $\mathcal{C}_{\backslash I}:=(C_i)_{i\in [n-2]\backslash I}$.

We will also encounter mixed volumes of convex bodies $C_1,\ldots,C_m$
that lie in a subspace $E\subset\mathbb{R}^n$ with $\dim(E)=m$. Such 
mixed volumes will be denoted as $\V_E(C_1,\ldots,C_m)$, or as
$\V_m(C_1,\ldots,C_m)$ when the subspace is clear from context.

\subsection{Mixed volumes and mixed area measures}

Mixed volumes and mixed area measures were introduced in section 
\ref{sec:basic}. For future reference, we begin by spelling out their 
basic properties more carefully.

Mixed volumes are defined by \eqref{eq:volpoly}. They satisfy the 
following \cite[\S 5.1]{Sch14}. (Here and in the sequel, we use the 
notation $\llbracket A\rrbracket := \sqrt{\det A^*A}$ for a linear map
$A$.)

\begin{lem}
\label{lem:mvprop}
Let $C,C',C_1,\ldots,C_n$ be convex bodies in $\mathbb{R}^n$.
\begin{enumerate}[a.]
\item $\V_n(C,\ldots,C) = \Vol_n(C)$.
\item $\V_n(C_1,\ldots,C_n)$ is symmetric and multilinear in its 
arguments.
\item $\V_n(C_1,\ldots,C_n)\ge 0$.
\item $\V_n(C,C_2,\ldots,C_n)\ge \V_n(C',C_2,\ldots,C_n)$ if
$C\supseteq C'$.
\item $\V_n(C_1,\ldots,C_n)$ is invariant under translation $C_i\leftarrow
C_i+v_i$ for $v_i\in\mathbb{R}^n$.
\item $\V_n(AC_1,\ldots,AC_n)=\llbracket A\rrbracket\,
\V_n(C_1,\ldots,C_n)$ for any linear map $A:\mathbb{R}^n\to\mathbb{R}^n$.
\end{enumerate}
\end{lem}

The identity \eqref{eq:mixvolarea} may be viewed as the definition of 
mixed area measures. The following basic properties are analogous to those 
of mixed volumes \cite[\S 5.1]{Sch14}.

\begin{lem}
\label{lem:maprop}
Let $C,C_1,\ldots,C_{n-1}$ be convex bodies in $\mathbb{R}^n$.
\begin{enumerate}[a.]
\item $S_{C_1,\ldots,C_{n-1}}$
is symmetric and multilinear in its arguments.
\item $S_{C_1,\ldots,C_{n-1}}\ge 0$.
\item $S_{C_1,\ldots,C_{n-1}}$ is invariant under translation $C_i
\leftarrow C_i+v_i$.
\item $\int \langle v,x\rangle\,S_{C_1,\ldots,C_{n-1}}(dx)=0$
for all $v\in\mathbb{R}^n$.
\end{enumerate}
\end{lem}

We now recall the basic continuity property of mixed volumes and mixed 
area measures. Recall that convex bodies $C^{(l)}$ converge to a convex 
body $C$ in the sense of Hausdorff convergence if and only if 
$\|h_{C^{(l)}}-h_C\|_\infty\to 0$, cf.\ \cite[Lemma 1.8.14]{Sch14}. 
Then we have the following result \cite[pp.\ 280--281]{Sch14}.

\begin{lem}
\label{lem:cont}
Suppose that $C_1^{(l)},\ldots,C_n^{(l)}$ are convex bodies in 
$\mathbb{R}^n$ such that $C_i^{(l)}\to C_i$ as $l\to\infty$ in the sense 
of Hausdorff convergence. Then
$$
        \V_n(C_1^{(l)},\ldots,C_n^{(l)})\to
        \V_n(C_1,\ldots,C_n),\qquad
        S_{C_1^{(l)},\ldots,C_{n-1}^{(l)}}\wto
        S_{C_1,\ldots,C_{n-1}}
$$
as $l\to\infty$, where the limit of measures is in the sense of weak 
convergence.
\end{lem}

In the case that all the convex bodies are polytopes, mixed area measures 
take a particularly simple form \cite[p.\ 279]{Sch14}.

\begin{lem}
\label{lem:mapoly}
Let $P_1,\ldots,P_{n-1}$ be polytopes in $\mathbb{R}^n$. Then 
$S_{P_1,\ldots,P_{n-1}}$ is atomic, that is, 
$\supp S_{P_1,\ldots,P_{n-1}} =
\{u\in S^{n-1}:S_{P_1,\ldots,P_{n-1}}(\{u\})>0\}$,
with
$$
	S_{P_1,\ldots,P_{n-1}}(\{u\}) =
	\V_{n-1}(F(P_1,u),\ldots,F(P_{n-1},u)).
$$
\end{lem}

\begin{rem}
In Lemma \ref{lem:mapoly} we have made a slight abuse of notation: the 
faces $F(P_i,u)$, $i=1,\ldots,n-1$ need not lie in a single 
$(n-1)$-dimensional subspace. However, by definition all these faces have 
$u$ as a normal direction, so that each face may be translated to lie in 
$u^\perp$. We implicitly define $\V_{n-1}(F(P_1,u),\ldots,F(P_{n-1},u))$ 
as the mixed volume in $u^\perp$ of the translated faces; this
convenient notation is consistent with the translation-invariance of mixed 
volumes.
\end{rem}

As the faces $F(P,u)$ play a fundamental role in what follows, let us 
briefly recall at this stage some associated notions. A \emph{facet} of a 
convex body $C$ in $\mathbb{R}^n$ is an $(n-1)$-dimensional face of $C$. 
We recall that every polytope has a finite number of facets. We also 
recall the following basic property \cite[\S 1.7]{Sch14}.

\begin{lem}
\label{lem:faces}
Let $C,C'$ be any convex bodies in $\mathbb{R}^n$ and $u,x\in\mathbb{R}^n$.
Then
$$
	h_{F(C,u)}(x)=\nabla_x h_C(u),
$$
where $\nabla_x$ denotes the directional derivative in direction $x$.
In particular,
$$
	F(C+C',u)=F(C,u)+F(C',u).
$$
\end{lem}

Consequently, we may observe that the mixed area measure in Lemma 
\ref{lem:mapoly} is in fact supported on a finite number of points. 
Indeed, Lemmas \ref{lem:mapoly} and \ref{lem:dim} imply that every 
$u\in\supp S_{P_1,\ldots,P_{n-1}}$ must satisfy $\dim 
F(P_1+\cdots+P_{n-1},u)\ge n-1$, that is, each such $u$ must be a facet 
normal of $P_1+\cdots+P_{n-1}$. As the Minkowski sum of polytopes is a 
polytope, $\supp S_{P_1,\ldots,P_{n-1}}$ must be finite.

Finally, the following basic property of faces will be useful. Here and in 
the sequel, we denote by $\proj_E$ the orthogonal projection onto
a subspace $E$ of $\mathbb{R}^n$.

\begin{lem}
\label{lem:faceproj}
For any convex body $C$ in $\mathbb{R}^n$, linear subspace 
$E\subseteq\mathbb{R}^n$, and $u\in\mathbb{R}^n$,
$$
        F(\proj_EC,u) = \proj_E F(C,\proj_Eu).
$$
\end{lem}

\begin{proof} 
Using Lemma \ref{lem:faces}, we can compute
\begin{align*}
        h_{F(\proj_EC,u)}(x) &=
        \nabla_x h_{\proj_EC}(u) =
        \nabla_{\proj_E x}h_C(\proj_Eu) 
	\\ &=
        h_{F(C,\proj_Eu)}(\proj_E x) =
        h_{\proj_E F(C,\proj_Eu)}(x)
\end{align*}
for every $x\in\mathbb{R}^n$, where we used $h_{\proj_E C}(u)=h_C(\proj_E u)$.
\end{proof}

\subsection{Projection formulae}

The relation between mixed volumes of convex bodies and their projections 
will play a recurring role in this paper. The following result captures 
this connection in a general setting \cite[Theorem 5.3.1]{Sch14}.

\begin{lem}
\label{lem:proj}
Let $E$ be an $m$-dimensional subspace of 
$\mathbb{R}^n$, let $C_1,\ldots,C_m$ be convex bodies in $E$,
and let $C_{m+1},\ldots,C_n$ be convex bodies in $\mathbb{R}^n$. Then
$$
	{n\choose m}
	\V_n(C_1,\ldots,C_n)
	=
	\V_E(C_1,\ldots,C_m)\,
	\V_{E^\perp}(\proj_{E^\perp}C_{m+1},\ldots,\proj_{E^\perp}C_n).
$$
\end{lem}

We will use Lemma \ref{lem:proj} in its full force many times. The special 
case $m=1$ is particularly important, however, so we highlight it 
separately.

\begin{cor}
\label{cor:segproj}
Let $C_1,\ldots,C_{n-1}$ be convex bodies in $\mathbb{R}^n$, and
let $u\in S^{n-1}$. Then
$$
	n\,\V_n([0,u],C_1,\ldots,C_{n-1}) =
	\V_{n-1}(\proj_{u^\perp}C_1,\ldots,\proj_{u^\perp}C_{n-1}).
$$
\end{cor}

When combined with Corollary \ref{cor:segproj}, the following observation 
expresses certain $n$-dimensional mixed volumes in terms of 
$(n-1)$-dimensional projections.

\begin{lem}
\label{lem:intseg}
Let $C_1,\ldots,C_{n-1}$ be convex bodies in $\mathbb{R}^n$. Then
$$
	\int_{S^{n-1}}
	\V_n([0,u],C_1,\ldots,C_{n-1}) \,\omega(du)
	=
	\kappa_{n-1}\V_n(B,C_1,\ldots,C_{n-1}),
$$
where $\omega$ denotes the Lebesgue measure on $S^{n-1}$
and $\kappa_{n-1}$ denotes the volume of the Euclidean unit ball in 
$\mathbb{R}^{n-1}$.
\end{lem}

\begin{proof}
Apply \eqref{eq:mixvolarea} and
$\int h_{[0,u]}(x)\,\omega(du) = \int 
\langle u,x\rangle_+\,\omega(du) = \kappa_{n-1}\,h_B(x)$.
\end{proof}

\subsection{Alexandrov-Fenchel inequality and equality}
\label{sec:3af}

The classical formulation of the Alexandrov-Fenchel inequality given in 
Theorem \ref{thm:af} is not the most general one: as was emphasized by 
Alexandrov \cite{Ale37,Ale96}, the convex body $K$ may be replaced by any 
difference of support functions $f$. We will often require this more 
general inequality and its equality cases. We presently make precise the 
connection between these formulations. The results of this section could 
be deduced from \cite[\S 7.4]{Sch14}, but we find it more insightful to 
give direct proofs.

We begin by spelling out three equivalent formulations of
Theorem \ref{thm:af}.

\begin{lem}
\label{lem:3af}
Let $\mathcal{C}=(C_1,\ldots,C_{n-2})$ be convex bodies in 
$\mathbb{R}^n$. The following are three equivalent formulations of the 
Alexandrov-Fenchel inequality:
\begin{enumerate}[a.]
\item For any convex bodies $K,L$,
$$
	\V_n(K,L,\mathcal{C})^2 \ge \V_n(K,K,\mathcal{C})\,
	\V_n(L,L,\mathcal{C}).
$$
\item For any difference of support functions $g$ and convex body
$L$,
$$
	\V_n(g,L,\mathcal{C})^2 \ge \V_n(g,g,\mathcal{C})\,
	\V_n(L,L,\mathcal{C}).
$$
\item 
For any difference of support functions $f$ 
and convex body $L$ with $\V_n(L,L,\mathcal{C})>0$,
$$
	\V_n(f,L,\mathcal{C})=0
	\quad\mbox{implies}\quad
	\V_n(f,f,\mathcal{C})\le 0.
$$
\end{enumerate}
Moreover, if $\V_n(L,L,\mathcal{C})>0$, then equality holds in part b
if and only if there exists $a\in\mathbb{R}$ such that equality holds in 
part c with $f=g-ah_L$.
\end{lem}

\begin{proof}
The implications $b\Rightarrow a$, $b\Rightarrow c$, and $c\Rightarrow b$ 
follow readily by choosing, respectively, 
$g=h_K$, $g=f$, and $f=g-ah_L$ with 
$a=\V_n(g,L,\mathcal{C})/\V_n(L,L,\mathcal{C})$ (we may assume 
$\V_n(L,L,\mathcal{C})>0$ in the latter case, as otherwise $b$ is 
trivial.)

To prove $a\Rightarrow b$, note first that if $g=h_K-ah_L$ for some 
$a\in\mathbb{R}$, the $ah_L$ term cancels on both sides of the inequality 
in $b$ by expanding the square, so that $a\Rightarrow b$ follows 
trivially. But if $g$ and $L$ are sufficiently smooth, then we may always 
write $g=h_K-ah_L$ for some $a>0$ and convex body $K$ \cite[Corollary 
2.2]{SvH18}; thus the implication $a\Rightarrow b$ follows under 
smoothness assumptions, and consequently in general by a standard 
approximation argument \cite[\S 3.4]{Sch14}.

Finally, suppose $\V_n(L,L,\mathcal{C})>0$.
Then it is immediate that $b$ holds with equality if and only if
$c$ holds with equality for $f=g-ah_L$ with 
$a=\V_n(g,L,\mathcal{C})/\V_n(L,L,\mathcal{C})$.
It remains to note that if $c$ holds with equality with
$f=g-ah_L$ for some $a\in\mathbb{R}$, then it follows from
$\V_n(f,L,\mathcal{C})=0$ that necessarily
$a=\V_n(g,L,\mathcal{C})/\V_n(L,L,\mathcal{C})$.
\end{proof}

In view of Lemma \ref{lem:3af}, to study the equality cases of the 
Alexandrov-Fenchel inequality it suffices to consider the formulation of 
part $c$ of Lemma \ref{lem:3af}. We presently reformulate the equality 
condition
\begin{equation}
\label{eq:eq3af}
	\V_n(f,L,\mathcal{C})=0\quad\mbox{and}\quad
	\V_n(f,f,\mathcal{C})=0
\end{equation}
using the first-order condition of optimality, following 
\cite[p.\ 80]{Ale96}.
For future reference, we consider a slightly more general situation
than arises in Lemma \ref{lem:3af}.

\begin{lem}
\label{lem:afeq}
Let $f$ be a difference of support functions, and let
$L$ and $\mathcal{C}=(C_1,\ldots,C_{n-2})$ be convex bodies in 
$\mathbb{R}^n$.
\begin{enumerate}[a.]
\itemsep\abovedisplayskip
\item Suppose $\V_n(L,L,\mathcal{C})>0$. Then \eqref{eq:eq3af} holds
if and only if $S_{f,\mathcal{C}}=0$.
\item Suppose $\V_n(L,L,\mathcal{C})=0$ and $S_{L,\mathcal{C}}\ne 0$.
Then \eqref{eq:eq3af} holds
if and only if there exists $a\in\mathbb{R}$ such that
$S_{f-ah_L,\mathcal{C}}=0$.
\end{enumerate}
\end{lem}

\begin{proof}
We first prove part $a$. If $S_{f,\mathcal{C}}=0$, then
$\int h_L dS_{f,\mathcal{C}}=\int f dS_{f,\mathcal{C}}=0$ and
\eqref{eq:mixvolarea} yields \eqref{eq:eq3af}. Conversely, suppose 
\eqref{eq:eq3af} holds, and let $g$ be any difference of support 
functions. As $\V_n(L,L,\mathcal{C})>0$, we can choose $a$ so 
that $\V_n(g-ah_L,L,\mathcal{C})=0$. Then
$$
	\varphi(\lambda) := 
	\V_n(f+\lambda[g-ah_L],f+\lambda[g-ah_L],\mathcal{C})
$$
satisfies $\varphi(\lambda)\le 0$ by Lemma 
\ref{lem:3af}($c$) and $\varphi(0)=0$ by \eqref{eq:eq3af}. Thus
$\varphi$ is a quadratic function with maximum at $0$,
so $\varphi'(0)=0$. Using $\V_n(f,L,\mathcal{C})=0$,
this yields
$$
	0 = \V_n(g,f,\mathcal{C}) =
	\frac{1}{n}\int g\,dS_{f,\mathcal{C}}.
$$
As we may choose $g$ to be any $C^2$ function by Lemma \ref{lem:c2}, we 
have $S_{f,\mathcal{C}}=0$.

We now prove part $b$. If $S_{f-ah_L,\mathcal{C}}=0$, then 
$n\V_n(f,L,\mathcal{C})=\int h_L dS_{f-ah_L,\mathcal{C}}=0$ as 
$\V_n(L,L,\mathcal{C})=0$; consequently, $n\V_n(f,f,\mathcal{C})=\int 
f dS_{f-ah_L,\mathcal{C}}=0$, proving \eqref{eq:eq3af}.
Conversely, suppose \eqref{eq:eq3af} holds. As $S_{L,\mathcal{C}}\ne 0$ we 
have $\V_n(B,L,\mathcal{C})>0$. Therefore:
\vspace{.5\abovedisplayskip}
\begin{enumerate}[$\bullet$]
\itemsep\abovedisplayskip
\item
We may choose $a\in\mathbb{R}$ so that $\V_n(f-ah_L,B,\mathcal{C})=0$.
\item
$\V_n(f-ah_L,f-ah_L,\mathcal{C})=0$ by \eqref{eq:eq3af} and 
$\V_n(L,L,\mathcal{C})=0$.
\item
$\V_n(B,B,\mathcal{C})>0$ as
$\V_n(B,L,\mathcal{C})>0$ and $L\subseteq cB$ for some $c>0$.
\end{enumerate}
\vspace{.5\abovedisplayskip}
We can now apply part $a$ with $L\leftarrow B$, 
$f\leftarrow f-ah_L$ to conclude. 
\end{proof}

For completeness, we conclude with a proof of Lemma \ref{lem:simpleeq}.

\begin{proof}[Proof of Lemma \ref{lem:simpleeq}]
Let $K,L$ and $\mathcal{C}=(C_1,\ldots,C_{n-2})$ be as in the statement of 
Lemma \ref{lem:simpleeq}. To prove $b\Rightarrow a$, it suffices to note that
integrating condition $b$ against $h_K$ and $h_L$
yields  
$\V_n(K,K,\mathcal{C})=a\V_n(K,L,\mathcal{C})=a^2\V_n(L,L,\mathcal{C})$
by \eqref{eq:mixvolarea}.
To prove $a\Rightarrow b$, 
note that the assumption $\V_n(K,L,\mathcal{C})>0$ and condition $a$ 
imply $\V_n(L,L,\mathcal{C})>0$. Thus Lemmas \ref{lem:3af} and 
\ref{lem:afeq} imply $S_{h_K-ah_L,\mathcal{C}}=0$ for some 
$a\in\mathbb{R}$. But integrating against $h_L$ yields
$\V_n(K,L,\mathcal{C})=a\V_n(L,L,\mathcal{C})$ by \eqref{eq:mixvolarea}, 
so $a>0$.
\end{proof}

\section{Overview of the proof}
\label{sec:overview}

The main result of this paper, Theorem \ref{thm:main}, is 
proved in sections \ref{sec:matrix}--\ref{sec:proofmain} below. Before 
we proceed to the details, however, we aim to give a 
high-level overview of the proof in order to help the
reader navigate the following sections.
At the most basic level, the proof proceeds by induction on the dimension 
$n$. The argument splits into two parts that require 
completely different ideas and techniques.

Throughout the proof of Theorem \ref{thm:main}, we will fix $n\ge 3$ and 
polytopes $\mathcal{P}=(P_1,\ldots,P_{n-2})$ in $\mathbb{R}^n$. Let us 
introduce at the outset a minimal dimensionality condition that will be 
assumed throughout most of this paper.

\begin{defn}
\label{defn:crit}
A collection of convex bodies $\mathcal{C}=(C_1,\ldots,C_{n-2})$ is 
\emph{critical} if $\dim(C_{i_1}+\cdots+C_{i_k})\ge k+1$
for all $k\in[n-2]$, $1\le i_1<\cdots<i_k\le n-2$.
\end{defn}

Note that if there exist $i_1<\cdots<i_k$ with 
$\dim(P_{i_1}+\cdots+P_{i_k})\le k$, then the bodies 
$(P_{i_1},\ldots,P_{i_k})$ factor on both sides of the Alexandrov-Fenchel 
inequality by Lemma~\ref{lem:proj}, and the problem reduces to a 
lower-dimensional one. For this reason, we may focus our attention on the 
case that $\mathcal{P}$ is critical, and the remaining cases will be 
easily dispensed with at the very end of the proof.

\subsection{The local Alexandrov-Fenchel inequality}
\label{sec:outlinelocalaf}

In order to perform induction on the dimension, we must understand how the 
extremals of the Alexandrov-Fenchel inequality in dimensions $n$ and $n-1$ 
are related. The purpose of the first part of the proof of Theorem 
\ref{thm:main} is to make this connection. To explain how this is done, we 
begin by discussing an apparently unrelated question.

In view of their definition \eqref{eq:mixvolarea}, it is natural to think 
of mixed area measures as local analogues of mixed volumes: they 
describe the behavior of mixed volumes in different normal directions.
The analogy is even more explicit in the polytope case, cf.\ Lemma 
\ref{lem:mapoly}. One might therefore wonder whether there exists an 
analogue of the Alexandrov-Fenchel inequality for mixed area measures. This 
question makes little sense in the formulation of Theorem \ref{thm:af}, 
of course, as one cannot square a measure. However, the question can be 
meaningfully formulated in the form of Lemma \ref{lem:3af}($c$): 
given convex bodies $L,C_1,\ldots,C_{n-3}$, is it true that
\begin{equation}
\label{eq:localaf}
	S_{f,L,C_1,\ldots,C_{n-3}} = 0 \quad 
	\stackrel{?}{\Longrightarrow}
	\quad
	S_{f,f,C_1,\ldots,C_{n-3}} \le 0
\end{equation}
for any difference of support functions $f$? We will refer to any 
statement of the form \eqref{eq:localaf} as a \emph{local 
Alexandrov-Fenchel inequality}.

Let us first explain why such an inequality would enable an 
induction argument, at least in the full-dimensional case. To this end, 
we make a simple observation.

\begin{lem}
\label{lem:weyl}
Let $\mathcal{C}=(C_1,\ldots,C_{n-2})$ be convex bodies in $\mathbb{R}^n$,
let $r\in[n-2]$, and suppose $C_r$ is full-dimensional.
If $S_{f,\mathcal{C}}=0$ and $S_{f,f,\mathcal{C}_{\backslash r}}\le 0$, 
then $S_{f,f,\mathcal{C}_{\backslash r}}=0$.
\end{lem}

\begin{proof}
By translation-invariance we may assume $0\in\intr C_r$, so 
that $h_{C_r}>0$. Now note that as $S_{f,\mathcal{C}}=0$, using 
\eqref{eq:mixvolarea} and the symmetry of mixed volumes yields
$$
	0 =
	\int f \,dS_{f,\mathcal{C}}
	=
	\int h_{C_r} \,dS_{f,f,\mathcal{C}_{\backslash r}}.
$$
The conclusion follows as
$S_{f,f,\mathcal{C}_{\backslash r}}\le 0$ and $h_{C_r}>0$.
\end{proof}

Now suppose we have equality in Theorem \ref{thm:af}, and assume for 
simplicity that $C_r$ is full-dimensional for some $r\in[n-2]$.
Then by Lemma \ref{lem:simpleeq}, we have
$$
	S_{f,\mathcal{C}}=0\quad
	\mbox{with}\quad f=h_K-ah_L
$$
for some $a>0$. If
the local Alexandrov-Fenchel inequality \eqref{eq:localaf} were to hold, 
we would obtain $S_{f,f,\mathcal{C}_{\backslash r}}\le 0$, and thus
$S_{f,f,\mathcal{C}_{\backslash r}}=0$ by Lemma \ref{lem:weyl}.
Integrating both measures against $h_{[0,u]}$ (for any $u\in S^{n-1}$) 
yields, by \eqref{eq:mixvolarea} and Corollary \ref{cor:segproj},
$$
	\V_{n-1}(\proj_{u^\perp}f,\proj_{u^\perp}\mathcal{C})=0
	\qquad\mbox{and}\qquad
	\V_{n-1}(\proj_{u^\perp}f,\proj_{u^\perp}f,
	\proj_{u^\perp}\mathcal{C}_{\backslash r})=0,
$$
where $\proj_Ef:=h_{\proj_EK}-ah_{\proj_EL}$ and 
$\proj_E\mathcal{C}:=(\proj_EC_1,\ldots,\proj_EC_{n-2})$. But the latter 
is nothing other than an equality case \eqref{eq:eq3af} of the 
Alexandrov-Fenchel inequality in $u^\perp$. Thus \emph{a local 
Alexandrov-Fenchel inequality would imply that extremality for the 
Alexandrov-Fenchel inequality in dimension $n$ is inherited by projection 
onto any $(n-1)$-dimensional subspace}, opening the door to induction.

Unfortunately, it turns out that this approach breaks down precisely when 
the Alexandrov-Fenchel inequality has nontrivial extremals. That the above 
conclusion must fail in this case is immediately evident from the 
classical fact that equality $\V_2(K,L)^2=\V_2(K,K)\,\V_2(L,L)$ holds in 
dimension $n=2$ if and only if $K,L$ are homothetic (cf.\ Remark 
\ref{rem:twodex} and the proof of Theorem \ref{thm:schneider}). Thus it 
cannot be the case that the projections of a nontrivial equality case of 
the Alexandrov-Fenchel inequality in dimension $n\ge 3$ yield equality in 
dimension $2$, as convex bodies in dimension $n\ge 3$ whose projections 
onto every hyperplane are homothetic must themselves be homothetic 
\cite{Sus32} (this is illustrated, for example, by Figure \ref{fig:cap}). 
In particular, it follows that the validity of the local 
Alexendrov-Fenchel inequality \eqref{eq:localaf} is contradicted by the 
presence of nontrivial extremals.

At first sight, the failure of \eqref{eq:localaf} appears to render the 
above approach useless for the study of the extremals. Remarkably, 
however, this turns out not to be the case. Recall that by Lemma 
\ref{lem:suppeq}, the measure $S_{f,\mathcal{C}}$ is unchanged if we 
modify $f$ outside the support of $S_{B,\mathcal{C}}$; in particular, as 
we characterize extremals only up to $S_{B,\mathcal{C}}$-a.e.\ 
equivalence, we are free to modify $f$ outside $\supp S_{B,\mathcal{C}}$ 
in the proof. On the other hand, the same property does \emph{not} hold 
for $S_{f,f,\mathcal{C}}$: this measure may change drastically if we 
modify $f$ outside the support of $S_{B,\mathcal{C}}$. One of the central 
ideas of this paper is that we can exploit the resulting degrees of 
freedom to force the validity of \eqref{eq:localaf}. More precisely, we 
will prove the following.

\begin{thm}[Local Alexandrov-Fenchel inequality]
\label{thm:localaf}
Let $\mathcal{P}=(P_1,\ldots,P_{n-2})$ be a critical collection of
polytopes in $\mathbb{R}^n$, and fix $r\in[n-2]$. Then for any difference 
of support functions $f$ so that
$S_{f,\mathcal{P}}=0$, there exists a difference of support functions 
$g$ so that $S_{g,\mathcal{P}}=0$, $S_{g,g,\mathcal{P}_{\backslash r}}\le 0$,
and $g(x)=f(x)$ for all $x\in\supp S_{B,\mathcal{P}}$.
\end{thm}

The proof of \ref{thm:localaf} is the main part of this paper in which we 
exploit the assumption that the reference bodies are polytopes (see 
section \ref{sec:discussion} for discussion). The simplification provided 
by this setting is that it enables us to reduce Theorem \ref{thm:localaf} 
to a finite-dimensional problem, which will be accomplished in sections 
\ref{sec:matrix}--\ref{sec:augment} by adapting ideas from Alexandrov's 
original proof of the Alexandrov-Fenchel inequality using strongly 
isomorphic polytopes \cite{Ale37} to the setting of arbitrary polytopes. 
It should be emphasized, however, that this reduction is merely a 
technical device: the entire difficulty of the proof lies in section 
\ref{sec:localaf}, where we prove the existence of a function $g$ with the 
requisite properties. We will ultimately reduce this problem to a system 
of linear equations, and the heart of the matter is to rule out the 
presence of degeneracies that would obstruct the existence of a solution.

\begin{rem}
\label{rem:weyl}
The simple argument in the proof of Lemma \ref{lem:weyl} is due to Weyl 
\cite{Wey17}. It is used in classical proofs of the Alexandrov-Fenchel 
inequality precisely to \emph{rule out} the existence of nontrivial 
extremals; see, e.g., \cite[p.\ 110]{BF87} or \cite[p.\ 396]{Sch14}.
It therefore appears rather surprising that such an argument provides a 
starting point for the study of nontrivial extremals. That this is in fact 
the case relies crucially on Theorem \ref{thm:localaf}, which is a central 
new idea of this paper that opens the door to the analysis of the 
extremals by induction on the dimension.

A different induction argument was exploited by Schneider \cite{Sch88} to 
investigate extremals of the Alexandrov-Fenchel inequality for zonoids, 
that is, limits of Minkowski sums of segments. In this setting, the 
relation between the extremals and their projections arises from Corollary 
\ref{cor:segproj}, but this appears as a very special property of this 
class of bodies. A notable aspect of our approach is that we are able 
to perform induction by projection in the absence of such special 
structure.
\end{rem}

\subsection{The gluing argument}

Once we have shown that extremality is preserved by projection onto 
hyperplanes, we must combine the information contained in the 
$(n-1)$-dimensional projections in order to characterize the 
$n$-dimensional extremals. This is the purpose of the second part of the 
proof of Theorem \ref{thm:main}.

At first sight, it seems evident that we may reconstruct a convex body 
from its projections, as $h_{\proj_EK}(x)=h_K(\proj_E x)$ for all $x$ by 
the definition of support functions. Thus if the function 
$\proj_{u^\perp}f$ were known for every $u$, the function $f$ would be 
uniquely determined. The situation we encounter is much more delicate, 
however, as only very limited information about the projections will 
follow from the induction hypothesis that Theorem \ref{thm:main} holds in 
dimension $n-1$.

To illustrate the difficulty, suppose for simplicity that all polytopes
in $\mathcal{P}$ are full-dimensional, and let $f$ be an equality case of 
the Alexandrov-Fenchel inequality in dimension $n$, that is, 
$S_{f,\mathcal{P}}=0$. We aim to prove the conclusion of Corollary 
\ref{cor:schneider}, that is, there exists $s\in\mathbb{R}^n$ so that 
$f(x)=\langle s,x\rangle$ for $x\in\supp S_{B,\mathcal{P}}$.  If we assume 
Corollary \ref{cor:schneider} holds in dimension $n-1$, then Theorem 
\ref{thm:localaf} and the arguments of the previous section show that 
there exists $s(u)\in\mathbb{R}^n$ such that
$$
	f(x) = \langle s(u),x\rangle \qquad
	\mbox{for all }x\in \supp 
	S_{\proj_{u^\perp}B,\proj_{u^\perp}\mathcal{P}_{\backslash r}}
	\cap \supp S_{B,\mathcal{P}}
$$
for every $u\in S^{n-1}$.
We now face two problems: the linear function $\langle 
s(u),x\rangle$ may \emph{a priori} depend on $u$; and we have only very 
limited information for any given $u$, as $\supp
S_{\proj_{u^\perp}B,\proj_{u^\perp}\mathcal{P}_{\backslash r}}
\cap \supp S_{B,\mathcal{P}}$ may only 
cover a very small part of $S^{n-1}\cap u^\perp$. We must therefore rule 
out, for example, that $f$ is piecewise linear on disjoint parts of the 
supports of the mixed area measures that arise for different $u$.

In the supercritical case (Definition \ref{defn:supercrit}), these issues 
will be resolved in section~\ref{sec:supercrit}, where we will glue 
together the linear functions $\langle s(u),x\rangle$ to form a single 
linear function $\langle s,x\rangle$. The idea behind the gluing argument 
is to show that there is sufficient overlap between the supports of the 
measures $S_{\proj_{u^\perp}B,\proj_{u^\perp}\mathcal{P}_{\backslash r}}$ 
for different $u$ so that all the vectors $s(u)$ must be consistent with a 
single vector $s$. It will turn out that the supercriticality assumption 
is preserved by the induction, so that a self-contained proof of Corollary 
\ref{cor:schneider} will already be achieved in section 
\ref{sec:supercrit}.

To complete the proof of Theorem \ref{thm:main} it remains to consider the 
critical case, that is, when $\dim(C_{i_1}+\cdots+C_{i_k})=k+1$ for some 
\emph{critical set} $i_1<\cdots<i_k$. It is in this situation that 
nontrivial degenerate functions (Definition \ref{defn:deg}) appear. The 
problem of gluing together these degenerate functions in dimension $n-1$ 
to form degenerate functions in dimension $n$ gives rise to numerous 
complications. We begin in section \ref{sec:panov} by characterizing what 
degenerate functions look like; they will turn out to be intimately 
connected to the critical sets. In section \ref{sec:propeller}, we will 
show that in the critical case, the supports of the relevant mixed area 
measures exhibit a striking phenomenon: they form geometric structures 
that we call \emph{propellers}, which are responsible for the formation of 
the degenerate extremals. We exploit these insights in section 
\ref{sec:crit} to solve the gluing problem for degenerate functions.  The 
proof of Theorem \ref{thm:main} is finally completed in section 
\ref{sec:proofmain}.

\part{The local Alexandrov-Fenchel inequality}

\section{Polytopes, graphs, and extremals}
\label{sec:matrix}

The aim of this section is to give a concrete formulation of the equality 
condition $S_{f,\mathcal{P}}=0$ in the case that 
$\mathcal{P}=(P_1,\ldots,P_{n-2})$ are polytopes. In particular, we will 
describe the underlying combinatorial structure, and introduce the basic 
objects and notations that will be used in the following sections.

\subsection{Basic constructions}
\label{sec:background}

We fix at the outset $n\ge 3$ and an arbitrary collection of polytopes 
$\mathcal{P}=(P_1,\ldots,P_{n-2})$ in $\mathbb{R}^n$. The 
notations and definitions that are introduced below will be in force
throughout sections \ref{sec:matrix}--\ref{sec:localaf}.

\subsubsection{The background polytope}

We begin by introducing a certain background polytope $P$ that will be 
fixed throughout the following constructions.

Recall that a polytope in $\mathbb{R}^n$ is called \emph{simple} if it has 
nonempty interior and each of its extreme points meets exactly $n$ facets.

\begin{lem}
\label{lem:simple}
There exists a polytope $P_0$ in $\mathbb{R}^n$ so that
$P_0 + P_1 + \cdots + P_n$ is simple.
\end{lem}

\begin{proof}
Let $R$ be any polytope in $\mathbb{R}^n$ with nonempty interior, and
define $Q:=R+P_1+\cdots+P_n$. Then by \cite[Lemma 2.4.14]{Sch14}, there 
exists a simple polytope $Q'$ so that each normal cone of an extreme point 
of $Q'$ is contained in the normal cone of an extreme point of $Q$.
As the normal cones of $Q'+Q$ are intersections of normal cones of $Q'$ 
and of $Q$ \cite[Theorem 2.2.1]{Sch14}, it follows that the normal 
cones of the extreme points of $Q'+Q$ coincide with the normal cones of 
the extreme points of $Q'$. Thus $Q'+Q$ is also simple. The proof is 
concluded by choosing $P_0=Q'+R$.
\end{proof}

In the sequel, we fix a polytope $P_0$ as in Lemma \ref{lem:simple},
and define
$$
	P:= P_0+P_1+\cdots+P_{n-2}.
$$
We will use $P$ to construct a certain graph structure, on which the 
various objects that will be encountered in the sequel are defined.

\begin{rem}
In this section we will only use the fact that $P$ is full-dimensional. 
The reason for choosing $P$ to be simple will become apparent in section 
\ref{sec:augment}.
\end{rem}

\subsubsection{The background graph}

Let $P^1,\ldots,P^N$ be the facets of $P$. We will frequently identify
a facet $P^i$ by its index $i\in [N]$. For each $i\in[N]$, 
we denote by $u_i\in S^{n-1}$ the outer unit normal vector of facet $P^i$.

Two facets $i,j\in[N]$ of $P$ are said to be \emph{neighboring} if they 
intersect in an $(n-2)$-dimensional face of $P$. We denote the set of such 
pairs as
$$
	E_P := \{ (i,j)\in [N]^2 : \dim(P^i\cap P^j)=n-2\}.
$$
For any $i\in[N]$, we will denote by
$$
	E_P^i := \{j\in [N]:(i,j)\in E_P\}
$$
the set of facets that are neighbors of facet $i$. One should view 
$([N],E_P)$ as a graph whose vertices are facets of $P$ and whose edges 
are neighboring facets.

As $P$ is full-dimensional, the angle $\theta_{ij}$ between the vectors 
$u_i$ and $u_j$ must satisfy $0<\theta_{ij}<\pi$ for any $(i,j)\in E_P$. 
Thus there is a unique shortest geodesic in the sphere between $u_i$ and 
$u_j$, which we denote as $e_{ij}\subset S^{n-1}$; note that the length of 
$e_{ij}$ is precisely $\theta_{ij}$. Geometrically, $e_{ij}$ is precisely 
the set of outer unit normal vectors of the $(n-2)$-dimensional face 
$P^i\cap P^j$ of $P$.

We further define for 
each $(i,j)\in E_P$ a vector $v_{ij}\in S^{n-1}$ such that
$v_{ij}\perp u_i$ by
$$
	u_j =: u_i \cos\theta_{ij} + v_{ij}\sin\theta_{ij}.
$$
Then $v_{ij}$ is the unit tangent vector to $e_{ij}$ at $u_i$, 
pointing toward $u_j$. Geometrically, if we view $P^i$ as an 
$(n-1)$-dimensional convex body in $\aff P^i$, then $v_{ij}$ is precisely
the outer unit normal vector of its facet $P^i\cap P^j = F(P^i,v_{ij})$.

The above definitions are illustrated in Figure
\ref{fig:graphP}. As is evident from the figure, one may naturally
view these definitions as a geometric realization of the 
combinatorial graph $([N],E_P)$, whose vertices are the vectors 
$\{u_i\}_{i\in[N]}$ and whose edges are the geodesics 
$\{e_{ij}\}_{(i,j)\in E_P}$. We will often implicitly identify these
viewpoints: we refer to both $i\in[N]$ 
and the associated vector $u_i$ as a vertex, and to $(i,j)\in E_P$ 
and the associated geodesic $e_{ij}$ as an edge, of the graph defined by 
$P$.
\begin{figure}
\centering
\begin{tikzpicture}[scale=.8]

\shade[ball color = blue, opacity = 0.15] (1,0) circle [radius=2];

\draw[thick, densely dashed] (3,0) arc [start angle = 0, end angle = 180, 
x radius = 2, y radius = .5];

\draw[thick] (3,0) arc [start angle = 0, end angle = -180, x radius = 2, 
y radius = .5];

\draw[thick] (1,2) [rotate=90] arc [start angle = 0, end angle = 
-180, x radius = 2, y radius = -1];

\draw[thick, densely dashed] (1,2) [rotate=90] arc [start angle 
= 0, end angle = 
-180, x radius = 2, y radius = 1];

\draw[thick] (1,2) [rotate=90] arc [start angle = 0, end angle = 
-180, x radius = 2, y radius = 1.75];

\draw[thick, densely dashed] (1,2) [rotate=90] arc [start angle 
= 0, end angle = 
-180, x radius = 2, y radius = -1.75];

\draw[fill=black] (1,2) circle [radius=.07];
\draw[fill=black] (1,-2) circle [radius=.07];
\draw[fill=black] (2.74,-.24) circle [radius=.07] node[below right] 
{$u_i$};
\draw[fill=black] (-0.74,.24) circle [radius=.07];
\draw[fill=black] (.03,-.45) circle [radius=.07];
\draw[fill=black] (1.97,.45) circle [radius=.07];

\draw (-.25,-.44) node[below] {$u_j$};

\draw (1.4,-.23) node {$e_{ij}$};

\draw[thick,->] (2.74,-.25) to (2,-.6);

\draw (3.7,-1.5) node {$v_{ij}$};
\draw[->] (3.4,-1.2) to [out=130,in=-45] (2.3,-.6);



\begin{scope}[scale=.9,xshift=-1cm]

\fill[blue!15] (-5.95,.55) -- (-4.03,1.55) -- 
(-4.03,-1.25) -- (-5.95,-2.25) -- (-5.95,.55);

\fill[blue!25] (-5.95,.55) -- (-8.47,.89) -- 
(-8.47,-1.91) -- (-5.95,-2.25) -- (-5.95,.55);

\fill[blue!5] (-8.47,.89) -- (-6.53,1.79) -- 
(-4.03,1.55) -- (-5.95,.55) -- (-8.47,.89);

\draw[thick] (-5.95,.55) -- (-4.03,1.55) -- (-4.03,-1.25) --
(-5.95,-2.25) -- (-5.95,.55);

\draw[thick] (-5.95,.55) -- (-8.47,.89) -- (-8.47,-1.91) --
(-5.95,-2.25);

\draw[thick] (-8.47,.89) -- (-6.53,1.79) -- (-4.03,1.55);

\draw[thick,densely dashed] (-6.53,1.79) -- (-6.53,-1.01) --
(-4.03,-1.25);

\draw[thick,densely dashed] (-6.53,-1.01) -- (-8.47,-1.91);
\end{scope}

\begin{scope}[xshift=-1cm]
\draw (-7.7,-.25) node[left] {$P=$};

\draw (-6.4,-.6) node {$P^j$};

\draw (-4.35,-.3) node {$P^i$};

\draw[->] (-3.5,-1.5) to [out=180,in=0] (-5.2,-1.2);
\draw (-2.7,-1.5) node {$P^i\cap P^j$};

\end{scope}

\end{tikzpicture}
\caption{A polytope $P$ in $\mathbb{R}^3$ and the associated geometric 
graph.\label{fig:graphP}}
\end{figure}
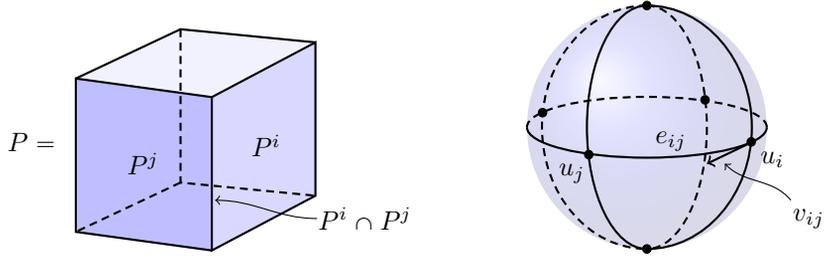

\subsubsection{Faces}

For any convex body $C$ in $\mathbb{R}^n$ and $i\in[N]$, $j\in E_P^i$,
we will denote
$$
	C^i := F(C,u_i),\qquad \quad
	C^{ij} := F(C^i,v_{ij}).
$$
We will frequently write $\mathcal{P}^i:=(P_1^i,\ldots,P_{n-2}^i)$ and
$\mathcal{P}^{ij}:=(P_1^{ij},\ldots,P_{n-2}^{ij})$, and 
analogously for other collections of bodies.

The notation $C^i$ is consistent with the notation $P^i$ for the facets of 
$P$, and we have $P^{ij}=P^i\cap P^j$. In particular, it follows from
Lemma \ref{lem:faces} that we can express the facets and $(n-2)$-faces of 
$P$ in terms of faces of the polytopes $P_r$ as
$$
	P^i = P^i_0 + \cdots + P^i_{n-2}\,\qquad\quad
	P^{ij} = P^{ij}_0 + \cdots + P^{ij}_{n-2}.
$$
In the sequel, we will apply these and similar 
consequences of the linearity of faces under Minkowski addition (Lemma 
\ref{lem:faces}) without further comment.

Note that $P^i_r$ and $P^{ij}_r$ are faces of the polytope $P_r$ by 
definition. However, in contrast to the analogous faces of $P$, it is not 
necessarily the case that $P_r^i$ is a facet and $P_r^{ij}$ is an 
$(n-2)$-face of $P_r$. Nonetheless, the following lemma shows that 
the normal cone of the face $P_r^{ij}$ of $P_r$ always contains $e_{ij}$.
In particular, it follows that $P_r^{ij}=P_r^{ji}$, which is not entirely 
obvious from the definition.

\begin{lem}
\label{lem:newfaceij}
For every $r$, $i\in[N]$, $j\in E_P^i$, and 
$u\in\relint e_{ij}$, we have
$$
	P_r^{ij} = F(P_r,u).
$$
\end{lem}

\begin{proof}
Recall that $e_{ij}$ is the set of outer unit normal vectors of
the face $P^i\cap P^j$ of $P$. But as any normal cone of a Minkowski sum
$P=P_0+\cdots+P_{n-2}$ of polytopes is contained in some normal cone of 
$P_r$ for each $r$ \cite[Theorem 2.2.1]{Sch14}, it follows that
$F(P_r,u)=F(P_r,v)\subseteq F(P_r,w)$ for all $u,v\in\relint e_{ij}$ and
$w\in e_{ij}$.

Choosing $w=u_i$, it follows that $F(P_r,u)\subseteq F(P_r,u_i)=P_r^i$.   
Thus $F(P_r,u)$ must be a face of $P_r^i$ that has $u$ as an outer normal
vector, so that $F(P_r,u)\subseteq F(P_r^i,u)$. On the other hand, as 
$F(P_r^i,u)$ is a face of $P_r$ with outer normal vector $u$, we must have
$F(P_r^i,u)\subseteq F(P_r,u)$ as well. Thus we have shown
$F(P_r,u)=F(P_r^i,u)$.

To conclude, note that as
$u\in\relint e_{ij}$, we may write $u=au_i+bu_j$ for some $a,b>0$,
so that $\proj_{u_i^\perp}u=cv_{ij}$ with $c=b\sin\theta_{ij}>0$.
It therefore follows from Lemma \ref{lem:faceproj} that
$F(P_r^i,u)=F(P_r^i,v_{ij})=P_r^{ij}$, concluding the proof.
\end{proof}

\subsection{The quantum graph}
\label{sec:qgraph}

The aim of this section is to describe the structure of the mixed area 
measure $S_{B,\mathcal{P}}$; this will be used in the next section to 
describe the extremal functions $f$ such that $S_{f,\mathcal{P}}=0$. It 
turns out that these objects are supported on a certain subgraph of 
the background graph defined by $P$ in the previous section. We will rely 
on the formulation developed in \cite{SvH19}, where the construction that 
arises here was called the ``quantum graph''.
Related representations of mixed volumes and mixed area measures
may be found in \cite{Ew88} and \cite[p.\ 437]{Sch14}.

Let us begin by describing the measure $S_{B,\mathcal{P}}$.

\begin{lem}
\label{lem:sbp}
For every continuous function $f:S^{n-1}\to\mathbb{R}$, we have
$$
	\int f\,dS_{B,\mathcal{P}} =
	\frac{1}{n-1}
	\sum_{(i,j)\in E_P:i<j}
	\V_{n-2}(P_1^{ij},\ldots,P_{n-2}^{ij})
	\int_{e_{ij}} f\,d\mathcal{H}^1,
$$
where $\mathcal{H}^1$ is the $1$-dimensional Hausdorff measure.
\end{lem}

\begin{proof}
This is an immediate consequence of \cite[Remark 5.11]{SvH19}.
\end{proof}

Lemma \ref{lem:sbp} shows that $S_{B,\mathcal{P}}$ is supported on the 
edges $\{e_{ij}\}_{(i,j)\in E_P}$ of the geometric graph defined in the 
previous section. However, not every edge appears in the support: some of 
the weights $\V_{n-2}(\mathcal{P}^{ij})$ may be zero. Thus the collection 
of reference polytopes $\mathcal{P}$ defines a subgraph of the graph 
defined by $P$. Let us define some notation to describe this subgraph. In 
the sequel, we will write
$$
	\omega_{ij} := \V_{n-2}(P_1^{ij},\ldots,P_{n-2}^{ij}).
$$
The \emph{active edges} of the graph defined by $\mathcal{P}$ are
$$
	E := \{ (i,j)\in E_P : \omega_{ij}>0\}.
$$
Similarly, the \emph{active vertices} of the graph defined by 
$\mathcal{P}$ are
$$
	V := \{ i\in[N]: \textstyle{\sum_{j\in[N]}} \omega_{ij}>0\},
$$
that is, $i\in V$ when $i$ is incident to at least one
active edge $(i,j)\in E$. Denote by 
$$
	E^i:=\{j\in V:(i,j)\in E\}
$$
the neighbors of $i\in V$ in the graph defined by $\mathcal{P}$.

We can now characterize the support of $S_{B,\mathcal{P}}$
as announced in Lemma \ref{lem:supp}.

\begin{lem}
\label{lem:suppsbp}
The following are equivalent for any $u\in S^{n-1}$:
\begin{enumerate}[a.]
\itemsep\abovedisplayskip
\item $u\in \supp S_{B,\mathcal{P}}$.
\item $u\in e_{ij}$ for some $(i,j)\in E$.
\item There are segments $I_i\subseteq F(P_i,u)$, $i\in [n-2]$
with linearly independent directions.
\item $\dim(F(P_{i_1},u)+\cdots+F(P_{i_k},u))\ge k$ for all
$k\in[n-2]$, $1\le i_1<\cdots<i_k\le n-2$.
\end{enumerate}
\end{lem}

\begin{proof}
That $a\Leftrightarrow b$ is immediate by Lemma \ref{lem:sbp}, while
$c\Rightarrow d$ is trivial.

We now prove $b\Rightarrow c$. Suppose first that $u\in\relint e_{ij}$ for 
$(i,j)\in E$. Then $\V_{n-2}(F(P_1,u),\ldots,F(P_{n-2},u))>0$ by the 
definition of $E$ and Lemma \ref{lem:newfaceij}, which implies $c$ by 
Lemma \ref{lem:dim}. It remains to consider the case where $u=u_i$ for 
some $i\in V$, so that $F(P_r,u)=P_r^i$. But by the definition of $V$, 
there exists $j$ so that $\V_{n-2}(P_1^{ij},\ldots,P_{n-2}^{ij})>0$,
so $c$ follows by Lemma \ref{lem:dim} and $P_r^{ij}\subseteq P_r^i$.

It remains to prove $d\Rightarrow a$. To this end, suppose that $d$
holds, and let $Q$ be any polytope in $\mathbb{R}^n$ that has a facet with 
outer normal direction $u$. Then by Lemma \ref{lem:dim}, we have
$\V_{n-1}(F(Q,u),F(P_1,u),\ldots,F(P_{n-2},u))>0$. Thus $a$ follows
as $u\in\supp S_{Q,\mathcal{P}}\subseteq \supp S_{B,\mathcal{P}}$ by
Lemmas \ref{lem:mapoly} and \ref{lem:maxsupp}.
\end{proof}

We now provide a useful description of the active vertices.

\begin{lem}
\label{lem:vertex}
Let $i\in[N]$. Then the following hold:
\begin{enumerate}[a.]
\itemsep\abovedisplayskip
\item $i\in V$ if and only 
if $\V_{n-1}(P^i,P_1^i,\ldots,P_{n-2}^i)>0$.
\item If $i\not\in V$, then
$\V_{n-1}(Q^i,P_1^i,\ldots,P_{n-2}^i)=0$ for every polytope $Q$.
\end{enumerate}
\end{lem}

\begin{proof}
We may assume without loss of generality (by translation) that 
$P_1^i,\ldots,P_{n-2}^i$ are convex bodies in $u_i^\perp$ and that $0\in 
\relint P^i$. As the facet 
normals of $P^i$ in $u_i^\perp$ are precisely $\{v_{ij}\}_{j\in E_P}$, it 
follows from Lemma \ref{lem:mapoly} that the mixed area measure 
$S_{P_1^i,\ldots,P_{n-2}^i}$ (computed in $u_i^\perp$) is supported on
$\{v_{ij}\}_{j\in E_P}$, and that
$$
	S_{P_1^i,\ldots,P_{n-2}^i}(\{v_{ij}\}) =
	\V_{n-2}(F(P_1^i,v_{ij}),\ldots,F(P_{n-2}^i,v_{ij})) =
	\omega_{ij}.
$$
Thus \eqref{eq:mixvolarea} implies
$$
	\V_{n-1}(Q^i,P_1^i,\ldots,P_{n-2}^i) =
	\frac{1}{n-1}
	\sum_{j\in E_P^i}
	h_{Q^i}(v_{ij})\,\omega_{ij}
$$
for any polytope $Q$, from which part $b$ follows immediately. To prove 
part $a$, recall that $P^i$ is a facet of $P$ by definition, so our 
assumptions imply that $P^i$ is a 
full-dimensional polytope in $u_i^\perp$ containing the origin in its 
interior. Thus $h_{P^i}(v_{ij})>0$ for all $j\in E_P^i$, and the 
conclusion of part $a$ follows.
\end{proof}

\subsection{The Alexandrov matrix}

The aim of this section is to give a combinatorial description of the 
equality cases of the Alexandrov-Fenchel inequality: we will show that the 
equality condition $S_{f,\mathcal{P}}=0$ can be equivalently formulated in 
terms of separate conditions on the edges and vertices of the graph 
defined by $\mathcal{P}$. Such a characterization appears in the
proof of \cite[Theorem 7.6.21]{Sch14} for the special case that
$P_1,\ldots,P_{n-2}$ are strongly isomorphic, which we do not assume here.

Define a symmetric matrix $\A\in\mathbb{R}^{N\times N}$ by
$$
	\A_{ij} :=
	1_{(i,j)\in E_P}\,\omega_{ij}\csc\theta_{ij}
	-1_{i=j}\sum_{k\in E^i_P}
	\omega_{ik}\cot\theta_{ik}.
$$
We will refer to $\A$ as the \emph{Alexandrov matrix}, as a special case 
of this matrix arises in a much more restrictive setting (of strongly 
isomorphic polytopes) in Alexandrov's original proof of the 
Alexandrov-Fenchel inequality \cite{Ale37}. We now show that 
$S_{f,\mathcal{P}}=0$ is equivalent to two conditions: $f$ is (piecewise)
linear on each edge in $E$, and the values of $f$ on the vertices define 
a vector in the kernel of $\A$.

\begin{prop}
\label{prop:qgraph}
Let $f:S^{n-1}\to\mathbb{R}$ be a difference of support functions.
Then $S_{f,\mathcal{P}}=0$ if and only if the following two conditions 
both hold:
\begin{enumerate}[1.]
\item For every $(i,j)\in E$, there exists
$t_{ij}\in\mathbb{R}^n$ such that
$f(x)=\langle t_{ij},x\rangle$ for $x\in e_{ij}$.
\item The vector $z:=(f(u_i))_{i\in[N]}\in \mathbb{R}^N$
satisfies $z\in\ker\A$.
\end{enumerate}
\end{prop}

\begin{proof}
That $S_{f,\mathcal{P}}=0$ may be equivalently stated as
$\V_n(g,f,\mathcal{P})=0$ for every difference of support functions $g$.
By \cite[Theorem 5.1]{SvH19}, this is equivalent to the statement that $f$ 
lies in the kernel of the self-adjoint operator defined in
\cite[Theorem 5.7 and Remark 5.11]{SvH19}, which is characterized by the 
following two conditions:
\begin{enumerate}[1.]
\itemsep\abovedisplayskip
\item $f$ is (piecewise) linear on each edge $e_{ij}$ for $(i,j)\in E$.
\item $f$ satisfies
$$
	\sum_{j\in E^i_P} \omega_{ij} \nabla_{v_{ij}}f(u_i)=0
	\quad\mbox{for every }i\in V.
$$
\end{enumerate}
It remains to show that the second condition is equivalent to
$z\in\ker\A$. To this end, let us parametrize the edge
$e_{ij}$ as
$$
	e_{ij}=\{x(\theta):0\le\theta\le\theta_{ij}\},\qquad
	x(\theta):=u_i\cos\theta+v_{ij}\sin\theta.
$$
By the first condition we can write $f(x)=\langle t,x\rangle$ on $e_{ij}$
for some vector $t$, so 
\begin{align*}
	f(x(\theta)) &=
	\langle t,u_i\rangle \cos\theta + 
	\langle t,v_{ij}\rangle \sin\theta
	\\ &=
	f(u_i) \cos\theta + 
	\frac{f(u_j) - f(u_i) \cos\theta_{ij}}{\sin\theta_{ij}}
	\,\sin\theta.
\end{align*}
Consequently
$$
	\nabla_{v_{ij}}f(u_i) =
	\frac{d}{d\theta}f(x(\theta))\bigg|_{\theta=0} =
	\frac{f(u_j) - f(u_i) \cos\theta_{ij}}{\sin\theta_{ij}}.
$$
Thus the second condition may be expressed equivalently as
$$
	0=\sum_{j\in E^i_P} \omega_{ij} \frac{f(u_j) - f(u_i) 
	\cos\theta_{ij}}{\sin\theta_{ij}} = (\A z)_i
	\quad\mbox{for all }i\in V.
$$
But $(\A z)_i=0$ always holds for $i\not\in 
V$ by the definition of $V$, so we have shown that the second condition 
above is equivalent to $z\in\ker\A$. 
\end{proof}

\begin{rem}
Instead of using the analytic theory of \cite{SvH19} as we have done here, 
one can give a more geometric proof by adapting the first part of the 
proof of \cite[Theorem 7.6.21]{Sch14} to the present setting. Conditions 
$1$ and $2$ in the proof of Proposition \ref{prop:qgraph} appear in 
\cite{Sch14} as (7.177) and (7.178), respectively.
\end{rem}

Let us emphasize that the $i$th row and column of $\A$ are zero for every 
$i\not\in V$. Thus the values $f(u_i)$ for $i\not\in V$ never actually 
appear in Proposition \ref{prop:qgraph}. This simply reflects the fact 
that $\{u_i\}_{i\not\in V}$ lie outside the support of 
$S_{B,\mathcal{P}}$, so these points play no role in the equality 
condition. Recall, however, that our ultimate aim is to prove the local 
Alexandrov-Fenchel inequality of Theorem \ref{thm:localaf}, in which 
points outside the support of $S_{B,\mathcal{P}}$ play a crucial role. We 
therefore resist the temptation to simply remove the zero rows and columns 
from the definition of $\A$ at this stage.

\section{Finite-dimensional reduction}
\label{sec:augment}

The previous section introduced a combinatorial formulation of the 
equality condition $S_{f,\mathcal{P}}=0$. In particular, Proposition 
\ref{prop:qgraph} shows that an extremal function $f$ is fully specified 
by its values $f(u_i)$ on the vertices $u_i$ of the graph defined by 
$\mathcal{P}$: its values on the rest of the support of 
$S_{B,\mathcal{P}}$ are then uniquely determined by linearity. In order to 
prove the local Alexandrov-Fenchel inequality, however, we will also need 
to reason about the measure $S_{f,f,\mathcal{P}_{\backslash r}}$, and 
there is no reason to expect that only the directions $\{u_i\}_{i\in[N]}$ 
will appear in its description.

The aim of this section is to introduce a basic geometric construction 
that will enable us to surmount this issue. This construction will 
simultaneously serve two purposes: it will enable us to reduce the local 
Alexandrov-Fenchel inequality to a finite-dimensional problem, and it will 
furnish the objects that appear in Proposition \ref{prop:qgraph} with a 
geometric interpretation that will be key to their analysis.

In this section, all assumptions and definitions of 
section \ref{sec:matrix} will be in force.

\subsection{Strongly isomorphic polytopes}

We begin by recalling the definition.

\begin{defn}
Polytopes $Q,Q'$ are said to be \emph{strongly isomorphic} if
$$
	\dim F(Q,u) = \dim F(Q',u)\quad\mbox{for all }u\in S^{n-1}.
$$
\end{defn}

The key feature of strongly isomorphic polytopes $Q,Q'$ is that they have 
identical facial structures: there is a bijection between the faces of $Q$ 
and $Q'$ such that each pair of identified faces has the same normal cone 
\cite[\S 2.4]{Sch14}. Consequently, no new faces are created when one 
takes Minkowski sums of strongly isomorphic polytopes. Let us record this 
basic fact for future reference \cite[Corollary 2.4.12]{Sch14}.

\begin{lem}
\label{lem:strongiso}
Let $Q,Q'$ be polytopes. Then all the polytopes $\lambda Q+\lambda' Q'$ 
with $\lambda,\lambda'>0$ are strongly isomorphic. If $Q,Q'$ are themselves 
strongly isomorphic, then all the polytopes $\lambda Q+\lambda' Q'$ with 
$\lambda,\lambda'\ge 0$ are strongly isomorphic. 
\end{lem}

The following simple observation will play an important role in the 
sequel.

\begin{lem}
\label{lem:strlin}
Let $Q$ be a polytope that is strongly isomorphic to $P$. Then 
for every $(i,j)\in E_P$, 
there exists $t_{ij}\in\mathbb{R}^n$ such that
$h_Q(x)=\langle t_{ij},x\rangle$ for $x\in e_{ij}$.
\end{lem}

\begin{proof}
Let $(i,j)\in E_P$. As $Q$ and $P$ are strongly isomorphic, 
$e_{ij}$ is the set of unit normal vectors to the face $Q^{ij}$ of $Q$.
Thus $Q^{ij}=F(Q,u)$ for any $u\in \relint 
e_{ij}$. If we therefore fix any $t_{ij}\in Q^{ij}$, then $h_Q(u)=\langle 
t_{ij},u\rangle$ for all $u\in \relint e_{ij}$ by \eqref{eq:facedef}, and 
the conclusion extends to the endpoints of $e_{ij}$ by continuity.
\end{proof}

The significance of Lemma \ref{lem:strlin} is immediately evident from 
Proposition \ref{prop:qgraph}: when $Q$ is strongly isomorphic to $P$, the 
function $f=h_Q-h_P$ automatically satisfies the piecewise linearity 
condition that characterizes the extremals of the Alexandrov-Fenchel 
inequality on the edges of the graph defined by $\mathcal{P}$ (note that 
as $P$ is strongly isomorphic to itself, Lemma \ref{lem:strlin} also 
applies to $Q=P$). We will shortly prove a strong converse to this 
statement: any extremal function $f$ of the Alexandrov-Fenchel inequality 
may be represented in such a form.

\subsection{Support vectors}
\label{sec:suppvec}

As is already anticipated by Proposition \ref{prop:qgraph}, we will 
frequently work with the restriction of support functions of convex bodies 
to the finite collection of directions $\{u_i\}_{i\in[N]}$. It will 
be convenient to introduce the following notation: for any 
convex body $C$ in $\mathbb{R}^n$, we define its \emph{support vector}
$\h_C\in\mathbb{R}^N$ by
$$
	(\h_C)_i := h_C(u_i),\qquad i\in[N].
$$
The following result shows that any vector $z\in\mathbb{R}^N$ can be 
expressed in terms of the support vector of a polytope that is strongly 
isomorphic to $P$. It is here that we make crucial use of the fact that
$P$ was chosen to be a simple polytope.

\begin{lem}
\label{lem:generate}
For any vector $z\in\mathbb{R}^N$, there exists a polytope $Q$ that 
is strongly isomorphic to $P$ and a scalar $a\in\mathbb{R}$ such that 
$z = \h_Q-a\h_P$.
\end{lem}

\begin{proof}
For any $y\in\mathbb{R}^N$, define 
$$
	Q_y := \bigcap_{i\in[N]} \{x\in\mathbb{R}^n:
	\langle u_i,x\rangle \le h_P(u_i)+y_i\}.
$$
As $P$ is a simple polytope, it follows from
\cite[Lemma 2.4.13]{Sch14} that there exists $\varepsilon>0$ such that 
$Q_y$ is strongly isomorphic to $P$ whenever $\|y\|_\infty\le\varepsilon$.
In particular, we then have $\h_{Q_y}=\h_P+y$ as $Q_y$ and $P$ have 
the same facet normals.
The conclusion follows by choosing
$Q:=a Q_{z/a}$ with $a:=\varepsilon^{-1}(1+\|z\|_\infty)$. 
\end{proof}

We can now explain a key implication of the above construction: it 
enables us to modify any equality case of the Alexandrov-Fenchel 
inequality outside the support of $S_{B,\mathcal{P}}$ in such a way that 
the relevant mixed area measures are supported in the finite set 
$\{u_i\}_{i\in [N]}$. It is by virtue of this procedure that we will be 
able to reduce Theorem \ref{thm:localaf} to a finite-dimensional problem.

\begin{cor}
\label{cor:augfindim}
Let $f$ be a difference of support functions so that
$S_{f,\mathcal{P}}=0$. Then there is a polytope $Q$ that is strongly 
isomorphic to $P$ and $a\in\mathbb{R}$ so that
$g=h_Q-ah_P$ satisfies
$g=f$ $S_{B,\mathcal{P}}$-a.e., $S_{g,\mathcal{P}}=0$, and
$S_{g,g,\mathcal{P}_{\backslash r}}$ is 
supported on $\{u_i\}_{i\in[N]}$ for all $r$.
\end{cor}

\begin{proof}
Choose any $z\in\mathbb{R}^N$ such that $z_i=f(u_i)$ for 
$i\in V$. Applying Lemma~\ref{lem:generate}, we find a polytope $Q$ that
is strongly isomorphic to $P$ and $a\in\mathbb{R}$ so that $g=h_Q-ah_P$ 
satisfies $g(u_i)=f(u_i)$ for all $i\in V$. Moreover, $f$ is linear on 
$e_{ij}$ for every $(i,j)\in E$ by Proposition \ref{prop:qgraph}, while $g$ 
satisfies the same property by Lemma \ref{lem:strlin}. Thus $f=g$ 
$S_{B,\mathcal{P}}$-a.e.\ by Lemma \ref{lem:suppsbp}. That 
$S_{g,\mathcal{P}}=S_{f,\mathcal{P}}=0$ now follows by 
Lemma~\ref{lem:suppeq}. Finally, as the facet normals of $Q+P$ are 
$\{u_i\}_{i\in[N]}$ by Lemma \ref{lem:strongiso}, we can conclude that 
$S_{g,g,\mathcal{P}_{\backslash r}}$ is supported in this set for any $r$ 
by Lemma \ref{lem:mapoly}. 
\end{proof}

It should be emphasized that Corollary \ref{cor:augfindim} does not in 
itself capture any aspect of the phenomenon described by Theorem 
\ref{thm:localaf}: it merely reduces the problem to a finite universe 
$\{u_i\}_{i\in[N]}$ of normal directions, but does not otherwise guarantee 
any properties of the measure $S_{g,g,\mathcal{P}_{\backslash r}}$. On the 
other hand, we have considerable freedom in the construction of $g$ in 
Corollary \ref{cor:augfindim}: we have only specified $g(u_i)$ for $i\in 
V$ in the proof, and we are therefore free to choose arbitrary values of 
$g(u_i)$ for $i\not\in V$. What we must show in the proof of the local 
Alexandrov-Fenchel inequality is that there exists a choice of the latter 
values which ensures that $S_{g,g,\mathcal{P}_{\backslash r}}\le 0$.

\subsection{The Alexandrov matrix revisited}

We now show that strongly isomorphic polytopes enable us to furnish 
the Alexandrov matrix of Proposition~\ref{prop:qgraph} with a 
geometric interpretation. To this end, it will be useful to introduce the 
following notation. For any $i\in[N]$, define a linear map 
$D_i:\mathbb{R}^N\to\mathbb{R}^{E^i_P}$ by 
$$
	(D_iz)_j := z_j \csc\theta_{ij} - z_i \cot\theta_{ij},\qquad
	j\in E_P^i,~i\in[N],~z\in\mathbb{R}^N.
$$
The significance of this definition is the following.

\begin{lem}
\label{lem:facevec}
Let $Q$ be a polytope that is strongly
isomorphic to $P$. Then
$$
	(D_i\h_Q)_j = h_{Q^i}(v_{ij})\qquad
	\mbox{for all }(i,j)\in E_P.
$$
\end{lem}

\begin{proof}
Fix $(i,j)\in E_P$. As $Q$ is strongly isomorphic to $P$, it must be the 
case that $F(Q^i,v_{ij})=Q^i\cap Q^j$.
Thus for fixed $x\in Q^i\cap Q^j$, we have
$$
	\langle x,u_i\rangle = h_Q(u_i),\qquad
	\langle x,u_j\rangle = h_Q(u_j),\qquad
	\langle x,v_{ij}\rangle = h_{Q^i}(v_{ij}).
$$
Taking the inner product with $x$ in the definition of $v_{ij}$ yields
$$
	h_Q(u_j) = h_Q(u_i)\cos\theta_{ij} +
	h_{Q^i}(v_{ij})\sin\theta_{ij},
$$
and the conclusion follows by rearranging this expression.
\end{proof}

As a consequence, we obtain the following geometric interpretation.

\begin{cor}
\label{cor:augprobrep}
Let $Q$ be strongly isomorphic to $P$ and let 
$a\in\mathbb{R}$. Denote
$$
	z := \h_Q-a\h_P,\qquad f := h_Q-ah_P,\qquad
	f^i := h_{Q^i}-ah_{P^i}.
$$
Then for every $i\in[N]$
$$
	(\A z)_i = 
	(n-1)\,\V_{n-1}(f^i,P^i_1,\ldots,P^i_{n-2}),
$$
and for any convex body $C$
$$
	\langle \h_C,\A z\rangle = 
	n(n-1)\,\V_n(C,f,P_1,\ldots,P_{n-2}).
$$
\end{cor}

\begin{proof}
It was shown in the proof of Lemma \ref{lem:vertex} that
$\omega_{ij}=S_{P_1^i,\ldots,P_{n-2}^i}(\{v_{ij}\})$. We may therefore 
rewrite the definition of $\A$ as
$$
	(\A z)_i =
	\sum_{j\in E^i_P} 
	(D_iz)_j\omega_{ij} =
	\sum_{j\in E^i_P}
	f^i(v_{ij})\,S_{P_1^i,\ldots,P_{n-2}^i}(\{v_{ij}\})
$$
using Lemma \ref{lem:facevec}.
But as $\{v_{ij}\}_{j\in E^i_P}$ are the facet normals of $P^i$ in $\aff 
P^i$, we obtain
$$
	(\A z)_i =
	\int f^i \, dS_{P_1^i,\ldots,P_{n-2}^i} =
	(n-1)\,\V_{n-1}(f^i,P_1^i,\ldots,P_{n-2}^i)
$$
by 
Lemma \ref{lem:mapoly} and \eqref{eq:mixvolarea}.
Now note that as $Q$ is strongly isomorphic to $P$, the facet normals of 
$Q+P$ are 
$\{u_i\}_{i\in[N]}$ by Lemma \ref{lem:strongiso}. Thus 
\begin{align*}
	\langle \h_C,\A z\rangle &=
	(n-1)
	\sum_{i\in[N]} h_C(u_i)\,\V_{n-1}(f^i,P_1^i,\ldots,P_{n-2}^i)
	\\ &
	=(n-1) \int h_C \, dS_{f,P_1,\ldots,P_{n-2}}
	= n(n-1)\,\V_n(C,f,P_1,\ldots,P_{n-2})
\end{align*}
by Lemma \ref{lem:mapoly} and \eqref{eq:mixvolarea}.
\end{proof}

\section{Proof of the local Alexandrov-Fenchel inequality}
\label{sec:localaf}

Now that the requisite machinery is in place, we proceed to the main part 
of the proof of Theorem \ref{thm:localaf}. Throughout this section, all 
assumptions and definitions of sections \ref{sec:matrix} and 
\ref{sec:augment} will be assumed without further comment.

Let us begin by reformulating Theorem \ref{thm:localaf} in a more 
combinatorial manner.

\begin{thm}
\label{thm:comblocalaf}
Assume that $\mathcal{P}=(P_1,\ldots,P_{n-2})$ is a critical collection
of polytopes.
Fix $r\in[n-2]$ and $z\in\ker\A$. Then there exist a polytope $Q$ that 
is strongly isomorphic to $P$ and $a\in\mathbb{R}$ such that the following 
hold:
\begin{enumerate}[1.]
\itemsep\abovedisplayskip
\item $(\h_Q-a\h_P)_i=z_i$ for every $i\in V$.
\item $\V_{n-1}(h_{Q^i}-ah_{P^i},h_{Q^i}-ah_{P^i},\mathcal{P}_{\backslash 
r}^i)\le 0$ for every $i\in[N]$.
\end{enumerate}
\end{thm}

With this result in hand, the conclusion of Theorem \ref{thm:localaf} 
follows readily:

\begin{proof}[Proof of Theorem \ref{thm:localaf}]
Fix $r\in[n-2]$ and a difference of support functions $f$ so that
$S_{f,\mathcal{P}}=0$. Then $z:=(f(u_i))_{i\in[N]}$ satisfies
$z\in\ker\A$ by Proposition \ref{prop:qgraph}. We can therefore apply 
Theorem \ref{thm:comblocalaf} to construct an associated polytope $Q$
and $a\in\mathbb{R}$. We claim that $g:=h_Q-ah_P$ satisfies the 
conclusion of Theorem \ref{thm:localaf}. 

To show this, note first that it follows exactly as in 
the proof of Corollary \ref{cor:augfindim} that
$g=f$ $S_{B,\mathcal{P}}$-a.e., that $S_{g,\mathcal{P}}=0$, and that
$S_{g,g,\mathcal{P}_{\backslash r}}$ is supported on 
$\{u_i\}_{i\in[N]}$. On the other hand, Lemma \ref{lem:mapoly} implies 
that
$$
	S_{g,g,\mathcal{P}_{\backslash r}}(\{u_i\}) =
	\V_{n-1}(h_{Q^i}-ah_{P^i},h_{Q^i}-ah_{P^i},
	\mathcal{P}_{\backslash r}^i).
$$
Thus the second property of Theorem \ref{thm:comblocalaf} implies 
$S_{g,g,\mathcal{P}_{\backslash r}}\le 0$.
\end{proof}

The rest of this section is devoted to the proof of Theorem 
\ref{thm:comblocalaf}. The proof consists of two parts. First, we will 
show that the second property of Theorem \ref{thm:comblocalaf} holds 
automatically for $i\in V$ by the Alexandrov-Fenchel inequality. We then 
show that $Q,a$ can be chosen in such a way that this property holds
also for $i\not\in V$.

\subsection{The active vertices}

The first observation of the proof of Theorem \ref{thm:comblocalaf} is 
that its second condition is automatically satisfied for the active 
vertices $i\in V$ whenever the first condition is satisfied, regardless of 
how $Q$ is chosen.

\begin{lem}
\label{lem:localafactive}
Let $Q$ be a polytope that is strongly isomorphic to $P$ and let
$a\in\mathbb{R}$. Suppose that
$\h_Q-a\h_P\in\ker\A$. Then for any $r\in[n-2]$, we have
$$
	\V_{n-1}(h_{Q^i}-ah_{P^i},h_{Q^i}-ah_{P^i},\mathcal{P}^i_{\backslash r})
	\le 0\quad\mbox{for all }i\in V.
$$
\end{lem}

\begin{proof}
By Corollary \ref{cor:augprobrep}, the assumption
$\h_Q-a\h_P\in\ker\A$ implies
$$
	\V_{n-1}(h_{Q^i}-ah_{P^i},\mathcal{P}^i)=0\quad
	\mbox{for all }i\in[N].
$$
By the Alexandrov-Fenchel inequality
$$
	0 
	\ge 
	\V_{n-1}(h_{Q^i}-ah_{P^i},h_{Q^i}-ah_{P^i},
	\mathcal{P}^i_{\backslash r}) \,
	\V_{n-1}(P_r^i,P_r^i,\mathcal{P}^i_{\backslash r}).
$$
The conclusion follows immediately for any $i$ such that
$\V_{n-1}(P_r^i,P_r^i,\mathcal{P}^i_{\backslash r})>0$.

Now suppose $i\in V$ but $\V_{n-1}(P_r^i,P_r^i,\mathcal{P}^i_{\backslash 
r})=0$. Then we can argue in a similar manner as in the proof of the 
second part of Lemma \ref{lem:afeq}. As $i\in V$, Lemma \ref{lem:vertex}
states that $\V_{n-1}(P^i,P_r^i,\mathcal{P}^i_{\backslash r})>0$.
Thus we may choose $b\in\mathbb{R}$ so that
$$
	\V_{n-1}(h_{Q^i}-ah_{P^i}-bh_{P^i_r},P^i,
	\mathcal{P}^i_{\backslash r})=0.
$$
The Alexandrov-Fenchel inequality now yields
$$
	0 \ge
	\V_{n-1}(h_{Q^i}-ah_{P^i}-bh_{P^i_r},
	h_{Q^i}-ah_{P^i}-bh_{P^i_r},
        \mathcal{P}^i_{\backslash r})\,
	\V_{n-1}(P^i,P^i,\mathcal{P}^i_{\backslash r}).
$$
But $i\in V$ implies
$\V_{n-1}(P^i,P^i,\mathcal{P}^i_{\backslash r}) \ge
\V_{n-1}(P^i,\mathcal{P}^i)>0$ by Lemma \ref{lem:vertex}. Thus
\begin{align*}
	0 &\ge
	\V_{n-1}(h_{Q^i}-ah_{P^i}-bh_{P^i_r},
	h_{Q^i}-ah_{P^i}-bh_{P^i_r},
        \mathcal{P}^i_{\backslash r}) \\
	&=\V_{n-1}(h_{Q^i}-ah_{P^i},
	h_{Q^i}-ah_{P^i},
        \mathcal{P}^i_{\backslash r}),
\end{align*}
where we used that 
$\V_{n-1}(h_{Q^i}-ah_{P^i},\mathcal{P}^i)=
\V_{n-1}(P^i_r,P^i_r,\mathcal{P}^i_{\backslash r})=0$.
\end{proof}

Now consider the setting of Theorem \ref{thm:comblocalaf} for a given
$z\in\ker\A$. As the $i$th row and column of 
$\A$ are zero for $i\not\in V$, we have $z'\in\ker\A$ 
whenever $z_i=z_i'$ for $i\in V$. To any such choice of $z'$, we can 
apply Lemma \ref{lem:generate} to obtain a polytope $Q$ that is strongly
isomorphic to $P$ 
and $a\in\mathbb{R}$ so that $z'=\h_Q-a\h_P$. Then:
\begin{enumerate}[1.]
\itemsep\abovedisplayskip
\item $(\h_Q-a\h_P)_i=z_i$ for every $i\in V$ (as $z_i=z_i'$ for $i\in 
V$).
\item $\V_{n-1}(h_{Q^i}-ah_{P^i},h_{Q^i}-ah_{P^i},\mathcal{P}_{\backslash
r}^i)\le 0$ for every $i\in V$ (by Lemma \ref{lem:localafactive}).
\end{enumerate}
Thus the only part of the proof of Theorem \ref{thm:comblocalaf} that 
remains is to ensure that the second condition holds for $i\not\in V$. On 
the other hand, in the above construction, the choice of $z_i'$ for 
$i\not\in V$ is completely arbitrary.

The present discussion provides us with a key intuition about why the 
local Alexandrov-Fenchel inequality has any hope of being true: we aim to 
satisfy $N-|V|$ nontrivial equations, but we are free to choose $N-|V|$ 
parameters. In other words, \emph{the number of degrees of freedom equals 
the number of equations we aim to satisfy}. This fact is not at all 
obvious from the formulation of Theorem \ref{thm:localaf}.

On the other hand, this idea alone cannot suffice to complete the proof: 
it is possible that the system of equations we aim to solve is degenerate, 
in which case no solution may exist. It is far from obvious, \emph{a 
priori}, why this situation cannot occur for some special choices of 
polytopes: had that been the case, there would have likely existed 
additional extremals of the Alexandrov-Fenchel inequality beyond the ones 
discussed in section \ref{sec:three}. The main difficulty in the remainder 
of the proof of Theorem \ref{thm:comblocalaf} is to rule out the existence 
of such degeneracies.

\subsection{Reduction to a linear system}

As was explained above, we now aim to choose the polytope $Q$ in such a 
way that the second condition of Theorem \ref{thm:comblocalaf} holds for 
$i\not\in V$. In essence, this requires us to find a solution to a system 
of quadratic inequalities. The manipulation of these inequalities is 
somewhat awkward, however, so we begin by introducing a simplification: we 
will reduce the problem to solving a system of linear equations, which are 
formulated in the following result.

\begin{prop}
\label{prop:comblin}
Assume that $\mathcal{P}=(P_1,\ldots,P_{n-2})$ is a critical collection
of polytopes.
Fix $r\in[n-2]$ and $z\in\mathbb{R}^N$. Then there exist a polytope $Q$ 
that is strongly isomorphic to $P$ and $a\in\mathbb{R}$ such that the 
following hold:
\begin{enumerate}[1.]
\itemsep\abovedisplayskip
\item $(\h_Q-a\h_P)_i=z_i$ for every $i\in V$.
\item $\V_{n-1}(h_{Q^i}-ah_{P^i},P^i,\mathcal{P}_{\backslash 
r}^i)=0$ for every $i\not\in V$.
\end{enumerate}
\end{prop}

Proposition \ref{prop:comblin} will be proved in the next section. Before 
we do so, let us show that it implies Theorem \ref{thm:comblocalaf}. As in 
Lemma \ref{lem:localafactive}, the transition from linear equations to 
quadratic inequalities is a consequence of the Alexandrov-Fenchel 
inequality.

\begin{proof}[Proof of Theorem \ref{thm:comblocalaf}]
Fix $r\in[n-2]$ and $z\in\ker\A$, and construct the polytope $Q$ as in 
Proposition \ref{prop:comblin}. Then the first condition of Theorem 
\ref{thm:comblocalaf} holds by construction. Moreover, as the $i$th column 
of $\A$ vanishes for $i\not\in V$, it follows that $\h_Q-a\h_P\in\ker\A$.
Thus the second condition of Theorem \ref{thm:comblocalaf} holds for $i\in 
V$ by Lemma \ref{lem:localafactive}.

Now let $i\not\in V$. Then by Proposition \ref{prop:comblin} and
the Alexandrov-Fenchel inequality
\begin{align*}
	0 &= \V_{n-1}(h_{Q^i}-ah_{P^i},P^i,\mathcal{P}_{\backslash r}^i)^2 
	\\ &\ge
	\V_{n-1}(h_{Q^i}-ah_{P^i},
	h_{Q^i}-ah_{P^i},
	\mathcal{P}_{\backslash r}^i)\,
	\V_{n-1}(P^i,P^i,
	\mathcal{P}_{\backslash r}^i).
\end{align*}
Thus the second condition of Theorem
\ref{thm:comblocalaf} holds provided
$\V_{n-1}(P^i,P^i,\mathcal{P}_{\backslash r}^i)>0$.

It remains to consider $i\in[N]$ such that
$\V_{n-1}(P^i,P^i,\mathcal{P}_{\backslash r}^i)=0$. By definition, $P^i$ 
are the facets of $P$, so $\dim P^i=n-1$. It therefore
follows from Lemma \ref{lem:dim} that
$\V_{n-1}(K^i,L^i,\mathcal{P}_{\backslash r}^i)=0$ for any convex bodies 
$K,L$. In particular, for such $i$
$$
	\V_{n-1}(h_{Q^i}-ah_{P^i},
	h_{Q^i}-ah_{P^i},
	\mathcal{P}_{\backslash r}^i) = 0.
$$
Thus the second condition of Theorem \ref{thm:comblocalaf}
is established for every $i\in[N]$.
\end{proof}

To clarify the computations in the next section, let us further express 
the linear system of Proposition \ref{prop:comblin} explicitly in a 
finite-dimensional form. To this end, we would like to represent the mixed 
volume $\V_{n-1}(h_{Q^i}-ah_{P^i},P^i,\mathcal{P}^i_{\backslash r})$ in 
terms of an Alexandrov matrix. In the present setting, however, the 
reference body $P_r$ has been replaced by $P$, so that Corollary 
\ref{cor:augprobrep} does not directly apply.

Note, however, that $P+\sum_{i\not\in r}P_i$ is strongly isomorphic to $P$ 
by Lemma \ref{lem:strongiso}. Therefore, if we replace the reference 
bodies $\mathcal{P}$ by $(P,\mathcal{P}_{\backslash r})$, then the 
background graph defined in section \ref{sec:background} remains 
unchanged, and all the subsequent constructions in sections 
\ref{sec:matrix}--\ref{sec:augment} extend \emph{verbatim} to this setting 
up to a change of notation. In particular, if we define the Alexandrov 
matrix associated to $(P,\mathcal{P}_{\backslash r})$ as
$$
	\bar\A_{ij} :=
	1_{(i,j)\in E_P}\,
	\V_{n-2}(P^{ij},\mathcal{P}^{ij}_{\backslash r})\csc\theta_{ij}
	-1_{i=j}\sum_{k\in E^i_P}
	\V_{n-2}(P^{ik},\mathcal{P}^{ik}_{\backslash r})
	\cot\theta_{ik},
$$
then Corollary \ref{cor:augprobrep} extends immediately to show that
whenever $z'=\h_Q-a\h_P$ for a polytope $Q$ that is strongly isomorphic to 
$P$ and $a\in\mathbb{R}$, we have
\begin{align*}
	&(\bar\A z')_i = 
	(n-1)\,\V_{n-1}(h_{Q^i}-ah_{P^i},P^i,\mathcal{P}^i_{\backslash r}),
	\\ &
	\langle \h_C,\bar\A z'\rangle = 
	n(n-1)\,\V_n(C,h_Q-ah_P,P,\mathcal{P}_{\backslash r})
\end{align*}
for any $i\in [N]$ and convex body $C$. 
If we can therefore show that 
the linear system 
\begin{equation}
\label{eq:localafsystem}
	\left\{
	\begin{aligned}
	z_i' &= z_i &&\mbox{for }i\in V,\\
	(\bar\A z')_i &= 0 &&\mbox{for }i\not\in V
	\end{aligned}
	\right.
\end{equation}
has a solution $z'\in\mathbb{R}^N$, the proof of Proposition
\ref{prop:comblin} would follow from Lemma \ref{lem:generate}.

\subsection{The Fredholm alternative}

We are now ready to complete the proof of Proposition \ref{prop:comblin}. 
To show that the linear system \eqref{eq:localafsystem} has a solution, we 
will verify the dual condition provided by the Fredholm alternative
$\mathop{\mathrm{ran}}\mathrm{M}=(\ker\mathrm{M}^*)^\perp$ of 
linear algebra. Surprisingly, it will turn out that the validity of this 
dual condition is itself a consequence of the equality condition of the 
Alexandrov-Fenchel inequality.

\begin{proof}[Proof of Proposition \ref{prop:comblin}]
We fix $r\in[n-2]$ and $z\in\mathbb{R}^N$ throughout the proof. Let us 
begin by rewriting the linear system \eqref{eq:localafsystem} as a single 
equation.
Let $V^c := [N]\backslash V$, and denote by $\mathrm{P}_V$ and 
$\mathrm{P}_{V^c}$ the orthogonal projections onto the subspaces of 
vectors supported on the coordinates $V$ and $V^c$, respectively. Then 
clearly \eqref{eq:localafsystem} has a solution $z'\in\mathbb{R}^N$ if and 
only if there exists $y\in\mathbb{R}^N$ such that
\begin{equation}
\label{eq:primal}
	\mathrm{P}_{V^c}\bar\A\mathrm{P}_{V^c}y = 
	-\mathrm{P}_{V^c}\bar\A\mathrm{P}_Vz
\end{equation}
(as then $z'=\mathrm{P}_{V^c}y+\mathrm{P}_Vz$ is a solution to 
\eqref{eq:localafsystem}).

To show there exists a solution to \eqref{eq:primal}, we will prove the 
following claim:
\begin{equation}
\label{eq:fred}
	\mathrm{P}_{V^c}w\in \ker\bar\A \quad
	\text{for every}\quad
	w\in\ker \mathrm{P}_{V^c}\bar\A\mathrm{P}_{V^c}.
\end{equation}
Let us first argue that this suffices to conclude the proof.
If \eqref{eq:fred} holds, then we clearly have
$\langle w,\mathrm{P}_{V^c}\bar\A\mathrm{P}_Vz\rangle =
\langle \bar\A\mathrm{P}_{V^c}w,\mathrm{P}_Vz\rangle = 0$ for every
$w\in \ker \mathrm{P}_{V^c}\bar\A\mathrm{P}_{V^c}$. The latter is 
precisely the dual condition for the existence of a solution $y$ to 
\eqref{eq:primal}. It therefore follows that there exists 
$z'\in\mathbb{R}^N$ satisfying \eqref{eq:localafsystem}, and the proof of 
Proposition \ref{prop:comblin} is concluded as explained at the end of the 
previous section.

It therefore remains to prove \eqref{eq:fred}. To this end, let us fix 
$w\in \ker \mathrm{P}_{V^c}\bar\A\mathrm{P}_{V^c}$. By Lemma 
\ref{lem:generate}, there exists a polytope $R$ that is strongly 
isomorphic to $P$ and $b\in\mathbb{R}$ such that
$\mathrm{P}_{V^c}w = \h_R-b\h_P$. We can therefore compute
$$
	\V_n(h_R-bh_P,h_R-bh_P,P,\mathcal{P}_{\backslash r}) =
	\frac{\langle \mathrm{P}_{V^c}w,\bar\A\mathrm{P}_{V^c}w\rangle}{n(n-1)} = 0.
$$
On the other hand, we have
\begin{align*}
	\V_n(h_R-bh_P,P_r,P,\mathcal{P}_{\backslash r}) &=
	\frac{1}{n}\sum_{i\in[N]}
	(\h_R-b\h_P)_i\,\V_{n-1}(P_r^i,P^i,\mathcal{P}_{\backslash r}^i)
	\\
	&=
	\frac{1}{n}\sum_{i\in V^c}
	(\h_R-b\h_P)_i\,\V_{n-1}(P^i,\mathcal{P}^i) = 0,
\end{align*}
where the first equality follows from Lemma \ref{lem:mapoly} and 
\eqref{eq:mixvolarea}, the second equality follows as
$(\h_R-b\h_P)_i=(\mathrm{P}_{V^c}w)_i=0$ for $i\in V$, and the third 
equality follows as $\V_{n-1}(P^i,\mathcal{P}^i)=0$ for $i\in V^c$ by
Lemma \ref{lem:vertex}. Finally, we have
$$
	\V_n(P_r,P_r,P,\mathcal{P}_{\backslash r})>0
$$
using that $\mathcal{P}$ is critical (Definition 
\ref{defn:crit}) and Lemma \ref{lem:dim}. Thus Lemma 
\ref{lem:afeq} yields
$$
	0 =S_{h_R-bh_P,P,\mathcal{P}_{\backslash r}}(\{u_i\}) =
	\V_{n-1}(h_{R^i}-bh_{P^i},P^i,\mathcal{P}_{\backslash r}^i) =
	\frac{(\bar\A\mathrm{P}_{V^c}w)_i}{n-1}
$$
for every $i\in[N]$, where we used Lemma \ref{lem:mapoly} in the second 
equality. In other words, we have shown that 
$\mathrm{P}_{V^c}w\in\ker\bar\A$, concluding the proof of 
\eqref{eq:fred}.
\end{proof}

\begin{rem}
Let us emphasize that the definition of the matrix $\bar\A$ depends on the 
choice of $r$, so that the polytope $Q$ and $a\in\mathbb{R}$ that are 
constructed in the proof of Proposition \ref{prop:comblin} will generally 
depend on $r$. This will not be a problem for our purposes, however, 
as we will fix $r$ when we implement the induction 
argument.
\end{rem}

\part{Gluing}

\section{The supercritical case}
\label{sec:supercrit}

The aim of this section is to complete our characterization of the 
extremals of the Alexandrov-Fenchel inequality in the supercritical case 
(Definition \ref{defn:supercrit}).

\begin{thm}
\label{thm:schneider}
Let $\mathcal{P}:=(P_1,\ldots,P_{n-2})$ be a supercritical collection of
polytopes in $\mathbb{R}^n$ {\textup(}$n\ge 2${\textup)}.
For any difference of support functions 
$f:S^{n-1}\to\mathbb{R}$, we have $S_{f,\mathcal{P}}=0$ if and only if
there exists $s\in\mathbb{R}^n$ so that
$f(x)=\langle s,x\rangle$ for all $x\in\supp S_{B,\mathcal{P}}$.
\end{thm}

Let us note that Theorem \ref{thm:schneider} is simply a reformulation of 
Corollary \ref{cor:schneider}.

\begin{proof}[Proof of Corollary \ref{cor:schneider}]
By Lemma \ref{lem:simpleeq}, the equality condition in 
Corollary \ref{cor:schneider} holds if and only if there exists $a>0$ such 
that $S_{f,\mathcal{P}}=0$ for $f=h_K-ah_L$. The conclusion now follows
immediately from Theorem \ref{thm:schneider}.
\end{proof}

\begin{rem}
\label{rem:twodex}
We fixed at the beginning of this paper (section \ref{sec:basic}) $n\ge 3$,
which has been assumed throughout without further comment.
In dimension $n=2$, the collection $\mathcal{P}$ is empty and the 
Alexandrov-Fenchel inequality reduces to Minkowski's first inequality 
\cite[Theorem 7.2.1]{Sch14} whose equality cases are elementary. The 
case $n=2$ does play a role in this paper, however, as it will be used as the 
base case for our induction arguments. For this reason, we have formulated 
Theorem \ref{thm:schneider} for $n\ge 2$. Note that the $n=2$ case is 
always supercritical by definition.
\end{rem}

Most of this section will be devoted to the proof of the induction step.
We therefore fix 
until further notice $n\ge 3$ and a supercritical collection of 
polytopes in $\mathbb{R}^n$. By translation-invariance of mixed area 
measures, the equality condition $S_{f,\mathcal{P}}=0$ is invariant under 
translation of the polytopes in $\mathcal{P}$, so there is no loss of 
generality in assuming that \emph{$P_i$ contains the origin in its 
relative interior for every $i\in[n-2]$}. Consequently, if we define for 
every $\alpha\subseteq[n-2]$ the linear space
$$
	\mathcal{L}_\alpha := \sspan\{P_i:i\in \alpha\} =
	\sspan{\textstyle \sum_{i\in\alpha}P_i},
$$
then $\dim \sum_{i\in\alpha}P_i=\dim\mathcal{L}_\alpha$
for any $\alpha\subseteq[n-2]$.
We will denote by $B_\alpha$ the Euclidean unit ball in 
$\mathcal{L}_\alpha$, and we write 
$\mathcal{L}_r := \mathcal{L}_{\{r\}}$, $B_r:=B_{\{r\}}$.

The above assumptions and notation will be assumed in the sequel without 
further comment. In particular, note that the supercriticality assumption 
may now be formulated as $\dim\mathcal{L}_\alpha\ge|\alpha|+2$ for every 
$\alpha\subseteq[n-2]$, $\alpha\ne\varnothing$. Let us also note the 
simple identity 
$\mathcal{L}_{\alpha\cup\beta}=\mathcal{L}_\alpha+\mathcal{L}_\beta$
that will be used many times.

\subsection{The induction hypothesis}

The proof of Theorem \ref{thm:schneider} proceeds by induction on $n$: in 
the induction step, we will assume the theorem has been proved in 
dimension $n-1$, and deduce its validity in dimension $n$. The aim of this 
section is to formulate the resulting induction hypothesis.
To this end, let us begin by stating a consequence of the
local Alexandrov-Fenchel inequality.

\begin{lem}
\label{lem:superpproj}
Fix $r\in[n-2]$ and a difference of support functions 
$f$ with $S_{f,\mathcal{P}}=0$. Then there exists a difference of 
support functions $g$ with the following properties:
\vspace{.5\abovedisplayskip}
\begin{enumerate}[1.]
\itemsep\abovedisplayskip
\item $g(x)=f(x)$ for all $x\in\supp S_{B,\mathcal{P}}$.
\item $\V_{n-1}(\proj_{u^\perp}g,\proj_{u^\perp}P_r,
\proj_{u^\perp}\mathcal{P}_{\backslash r})=0$
for all $u\in S^{n-1}$.
\item $\V_{n-1}(\proj_{u^\perp}g,\proj_{u^\perp}g,
\proj_{u^\perp}\mathcal{P}_{\backslash r})=0$ for all
$u\in S^{n-1}\cap\mathcal{L}_r$.
\item $\V_{n-1}(\proj_{u^\perp}P_r,\proj_{u^\perp}P_r,
\proj_{u^\perp}\mathcal{P}_{\backslash r})>0$ 
for all $u\in S^{n-1}$.
\end{enumerate}
\vspace{.5\abovedisplayskip}
Here the projections $\proj_{u^\perp}g$, 
$\proj_{u^\perp}\mathcal{P}_{\backslash r}$
are as defined in section \ref{sec:outlinelocalaf}.
\end{lem}

\begin{proof}
By Theorem \ref{thm:localaf},
there exists $g=f$ $S_{B,\mathcal{P}}$-a.e.\ such that
$S_{g,\mathcal{P}}=0$ and $S_{g,g,\mathcal{P}_{\backslash r}}\le 0$.
Let us check that each of the claimed properties holds for $g$.
The first property holds by construction. To prove the second property, 
note that 
$$
	0 = \int h_{[0,u]} \,dS_{g,\mathcal{P}} =
	\V_{n-1}(\proj_{u^\perp}g,\proj_{u^\perp}P_r,
	\proj_{u^\perp}\mathcal{P}_{\backslash r}),
$$
where we have used Corollary \ref{cor:segproj} and \eqref{eq:mixvolarea}.

The third property is analogous to Lemma \ref{lem:weyl}, but in the 
present case we cannot assume that $P_r$ is full-dimensional. We first 
note that as $S_{g,\mathcal{P}}=0$, we have
$$
	0 = \int g\,dS_{g,\mathcal{P}} =
	\int h_{P_r}\,dS_{g,g,\mathcal{P}_{\backslash r}}
$$
using \eqref{eq:mixvolarea} and the symmetry of mixed volumes. On the 
other hand, as $S_{g,g,\mathcal{P}_{\backslash r}}\le 0$ by construction,
it follows that $1_{h_{P_r}>0}dS_{g,g,\mathcal{P}_{\backslash r}}=0$. 
Now note that as we assumed $0\in\relint P_r$, there exists 
$\varepsilon>0$ so 
that $\varepsilon[0,u]\subseteq P_r$ for every
$u\in S^{n-1}\cap\mathcal{L}_r$. In particular, this implies
$\varepsilon h_{[0,u]}\le h_{P_r}$
and thus $\{x:h_{[0,u]}(x)>0\} 
\subseteq \{x:h_{P_r}(x)>0\}$
whenever $u\in S^{n-1}\cap\mathcal{L}_r$.
We can therefore conclude that
$$
	0 = \int h_{[0,u]}\,dS_{g,g,\mathcal{P}_{\backslash r}} =
	\V_{n-1}(\proj_{u^\perp}g,\proj_{u^\perp}g,
	\proj_{u^\perp}\mathcal{P}_{\backslash r})
$$
for any $u\in S^{n-1}\cap\mathcal{L}_r$
using Corollary \ref{cor:segproj} and \eqref{eq:mixvolarea}.

It remains to verify the fourth property, which is a consequence of the 
supercriticality assumption. As
$\dim (\sum_{i\in\alpha}P_i)\ge|\alpha|+2$ for all
$\alpha\ne\varnothing$, it follows readily that
$\dim(\sum_{i\in\alpha}\proj_{u^\perp}P_i) \ge |\alpha|+1$ for
$\alpha\ne\varnothing$, and thus also
$\dim(\proj_{u^\perp}P_r + \sum_{i\in\alpha}\proj_{u^\perp}P_i) \ge 
|\alpha|+1$ for all $\alpha$. The fourth property now 
follows from Lemma \ref{lem:dim}.
\end{proof}

From now on, we will fix $r\in[n-2]$ and a difference of support functions 
$f$ with $S_{f,\mathcal{P}}=0$, and construct the difference of support 
functions $g$ as in Lemma \ref{lem:superpproj}. In particular, Lemma 
\ref{lem:superpproj} ensures that the projection $\proj_{u^\perp}g$ yields 
an equality case \eqref{eq:eq3af} of the Alexandrov-Fenchel inequality in 
dimension $n-1$ for any $u\in S^{n-1}\cap\mathcal{L}_r$. If we now assume 
that the conclusion of Theorem \ref{thm:schneider} is valid in dimension 
$n-1$, this will yield an explicit characterization of $\proj_{u^\perp}g$ 
that will serve as the induction hypothesis for the proof of Theorem 
\ref{thm:schneider} in dimension $n$.

Theorem \ref{thm:schneider} is only valid, however, if the 
supercriticality assumption is satisfied. In order to implement the above 
program, we must therefore show that the supercriticality assumption on 
$\mathcal{P}$ is inherited by $\proj_{u^\perp}\mathcal{P}_{\backslash r}$. 
We will presently show that this is in fact the case for almost all 
directions $u$, which will suffice for our purposes. More precisely, let 
us define the sets
$$
	N := \bigcup_{\substack{%
	\alpha\subseteq[n-2]\backslash\{r\}: \\
	\dim\mathcal{L}_\alpha=|\alpha|+2}}
	S^{n-1}\cap\mathcal{L}_r\cap\mathcal{L}_\alpha,
	\qquad\quad
	U := (S^{n-1}\cap\mathcal{L}_r)\backslash N.
$$
Then we have the following lemma. Here and in the remainder of this paper, 
we will frequently use the following simple linear algebra fact without 
further comment: for any linear subspace $E\subseteq\mathbb{R}^n$ and 
$u\in S^{n-1}$, we have $\dim(\proj_{u^\perp}E)=\dim E$ if $u\not\in E$, 
whereas $\dim(\proj_{u^\perp}E)=\dim E-1$ if $u\in E$.

\begin{lem}
\label{lem:superumeasure1}
The following hold:
\begin{enumerate}[a.]
\itemsep\abovedisplayskip
\item $\proj_{u^\perp}\mathcal{P}_{\backslash r}$ is supercritical
for every $u\in U$.
\item $U$ has full measure with respect to the
uniform measure on $S^{n-1}\cap\mathcal{L}_r$.
\end{enumerate}
\end{lem}

\begin{proof}
To prove part $a$, consider $u\in S^{n-1}\cap\mathcal{L}_r$ such that 
$\proj_{u^\perp}\mathcal{P}_{\backslash r}$ is not supercritical. Then 
$\dim(\sum_{i\in\alpha}\proj_{u^\perp}P_i)<|\alpha|+2$ for some 
$\alpha\subseteq[n-2]\backslash\{r\}$, $\alpha\ne\varnothing$.
On the other hand, as $\mathcal{P}$ is supercritical, we have
$\dim(\sum_{i\in\alpha}P_i)\ge|\alpha|+2$. By the above linear algebra 
fact, this can only occur
if $\dim(\sum_{i\in\alpha}P_i)=|\alpha|+2$ and 
$u\in\mathcal{L}_\alpha$, so that $u\in N$. Thus if $u\in 
U$, then $\proj_{u^\perp}\mathcal{P}_{\backslash r}$ must be 
supercritical.

To prove part $b$, it suffices to show that $N$ is the intersection of
$S^{n-1}\cap\mathcal{L}_r$ with hyperplanes of codimension at least 
one. That is, for any $\alpha\subseteq[n-2]\backslash\{r\}$ such that
$\dim\mathcal{L}_\alpha=|\alpha|+2$, we claim that
$\dim(\mathcal{L}_\alpha\cap\mathcal{L}_r)<\dim\mathcal{L}_r$. Indeed, 
if this is not the case, we must have $\mathcal{L}_r\subseteq
\mathcal{L}_\alpha$. But that would imply that
$\dim\mathcal{L}_{\alpha\cup\{r\}}=\dim\mathcal{L}_\alpha=|\alpha|+2$, 
contradicting the supercriticality assumption on $\mathcal{P}$.
\end{proof}

Combining the above observations, we can now formally state the induction 
hypothesis (recall that $r,f,g$ have been fixed in the remainder of this 
section).

\begin{cor}
\label{cor:superinduct}
Suppose that Theorem \ref{thm:schneider} has been proved in dimension 
$n-1$. Then for any $u\in U$, there exists $s(u)\in u^\perp$ such that
$$
	g(x) = \langle s(u),x\rangle \quad
	\mbox{for all }x\in\supp S_{[0,u],B,\mathcal{P}_{\backslash r}}.
$$
\end{cor}

\begin{proof}
Applying Lemma \ref{lem:afeq} in $u^\perp$ and Lemma
\ref{lem:superpproj}, we obtain 
$S_{\proj_{u^\perp}g,\proj_{u^\perp}\mathcal{P}_{\backslash r}}=0$.
As $\proj_{u^\perp}\mathcal{P}_{\backslash r}$ is supercritical by Lemma 
\ref{lem:superumeasure1}, applying Theorem
\ref{thm:schneider} in $u^\perp$ yields
$$
	\proj_{u^\perp}g(x) = \langle s(u),x\rangle \quad
        \mbox{for all }x\in\supp S_{\proj_{u^\perp}B,\proj_{u^\perp}
        \mathcal{P}_{\backslash r}}.
$$
But as 
$\proj_{u^\perp}g(x)=g(\proj_{u^\perp}x)$ and as
$S_{\proj_{u^\perp}B,\proj_{u^\perp}\mathcal{P}_{\backslash r}}$ is 
supported in $u^\perp$ by definition, we may remove
$\proj_{u^\perp}$ on the left-hand side. The conclusion now follows 
as $\supp S_{\proj_{u^\perp}B,\proj_{u^\perp}
\mathcal{P}_{\backslash r}}=\supp S_{[0,u],B,\mathcal{P}_{\backslash r}}$
by Corollary \ref{cor:segproj}
(see Remark \ref{rem:lowdimma} below).
\end{proof}

\begin{rem}
\label{rem:lowdimma}
In the proof of Corollary \ref{cor:superinduct}, we encountered a mixed 
area measure of the form $S_{\proj_{u^\perp}C_1,\ldots,\proj_{u^\perp}C_{n-2}}$ 
for convex bodies $C_1,\ldots,C_{n-2}$ in $\mathbb{R}^n$. By convention, 
this notation will be taken to mean that 
$\proj_{u^\perp}C_1,\ldots,\proj_{u^\perp}C_{n-2}$ are viewed as convex 
bodies in $u^\perp$, and that the mixed area measure is computed in this 
space. Even though we 
do not specify explicitly in the notation in which space the mixed area 
measure is computed, this will always be clear from context. For example, 
note that the collection 
$\proj_{u^\perp}C_1,\ldots,\proj_{u^\perp}C_{n-2}$ consists of $n-2$ 
bodies, so its mixed area measure only makes sense in an 
$(n-1)$-dimensional space.

Projected mixed area measures may be equivalently expressed as mixed area 
measures in $\mathbb{R}^n$ by Corollary \ref{cor:segproj}. Indeed, note 
that
$$
	\int h\,dS_{[0,u],C_1,\ldots,C_{n-2}} =
	\frac{1}{n-1}\int 
	h\,dS_{\proj_{u^\perp}C_1,\ldots,\proj_{u^\perp}C_{n-2}}
$$
for 
any convex bodies $C_1,\ldots,C_{n-2}$ in $\mathbb{R}^n$
and any difference of support functions $h$
by Corollary \ref{cor:segproj} and \eqref{eq:mixvolarea}, where we used 
again that $\proj_{u^\perp}h(x)=h(\proj_{u^\perp}x)$. As we may choose $h$ 
to be any $C^2$ function by Lemma \ref{lem:c2}, it follows that
$$
	(n-1)\,
	S_{[0,u],C_1,\ldots,C_{n-2}} =
	S_{\proj_{u^\perp}C_1,\ldots,\proj_{u^\perp}C_{n-2}}.
$$
This is, of course, the direct counterpart of 
Corollary \ref{cor:segproj} for mixed area measures. Let us note,
in particular, that $\supp S_{[0,u],C_1,\ldots,C_{n-2}} \subset u^\perp$.
\end{rem}

\subsection{The gluing argument}

We now aim to show that the induction hypothesis of Corollary 
\ref{cor:superinduct} implies the conclusion of Theorem 
\ref{thm:schneider} in dimension $n$. To this end, the main issue we 
encounter is to show that $s(u)$ may be replaced by a single vector 
$s\in\mathbb{R}^n$ that is independent of $u$. That is, we must ``glue'' 
together the linear functions obtained for different $u\in U$ to obtain a 
single linear function.

As a first step, we observe that supports of the measures
$S_{[0,u],B,\mathcal{P}_{\backslash r}}$ for 
different $u\in U$ have a small but nontrivial overlap.

\begin{lem}
\label{lem:supercommon}
Let $u,v\in U$ be linearly independent. Then
$$
	\supp S_{[0,u],[0,v],\mathcal{P}_{\backslash r}}
	\subseteq
	\supp S_{[0,u],B,\mathcal{P}_{\backslash r}}
	\cap
	\supp S_{[0,v],B,\mathcal{P}_{\backslash r}},
$$
and
$$
	\sspan\supp S_{[0,u],[0,v],\mathcal{P}_{\backslash r}} = \{u,v\}^\perp.
$$
\end{lem}

\begin{proof}
The first claim is immediate by Lemma \ref{lem:maxsupp}.
To prove the second claim, note first that
$\sspan\supp S_{[0,u],[0,v],\mathcal{P}_{\backslash 
r}}\subseteq\{u,v\}^\perp$ by Remark \ref{rem:lowdimma}.
Now suppose the inclusion is strict. Then $\supp 
S_{[0,u],[0,v],\mathcal{P}_{\backslash r}}\subset w^\perp$ for 
some $w\in S^{n-1}\cap\{u,v\}^\perp$, so
$$
	0 = 
	\int \langle w,x\rangle_+ \,
	S_{[0,u],[0,v],\mathcal{P}_{\backslash r}}(dx)
	= n
	\,\V_n([0,w],[0,u],[0,v],\mathcal{P}_{\backslash r})
$$
using $h_{[0,w]}(x)=\langle w,x\rangle_+$ and
\eqref{eq:mixvolarea}. Now note that 
$$
	\dim({\textstyle\sum_{i\in\alpha}P_i})\ge|\alpha|+2
	\qquad\mbox{and}\qquad
	\dim([0,v]+{\textstyle\sum_{i\in\alpha}P_i})\ge|\alpha|+3
$$
for every $\alpha\subseteq[n-2]\backslash\{r\}$, $\alpha\ne\varnothing$ by 
the supercriticality assumption and the definition of $U$. As $u,v,w$ are 
linearly independent, it follows from Lemma 
\ref{lem:dim} that $\V_n([0,w],[0,u],[0,v],\mathcal{P}_{\backslash r})>0$,
which entails a contradiction.
\end{proof}

We can now conclude the following.

\begin{cor}
\label{cor:superauv}
Suppose the conclusion of Corollary \ref{cor:superinduct} holds. Then
there exists a function $a:U\times U\to\mathbb{R}$ such that
$s(u)-s(v)=a(u,v)u-a(v,u)v$ whenever $u,v\in U$ are linearly independent.
\end{cor}

\begin{proof}
Let $u,v\in U$ be linearly independent. 
By Corollary \ref{cor:superinduct} and
Lemma \ref{lem:supercommon}, 
$$
	\langle s(u),x\rangle = g(x) =
	\langle s(v),x\rangle\quad\mbox{for all }
	x\in \supp S_{[0,u],[0,v],\mathcal{P}_{\backslash r}}.
$$
Thus Lemma \ref{lem:supercommon} implies 
$s(u)-s(v)\perp \{u,v\}^\perp$, so that
$$
	s(u)-s(v) = a(u,v)u + b(u,v)v
$$
for some functions $a,b$. But exchanging 
the roles of $u,v$, we obtain
$$
	a(u,v)u + b(u,v)v =
	s(u)-s(v) = - (s(v)-s(u)) =
	-a(v,u)v-b(v,u)u,
$$
which implies
$b(u,v)=-a(v,u)$ as $u,v$ are linearly independent.
\end{proof}

Next, we show that the function $a(u,v)$ may be chosen to be
independent of $v$.

\begin{lem}
\label{lem:superconst}
Suppose the conclusion of Corollary \ref{cor:superinduct} holds,
and let $v,w\in U$ be linearly independent. Then there is a function 
$b:U\to\mathbb{R}$ such that the function $u\mapsto s(u)-b(u)u$ is 
constant on $U\backslash \sspan\{v,w\}$.
\end{lem}

\begin{proof}
Let the function $a$ be as in Corollary \ref{cor:superauv}. Consider first 
any linearly independent $u,v,w\in U$. Then we obtain by Corollary 
\ref{cor:superauv}
\begin{align*}
	0 &= s(u)-s(v) + s(v)-s(w) + s(w)-s(u) \\
	&= (a(u,v)-a(u,w))u+(a(v,w)-a(v,u))v + 
	(a(w,u)-a(w,v))w.
\end{align*}
Thus $a(u,v)=a(u,w)$ by linear independence of $u,v,w$.

Let us now fix any linearly independent $v,w\in U$, and let
$b(u):=a(u,v)$ for $u\in U$. As $u,v,w$ are linearly independent for any
$u\in U\backslash \sspan\{v,w\}$, we have $b(u)=a(u,v)=a(u,w)$ and
$b(w)=a(w,v)=a(w,u)$ for all such $u$. Therefore
$$
	s(u)-b(u)u =
	s(w) + s(u)-s(w) - b(u)u 
	= s(w) -b(w)w
$$
for every $u\in U\backslash \sspan\{v,w\}$ by Corollary 
\ref{cor:superauv}.
\end{proof}

Putting together the preceding arguments, we obtain the following.

\begin{lem}
\label{lem:superglue}
Suppose that Theorem \ref{thm:schneider} has been proved in dimension 
$n-1$. Then there exists $s\in\mathbb{R}^n$ such that
$$
	g(x)=\langle s,x\rangle \quad\mbox{for all }x\in\supp
	S_{B,B_r,\mathcal{P}_{\backslash r}}.
$$
\end{lem}

\begin{proof}
We begin by noting that $\dim P_r\ge 3$ by the supercriticality 
assumption. Therefore, as $U$ has full measure in 
$S^{n-1}\cap\mathcal{L}_r$ by Lemma \ref{lem:superumeasure1}, we may 
choose linearly independent $v,w\in U$. Moreover, as $\dim P_r\ge 3$ and 
$\dim\sspan\{u,v\}=2$, it follows that $U\backslash\sspan\{v,w\}$ still 
has full measure.

By Corollary \ref{cor:superinduct} and Lemma \ref{lem:superconst},
there exists a function $b:U\to\mathbb{R}$ and $s\in\mathbb{R}^n$ so that
$s(u)-b(u)u=s$ for all $u\in U\backslash\sspan\{v,w\}$. Thus
Corollary \ref{cor:superinduct} yields
$$
	g(x)=\langle s,x\rangle\quad\mbox{for all }x\in
	\supp S_{[0,u],B,\mathcal{P}_{\backslash r}}
	\mbox{ and }
	u\in U\backslash\sspan\{v,w\},
$$
where
we used that $\langle s,x\rangle = \langle s(u),x\rangle$ for 
$x\in u^\perp$.

Now note that it follows as in the proof of Lemma \ref{lem:intseg} and 
Remark \ref{rem:lowdimma} that $\int S_{[0,u],B,\mathcal{P}_{\backslash 
r}}\,\omega_r(du) = \kappa_{\dim P_r-1}\,S_{B_r,B,\mathcal{P}_{\backslash
r}}$, where $\omega_r$ denotes the uniform measure on   
$S^{n-1}\cap\mathcal{L}_r$. As $U\backslash\sspan\{v,w\}$ has full 
$\omega_r$-measure, we can compute
\begin{align*}
	0 &=
	\int_{U\backslash\sspan\{v,w\}}
	\bigg(\int |g(x)-\langle s,x\rangle|
	\,S_{[0,u],B,\mathcal{P}_{\backslash r}}(dx)\bigg)\,\omega_r(du)
	\\
	&= \kappa_{\dim P_r-1}
	\int |g(x)-\langle s,x\rangle|\,
	S_{B_r,B,\mathcal{P}_{\backslash r}}(dx).
\end{align*}
The conclusion follows by the continuity of $g(x)-\langle s,x\rangle$.
\end{proof}

We have now almost concluded the induction step in the proof of Theorem 
\ref{thm:schneider}, but there is a remaining subtlety: in Lemma 
\ref{lem:superglue} we have shown that $g(x)=\langle s,x\rangle$ for 
$x\in\supp S_{B,B_r,\mathcal{P}_{\backslash r}}$, while the conclusion of 
Theorem \ref{thm:schneider} states that this holds for $x\in\supp 
S_{B,P_r,\mathcal{P}_{\backslash r}}$. That the latter follows from the 
former is an immediate consequence of the following lower-dimensional 
analogue of Lemma \ref{lem:maxsupp}.

\begin{lem}
\label{lem:lowmaxsupp}
For any convex bodies $\mathcal{C}=(C_1,\dots,C_{n-2})$ in $\mathbb{R}^n$,
we have
$$
	\supp S_{P_r,\mathcal{C}}\subseteq\supp S_{B_r,\mathcal{C}}.
$$
\end{lem}

\begin{proof}
Let $K$ be any convex body in $\mathbb{R}^n$ such that $h_K$ is a $C^2$ 
function on $S^{n-1}$. It is shown in \cite[Lemma 5.4]{SvH19} that we
have
$$
	\int h\,dS_{K,\mathcal{C}} \le
	\|\nabla^2 h_K\|_{L^\infty(S^{n-1})}
	\int h\,dS_{B,\mathcal{C}}
$$
for any difference of support functions $h:S^{n-1}\to\mathbb{R}_+$.
Let us now define $\Pi_\varepsilon := \proj_{\mathcal{L}_r} +
\varepsilon\proj_{\mathcal{L}_r^\perp}$. Replacing
$\mathcal{C}\leftarrow\Pi_\varepsilon^{-1}\mathcal{C}$ 
and $h\leftarrow h\circ\Pi_\varepsilon^{-1}$ in the above inequality
yields
$$
	\int h\,dS_{\Pi_\varepsilon K,\mathcal{C}} \le
	\|\nabla^2 h_K\|_{L^\infty(S^{n-1})}
	\int h\,dS_{\Pi_\varepsilon B,\mathcal{C}},
$$
where we have used \eqref{eq:mixvolarea} and part $f$ of 
Lemma \ref{lem:mvprop}.
Letting $\varepsilon\to 0$ yields
$$
	\int h\,dS_{\proj_{\mathcal{L}_r}K,\mathcal{C}} \le
	\|\nabla^2 h_K\|_{L^\infty(S^{n-1})}
	\int h\,dS_{B_r,\mathcal{C}}
$$
by Lemma \ref{lem:cont}. In particular, using Lemma \ref{lem:c2},
this implies that
$$
	\supp S_{\proj_{\mathcal{L}_r}K,\mathcal{C}} \subseteq
	\supp S_{B_r,\mathcal{C}}
$$
for any convex body $K$ in $\mathbb{R}^n$ such that
$h_K$ is $C^2$ on $S^{n-1}$.

By a classical approximation argument \cite[Theorem 3.4.1]{Sch14}, we can 
find a sequence of convex bodies $K^{(l)}$ so that $h_{K^{(l)}}$ is $C^2$ 
for each $l$, and $K^{(l)}\to P_r$ in Hausdorff distance. Thus 
$S_{\proj_{\mathcal{L}_r}K^{(l)},\mathcal{C}}\wto S_{P_r,\mathcal{C}}$ by 
Lemma \ref{lem:cont}. But as each 
$S_{\proj_{\mathcal{L}_r}K^{(l)},\mathcal{C}}$ is supported in $\supp 
S_{B_r,\mathcal{C}}$, this must be the case for the limiting measure as 
well.
\end{proof}

We can now conclude the proof of Theorem \ref{thm:schneider}.

\begin{proof}[Proof of Theorem \ref{thm:schneider}]
The \emph{if} direction of Theorem \ref{thm:schneider} follows directly 
from Lemmas \ref{lem:lineq} and \ref{lem:suppeq}, so it suffices to 
consider the \emph{only if} direction.

Suppose first that Theorem \ref{thm:schneider} has been proved in dimension 
$n-1$ for some $n\ge 3$. Then we claim that Theorem \ref{thm:schneider} 
holds also in dimension $n$. Indeed, let $f:S^{n-1}\to\mathbb{R}$ be a 
difference of support functions such 
that $S_{f,\mathcal{P}}=0$, and let $g$ be the function
constructed in Lemma \ref{lem:superpproj} (for any $r\in[n-2]$ 
that is fixed throughout the proof).
By Lemmas \ref{lem:superglue} and \ref{lem:lowmaxsupp}, there exists
$s\in\mathbb{R}^n$ so that
$$
	g(x)=\langle s,x\rangle \quad\mbox{for all }x\in \supp
	S_{B,\mathcal{P}}.
$$
The claim follows as $f(x)=g(x)$ for all $x\in\supp 
S_{B,\mathcal{P}}$ by Lemma \ref{lem:superpproj}.

It remains to prove the base case $n=2$. More precisely, we claim the 
following: for any difference of support functions $f:S^1\to\mathbb{R}$ 
such that $S_f=0$, there must exist $s\in\mathbb{R}^2$ so that 
$f(x)=\langle s,x\rangle$ for all $x\in S^1$. This is a classical fact; 
for example, it may be deduced from the equality case of the 
Brunn-Minkowski inequality as in \cite[Theorem 7.2.1]{Sch14}. Let us give 
another proof here in order to illustrate a method that will be used again 
in section \ref{sec:proplin} in an essential manner.

Suppose $f$ does not satisfy $f=\langle s,\cdot\rangle$ for any $s$.
Then the Hahn-Banach theorem implies \cite[Corollary IV.3.15]{Con90} that 
there is a finite signed measure $\sigma$ on $S^1$ so that
$$
	\int f\,d\sigma > 0\qquad\mbox{and}\qquad
	\int x\,\sigma(dx)=0.
$$
Let $\sigma=\sigma^+-\sigma^-$ be the Hahn-Jordan decomposition of 
$\sigma$, and 
let $m:=\int x\,\sigma^\pm(dx)$ and
$\mu^\pm := \sigma^{\pm} + \|m\|\delta_{-m/\|m\|} + S_B$. Then $\mu^\pm$ 
are nonnegative measures on $S^1$ so that $\int x\,\mu^\pm(dx)=0$ 
and $\sspan\supp \mu^{\pm}=\mathbb{R}^2$. By the Minkowski existence 
theorem
\cite[Theorem 8.2.2]{Sch14}, there exist convex bodies $C^\pm$ in 
$\mathbb{R}^2$ so that $\mu^\pm = S_{C^\pm}$. But then we obtain
using \eqref{eq:mixvolarea} and the symmetry of mixed volumes
$$
	\int f\,d\sigma = 
	\int f\,dS_{C^+}-\int f\,dS_{C^-} =
	\int h_{C^+}\,dS_f - \int h_{C^-}\,dS_f = 0,
$$
which entails the desired contradiction.
\end{proof}

\begin{rem}
Let us highlight a surprising aspect of the proof of Theorem 
\ref{thm:schneider}. By Lemma 
\ref{lem:suppeq}, the equality condition $S_{f,\mathcal{P}}=0$ can only
determine $f$ on the support of $S_{B,\mathcal{P}}$. However, in Lemma 
\ref{lem:superglue} we have characterized the function $g$ on the support 
of $S_{B,B_r,\mathcal{P}_{\backslash r}}$. The latter set
is often much larger than the former. For example, if 
$P_1=\cdots=P_{n-2}=P$ is 
a full-dimensional polytope, then $\supp S_{B,\mathcal{P}}$ is the set of 
normal directions of $(n-2)$-dimensional faces of $P$, but $\supp 
S_{B,B,\mathcal{P}_{\backslash r}}$ is the set of normal directions of 
$(n-3)$-dimensional faces of $P$ (cf.\ \cite[Theorem 4.5.3]{Sch14}).

Nonetheless, there is no contradiction, as Theorem \ref{thm:localaf} only 
ensures that $f=g$ on the smaller set $\supp S_{B,\mathcal{P}}$. The 
phenomenon exhibited here should be viewed as another manifestation of the 
fact that the local Alexandrov-Fenchel inequality fixes many degrees of 
freedom of the extremal functions. 
\end{rem}

\section{Structure of critical sets}
\label{sec:panov}

We now turn to the study of the extremals of the Alexandrov-Fenchel 
inequality in the critical case (Definition \ref{defn:crit}).
The new feature that arises when $\mathcal{P}$ is critical is the 
appearance of $\mathcal{P}$-degenerate functions (Definition 
\ref{defn:deg}). Their analysis requires several new 
ideas, whose development will occupy us throughout sections 
\ref{sec:panov}--\ref{sec:crit}.

The definition of the critical case differs from the supercritical case 
only in that there may now exist indices $i_1<\cdots<i_k$ so that 
$\dim(P_{i_1}+\cdots+P_{i_k})=k+1$. Such \emph{critical sets} of indices 
will prove to be intimately connected to the structure of 
$\mathcal{P}$-degenerate pairs and functions. For example, we will show 
that for any $\mathcal{P}$-degenerate pair $(M,N)$, the bodies $M,N$ must 
be contained (up to translation) in the affine hull of 
$P_{i_1}+\cdots+P_{i_k}$ for some critical set $i_1<\cdots<i_k$.

In this section, we begin the analysis of the critical case by obtaining a 
classification of the critical sets, which will be used to give an 
explicit description of the structure of $\mathcal{P}$-degenerate 
functions. In section \ref{sec:propeller}, we undertake a detailed 
study of the geometric structure of critical mixed area measures. These 
results will be employed in section \ref{sec:crit} to prove Theorem 
\ref{thm:main} in the critical case.

Throughout this section, we fix $n\ge 3$ and a critical collection 
$\mathcal{P}=(P_1,\ldots,P_{n-2})$ of polytopes in $\mathbb{R}^n$. As in 
section \ref{sec:supercrit}, we will assume without loss of generality 
that $P_i$ contains the origin in its relative interior for every 
$i\in[n-2]$, and we define the spaces $\mathcal{L}_\alpha$ and  
balls $B_\alpha$ as in the supercritical case. The criticality 
assumption may then be formulated as $\dim\mathcal{L}_\alpha\ge|\alpha|+1$ 
for every $\alpha\subseteq[n-2]$, $\alpha\ne\varnothing$.

\subsection{Critical sets}

The following definition will play a central role in the sequel.

\begin{defn}
\label{defn:critset}
Let $\mathcal{C}=(C_1,\ldots,C_m)$ be any collection of convex bodies.
\vspace{.5\abovedisplayskip}
\begin{enumerate}[a.]
\itemsep\abovedisplayskip
\item $\alpha\subseteq[m]$ is called \emph{$\mathcal{C}$-critical} if
$\dim(\sum_{i\in\alpha}C_i)=|\alpha|+1$.
\item $\alpha\subseteq[m]$ is called \emph{$\mathcal{C}$-maximal} if it is 
$\mathcal{C}$-critical, and there is no $\mathcal{C}$-critical set 
$\beta\supsetneq\alpha$.
\end{enumerate}
\vspace{.5\abovedisplayskip}
A $\mathcal{P}$-critical ($\mathcal{P}$-maximal) set
$\alpha\subseteq[n-2]$ will simply be called
\emph{critical} (\emph{maximal}).
\end{defn}

The analysis of degenerate functions will be greatly facilitated by the 
fact that the family of critical sets is organized in a very simple 
manner. The following lemma and its corollary are due to Panov 
\cite[Lemma 6]{Pan85}.

\begin{lem}
\label{lem:panov}
Let $\alpha,\alpha'$ be critical sets. If 
$\alpha\cap\alpha'\ne\varnothing$, then $\alpha\cup\alpha'$ is a critical 
set.
\end{lem}

\begin{proof}
For any $\beta,\beta'\subset[n-2]$, we have
$\mathcal{L}_{\beta\cup\beta'}=\mathcal{L}_\beta+\mathcal{L}_{\beta'}$
and $\mathcal{L}_{\beta\cap\beta'}\subseteq
\mathcal{L}_{\beta}\cap\mathcal{L}_{\beta'}$
by the definition of $\mathcal{L}_\beta$.
On the other hand, as we assumed $\mathcal{P}$ is critical and
$\alpha\cap\alpha'\ne\varnothing$, we have
$\dim\mathcal{L}_{\alpha\cup\alpha'}\ge |\alpha\cup\alpha'|+1$
and
$\dim\mathcal{L}_{\alpha\cap\alpha'}\ge |\alpha\cap\alpha'|+1$.
Therefore 
\begin{align*}
	|\alpha\cup\alpha'|+1 &\le
	\dim\mathcal{L}_{\alpha\cup\alpha'} 
	= 
	\dim \mathcal{L}_\alpha + \dim\mathcal{L}_{\alpha'}
	-\dim(\mathcal{L}_\alpha\cap\mathcal{L}_{\alpha'}) \\
	&\le
	\dim \mathcal{L}_\alpha + \dim\mathcal{L}_{\alpha'}
	-\dim\mathcal{L}_{\alpha\cap\alpha'} \\
	&\le
	(|\alpha|+1) + (|\alpha'|+1)
	- (|\alpha\cap\alpha'|+1)
	\\ &= |\alpha\cup\alpha'|+1,
\end{align*}
where we used that $\dim \mathcal{L}_\alpha=|\alpha|+1$ and
$\dim\mathcal{L}_{\alpha'}=|\alpha'|+1$ as $\alpha,\alpha'$ are 
critical.
It follows that $\dim\mathcal{L}_{\alpha\cup\alpha'}=
|\alpha\cup\alpha'|+1$, so $\alpha\cup\alpha'$ is critical.
\end{proof}

The key consequence of Lemma \ref{lem:panov} is that distinct 
\emph{maximal} sets $\alpha,\alpha'$ must be disjoint. This structure is 
also reflected in the associated linear spaces: if $\alpha,\alpha'$ are 
distinct maximal sets, then $\mathcal{L}_\alpha,\mathcal{L}_{\alpha'}$
are linearly independent.

\begin{cor}
\label{cor:panov}
Let $\alpha\ne\alpha'$ be maximal sets. 
Then $\alpha\cap\alpha'=\varnothing$ and
$\mathcal{L}_\alpha\cap\mathcal{L}_{\alpha'}=\{0\}$.
\end{cor}

\begin{proof}
Let $\alpha,\alpha'$ be distinct maximal sets. Then
$\alpha\cup\alpha'$ cannot be a critical set:
as either $\alpha\cup\alpha'\supsetneq\alpha$ or   
$\alpha\cup\alpha'\supsetneq\alpha'$, this would contradict maximality of 
$\alpha,\alpha'$. Thus $\alpha\cap\alpha'=\varnothing$, as otherwise
$\alpha\cup\alpha'$ would be a 
critical set by Lemma \ref{lem:panov}.

Now note that as $\alpha\cup\alpha'$ is not a critical set and
$\mathcal{P}$ is critical, we have
$$
	|\alpha\cup\alpha'|+2 \le
	\dim\mathcal{L}_{\alpha\cup\alpha'}
	\le
	\dim\mathcal{L}_\alpha+\dim\mathcal{L}_{\alpha'}
	= |\alpha|+|\alpha'|+2 = |\alpha\cup\alpha'|+2,
$$
where we used that $\alpha,\alpha'$ are critical sets and
$\alpha\cap\alpha'=\varnothing$. Thus
$$
	\dim(\mathcal{L}_\alpha\cap\mathcal{L}_{\alpha'}) =
	\dim\mathcal{L}_{\alpha\cup\alpha'} -
	\dim\mathcal{L}_\alpha -
	\dim\mathcal{L}_{\alpha'} = 0,
$$
completing the proof.
\end{proof}

In view of Corollary \ref{cor:panov}, we obtain the following picture. 
Associated to the critical collection $\mathcal{P}$ of polytopes is its 
collection $\{\alpha_1,\ldots,\alpha_\ell\}$ of disjoint maximal sets. Any 
critical set $\beta$ is contained in exactly one of the maximal sets 
$\alpha_i$. Moreover, the linear spaces $\mathcal{L}_{\alpha_i}$ are 
pairwise (but not jointly) linearly independent. The same properties
extend \emph{verbatim} to any critical collection $\mathcal{C}$ of convex 
bodies.

Let us finally record a simple observation.

\begin{lem}
\label{lem:critinc}
Let $\alpha\subseteq[n-2]$ be a critical set and 
$\beta\subseteq[n-2]$ be arbitrary.
Then 
$$
	\mathcal{L}_\beta\subseteq\mathcal{L}_\alpha
	\quad\mbox{if and only if}\quad
	\beta\subseteq\alpha.
$$
\end{lem}

\begin{proof}
If $\beta\subseteq\alpha$, then
$\mathcal{L}_\beta\subseteq\mathcal{L}_\alpha$ by definition.
Conversely, if $\mathcal{L}_\beta\subseteq\mathcal{L}_\alpha$, then
$$
	|\alpha|+1\le |\alpha\cup\beta|+1
	\le\dim\mathcal{L}_{\alpha\cup\beta} =
	\dim\mathcal{L}_\alpha = |\alpha|+1,
$$
where we used that $\mathcal{P}$ is critical,
that $\mathcal{L}_\beta\subseteq\mathcal{L}_\alpha$, and
that $\alpha$ is a critical set, respectively. Thus 
$|\alpha|=|\alpha\cup\beta|$, which implies
$\beta\subseteq\alpha$.
\end{proof}

\subsection{Degenerate pairs and functions}
\label{sec:deg}

We now use the above classification of critical sets to obtain a better 
understanding of Definition \ref{defn:deg}. For simplicity, 
$\mathcal{P}$-degenerate pairs and functions will henceforth be called 
\emph{degenerate pairs} and \emph{degenerate functions}, respectively.
However, the same structure will apply \emph{verbatim} to 
$\mathcal{C}$-degenerate pairs and functions for any critical collection 
$\mathcal{C}$ of convex bodies.

Let us begin by introducing a more precise definition.

\begin{defn}
\label{defn:maxdeg}
Let $\alpha$ be a maximal set and $M,N$ be convex bodies
in $\mathbb{R}^n$.
\begin{enumerate}[a.]
\itemsep\abovedisplayskip
\item $(M,N)$ is called an \emph{$\alpha$-degenerate pair} if
$$
	M,N\subset\mathcal{L}_\alpha\qquad
	\mbox{and}\qquad
	\V_{\mathcal{L}_\alpha}(M,\mathcal{P}_\alpha)=
	\V_{\mathcal{L}_\alpha}(N,\mathcal{P}_\alpha).
$$
\item A function $f:S^{n-1}\to\mathbb{R}$ is called an 
\emph{$\alpha$-degenerate 
function} if $f=h_M-h_N$ for some $\alpha$-degenerate pair $(M,N)$.
\end{enumerate}
If $\mathcal{C}$ is a critical collection of convex bodies and
$\alpha$ is $\mathcal{C}$-maximal, the analogous definitions will be
referred to as $(\mathcal{C},\alpha)$-degenerate pairs and functions.
\end{defn}

As a first step towards understanding Definition \ref{defn:maxdeg}, let us 
note for any maximal (hence also critical) set $\alpha$, we have 
$\dim\mathcal{L}_\alpha=|\alpha|+1$ and $P_i\subset\mathcal{L}_\alpha$ for 
every $i\in\alpha$. Thus the mixed volume 
$\V_{\mathcal{L}_\alpha}(M,\mathcal{P}_\alpha)$ is indeed well defined: 
this is the mixed volume of $|\alpha|+1$ convex bodies in the 
$(|\alpha|+1)$-dimensional space $\mathcal{L}_\alpha$.

We will now show that in the present setting ($\mathcal{P}$ is 
critical), any degenerate pair 
or function in the sense of Definition \ref{defn:deg} is in fact an 
$\alpha$-degenerate pair or function up to translation. In other words, 
degenerate pairs must always be contained in translates of
$\mathcal{L}_\alpha$ for some maximal set $\alpha$, which provides an 
explicit geometric description of the dimensionality property that is 
implicit in Definition \ref{defn:deg}.

\begin{lem}
\label{lem:twodegs}
$(M,N)$ is a degenerate pair 
if and only if $M$ is not a translate of $N$ and
$(M+v,N+w)$ is an $\alpha$-degenerate pair for some maximal 
set $\alpha$ and $v,w\in\mathbb{R}^n$. Thus $f$ is a degenerate function 
if and only if $f$ is nonlinear and
$f-\langle v,\cdot\rangle$ is an $\alpha$-degenerate 
function for some maximal set $\alpha$ and $v\in\mathbb{R}^n$.
\end{lem}

\begin{proof}
We begin by noting that for any critical set $\alpha$ and
convex body $K\subset\mathcal{L}_\alpha$, 
Lemma \ref{lem:proj} implies the projection formula
$$
        {n\choose |\alpha|+1}\,
        \V_n(K,B,\mathcal{P})
        =
        \V_{\mathcal{L}_\alpha}(K,\mathcal{P}_\alpha)\,
        \V_{\mathcal{L}_\alpha^\perp}(
	\proj_{\mathcal{L}_\alpha^\perp}B,
	\proj_{\mathcal{L}_\alpha^\perp}\mathcal{P}_{\backslash\alpha}).
$$
Moreover, as $\V_n(B_\alpha,B,\mathcal{P})>0$
by Lemma \ref{lem:dim} and the assumption that $\mathcal{P}$ is critical, 
it follows that
$\V_{\mathcal{L}_\alpha^\perp}(\proj_{\mathcal{L}_\alpha^\perp}B,
\proj_{\mathcal{L}_\alpha^\perp}\mathcal{P}_{\backslash\alpha})>0$.

Consider first an $\alpha$-degenerate pair $(M,N)$ for some maximal set 
$\alpha$, where $M,N$ are not translates. We claim that $(M,N)$ is a 
degenerate pair. Indeed, condition \eqref{eq:degp1} follows from Lemma 
\ref{lem:dim} as 
$\dim(M+N+\sum_{i\in\alpha}P_i)=\dim\mathcal{L}_\alpha=|\alpha|+1$, while 
condition \eqref{eq:degp2} follows from the projection formula and 
Definition \ref{defn:maxdeg}.

Now consider a degenerate pair $(M,N)$. As $M$ is not a translate of $N$, 
at least one of $M,N$ must have nonzero dimension. But as $\mathcal{P}$ is 
critical, $\V_n(K,B,\mathcal{P})>0$ whenever $\dim K\ge 1$ by Lemma 
\ref{lem:dim}. Thus \eqref{eq:degp2} implies that $\dim(M)\ge 1$ and 
$\dim(N)\ge 1$. On the other hand, it cannot be the case that 
$\dim(M+N)=1$. Indeed, if that were the case, then $M,N$ must be segements 
with parallel directions; moreover, \eqref{eq:degp2} then implies that 
$M,N$ have equal length, so that $M,N$ are translates. This case is 
therefore ruled out by the definition of a degenerate pair.

We have now shown that any degenerate pair $(M,N)$ must satisfy
$$
	\dim(M)\ge 1,\qquad
	\dim(N)\ge 1,\qquad
	\dim(M+N)\ge 2.
$$
Together with the assumption that $\mathcal{P}$ is critical, it follows
from Lemma \ref{lem:dim} and \eqref{eq:degp1} that there
must exist $\alpha'\subseteq[n-2]$, $\alpha'\ne\varnothing$ such that
$$
	\dim(M+N+{\textstyle\sum_{i\in\alpha'}P_i})\le|\alpha'|+1.
$$
On the other hand, as $\mathcal{P}$ is critical we have 
$\dim(\sum_{i\in\alpha'}P_i)\ge|\alpha'|+1$. The only way this can happen 
is 
if $\dim\mathcal{L}_{\alpha'}=\dim(\sum_{i\in\alpha'}P_i)=|\alpha'|+1$ 
(that is, $\alpha'$ is critical) and there exist $v,w\in\mathbb{R}^n$ so 
that $M+v$ and $N+w$ lie in $\mathcal{L}_{\alpha'}$.

Now let $\alpha$ be the maximal set containing $\alpha'$. 
Then $M,N\subset\mathcal{L}_{\alpha'}\subseteq\mathcal{L}_\alpha$. 
Moreover, by the projection formula, the normalization condition of 
Definition \ref{defn:maxdeg} follows from \eqref{eq:degp2}.
Thus we have shown that $(M+v,N+w)$ is an $\alpha$-degenerate pair.

Finally, the equivalence between degenerate and $\alpha$-degenerate 
functions is an immediate consequence of the corresponding equivalence for 
pairs.
\end{proof}

Lemma \ref{lem:twodegs} explains the basic structure of the extremals of 
the Alexandrov-Fenchel inequality that appears in Theorem \ref{thm:main}. 
Note that for a given maximal set $\alpha$, any linear combination of 
$\alpha$-degenerate functions is again an $\alpha$-degenerate function by 
definition. On the other hand, if $f$ is an $\alpha$-degenerate function 
and $f'$ is an $\alpha'$-degenerate function for distinct maximal sets 
$\alpha,\alpha'$, then linear combinations of $f,f'$ need not be 
degenerate. Each maximal set $\alpha$ will therefore give rise to (at 
most) one $\alpha$-degenerate pair in the statement of Theorem 
\ref{thm:main}.

\subsection{An intrinsic description}

So far we have defined degenerate functions as differences of support 
functions of degenerate pairs of convex bodies. However, in the proof of 
Theorem \ref{thm:main}, it will be necessary to construct degenerate 
functions directly by gluing together lower-dimensional degenerate 
functions. To this end, we now introduce a more intrinsic perspective on 
degenerate functions that does not require the auxiliary construction of a 
degenerate pair.

Before we proceed, we state a variant of the projection 
formula of Lemma \ref{lem:proj} in terms of mixed area measures, which 
will be needed below.

\begin{lem}
\label{lem:maproj}
Let $C_1,\ldots,C_{n-1}$ be convex bodies in $\mathbb{R}^n$, and
suppose that $C_1,\ldots,C_k$ lie in a subspace $E$ with 
$\dim E=k+1$. Then
$$
	{n-1\choose k} \int \varphi(\proj_E x)
	\,S_{C_1,\ldots,C_{n-1}}(dx) =
	\V_{E^\perp}(\proj_{E^\perp}C_{k+1},\ldots,\proj_{E^\perp}C_n)
	\int \varphi\,dS_{C_1,\ldots,C_k}
$$
for any $1$-homogeneous function $\varphi:E\to\mathbb{R}$
that is $S_{C_1,\ldots,C_k}$-integrable.
\end{lem}

\begin{proof}
Suppose first that the restriction of $\varphi$ to $S^{n-1}\cap E$ is a 
$C^2$ function. Then we may write $\varphi=h_K-h_L$ for convex bodies 
$K,L$ in $E$ by Lemma \ref{lem:c2}. Moreover, by the definition of support 
functions, $\varphi(\proj_Ex)=h_K(x)-h_L(x)$ for any $x\in S^{n-1}$ as 
$K,L\subset E$. The conclusion now follows from Lemma \ref{lem:proj} and 
\eqref{eq:mixvolarea}.

Now define the map $\iota:S^{n-1}\backslash E^\perp\to S^{n-1}\cap E$ as 
$\iota(x):=\proj_Ex/\|\proj_Ex\|$. By $1$-homogeneity of $\varphi$, the 
identity in the statement of the lemma may be written as
$$
	{n-1\choose k} \int \varphi\circ\iota ~ d\mu =
	\V_{E^\perp}(\proj_{E^\perp}C_{k+1},\ldots,\proj_{E^\perp}C_n)
	\int \varphi\,dS_{C_1,\ldots,C_k},
$$
where the measure $\mu(dx):=\|\proj_E x\|\,S_{C_1,\ldots,C_{n-1}}(dx)$
is supported on $S^{n-1}\backslash E^\perp$. As we have shown this 
identity holds for any $\varphi$ of class $C^2$, it follows that
\begin{equation}
\label{eq:pullback}
	{n-1\choose k}\, \mu\circ\iota^{-1} = 
	\V_{E^\perp}(\proj_{E^\perp}C_{k+1},\ldots,\proj_{E^\perp}C_n)
	\,S_{C_1,\ldots,C_k}
\end{equation}
as measures on $S^{n-1}\cap E$. The conclusion follows for any integrable
$1$-homogeneous function $\varphi:E\to\mathbb{R}$ by integrating this 
identity.
\end{proof}

\begin{rem}
\label{rem:raycont}
Suppose $C_1,\ldots,C_k$ are polytopes in Lemma \ref{lem:maproj}.
Then $S_{C_1,\ldots,C_k}$ has finite support by Lemma 
\ref{lem:mapoly}. Thus \eqref{eq:pullback} shows that the measure
$S_{C_1,\ldots,C_{n-1}}\circ\proj_E^{-1}$ is supported on a finite union
of rays emanating from the origin with directions in 
$\supp S_{C_1,\ldots,C_k}$. We now observe that any $1$-homogeneous 
function $\varphi$ is continuous on such a set: it is linear on each ray 
and zero at the origin. This implies that in the polytope setting, the
function $x\mapsto\varphi(\proj_E x)$ is continuous on
$\supp S_{C_1,\ldots,C_{n-1}}$ for any $1$-homogeneous function $\varphi$.
This observation will be used below.
\end{rem}

We can now introduce the main idea of this section: 
$\alpha$-degenerate functions may be intrinsically described in terms of 
$1$-homogeneous functions on $\mathcal{L}_\alpha$.

\begin{lem}
\label{lem:deghomog}
Let $\alpha$ be a maximal set.
\begin{enumerate}[a.]
\itemsep\abovedisplayskip
\item
For any $\alpha$-degenerate function $f$, there exists a 
$1$-homogeneous 
function $\varphi:\mathcal{L}_\alpha\to\mathbb{R}$ with
$\int \varphi\,dS_{\mathcal{P}_\alpha}=0$ so
that $f(x)=\varphi(\proj_{\mathcal{L}_\alpha}x)$ for all $x\in S^{n-1}$.
\item
For any $1$-homogeneous function $\varphi:\mathcal{L}_\alpha\to\mathbb{R}$ 
with $\int \varphi\,dS_{\mathcal{P}_\alpha}=0$, there exists an
$\alpha$-degenerate function $f$
so that $f(x)=\varphi(\proj_{\mathcal{L}_\alpha}x)$ for all
$x\in\supp S_{B,\mathcal{P}}$.
\end{enumerate}
\end{lem}

\begin{proof}
To prove part $a$, write $f=h_M-h_N$ for some $\alpha$-degenerate pair 
$(M,N)$. As $M,N\subset\mathcal{L}_\alpha$, we have 
$h_M(x)=h_M(\proj_{\mathcal{L}_\alpha}x)$ and 
$h_N(x)=h_N(\proj_{\mathcal{L}_\alpha}x)$ for all $x\in S^{n-1}$ by the 
definition of support functions. Now define $\varphi$ to be the 
restriction of $h_M-h_N$ to $\mathcal{L}_\alpha$. Then $\varphi$ 
is $1$-homogeneous, $f(x)=\varphi(\proj_{\mathcal{L}_\alpha}x)$ for 
all $x\in S^{n-1}$, and
$$
	\frac{1}{|\alpha|+1}\int 
	\varphi\,dS_{\mathcal{P}_\alpha}= 
	\V_{\mathcal{L}_\alpha}(M,\mathcal{P}_\alpha)-
	\V_{\mathcal{L}_\alpha}(N,\mathcal{P}_\alpha)=0
$$
by \eqref{eq:mixvolarea} and the definition of an $\alpha$-critical pair.

The same argument would apply \emph{verbatim} in the converse direction if 
$\varphi$ can be written as a difference of support functions. This is not 
clear, however, as we did not make any regularity assumption on $\varphi$. 
To work around this issue, we will exploit that $\mathcal{P}$ are 
polytopes to create a modification of $\varphi$ with the requisite 
property.

More precisely, part $b$ is proved as follows. As 
$\mathcal{P}$ are polytopes, $\supp S_{\mathcal{P}_\alpha}$ is a finite 
subset of $S^{n-1}\cap\mathcal{L}_\alpha$ by Lemma \ref{lem:mapoly}. Thus 
we can choose a $C^2$ function 
$\eta:S^{n-1}\cap\mathcal{L}_\alpha\to\mathbb{R}$ so that 
$\varphi(x)=\eta(x)$ for all $x\in\supp S_{\mathcal{P}_\alpha}$. 
By Lemma \ref{lem:c2}, there exist convex bodies 
$M,N\subset\mathcal{L}_\alpha$ so that
$\eta(x)=h_M(x)-h_N(x)$ for all $x\in S^{n-1}\cap\mathcal{L}_\alpha$.
We claim that $f:=h_M-h_N$ has the properties 
stated in part $b$.
Indeed, note that
$$
	\V_{\mathcal{L}_\alpha}(M,\mathcal{P}_\alpha)-
	\V_{\mathcal{L}_\alpha}(N,\mathcal{P}_\alpha) =
	\frac{1}{|\alpha|+1}\int 
	f\,dS_{\mathcal{P}_\alpha}=
	\frac{1}{|\alpha|+1}\int 
	\varphi\,dS_{\mathcal{P}_\alpha}= 0,
$$
where we used \eqref{eq:mixvolarea} in the first equality and
$f=\varphi$ $S_{\mathcal{P}_\alpha}$-a.e.\ in the second equality.
Thus $(M,N)$ is an $\alpha$-degenerate pair and $f$ is an 
$\alpha$-degenerate function. On the other hand,
as $f=\varphi$ on $\supp S_{\mathcal{P}_\alpha}$, we obtain
\begin{align*}
	0 &= \V_{\mathcal{L}_\alpha^\perp}(
	\proj_{\mathcal{L}_\alpha^\perp}B,
	\proj_{\mathcal{L}_\alpha^\perp}\mathcal{P}_{\backslash\alpha})
	\int |f-\varphi|\,dS_{\mathcal{P}_\alpha}
	\\ &=
	{n-1\choose |\alpha|}\int
	|f(x)
	-\varphi(\proj_{\mathcal{L}_\alpha}x)|
	\,S_{B,\mathcal{P}}(dx)
\end{align*}
by Lemma \ref{lem:maproj}, where we used that 
$f(x)=f(\proj_{\mathcal{L}_\alpha}x)$ as $M,N\subset\mathcal{L}_\alpha$.
Thus $f(x)=\varphi(\proj_{\mathcal{L}_\alpha}x)$ for all $x\in\supp 
S_{B,\mathcal{P}}$ by Remark \ref{rem:raycont}, completing the proof.
\end{proof}

\section{Propeller geometry}
\label{sec:propeller}

We have seen in the previous section that the appearance of degenerate 
functions is intimately connected to the critical sets of the reference 
bodies $\mathcal{P}$. In this section, we will develop a new geometric 
phenomenon that explains the origin of this behavior: we will show that 
the supports of critical mixed area measures exhibit certain geometric 
structures that we call \emph{propellers}, in view of their resemblance to 
the propeller of a Mississippi steamboat. These propellers will play a 
crucial role in the proof of Theorem \ref{thm:main} in the critical case.

This section is organized as follows. We first introduce the propeller 
structure in section \ref{sec:propthm}. 
In the proof of Theorem \ref{thm:main}, this structure will be exploited 
in two nontrivial ways: to glue together lower-dimensional degenerate 
functions, and to decouple the contributions arising from distinct maximal 
sets. We develop both these methods in an abstract setting in sections 
\ref{sec:propglue} and \ref{sec:proplin}, respectively. While the basic 
principles can be understood independently of the rest of the paper, their 
power will become clear when they are applied in section \ref{sec:crit}.

\subsection{The propeller}
\label{sec:propthm}

The following theorem describes the propeller structure.

\begin{thm}
\label{thm:propeller}
Let $C_1,\ldots,C_{n-1}$ be convex bodies in $\mathbb{R}^n$, and suppose 
$C_1,\ldots,C_k$ lie in a subspace $E$ with $\dim E=k+1$. 
Define the space $F_z := \sspan\{E^\perp,z\}$ and
halfspace $F_z^+ := \{x\in F_z:\langle z,x\rangle>0\}$
for $z\in E$. Then
$$
	\supp S_{C_1,\ldots,C_{n-1}} \subset
	E^\perp \cup \bigcup_{z\in\supp S_{C_1,\ldots,C_k}} F_z^+,
$$
and
\begin{equation}
\label{eq:propeller}
\begin{aligned}
	& {n-1\choose k} \int f(x)\,1_{x\not\in E^\perp}\,
	S_{C_1,\ldots,C_{n-1}}(dx) \\
	&=
	\int\bigg(
	\int f(x)\,1_{\langle z,x\rangle>0}\,
	S_{\proj_{F_z}C_{k+1},\ldots,\proj_{F_z}C_{n-1}}(dx)
	\bigg)S_{C_1,\ldots,C_k}(dz)
\end{aligned}
\end{equation}
for any bounded measurable function $f:S^{n-1}\to\mathbb{R}$.
\end{thm}

Informally, Theorem \ref{thm:propeller} states that 
$S_{C_1,\ldots,C_{n-1}}$ is supported in a union of halfspaces 
(``blades'') centered around $E^\perp$ with orthogonal direction in 
$\supp S_{C_1,\ldots,C_k}\subset E$. Moreover, $S_{C_1,\ldots,C_{n-1}}$ 
agrees on each blade with the mixed area measure of the projections of 
$C_{k+1},\ldots,C_{n-1}$ onto the subspace in which that blade lies.
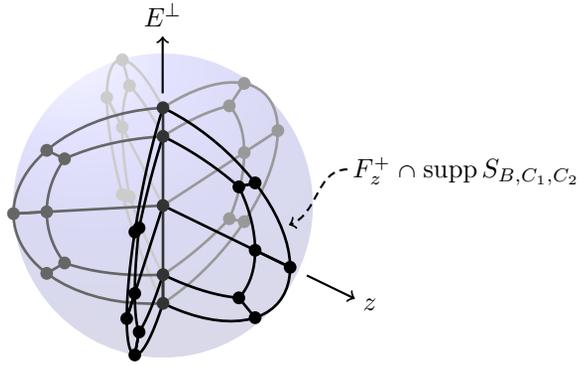
\begin{figure}
\centering
\begin{tikzpicture}

\node[inner sep=0pt] (0,0) {\includegraphics{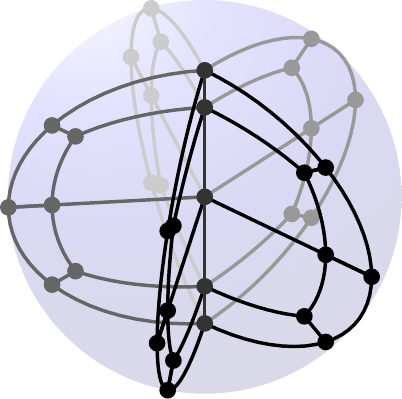}};

\draw[thick,->] (1.959,-0.894) to (2.6,-1.2);
\draw (2.8,-1.275) node {$z$};

\draw[thick,->] (0.0355,1.5) to (0.0355,2.25) node[above] {$E^\perp$};

\draw[densely dashed,thick,->] (2.5,0.5) to[out=180,in=20] (1.72,-0.23);
\draw (2.45,0.48) node[right] {$F_z^+\cap\supp S_{B,C_1,C_2}$};

\draw (-2.45,0.48) node[left] {\phantom{$F_z^+\cap\supp S_{B,C_1,C_2}$}};

\end{tikzpicture}

\caption{Illustration of a propeller structure in $\mathbb{R}^4$.\label{fig:propeller}}
\end{figure}

\begin{example}
The propeller structure is illustrated in Figure \ref{fig:propeller}. Here 
$C_1,C_2$ are polytopes in $\mathbb{R}^4$, where $C_1$ is a 
pentagon contained in the plane $E=\sspan\{e_1,e_2\}$ spanned by the first 
two coordinate directions, and $C_2$ is a full-dimensional polytope (in 
the figure, we chose the Minkowski sum of a cube and an octahedron). We 
have visualized the support of $S_{B,C_1,C_2}$ by projecting it onto 
$\sspan\{e_1,e_2,e_3\}$, which yields a geometric graph in the unit ball 
of $\mathbb{R}^3$ (cf.\ section \ref{sec:qgraph}). The propeller structure 
is immediately evident in the picture: the ``blades'' of the propeller lie
in the halfspaces $F_z^+$, while the ``shaft'' of the propeller lies in 
$E^\perp$. There are five blades, corresponding to the five facet normals
of the pentagon $C_1$.
\end{example}

\begin{example}
\label{ex:ewald}
Suppose $C_1,\ldots,C_{n-2}$ all lie in a subspace $E=w^\perp$. Then we 
may apply Theorem \ref{thm:propeller} with $k=n-1$ to investigate the
measure $S_{B,C_1,\ldots,C_{n-2}}$. In this case $S^{n-1}\cap F_z^+$ is 
merely a semicircular arc from $w$ to $-w$ passing through $z$, and
$S_{\proj_{F_z}B}$ is the uniform measure on this arc. The 
propeller then takes the explicit ``striped watermelon'' form that was 
studied in \cite[\S 8]{SvH19} and implicitly in \cite{Sch94}.
\end{example}

Let us now discuss the interaction between degenerate functions and the 
propeller structure. In the setting of Theorem \ref{thm:propeller}, 
a degenerate function may be expressed as $f(x)=\varphi(\proj_Ex)$ for a 
$1$-homogeneous function $\varphi:E\to\mathbb{R}$ (cf.\ Lemma 
\ref{lem:deghomog}). Now note that for any $z\in S^{n-1}\cap E$, we have 
$\proj_E F_z=\sspan\{z\}$ by  definition, so that 
$\proj_Ex=\langle x,z\rangle z$ for any $x\in F_z$. We therefore obtain
$$
	f(x) = \varphi(z)\langle z,x\rangle
	\quad\mbox{for all }x\in F_z^+,~z\in S^{n-1}\cap E
$$
by the homogeneity of $\varphi$. On the other hand, clearly
$f(x)=0$ for $x\in E^\perp$. Thus we have shown that \emph{a 
degenerate function is linear on each blade of the propeller, and vanishes 
on the shaft}. This provides a geometric explanation for why degenerate 
functions are extremals of the Alexandrov-Fenchel inequality: the 
conditions of Proposition \ref{prop:qgraph} are satisfied for degenerate 
functions precisely because the propeller structure creates a geometric 
mechanism for this to happen.

\begin{rem}
The propeller structure was already hinted at by the observation of Remark
\ref{rem:raycont} in the previous section: it is evident from 
the propeller structure that the projection of $\supp 
S_{C_1,\ldots,C_{n-1}}$ on $E$ is supported on rays emanating from the 
origin in the directions in $\supp S_{C_1,\ldots,C_k}$.
More generally, the reader may verify that Lemma 
\ref{lem:maproj} can be deduced directly from Theorem 
\ref{thm:propeller} and Corollary \ref{cor:segproj}.
\end{rem}

\begin{rem}
In Theorem \ref{thm:propeller} we only considered the effect of a single 
critical set on the geometry of the mixed area measure. However, many 
critical sets may coexist for the same collection of bodies: this is not 
ruled out by the assumptions of Theorem \ref{thm:propeller}, where we 
singled out one critical set for analysis. When distinct maximal sets are 
present, the geometry of the mixed area measure will feature several 
propellers that are superimposed in different directions. Such 
``propellers within propellers'' are hard to visualize, and we will not 
attempt to do so. Nonetheless, this situation must be addressed in the 
proof of Theorem \ref{thm:main}, which will be done using a technique that 
is developed in section \ref{sec:proplin} below.
\end{rem}

We now turn to the proof of Theorem \ref{thm:propeller}.

\begin{proof}[Proof of Theorem \ref{thm:propeller}]
The statement about the support of $S_{C_1,\ldots,C_{n-1}}$ follows 
immediately from \eqref{eq:propeller}. We will first prove 
\eqref{eq:propeller} in the case that $C_1,\ldots,C_{n-1}$ are polytopes, 
and then derive the general case by approximation.

\vspace{\abovedisplayskip}

\textbf{Step 1}.
Suppose that $C_1,\ldots,C_{n-1}$ are polytopes. Fix any
$x\in S^{n-1}\backslash E^\perp$ and let $z:=\proj_Ex/\|\proj_Ex\|$.
Then $x\in F_z$ by definition.
As $C_1+\cdots+C_k\subset E$, we have
$$
        F(C_1+\cdots+C_k,x) =
        F(C_1+\cdots+C_k,z)
        \subset
        az+E\cap z^\perp
$$
for some constant $a$ by Lemma \ref{lem:faceproj}.
In particular, $\dim F(C_1+\cdots+C_k,x)\le k$. We can therefore write
using Lemmas \ref{lem:mapoly} and \ref{lem:proj}
\begin{align*}
        &{n-1\choose k}
        S_{C_1,\ldots,C_{n-1}}(\{x\}) 
        = {n-1\choose k}\V_{n-1}(F(C_1,x),\ldots,F(C_{n-1},x))
\\      &=\V_{E\cap z^\perp}(F(C_1,z),\ldots,F(C_k,z))\,
        \V_{F_z}(\proj_{F_z}F(C_{k+1},x),\ldots,\proj_{F_z}F(C_{n-1},x))
\phantom{\bigg(}
\\	&=
        S_{C_1,\ldots,C_k}(\{z\})\,
        S_{\proj_{F_z}C_{k+1},\ldots,\proj_{F_z}C_{n-1}}(\{x\}),
\end{align*}
where we used in the last line that
$\proj_{F_z}F(C_i,x)=F(\proj_{F_z}C_i,x)$ by Lemma \ref{lem:faceproj}.

Now note that for any $u\in E$, we have $x\in F_u$ if and only if $u=z$
or $u=-z$. In particular, as $\langle z,x\rangle>0$ and
$\supp S_{\proj_{F_u}C_{k+1},\ldots,\proj_{F_u}C_{n-1}}\subset F_u$,
we have
$$
	1_{\langle u,x\rangle>0}\,
        S_{\proj_{F_u}C_{k+1},\ldots,\proj_{F_u}C_{n-1}}(\{x\}) =0
	\quad
	\mbox{for all }u\in E,~u\ne z.
$$
We therefore obtain 
$$
        {n-1\choose k} S_{C_1,\ldots,C_{n-1}}(\{x\}) = 
	\sum_{u} 
        1_{\langle u,x\rangle>0}\,
        S_{\proj_{F_u}C_{k+1},\ldots,\proj_{F_u}C_{n-1}}(\{x\})\,
        S_{C_1,\ldots,C_k}(\{u\}).
$$
As this identity holds for any $x\in S^{n-1}\backslash 
E^\perp$, \eqref{eq:propeller} follows
from Lemma \ref{lem:mapoly}.

\vspace{\abovedisplayskip}

\textbf{Step 2.} We now aim to show that \eqref{eq:propeller} remains 
valid when $C_1,\ldots,C_{n-1}$ are arbitrary convex bodies. We first 
claim that the result is equivalent 
to
\begin{equation}
\label{eq:contpropeller}
\begin{aligned}
        & {n-1\choose k}
        \int g(x)\,\|\proj_Ex\|\, S_{C_1,\ldots,C_{n-1}}(dx) 
        \\ &
        =\int \bigg(
	\int g(x)\,\langle z,x\rangle_+ \,
	S_{\proj_{F_z}C_{k+1},\ldots,
        \proj_{F_z}C_{n-1}}(dx)
	\bigg)
	S_{C_1,\ldots,C_k}(dz)
\end{aligned}
\end{equation}
for every continuous function $g:S^{n-1}\to\mathbb{R}$. That 
\eqref{eq:propeller} 
implies \eqref{eq:contpropeller} follows by choosing 
$f(x)=g(x)\|\proj_Ex\|$ and using that $\|\proj_Ex\|=|\langle z,x\rangle|$ 
on $F_z$. Conversely, suppose \eqref{eq:contpropeller} holds; by a 
standard approximation argument, it extends to any nonnegative measurable 
function $g$. Choosing $g(x)=f(x) 1_{x\not\in E^\perp}\|\proj_Ex\|^{-1}$ 
yields \eqref{eq:propeller} for nonnegative $f$, and the conclusion 
follows by linearity.

The advantage of \eqref{eq:contpropeller} is that the integrands are 
continuous, so we may use weak convergence. Fix a continuous 
function $g:S^{n-1}\to\mathbb{R}$, and choose polytopes 
$C_1^{(l)},\ldots,C_{n-1}^{(l)}$ so that $C_1^{(l)},\ldots,C_k^{(l)} 
\subset E$ and $C_r^{(l)}\to C_r$ in Hausdorff distance for all $r$ (the 
existence of such approximations is elementary \cite[Theorem 
1.8.16]{Sch14}). We have already shown in the first part of the proof that 
\eqref{eq:contpropeller} holds for the polytopes 
$C_1^{(l)},\ldots,C_{n-1}^{(l)}$. We would like to show the identity 
remains valid as $l\to\infty$. As mixed area measures are continuous by 
Lemma \ref{lem:cont}, it suffices by a standard weak convergence argument 
\cite[Theorem 4.27]{Kal02} to show that
\begin{align*}
	&\int g(x)\,\langle z_l,x\rangle_+ \,
	S_{\proj_{F_{z_l}}C_{k+1}^{(l)},\ldots,
        \proj_{F_{z_l}}C_{n-1}^{(l)}}(dx)
	\\ &\qquad\xrightarrow{l\to\infty}
	\int g(x)\,\langle z,x\rangle_+ \,
	S_{\proj_{F_z}C_{k+1},\ldots,
        \proj_{F_z}C_{n-1}}(dx)
\end{align*}
for any sequence $z_l\to z\in S^{n-1}\cap E$. But this follows readily
from Lemma \ref{lem:cont} as
\begin{align*}
	\|h_{\proj_{F_{z_l}}C_r^{(l)}}-h_{\proj_{F_z}C_r}\|_\infty
	&\le
	\|h_{\proj_{F_{z_l}}C_r^{(l)}}-h_{\proj_{F_{z_l}}C_r}\|_\infty +
	\|h_{\proj_{F_{z_l}}C_r}-h_{\proj_{F_z}C_r}\|_\infty
\\
	&\le
	\|h_{C_r^{(l)}}-h_{C_r}\|_\infty +
	\|h_{\proj_{F_{z_l}}C_r}-h_{\proj_{F_z}C_r}\|_\infty
	\xrightarrow{l\to\infty}0
\end{align*}
(that is, $\proj_{F_{z_l}}C_r^{(l)}\to \proj_{F_z}C_r$ in Hausdorff 
distance) and $\|\langle z_l,\cdot\rangle- \langle 
z,\cdot\rangle\|_\infty\to 0$.
\end{proof}

\subsection{The gluing principle}
\label{sec:propglue}

One of the difficulties we will encounter in the proof of Theorem 
\ref{thm:main} is that we must glue together degenerate functions 
in $(n-1)$-dimensional hyperplanes to form a degenerate function in 
dimension $n$. This will be accomplished using the following application 
of the propeller structure.

\begin{lem}
\label{lem:propglue}
Let $C_{k+1},\ldots,C_{n-1}$ be convex bodies in $\mathbb{R}^n$, and let
$E$ be a subspace of dimension $\dim E=k+1$. Assume that
$$
	\V_{E^\perp}(\proj_{E^\perp}C_{k+1},\ldots,\proj_{E^\perp}C_{n-1})>0.
$$
Then for any $1$-homogeneous function $h:\mathbb{R}^n\to\mathbb{R}$, 
there exists a $1$-homogeneous function $\varphi:E\to\mathbb{R}$ with the 
following property:
for every collection of convex bodies $K_1,\ldots,K_k$ in $E$ such that
we have
$$
	h(x)=\tilde\varphi(\proj_Ex)\quad\mbox{for all }x\in\supp
	S_{K_1,\ldots,K_k,C_{k+1},\ldots,C_{n-1}}
$$
for some $1$-homogeneous function 
$\tilde\varphi:E\to\mathbb{R}$,
we have in fact
$$
	h(x)=\varphi(\proj_Ex)\quad\mbox{for all }x\in\supp
        S_{K_1,\ldots,K_k,C_{k+1},\ldots,C_{n-1}}.
$$
\end{lem}
\medskip

The point of Lemma \ref{lem:propglue} is that the function 
$\tilde\varphi$ depends on the choice of bodies $K_1,\ldots,K_k$, 
while $\varphi$ does not. Thus $\varphi$ may be viewed as having ``glued 
together'' the functions $\tilde\varphi$ over all choices of 
$K_1,\ldots,K_k$ in $E$ for which $\tilde\varphi$ exists. In the 
proof of Theorem \ref{thm:main}, $\tilde\varphi$ 
will be degenerate functions of the $(n-1)$-dimensional projections,
and $\varphi$ will be the degenerate function in dimension $n$.

\begin{proof}[Proof of Lemma \ref{lem:propglue}]
Throughout the proof we adopt the same notation as in Theorem 
\ref{thm:propeller}. First, we define $\varphi(z)$ for $z\in S^{n-1}\cap 
E$ as
$$
	\varphi(z) :=
	\frac{\int h(x)\,1_{\langle z,x\rangle>0}\,
	S_{\proj_{F_z}C_{k+1},\ldots,\proj_{F_z}C_{n-1}}(dx)}{
	\V_{E^\perp}(\proj_{E^\perp}C_{k+1},\ldots,\proj_{E^\perp}C_{n-1})
	}.
$$
We may extend $\varphi:E\to\mathbb{R}$ to a $1$-homogeneous function
as $\varphi(z):=\|z\|\varphi(z/\|z\|)$ for $z\in E\backslash\{0\}$.
We now show that $\varphi$ satisfies the requisite property.

To this end, let
$\tilde\varphi:E\to\mathbb{R}$ be any $1$-homogeneous function, 
and fix any $z\in S^{n-1}\cap E$. As
$\proj_Ex = \langle x,z\rangle z$ for any $x\in F_z$, we obtain
$$
	\tilde\varphi(\proj_E x) =
	\tilde\varphi(z) \langle z,x\rangle 
	\qquad\mbox{for any }x\in F_z^+.
$$
Integrating this identity against
$1_{\langle z,x\rangle>0}\,S_{\proj_{F_z}C_{k+1},\ldots,\proj_{F_z}C_{n-1}}(dx)$ 
yields
\begin{equation}
\label{eq:condex}
\begin{aligned}
	&\int
	\tilde\varphi(\proj_E x)
	\,1_{\langle z,x\rangle>0}\,
	S_{\proj_{F_z}C_{k+1},\ldots,\proj_{F_z}C_{n-1}}(dx)
	\\
	&=
	\tilde\varphi(z)
	\int
	\langle z,x\rangle_+
	\,S_{\proj_{F_z}C_{k+1},\ldots,\proj_{F_z}C_{n-1}}(dx)
	\\ 
	&=
	\tilde\varphi(z)
	\,\V_{E^\perp}(\proj_{E^\perp}C_{k+1},
	\ldots,\proj_{E^\perp}C_{n-1}),
\end{aligned}
\end{equation}
where we used Corollary \ref{cor:segproj} in the last line.

Now let $K_1,\ldots,K_k$ be convex bodies in $E$ such that
$$
	h(x)=\tilde\varphi(\proj_Ex)\quad\mbox{for all }x\in\supp
	S_{K_1,\ldots,K_k,C_{k+1},\ldots,C_{n-1}}.
$$
Then we also have
$$
	h(x)=\tilde\varphi(\proj_Ex)\quad\mbox{for all }x\in
	\supp(1_{\langle z,\cdot\rangle>0}
	\,dS_{\proj_{F_z}C_{k+1},\ldots,\proj_{F_z}C_{n-1}})
$$
for any $z\in\supp S_{K_1,\ldots,K_k}$ by Theorem \ref{thm:propeller}.
Thus
$$
	\tilde\varphi(z) =
	\varphi(z)
	\quad\mbox{for all }z\in\supp S_{K_1,\ldots,K_k}
$$
by \eqref{eq:condex} and the definition of $\varphi$. But then we may also 
conclude that
$$
	h(x)=\tilde\varphi(\proj_Ex)=
	\varphi(\proj_Ex)
	\quad\mbox{for all }x\in\supp
        S_{K_1,\ldots,K_k,C_{k+1},\ldots,C_{n-1}},
$$
because $\proj_Ex\in \|\proj_Ex\|\supp S_{K_1,\ldots,K_k}$ for every
$x\in \supp S_{K_1,\ldots,K_k,C_{k+1},\ldots,C_{n-1}}$ by Theorem 
\ref{thm:propeller}, and as 
$\tilde\varphi$ and $\varphi$ are both $1$-homogeneous.
\end{proof}

\subsection{Linear relations}
\label{sec:proplin}

Lemma \ref{lem:propglue} is only applicable when a single degenerate 
function appears. In general there may be multiple degenerate functions 
corresponding to different maximal sets, and we will need a way to 
decouple their analysis. The technique that will be used for this purpose 
is developed in this section. The utility of the following result will be 
far from obvious at this point, but we will see in section 
\ref{sec:decouplingarg} that it plays a key role in our proofs.

Unlike the other results of this section, we formulate the following 
result only for polytopes, which will suffice for our purposes. The 
polytope assumption is convenient in the proof, but does not appear to be 
of fundamental importance.

\begin{prop}
\label{prop:proplin}
Let $C_1,\ldots,C_{n-1}$ be polytopes in $\mathbb{R}^n$. Suppose that
$C_1,\ldots,C_k$ lie in a subspace $E$ with $\dim E=k+1$
and satisfy the criticality condition of Definition \ref{defn:crit}.
Let $h:S^{n-1}\to\mathbb{R}$ be a function such that
$$
	h(x)=0\quad\mbox{for all }x\in E^\perp\cap 
	\supp S_{C_1,\ldots,C_{n-1}},
$$
and such that
$$
	\int h\,dS_{Q,\ldots,Q,C_{k+1},\ldots,C_{n-1}}=0
$$
for every full-dimensional polytope $Q$ in $E$. Then there exists
$w\in E$ so that
$$
	\int h(x)\,1_{\langle z,x\rangle>0}\,
	S_{\proj_{F_z}C_{k+1},\ldots,\proj_{F_z}C_{n-1}}(dx)
	=\langle w,z\rangle
$$
for all $z\in S^{n-1}\cap E$,
where $F_z$ is as defined in Theorem \ref{thm:propeller}.
\end{prop}

The reader should keep in mind the case where $h$ is a degenerate function 
corresponding to a critical set disjoint from $[k]$.
Then the bodies $C_1,\ldots,C_k$ factor out of the integral $\int 
h\,dS_{C_1,\ldots,C_{n-1}}$ by Lemma \ref{lem:maproj}, and may thus be 
replaced by any other body $Q$ in $E$. This motivates the assumption of 
Proposition \ref{prop:proplin}. The conclusion of Proposition 
\ref{prop:proplin} then states that the average of $h$ over each blade of 
the propeller generated by $C_1,\ldots,C_k$ must be linearly related 
across the blades. This is not at all clear from Theorem 
\ref{thm:propeller}, which specifies the mixed area measure on each blade 
but does not explain the relations between different blades.

The proof of Proposition \ref{prop:proplin} is based on a duality
argument that is similar to the one used at the end of the proof of 
Theorem \ref{thm:schneider}. Let us begin by formulating a simple 
consequence of the Minkowski existence theorem.

\begin{lem}
\label{lem:minkexsigned}
Let $\sigma$ be a signed measure on $S^{n-1}\cap E$ that is supported on a 
finite number of points and satisfies
$\int x\,\sigma(dx)=0$. Then there exist full-dimensional 
polytopes $Q,Q'$ in $E$ so that 
$\sigma=S_{Q,\ldots,Q}-S_{Q',\ldots,Q'}$.
\end{lem}

\begin{proof}
Let $\sigma=\sigma^+-\sigma^-$ be the Hahn-Jordan decomposition of 
$\sigma$ and 
$m:=\int x\,\sigma^\pm(dx)$.
Let $R$ be any full-dimensional polytope in $E$, and define
$$
	\mu^{\pm}:=\sigma^\pm + \|m\|\delta_{-m/\|m\|} + S_{R,\ldots,R}.
$$
Then $\mu^\pm$ are finitely supported measures, $\int 
x\,d\mu^{\pm}=0$, and $\sspan\supp\mu^\pm=E$. The 
Minkowski existence theorem \cite[Theorem 8.2.1]{Sch14} therefore yields
the existence of full-dimensional polytopes $Q,Q'$ in $E$ so that
$\mu^+=S_{Q,\ldots,Q}$ and $\mu^-=S_{Q',\ldots,Q'}$. The conclusion 
now follows as $\sigma=\mu^+-\mu^-$.
\end{proof}

We will exploit Lemma \ref{lem:minkexsigned} through a duality argument.

\begin{cor}
\label{cor:minkdual}
Let $\varrho:S^{n-1}\cap E\to\mathbb{R}$ be any function such that
$\int \varrho\,dS_{Q,\ldots,Q} = 0$
for every full-dimensional polytope $Q$ in $E$. Then 
$\varrho=\langle w,\cdot\rangle$ for some $w\in E$.
\end{cor}

\begin{proof}
We first claim that for any finite set $\Omega\subset S^{n-1}\cap E$, 
there exists $w_\Omega\in E$ so that $\varrho = \langle 
w_\Omega,\cdot\rangle$ on $\Omega$. Indeed, suppose this is not the 
case; then by the Hahn-Banach theorem, there is a 
signed measure with support in $\Omega$ so that
$$
	\int\varrho \,d\sigma>0\qquad
	\mbox{and}\qquad
	\int x\,\sigma(dx)=0.
$$
This is contradicted by Lemma \ref{lem:minkexsigned} and the assumption.

Now let $w := w_{\{v_1,\ldots,v_{k+1}\}}$, where
$\{v_1,\ldots,v_{k+1}\}$ is a basis of $E$. Then we have
$$
	\langle w,v_i\rangle = \varrho(v_i) =
	\langle w_{\{x,v_1,\ldots,v_{k+1}\}},v_i\rangle
$$
for every $x\in S^{n-1}\cap E$ and $i$. Thus
$w_{\{x,v_1,\ldots,v_{k+1}\}} =w$, so that
$$
	\varrho(x) = \langle w_{\{x,v_1,\ldots,v_{k+1}\}},x\rangle
	=\langle w,x\rangle
$$
for every $x\in S^{n-1}\cap E$.
\end{proof}

\begin{rem}
The reason for the somewhat roundabout finite-dimensional argument is that
we did not assume any regularity (for example, continuity) of $\varrho$, 
so the Hahn-Banach theorem cannot be applied directly in infinite 
dimension.
\end{rem}

We now formulate a useful consequence of the criticality condition. It is 
this part of the proof that is facilitated by the polytope assumption.

\begin{lem}
\label{lem:perpsupp}
Let $C_1,\ldots,C_{n-1}$ be polytopes in $\mathbb{R}^n$. Suppose that
$C_1,\ldots,C_k$ lie in a subspace $E$ with $\dim E=k+1$
and satisfy the criticality condition of Definition~\ref{defn:crit}.
Then
for any full-dimensional polytope $Q$ in $E$, we have
$$
	E^\perp\cap \supp S_{C_1,\ldots,C_{n-1}} =
	E^\perp\cap \supp S_{Q,\ldots,Q,C_{k+1},\ldots,C_{n-1}}.
$$
\end{lem}

\begin{proof}
Fix $u\in E^\perp$.
As $C_1,\ldots,C_k,Q\subset E$, Lemma \ref{lem:mapoly} yields
\begin{align*}
	S_{C_1,\ldots,C_{n-1}}(\{u\}) &=
	\V_{u^\perp}(C_1,\ldots,C_k,F(C_{k+1},u),\ldots,
	F(C_{n-1},u)),\\
	S_{Q,\ldots,Q,C_{k+1},\ldots,C_{n-1}}(\{u\}) &=
	\V_{u^\perp}(Q,\ldots,Q,F(C_{k+1},u),\ldots,
	F(C_{n-1},u)).
\end{align*}
Thus $S_{C_1,\ldots,C_{n-1}}(\{u\})>0$ implies 
$S_{Q,\ldots,Q,C_{k+1},\ldots,C_{n-1}}(\{u\})>0$ by Lemma \ref{lem:dim},
as $Q$ is full-dimensional. It remains to prove the converse implication.

To this end, suppose $S_{Q,\ldots,Q,C_{k+1},\ldots,C_{n-1}}(\{u\})>0$.
Then by Lemma \ref{lem:dim}, there exist segments $I_1,\ldots,I_k\subseteq
Q$ and $I_r\subseteq F(C_r,u)$ for $r=k+1,\ldots,n-1$ so that
$$
	\V_{u^\perp}(I_1,\ldots,I_{n-1})>0.
$$
Thus Lemma \ref{lem:proj} yields
$$
	0 <
	{n-1\choose k} \V_{u^\perp}(Q,\ldots,Q,I_{k+1},\ldots,I_{n-1})
	=
	\Vol_G(\proj_{G}Q)\,
	\V_H(I_{k+1},\ldots,I_{n-1}),
$$
where $H$ is the linear span of the directions of the segments
$I_{k+1},\ldots,I_{n-1}$ and $G:=H^\perp\cap u^\perp$.
This implies that $\dim(\proj_{G}Q)=k$.
But as $\dim(Q)=k+1$ by assumption, the map $\proj_{G}|_E$ must have 
a one-dimensional kernel. Therefore
$$
	\dim\Bigg(\sum_{i\in\alpha}\proj_{G}C_i\Bigg)
	\ge
	\dim\Bigg(\sum_{i\in\alpha}C_i\Bigg)-1\ge|\alpha|
	\quad\mbox{for all }\alpha\subseteq[k],
$$
where we used the criticality assumption in the second inequality.
Therefore
\begin{align*}
	0 &<
	\V_{G}(\proj_{G}C_1,\ldots,
	\proj_{G}C_k)\,\V_H(I_{k+1},\ldots,I_{n-1}) \\
	&=
	{n-1 \choose k}
	\V_{u^\perp}(C_1,\ldots,C_k,I_{k+1},\ldots,I_{n-1}) \\
	&\le
	{n-1 \choose k}
	\V_{u^\perp}(C_1,\ldots,C_k,F(C_{k+1},u),\ldots,F(C_{n-1},u))
\end{align*}
by Lemmas \ref{lem:dim} and \ref{lem:proj}. Thus we have shown that
$S_{Q,\ldots,Q,C_{k+1},\ldots,C_{n-1}}(\{u\})>0$ implies
$S_{C_1,\ldots,C_{n-1}}(\{u\})>0$, completing the proof.
\end{proof}

We can now complete the proof of Proposition \ref{prop:proplin}.

\begin{proof}[Proof of Proposition \ref{prop:proplin}]
Define the function $\varrho:S^{n-1}\cap E\to\mathbb{R}$ as
$$
	\varrho(z) :=
	\int h(x)\,1_{\langle z,x\rangle>0}\,
	S_{\proj_{F_z}C_{k+1},\ldots,\proj_{F_z}C_{n-1}}(dx).
$$
Then we have for any convex body $Q$ in $E$
$$
	\int \varrho\,dS_{Q,\ldots,Q} =
	{n-1\choose k}
	\int h(x)\,1_{x\not\in E^\perp}\,
	S_{Q,\ldots,Q,C_{k+1},\ldots,C_{n-1}}(dx)
$$
by Theorem \ref{thm:propeller}. But by the first assumption 
on $h$ and Lemma \ref{lem:perpsupp}, we have
$h(x)=0$ for $x\in E^\perp\cap\supp 
S_{Q,\ldots,Q,C_{k+1},\ldots,C_{n-1}}$ when $Q$ is a full-dimensional 
polytope in $E$. Thus the second assumption on $h$ 
shows that $\int \varrho\,dS_{Q,\ldots,Q} = 0$
for every full-dimensional polytope $Q$ in $E$. The conclusion follows 
from Corollary \ref{cor:minkdual}.
\end{proof}

\section{The critical case}
\label{sec:crit}

In this section, we complete the extremal characterization of the 
Alexandrov-Fenchel inequality in the critical case. More precisely, we 
will prove the following.

\begin{thm}
\label{thm:crit}
Let $\mathcal{P}=(P_1,\ldots,P_{n-2})$ be polytopes in $\mathbb{R}^n$ 
that contain the origin in their relative interior. Assume that
$\mathcal{P}$ is critical but not supercritical, and denote by 
$\alpha_0,\ldots,\alpha_\ell$ the associated maximal sets. Then for any 
difference of support functions $f:S^{n-1}\to\mathbb{R}$, we have 
$S_{f,\mathcal{P}}=0$ if and only if 
$$
	f(x) = \langle s,x\rangle + \sum_{j=0}^\ell g_j(x)
	\quad\mbox{for all }x\in\supp S_{B,\mathcal{P}}
$$
holds for some $s\in\mathbb{R}^n$ and $\alpha_j$-degenerate function 
$g_j$, $j=0,\ldots,\ell$.
\end{thm}

The proof of Theorem \ref{thm:crit} proceeds by induction on $n$. Just as 
in the supercritical case, it will turn out that the criticality 
assumption (Definition \ref{defn:crit}) is preserved by the induction. The 
induction hypothesis may therefore give rise to a supercritical case, 
which is already covered by Theorem \ref{thm:schneider}, or to a critical 
case, to which we may apply Theorem \ref{thm:crit} in lower dimension. 

The following setting will be assumed throughout this section. We fix 
$n\ge 3$ and a collection of polytopes $\mathcal{P}=(P_1,\ldots,P_{n-2})$ 
in $\mathbb{R}^n$ that contain the origin in their relative interior. We 
assume that $\mathcal{P}$ is critical but not supercritical, that is, 
there exists at least one critical set. We denote the maximal sets by 
$\alpha_0,\ldots,\alpha_\ell$. The spaces $\mathcal{L}_\alpha$ and the 
balls $B_\alpha$ are defined as in section \ref{sec:supercrit}. In 
particular, the criticality assumption may be formulated as 
$\dim\mathcal{L}_\alpha\ge|\alpha|+1$ for every $\alpha\subseteq[n-2]$, 
$\alpha\ne\varnothing$.

\subsection{The induction hypothesis}

In the induction step, we will assume that Theorem \ref{thm:crit} has been 
proved in dimension $n-1$, and deduce its validity in dimension $n$. The 
aim of this section is to formulate the resulting induction hypothesis.

As in section \ref{sec:supercrit}, we will begin by applying the local 
Alexandrov-Fenchel inequality. To this end, we must choose an index 
$r\in[n-2]$ to which Theorem \ref{thm:localaf} will be applied. Unlike in 
the supercitical case, however, we may not choose $r$ arbitrarily: the 
entire argument will be based on the fact that we will choose $r$ to lie 
inside one of the maximal sets. We therefore fix an element $r\in\alpha_0$ 
at the outset, which will be used throughout the proof without further 
comment. As $\alpha_0$ will play a special role throughout 
the proof, we will define henceforth $\gamma:=\alpha_0$ in order to 
distinguish it in the notation from the remaining maximal sets 
$\alpha_1,\ldots,\alpha_\ell$.

Our starting point is the following direct analogue of Lemma 
\ref{lem:superpproj}.

\begin{lem}
\label{lem:critproj}
Let $f$ be a difference of support functions 
such that $S_{f,\mathcal{P}}=0$. Then there exists a difference of 
support functions $g$ with the following properties:
\vspace{.5\abovedisplayskip}
\begin{enumerate}[1.]
\itemsep\abovedisplayskip
\item $g(x)=f(x)$ for all $x\in\supp S_{B,\mathcal{P}}$.
\item $\V_{n-1}(\proj_{u^\perp}g,\proj_{u^\perp}P_r,
\proj_{u^\perp}\mathcal{P}_{\backslash r})=0$
for all $u\in S^{n-1}$.
\item $\V_{n-1}(\proj_{u^\perp}g,\proj_{u^\perp}g,
\proj_{u^\perp}\mathcal{P}_{\backslash r})=0$ for all
$u\in S^{n-1}\cap\mathcal{L}_r$.
\item $S_{\proj_{u^\perp}P_r,\proj_{u^\perp}\mathcal{P}_{\backslash r}}\ne 
0$ for all $u\in S^{n-1}$.
\end{enumerate}
\end{lem}

\begin{proof}
By Theorem \ref{thm:localaf}, there exists $g=f$ 
$S_{B,\mathcal{P}}$-a.e.\ such that
$S_{g,\mathcal{P}}=0$ and $S_{g,g,\mathcal{P}_{\backslash r}}\le 0$.
We must show that each of the claimed properties holds for $g$.
The proof of properties $1$--$3$ is identical to the proof of these 
properties in Lemma~\ref{lem:superpproj}. 
To prove property $4$, recall that 
$S_{\proj_{u^\perp}P_r,\proj_{u^\perp}\mathcal{P}_{\backslash r}}= 
(n-1)\,S_{[0,u],\mathcal{P}}$ (cf.\ Remark \ref{rem:lowdimma}). If there 
were to exist $u\in S^{n-1}$ so that $S_{[0,u],\mathcal{P}}=0$, then 
integrating against $h_B$ and using \eqref{eq:mixvolarea} would yield 
$\V_n([0,u],B,\mathcal{P})=0$. But this contradicts the criticality 
assumption by Lemma \ref{lem:dim}, concluding the proof.
\end{proof}

We would like to exploit Lemma \ref{lem:critproj} by applying Theorem 
\ref{thm:crit} (or Theorem~\ref{thm:schneider}) in $u^\perp$. In order to 
do this, we must understand what happens to the criticality assumption 
under projection onto $u^\perp$. We will presently show that such a 
projection preserves not just the criticality assumption, but even the 
collection of maximal sets, for almost every choice of $u$. To this end, 
we define
$$
        N := \bigcup_{\substack{%
	\alpha\subseteq[n-2]\backslash\{r\}: \\ 
        \dim\mathcal{L}_\alpha\le |\alpha|+2,~\mathcal{L}_r\not\subseteq
	\mathcal{L}_\alpha}}
        S^{n-1}\cap\mathcal{L}_r\cap\mathcal{L}_\alpha,
        \qquad\quad
        U := (S^{n-1}\cap\mathcal{L}_r)\backslash N
$$
in the remainder of this section. Then we have the following.

\begin{lem}
\label{lem:umeasure1}
The following hold for every $u\in U$:
\vspace{.5\abovedisplayskip}
\begin{enumerate}[a.]
\itemsep\abovedisplayskip  
\item $\proj_{u^\perp}\mathcal{P}_{\backslash r}$ is critical.
\item The $\proj_{u^\perp}\mathcal{P}_{\backslash r}$-maximal sets
are precisely $\gamma\backslash\{r\},\alpha_1,\ldots,\alpha_\ell$.
\item The map $\proj_{u^\perp}|_{\mathcal{L}_{\alpha_i}}:
\mathcal{L}_{\alpha_i}\to\proj_{u^\perp}\mathcal{L}_{\alpha_i}$ is
a bijection for $i=1,\ldots,\ell$.
\item $U$ has full measure with respect to the
uniform measure on $S^{n-1}\cap\mathcal{L}_r$.
\end{enumerate}
\vspace{.5\abovedisplayskip}
If $\gamma=\{r\}$ is a singleton, then part $b$ should be understood to 
say that the $\proj_{u^\perp}\mathcal{P}_{\backslash r}$-maximal sets
are precisely $\alpha_1,\ldots,\alpha_\ell$.
\end{lem}

\begin{proof}
To prove part $a$, suppose that
$\proj_{u^\perp}\mathcal{P}_{\backslash r}$ is not critical. Then we must 
have
$\dim(\sum_{i\in\alpha}\proj_{u^\perp}P_i)<|\alpha|+1$ for some 
$\alpha\subseteq[n-2]\backslash\{r\}$, $\alpha\ne\varnothing$. But as
$\mathcal{P}$ is critical, we have 
$\dim(\sum_{i\in\alpha}P_i)\ge|\alpha|+1$.
This can happen only if
$\dim(\sum_{i\in\alpha}P_i)=|\alpha|+1$ and $u\in\mathcal{L}_\alpha$. 
Thus $\alpha$ is a critical set with
$r\not\in\alpha$, so that
$\mathcal{L}_r\not\subseteq\mathcal{L}_\alpha$ by Lemma 
\ref{lem:critinc}. Therefore $u\not\in U$ by the definition of $U$, which 
entails a contradiction.

To prove part $b$, we must prove two distinct properties: that
$\gamma\backslash\{r\},\alpha_1,\ldots,\alpha_\ell$ are
$\proj_{u^\perp}\mathcal{P}_{\backslash r}$-maximal sets, and that no 
other
$\proj_{u^\perp}\mathcal{P}_{\backslash r}$-maximal sets exist.

\begin{claim}
$\alpha_i$ is $\proj_{u^\perp}\mathcal{P}_{\backslash r}$-maximal
for all $i=1,\ldots,\ell$.
\end{claim}

\begin{proof}
Fix a maximal set $\alpha=\alpha_i$ for some $i=1,\ldots,\ell$.
Then $\alpha$ is disjoint from $\gamma\ni r$ by Corollary \ref{cor:panov}, 
so $\alpha\subseteq [n-2]\backslash\{r\}$. Therefore
$$	
	|\alpha|+1\le\dim({\textstyle\sum_{i\in\alpha}\proj_{u^\perp}P_i})\le
	\dim({\textstyle\sum_{i\in\alpha}P_i})=|\alpha|+1,
$$
where the first inequality holds 
by part $a$, and the equality holds as $\alpha$ is critical. Thus we have 
shown that $\alpha$ is $\proj_{u^\perp}\mathcal{P}_{\backslash 
r}$-critical.

To show $\alpha$ is $\proj_{u^\perp}\mathcal{P}_{\backslash r}$-maximal, 
we must show there does not exist a 
$\proj_{u^\perp}\mathcal{P}_{\backslash r}$-critical set 
$\beta\subseteq[n-2]\backslash\{r\}$ so that
$\alpha\subsetneq\beta$. Indeed, suppose such 
a set $\beta$ does exist. Then
$\dim(\sum_{i\in\beta}\proj_{u^\perp}P_i)=|\beta|+1$, so there are
two possibilities:
\begin{enumerate}[i.]
\item
$\dim(\sum_{i\in\beta}P_i)=|\beta|+1$ 
and $u\not\in\mathcal{L}_\beta$; or
\item $\dim(\sum_{i\in\beta}P_i)=|\beta|+2$ 
and $u\in\mathcal{L}_\beta$. 
\end{enumerate}
In case i, we have $\alpha\subsetneq\beta$ for a critical set $\beta$,
which contradicts the maximality of $\alpha$. On the other hand,
in case ii, we must have $\mathcal{L}_r\subseteq\mathcal{L}_\beta$ 
by the definition of $U\ni u$, so that $\dim\mathcal{L}_{\beta\cup\{r\}}=
\dim\mathcal{L}_\beta=|\beta|+2$. Thus in this case 
$\alpha\subsetneq\beta\cup\{r\}$ and $\beta\cup\{r\}$ is a critical set, 
which contradicts again the maximality of $\alpha$.
\end{proof}

\begin{claim}
If $\gamma\backslash\{r\}\ne\varnothing$, then
$\gamma\backslash\{r\}$ is $\proj_{u^\perp}\mathcal{P}_{\backslash 
r}$-maximal.
\end{claim}

\begin{proof}
As $r\in\gamma$, we have $u\in\mathcal{L}_r\subseteq\mathcal{L}_\gamma$
by the definition of $U$. Therefore
$$	
	|\gamma|=|\gamma\backslash\{r\}|+1
	\le
	\dim({\textstyle\sum_{i\in\gamma\backslash\{r\}}\proj_{u^\perp}P_i})
	\le
	\dim\mathcal{L}_\gamma-1=|\gamma|,
$$
where the first inequality holds by part $a$, the second inequality 
holds as $u\in\mathcal{L}_\gamma$, and the last equality holds as $\gamma$ 
is critical. Thus $\gamma\backslash\{r\}$ is
$\proj_{u^\perp}\mathcal{P}_{\backslash r}$-critical.

Now suppose $\gamma\backslash\{r\}\subsetneq\beta\subseteq
[n-2]\backslash\{r\}$ for some
$\proj_{u^\perp}\mathcal{P}_{\backslash r}$-critical set $\beta$. Then
one of the 
possibilities i or ii in the proof of the previous claim must hold.

In case i, let $\beta'$ be the maximal set containing the critical set
$\beta$. As $\beta\backslash\gamma\ne\varnothing$, we must have 
$\gamma\ne\beta'$. But as
$\gamma\cap\beta'\supseteq\gamma\backslash\{r\}\ne\varnothing$,
this contradicts maximality of $\gamma$ by Corollary \ref{cor:panov}.
On the other hand,
in case ii, we must have $\mathcal{L}_r\subseteq\mathcal{L}_\beta$
by the definition of $U\ni u$, so that $\dim\mathcal{L}_{\beta\cup\{r\}}=
\dim\mathcal{L}_\beta=|\beta|+2$. Thus 
$\gamma\subsetneq\beta\cup\{r\}$ and $\beta\cup\{r\}$ is a critical set,  
contradicting again the maximality of $\gamma$. 
\end{proof}

\begin{claim}
There exists no
$\proj_{u^\perp}\mathcal{P}_{\backslash r}$-critical set
$\beta\subseteq [n-2]\backslash
(\gamma\cup\alpha_1\cup\cdots\cup\alpha_\ell)$.
\end{claim}

\begin{proof}
Suppose $\beta\subseteq [n-2]\backslash
(\gamma\cup\alpha_1\cup\cdots\cup\alpha_\ell)$ is a
$\proj_{u^\perp}\mathcal{P}_{\backslash r}$-critical set.
Then one of the possibilities i or ii in the proofs of the previous claims 
must hold. Case i is impossible, as it would imply that $\beta$ is a 
critical set that is disjoint from all maximal sets.
In case ii, we must have $\mathcal{L}_r\subseteq\mathcal{L}_\beta$
by the definition of $U\ni u$, so that $\dim\mathcal{L}_{\beta\cup\{r\}}=
\dim\mathcal{L}_\beta=|\beta|+2$. 
Thus $\beta\cup\{r\}$ is a critical set containing $r$. But this would 
imply by Corollary 
\ref{cor:panov} that $\beta\cup\{r\}\subseteq\gamma$, which is impossible 
as $\beta$ and $\gamma$ are disjoint. Thus we have shown the desired 
contradiction.
\end{proof}

Now recall that distinct maximal sets must be disjoint by Corollary 
\ref{cor:panov}. Thus the combination of the above three claims
concludes the proof of part $b$.

To prove part $c$, it suffices to note that for any $i=1,\ldots,\ell$,
$$
	\dim\mathcal{L}_{\alpha_i} =
	|\alpha_i|+1
	=
	\dim\proj_{u^\perp}\mathcal{L}_{\alpha_i},
$$
where the first equality holds as $\alpha_i$ is a critical set and the 
second equality holds as $\alpha_i$ is a
$\proj_{u^\perp}\mathcal{P}_{\backslash 
r}$-critical set by part $b$. The conclusion follows immediately.

Finally, we prove part $d$. To this end, note that by definition, $N$ lies 
in a union of subspaces $\mathcal{L}_\alpha\not\supseteq\mathcal{L}_r$, 
so that $\dim(\mathcal{L}_\alpha\cap\mathcal{L}_r)<\dim\mathcal{L}_r$ for 
each of these spaces.
In other words, $N$ is the intersection of 
$S^{n-1}\cap\mathcal{L}_r$ with hyperplanes of codimension at least one, 
and is therefore a set of zero measure.
\end{proof}

In the rest of this section, we fix a difference of support functions 
$f$ with $S_{f,\mathcal{P}}=0$, and construct the difference of support 
functions $g$ as in Lemma \ref{lem:critproj}. Then Lemmas 
\ref{lem:critproj} and \ref{lem:umeasure1} ensure that the projection 
$\proj_{u^\perp}g$ yields a critical equality case \eqref{eq:eq3af} of the 
Alexandrov-Fenchel inequality in dimension $n-1$ for every $u\in U$.

We would like to combine this fact with Theorem \ref{thm:crit} in 
dimension $n-1$ to create the induction hypothesis for its proof in 
dimension $n$. In the critical case, however, there is a new subtlety: 
applying Theorem \ref{thm:crit} in $u^\perp$ yields 
$\proj_{u^\perp}\mathcal{P}_{\backslash r}$-degenerate functions, while we 
must construct $\mathcal{P}$-degenerate functions to prove Theorem 
\ref{thm:crit} in dimension $n$. We will address this problem by using 
condition $c$ of Lemma \ref{lem:umeasure1} to map between the two types of 
degenerate functions. In addition, we must handle the new term that may 
now arise from part $b$ of Lemma \ref{lem:afeq}, which did not appear in 
the supercritical case. Both issues will be resolved presently.

\begin{lem}
\label{lem:critinduct}
Suppose that Theorem \ref{thm:crit} has been proved in dimension
$n-1$. Then for every $u\in U$, we have
$$
        g(x) = \langle s(u),x\rangle+
	\sum_{j=0}^\ell g_{j,u}(x) \quad
        \mbox{for all }x\in\supp S_{[0,u],B,\mathcal{P}_{\backslash r}}
$$
for some $s(u)\in u^\perp$ and
$\alpha_j$-degenerate function $g_{j,u}$, $j=0,\ldots,\ell$.
\end{lem}

\begin{proof}
Fix $u\in U$. Applying Lemma \ref{lem:afeq} in $u^\perp$ and Lemma
\ref{lem:critproj}, we find that
$$
	S_{\proj_{u^\perp}g-a_uh_{\proj_{u^\perp}P_r},
	\proj_{u^\perp}\mathcal{P}_{\backslash r}}=0
	\quad\mbox{for some }a_u\in\mathbb{R}.
$$
Now note that $\proj_{u^\perp}\mathcal{P}_{\backslash r}$ is critical
with critical sets $\gamma\backslash\{r\},\alpha_1,\ldots,\alpha_\ell$
by Lemma \ref{lem:umeasure1}. If $\ell=0$ and $\gamma=\{r\}$, 
then $\proj_{u^\perp}\mathcal{P}_{\backslash r}$ is supercritical and we 
may apply Theorem \ref{thm:schneider} in $u^\perp$; otherwise we may apply 
Theorem \ref{thm:crit} in $u^\perp$. In either case
$$
        \proj_{u^\perp}g(x) = \langle s(u),x\rangle 
	+ a_uh_{\proj_{u^\perp}P_r}(x)
	+ \sum_{j=0}^\ell \tilde g_{j,u}(x)
	\quad
        \mbox{for all }x\in\supp S_{\proj_{u^\perp}B,\proj_{u^\perp}
        \mathcal{P}_{\backslash r}}
$$
holds for some $s(u)\in u^\perp$, 
$(\proj_{u^\perp}\mathcal{P}_{\backslash r},\alpha_j)$-degenerate 
functions $\tilde g_{j,u}$, $j=1,\ldots,\ell$, and 
$(\proj_{u^\perp}\mathcal{P}_{\backslash r},
\gamma\backslash\{r\})$-degenerate function $\tilde g_{0,u}$ (if 
$\gamma=\{r\}$, we simply set $\tilde g_{0,u}=0$).
But as $S_{\proj_{u^\perp}B,\proj_{u^\perp}\mathcal{P}_{\backslash r}}$ is
supported in $u^\perp$ by definition, we may remove
$\proj_{u^\perp}$ on the left-hand side of this identity. We therefore 
obtain
\begin{equation}
\label{eq:interimcritg}
        g(x) = \langle s(u),x\rangle 
	+ a_uh_{P_r}(x)
	+ \sum_{j=0}^\ell \tilde g_{j,u}(x)
	\quad
        \mbox{for all }x\in\supp S_{[0,u],B,\mathcal{P}_{\backslash r}}
\end{equation}
as $\supp S_{\proj_{u^\perp}B,\proj_{u^\perp}
\mathcal{P}_{\backslash r}}=\supp S_{[0,u],B,\mathcal{P}_{\backslash r}}$
by Remark \ref{rem:lowdimma}.

It remains to reformulate the above identity in terms of 
$\alpha_j$-degenerate (as opposed to 
$(\proj_{u^\perp}\mathcal{P}_{\backslash r},\alpha_j)$-degenerate) 
functions. To this end, we consider two cases.

\begin{claim}
For every $j=1,\ldots,\ell$, there exists an
$\alpha_j$-degenerate function $g_{j,u}$ so that
$$
	g_{j,u}(x)=\tilde g_{j,u}(x)
	\quad\mbox{for all }x\in u^\perp.
$$ 
\end{claim}

\begin{proof}
Let $\tilde g_{j,u}=h_{\tilde M}-h_{\tilde N}$ for a 
$(\proj_{u^\perp}\mathcal{P}_{\backslash r},\alpha_j)$-degenerate pair
$(\tilde M,\tilde N)$. In particular, $\tilde M,\tilde N\subset
\proj_{u^\perp}\mathcal{L}_{\alpha_j}$. Therefore, by part $c$ of Lemma 
\ref{lem:umeasure1}, we can uniquely define convex bodies
$M,N\subset\mathcal{L}_{\alpha_j}$ so that
$\tilde M=\proj_{u^\perp}M$ and $\tilde N=\proj_{u^\perp}N$. 
We claim that the function $g_{j,u}:=h_M-h_N$ satisfies the requisite 
properties.

Indeed, note that $g_{j,u}(\proj_{u^\perp}x)=\tilde g_{j,u}(x)$ by 
construction, so $g_{j,u}(x)=\tilde g_{j,u}(x)$ for all $x\in u^\perp$.
On the other hand, note that
for any convex body $K\subset\mathcal{L}_{\alpha_j}$
$$
	\V_{\proj_{u^\perp}\mathcal{L}_{\alpha_j}}(\proj_{u^\perp}K,
	\proj_{u^\perp}\mathcal{P}_{\alpha_j}) =
	\llbracket\proj_{u^\perp}|_{\mathcal{L}_{\alpha_j}}\rrbracket\,
	\V_{\mathcal{L}_{\alpha_j}}(K,\mathcal{P}_{\alpha_j}),
$$
by part $f$ of Lemma \ref{lem:mvprop}.
As $(\tilde M,\tilde N)$ is a $(\proj_{u^\perp}\mathcal{P}_{\backslash
r},\alpha_j)$-degenerate pair, we obtain
$$
	\llbracket\proj_{u^\perp}|_{\mathcal{L}_{\alpha_j}}\rrbracket\,
        \V_{\mathcal{L}_{\alpha_j}}(h_M-h_N,\mathcal{P}_{\alpha_j})
	=
	\V_{\proj_{u^\perp}\mathcal{L}_{\alpha_j}}(
	h_{\tilde M}-h_{\tilde N},
	\proj_{u^\perp}\mathcal{P}_{\alpha_j}) = 0.
$$
As $\llbracket\proj_{u^\perp}|_{\mathcal{L}_{\alpha_j}}\rrbracket>0$ by
part $c$ of Lemma
\ref{lem:umeasure1}, and as $M,N\subset \mathcal{L}_{\alpha_j}$,
we have verified that $(M,N)$ is an $\alpha_j$-degenerate 
pair. Thus $g_{j,u}$ is an $\alpha_j$-degenerate function.
\end{proof}

\begin{claim}
There exists a $\gamma$-degenerate function $g_{0,u}$ so that
$$
	g_{0,u}(x) = a_u h_{P_r}(x) +
	\tilde g_{0,u}(x)\quad\mbox{for all }x\in u^\perp.
$$
\end{claim}

\begin{proof}
Fix $b\in\mathbb{R}$, and define $g_{0,u}$ as
$$
	g_{0,u}:= a_u h_{P_r}+\tilde g_{0,u} + b h_{[0,u]}.
$$
Clearly $h_{[0,u]}(x)=\langle u,x\rangle_+=0$ and thus
$g_{0,u}(x)=a_uh_{P_r}(x)+\tilde g_{0,u}(x)$ for $x\in u^\perp$. We claim 
$g_{0,u}$ is $\gamma$-degenerate for a suitable choice of 
$b\in\mathbb{R}$.

If $\tilde g_{0,u}\ne 0$, then $\tilde g_{0,u}=h_M-h_N$ for a 
$(\proj_{u^\perp}\mathcal{P}_{\backslash 
r},\gamma\backslash\{r\})$-degenerate pair 
$(M,N)$. As $r\in\gamma$, we have 
$u\in\mathcal{L}_r\subseteq\mathcal{L}_\gamma$ by the definition of $U$, 
so
$M,N\subset 
\proj_{u^\perp}\mathcal{L}_{\gamma\backslash\{r\}}\subset\mathcal{L}_\gamma$.
But as $r\in\gamma$, we always have 
$[0,u],P_r\subset\mathcal{L}_\gamma$. Thus $g_{0,u}$ is a 
difference of support functions of convex bodies in $\mathcal{L}_\gamma$.
On the other hand, we can choose $b\in\mathbb{R}$ so that
$$
	\V_{\mathcal{L}_\gamma}(g_{0,u},\mathcal{P}_\gamma)=
	\V_{\mathcal{L}_\gamma}(a_u h_{P_r}+\tilde g_{0,u},
	\mathcal{P}_\gamma) +
	b\, \V_{\mathcal{L}_\gamma}([0,u],\mathcal{P}_\gamma)
	=0
$$
as $\V_{\mathcal{L}_\gamma}([0,u],\mathcal{P}_\gamma)>0$ by the 
criticality assumption and Lemma \ref{lem:dim}. Thus we have shown 
that $g_{0,u}$ is $\gamma$-degenerate, completing the proof.
\end{proof}

As $\supp S_{[0,u],B,\mathcal{P}_{\backslash r}}\subset u^\perp$,
the conclusion of Lemma \ref{lem:critinduct} follows 
readily by combining the above two claims with the identity 
\eqref{eq:interimcritg}.
 \end{proof}

\subsection{The decoupling argument}
\label{sec:decouplingarg}

With Lemma \ref{lem:critinduct} in hand, the main difficulty we face is to 
remove the dependence of $s(u)$ and $g_{j,u}$ on $u$. We therefore begin, 
as in the supercritical case, by investigating the overlap between the 
supports of the measures $S_{[0,u],B,\mathcal{P}_{\backslash r}}$ for 
different $u\in U$. Surprisingly, the situation in the critical case 
proves to be completely different than in the supercritical case: the 
overlap between the supports does not depend on the choice of $u$.

In the remainder of this section, we define the polytope
$$
	P_\gamma := \sum_{i\in\gamma}P_i.
$$
We state at the outset a simple technical lemma that will be used several 
times.

\begin{lem}
\label{lem:posmixcrit}
$\V_n(K,P_\gamma,\mathcal{P})>0$ whenever
$\dim(K)\ge 1$ and $0\in K\not\subset\mathcal{L}_\gamma$.
\end{lem}

\begin{proof}
Suppose $\V_n(K,P_\gamma,\mathcal{P})=0$. As $\mathcal{P}$ is critical 
(which implies $\dim P_\gamma\ge 2$), we must have
$\dim(K+P_\gamma+\sum_{i\in\alpha}P_i)\le|\alpha|+1$ for some
$\alpha\subseteq[n-2]$ by Lemma \ref{lem:dim}.
As $\dim(\sum_{i\in\alpha}P_i)\ge|\alpha|+1$
by the criticality assumption, this can happen only if
$\alpha$ is a critical set and $K,P_\gamma\subset\mathcal{L}_\alpha$.
But $P_\gamma\subset\mathcal{L}_\alpha$ implies $\gamma\subseteq\alpha$
by Lemma \ref{lem:critinc}. Thus $\alpha=\gamma$ by the 
maximality of $\gamma$, which contradicts the assumption
$K\not\subset\mathcal{L}_\gamma$.
\end{proof}

We can now formulate the key property of  
$\supp S_{[0,u],B,\mathcal{P}_{\backslash r}}$.

\begin{lem}
\label{lem:critcommon}
For every $u\in U$, we have
$$
	\supp 
	S_{P_\gamma,\mathcal{P}}
	\subseteq \supp S_{[0,u],B,\mathcal{P}_{\backslash r}},
$$
and
$$
	\sspan\supp 
	S_{P_\gamma,\mathcal{P}} =
	\mathcal{L}_\gamma^\perp.
$$
\end{lem}

\begin{proof}
We begin by noting that $\supp S_{[0,u],P_\gamma,\mathcal{P}_{\backslash 
r}}\subseteq \supp S_{[0,u],B,\mathcal{P}_{\backslash r}}$ by
Lemma \ref{lem:maxsupp}. As
$u\in\mathcal{L}_r\subseteq\mathcal{L}_\gamma$ by the definition of $U$, 
we have $[0,u],P_\gamma\subset\mathcal{L}_\gamma$, so that
$$
	\dim(\textstyle{[0,u]+P_\gamma+\sum_{i\in\gamma\backslash\{r\}}P_i})=
	\dim\mathcal{L}_\gamma=|\gamma|+1.
$$
Applying Lemma \ref{lem:proj} as in Remark \ref{rem:lowdimma} yields
\begin{align*}
	&{n-1\choose |\gamma|+1}\,
	S_{[0,u],P_\gamma,\mathcal{P}_{\backslash r}}
	=
	\V_{\mathcal{L}_\gamma}([0,u],P_\gamma,
	\mathcal{P}_{\gamma\backslash\{r\}})
	\,
	S_{\proj_{\mathcal{L}_\gamma^\perp}\mathcal{P}_{\backslash\gamma}},
	\\
	&{n-1\choose |\gamma|+1}\,
	S_{P_\gamma,\mathcal{P}}
	=
	\V_{\mathcal{L}_\gamma}(P_\gamma,
	\mathcal{P}_{\gamma})
	\,
	S_{\proj_{\mathcal{L}_\gamma^\perp}\mathcal{P}_{\backslash\gamma}}.
\end{align*}
But as $\mathcal{P}$ is 
critical, $(P_\gamma,\mathcal{P}_{\backslash r})$ is critical as well, so 
$\V_{\mathcal{L}_\gamma}([0,u],P_\gamma, 
\mathcal{P}_{\gamma\backslash\{r\}})>0$ and 
$\V_{\mathcal{L}_\gamma}(P_\gamma, \mathcal{P}_{\gamma})>0$ by Lemma 
\ref{lem:dim}. Thus we have proved the first claim.

Now note that the second identity above implies
$\sspan\supp S_{P_\gamma,\mathcal{P}}\subseteq\mathcal{L}_\gamma^\perp$. 
If the inclusion were to be strict, 
then $\supp S_{P_\gamma,\mathcal{P}}\subset w^\perp$ for some
$w\in S^{n-1}\cap\mathcal{L}_\gamma^\perp$, so
$$
        0 =
	\int \langle w,x\rangle_+\,S_{P_\gamma,\mathcal{P}}(dx) =
	n\,\V_n([0,w],P_\gamma,\mathcal{P})
$$
using $h_{[0,w]}(x)=\langle w,x\rangle_+$ and \eqref{eq:mixvolarea}. As 
$[0,w]\not\subset\mathcal{L}_\gamma$, this 
entails a contradiction by Lemma \ref{lem:posmixcrit}. Thus the second
claim is proved.
\end{proof}

Combining the above results, we conclude the following.

\begin{cor}
\label{cor:crittogether}
If the conclusion of Lemma \ref{lem:critinduct} holds, then 
for any $u,v\in U$
$$
	\langle s(u)-s(v),x\rangle +
	\sum_{j=1}^\ell \{g_{j,u}(x)-g_{j,v}(x)\} = 0
	\quad\mbox{for all }x\in\supp 
	S_{P_\gamma,\mathcal{P}}.
$$
\end{cor}

\begin{proof}
As $g_{0,u}$ is a $\gamma$-degenerate function,
$g_{0,u}(x)=g_{0,u}(\proj_{\mathcal{L}_{\gamma}}x)=0$ for all
$x\in\mathcal{L}_\gamma^\perp$ by Lemma \ref{lem:deghomog}.
The claim follows immedately from Lemmas \ref{lem:critinduct} and 
\ref{lem:critcommon}.
\end{proof}

Let us emphasize that the conclusion of Corollary \ref{cor:crittogether} 
is of a fundamentally different nature than the analogous statement in the 
supercritical case: as the mixed area measure that appears here does not 
depend on $u,v$, there is no need to ``glue'' the functions $\langle 
s(u),\cdot\rangle+\sum_{j=1}^\ell g_{j,u}$ for different $u$. (We will 
still have to solve a gluing problem for $g_{0,u}$, which is addressed in 
the next section.) The problem we face here is that Corollary 
\ref{cor:crittogether} only provides information on 
$\supp S_{P_\gamma,\mathcal{P}}$, while we must characterize these 
functions on $\supp S_{B,\mathcal{P}}$ to prove Theorem \ref{thm:crit}. 
As a first step towards this goal, let us make the following 
basic observation.

\begin{lem}
\label{lem:extend}
Let $h$ be any $\alpha_i$-degenerate function for some $i=1,\ldots,\ell$.
Then the following statements are equivalent:
\begin{enumerate}[a.]
\item $h(x)=0$ for all $x\in\supp S_{P_\gamma,\mathcal{P}}$.
\item $h(x)=0$ for all $x\in\supp S_{\mathcal{P}_{\alpha_i}}$.
\item $h(x)=0$ for all $x\in\supp S_{B,B,\mathcal{P}_{\backslash r}}$.
\end{enumerate}
\end{lem}

\begin{proof}
Lemmas \ref{lem:maproj} and \ref{lem:deghomog} imply that
\begin{align*}
	&{n-1\choose |\alpha_i|}
	\int |h|\,dS_{B,B,\mathcal{P}_{\backslash r}} =
	\V_{\mathcal{L}_{\alpha_i}^\perp}(\proj_{\mathcal{L}_{\alpha_i}^\perp}B,
	\proj_{\mathcal{L}_{\alpha_i}^\perp}B,
	\proj_{\mathcal{L}_{\alpha_i}^\perp}\mathcal{P}_{\backslash
	\{\alpha_i,r\}})
	\int |h| \,dS_{\mathcal{P}_{\alpha_i}},
	\\
	&{n-1\choose |\alpha_i|}
	\int |h|\,dS_{P_\gamma,\mathcal{P}} =
	\V_{\mathcal{L}_{\alpha_i}^\perp}(\proj_{\mathcal{L}_{\alpha_i}^\perp}
	P_\gamma,
	\proj_{\mathcal{L}_{\alpha_i}^\perp}\mathcal{P}_{\backslash\alpha_i})
	\int |h| \,dS_{\mathcal{P}_{\alpha_i}}.
\end{align*}
Thus the 
conclusion follows provided the two mixed volumes in the above identities
are positive. To show that this is in fact the case, we note that
$$
	\V_{\mathcal{L}_{\alpha_i}}(B_{\alpha_i},\mathcal{P}_{\alpha_i})\,
	\V_{\mathcal{L}_{\alpha_i}^\perp}(\proj_{\mathcal{L}_{\alpha_i}^\perp}
	P_\gamma,
	\proj_{\mathcal{L}_{\alpha_i}^\perp}\mathcal{P}_{\backslash\alpha_i})
	=
	{n\choose |\alpha_i|+1}
	\V_n(B_{\alpha_i},P_\gamma,\mathcal{P})
	>0
$$
by Lemmas \ref{lem:proj} and \ref{lem:posmixcrit} 
(here we used $B_{\alpha_i}\not\subset \mathcal{L}_{\gamma}$
by Corollary \ref{cor:panov}). Therefore
$$
	c\,
	\V_{\mathcal{L}_{\alpha_i}^\perp}(\proj_{\mathcal{L}_{\alpha_i}^\perp}B,
        \proj_{\mathcal{L}_{\alpha_i}^\perp}B,
        \proj_{\mathcal{L}_{\alpha_i}^\perp}\mathcal{P}_{\backslash
        \{\alpha_i,r\}})
	\ge
	\V_{\mathcal{L}_{\alpha_i}^\perp}(\proj_{\mathcal{L}_{\alpha_i}^\perp}
	P_\gamma,
	\proj_{\mathcal{L}_{\alpha_i}^\perp}\mathcal{P}_{\backslash\alpha_i})
	>0,
$$
for some $c>0$, where we used $P_\gamma\subset\mathrm{diam}(P_\gamma)B$ and
$P_r\subset \mathrm{diam}(P_r)B$.
\end{proof}

The difficulty in the application of Lemma \ref{lem:extend} is that it 
applies only to an individual $\alpha_i$-degenerate function, which is 
used crucially in its proof. On the other hand, Corollary 
\ref{cor:crittogether} involves a sum of degenerate functions for 
different sets $\alpha_i$. The problem we must now address is 
therefore to decouple the different terms of the sum in Corollary 
\ref{cor:crittogether}. This will be accomplished using Proposition 
\ref{prop:proplin}. 

\begin{prop}
\label{prop:decouplingmagic}
Let $t\in\mathbb{R}^n$, and let $f_j$ be an $\alpha_j$-degenerate function 
for every $j=1,\ldots,\ell$. Suppose that we have
\begin{equation}
\label{eq:decoupleaspt}
	\langle t,x\rangle + 
	\sum_{j=1}^\ell f_j(x)=0\quad
	\mbox{for all }x\in\supp S_{P_\gamma,\mathcal{P}}.
\end{equation}
Then for every $j=1,\ldots,\ell$,
there exists $w_j\in\mathcal{L}_{\alpha_j}$ so that
$$
	f_j(x) = \langle w_j,x\rangle\quad\mbox{for all }
	x\in\supp S_{B,B,\mathcal{P}_{\backslash r}}
$$
and $t+w_1+\cdots+w_\ell\in\mathcal{L}_\gamma$.
\end{prop}

\begin{proof}
Fix $i\in[\ell]$ until further notice.

\medskip

\textbf{Step 1.}
We aim to apply Proposition \ref{prop:proplin} with
$(C_1,\ldots,C_{n-1})\leftarrow (P_\gamma,\mathcal{P})$,
$(C_1,\ldots,C_k)\leftarrow \mathcal{P}_{\alpha_i}$, and
$E \leftarrow \mathcal{L}_{\alpha_i}$
to the function
$$
	h(x) := \langle t,x\rangle +
	\sum_{j\ne i} f_j(x).
$$
To this end, let us verify the assumptions of Proposition 
\ref{prop:proplin} are satisfied. 
The requisite criticality assumptions follow immediately as $\mathcal{P}$ 
is critical and $\alpha_i$ is a critical set.
Now note that as $f_i$ is $\alpha_i$-degenerate, it follows
from Lemma \ref{lem:deghomog} that $f_i(x)=0$ for 
$x\in\mathcal{L}_{\alpha_i}^\perp$. Thus \eqref{eq:decoupleaspt} 
implies that
$$
	h(x)=0\quad\mbox{for all }x\in \mathcal{L}_{\alpha_i}^\perp
	\cap\supp S_{P_\gamma,\mathcal{P}}.
$$
On the other hand, note that
by Lemmas \ref{lem:maproj} and \ref{lem:deghomog}, we have
$$
	{n-1\choose |\alpha_j|}
        \int f_j\,dS_{\mathcal{Q}_{\alpha_i},P_\gamma,
	\mathcal{P}_{\backslash\alpha_i}} =
	\V_{\mathcal{L}_{\alpha_j}^\perp}(
	\proj_{\mathcal{L}_{\alpha_j}^\perp}\mathcal{Q}_{\alpha_i},
	\proj_{\mathcal{L}_{\alpha_j}^\perp}P_\gamma,
	\proj_{\mathcal{L}_{\alpha_j}^\perp}\mathcal{P}_{\backslash\{\alpha_i,
	\alpha_j\}}
	)
	\int f_j\,dS_{\mathcal{P}_{\alpha_j}}
$$
for any $j\ne i$ and convex bodies $\mathcal{Q}_{\alpha_i}=
(Q_l)_{l\in\alpha_i}$ in $\mathcal{L}_{\alpha_i}$.
But the right-hand side 
vanishes as $\V_{\mathcal{L}_{\alpha_j}}(f_j,\mathcal{P}_{\alpha_j})=0$ by 
the definition of an $\alpha_j$-degenerate function; thus
$$
	\int h\,dS_{Q,\ldots,Q,P_\gamma,
        \mathcal{P}_{\backslash\alpha_i}}=0
$$
for any full-dimensional polytope $Q$ in $\mathcal{L}_{\alpha_i}$, where
we used that the integral of the linear part of $h$ vanishes by Lemma 
\ref{lem:maprop}. The assumptions of Proposition \ref{prop:proplin} 
are therefore satisfied. Consequently, there exists 
$w_i'\in\mathcal{L}_{\alpha_i}$ so that
\begin{equation}
\label{eq:linreloutput}
	\int h(x)\,1_{\langle z,x\rangle>0}
	\,S_{\proj_{F_z}P_\gamma,
	\proj_{F_z}\mathcal{P}_{\backslash\alpha_i}}(dx)=
	\langle w_i',z\rangle
\end{equation}
for all $z\in S^{n-1}\cap\mathcal{L}_{\alpha_i}$, where
$F_z := \sspan\{z,\mathcal{L}_{\alpha_i}^\perp\}$.

\medskip

\textbf{Step 2.}
Now note that $\proj_{\mathcal{L}_{\alpha_i}}x=\langle x,z\rangle z$
for every $x\in F_z$. Thus
$$
	f_i(x) = f_i(z)\langle z,x\rangle
	\quad\mbox{for all }x\in F_z^+
$$
by Lemma \ref{lem:deghomog}. On the other hand, \eqref{eq:decoupleaspt}
and Theorem \ref{thm:propeller} imply that
$$
	h(x)+f_i(x)=0\quad
	\mbox{for all }x\in\supp (1_{\langle z,\cdot\rangle>0}\,
	dS_{\proj_{F_z}P_\gamma,
        \proj_{F_z}\mathcal{P}_{\backslash\alpha_i}})
$$
holds for every $z\in\supp S_{\mathcal{P}_{\alpha_i}}$.
Substituting these identities
in \eqref{eq:linreloutput} yields
$$
	\V_{\mathcal{L}_{\alpha_i}^\perp}(
	\proj_{\mathcal{L}_{\alpha_i}^\perp}P_\gamma,
	\proj_{\mathcal{L}_{\alpha_i}^\perp}\mathcal{P}_{\backslash\alpha_i})
	\,f_i(z) =
        -\langle w_i',z\rangle
$$
for every $z\in\supp S_{\mathcal{P}_{\alpha_i}}$, where we used
Corollary \ref{cor:segproj}. But we already showed in the proof of Lemma 
\ref{lem:extend} that $\V_{\mathcal{L}_{\alpha_i}^\perp}(
\proj_{\mathcal{L}_{\alpha_i}^\perp}P_\gamma,        
\proj_{\mathcal{L}_{\alpha_i}^\perp}\mathcal{P}_{\backslash\alpha_i})>0$.
Therefore
$$
	f_i(z) = \langle w_i,z\rangle\quad
	\mbox{for all }z\in\supp S_{\mathcal{P}_{\alpha_i}},
$$
where $w_i := -\V_{\mathcal{L}_{\alpha_i}^\perp}(
\proj_{\mathcal{L}_{\alpha_i}^\perp}P_\gamma,
\proj_{\mathcal{L}_{\alpha_i}^\perp}\mathcal{P}_{\backslash\alpha_i})^{-1}
w_i'\in\mathcal{L}_{\alpha_i}$.

\medskip

\textbf{Step 3.} As $i\in[\ell]$ was arbitrary, we have now 
constructed for every $j=1,\ldots,\ell$ a vector 
$w_j\in\mathcal{L}_{\alpha_j}$ such that
$f_j(z)-\langle w_j,z\rangle=0$ for $z\in\supp 
S_{\mathcal{P}_{\alpha_j}}$. As
$f_j-\langle w_j,\cdot\rangle$
is an $\alpha_j$-degenerate function, we conclude by
Lemma \ref{lem:extend} that
$$
	f_j(x)=\langle w_j,x\rangle
	\quad\mbox{for all }x\in\supp S_{B,B,\mathcal{P}_{\backslash r}}
	\mbox{ and }x\in\supp S_{P_\gamma,\mathcal{P}}.
$$
In particular, we obtain using \eqref{eq:decoupleaspt}
$$
	\langle t+w_1+\cdots+w_\ell,x\rangle = 0
	\quad\mbox{for all }x\in\supp S_{P_\gamma,\mathcal{P}}.
$$
Thus $\proj_{\mathcal{L}_\gamma^\perp}(t+w_1+\cdots+w_\ell)=0$
by Lemma \ref{lem:critcommon}, completing the proof.
\end{proof}

We can now put together all the ideas of this section.

\begin{cor}
\label{cor:critnogamma}
Suppose that Theorem \ref{thm:crit} has been proved in dimension $n-1$. 
Then there exist $s\in\mathbb{R}^n$, an $\alpha_j$-degenerate function 
$g_j$ for every $j=1,\ldots,\ell$, and $t(u)\in\mathcal{L}_\gamma$ for 
every $u\in U$, so that the following holds for all $u\in U$:
$$
	g(x) -
	\langle s,x\rangle -
	\sum_{j=1}^\ell g_j(x) =
	\langle t(u),x\rangle +
	g_{0,u}(x)\quad
	\mbox{for all }x\in \supp S_{[0,u],B,\mathcal{P}_{\backslash r}}.
$$
\end{cor}

\begin{proof}
Fix any $v\in U$, and define $s:=s(v)$ and $g_j:=g_{j,v}$ for 
$j=1,\ldots,\ell$.
Then Lemma \ref{lem:critinduct}, Corollary \ref{cor:crittogether} and 
Proposition \ref{prop:decouplingmagic} yield 
$w_j(u)\in\mathcal{L}_{\alpha_j}$ 
so that
$$
	g_{j,u}(x)=g_j(x)+\langle w_j(u),x\rangle
	\quad\mbox{for all }x\in\supp S_{B,B,\mathcal{P}_{\backslash r}}
$$
and
$$
	t(u):=s(u)-s+w_1(u)+\cdots+w_\ell(u)\in\mathcal{L}_\gamma
$$
for every $u\in U$. As $\supp S_{[0,u],B,\mathcal{P}_{\backslash 
r}}\subseteq\supp S_{B,B,\mathcal{P}_{\backslash r}}$ for every $u$
by Lemma \ref{lem:maxsupp}, the conclusion follows immediately
from Lemma \ref{lem:critinduct}.
\end{proof}

\subsection{The gluing argument}

With Corollary \ref{cor:critnogamma} in hand, it remains to glue the 
$\gamma$-degenerate functions $g_{0,u}$ for different $u\in U$. This will 
be accomplished using Lemma \ref{lem:propglue}. We remind the 
reader that the function $f$ with $S_{f,\mathcal{P}}=0$ was fixed at the 
beginning of the proof, and that $g$ was constructed from $f$ by 
Lemma \ref{lem:critproj}.

\begin{lem}
\label{lem:crit}
Suppose that Theorem \ref{thm:crit} has been proved in dimension $n-1$. 
Then there exist $s\in\mathbb{R}^n$ and an $\alpha_j$-degenerate function 
$g_j$ for $j=0,\ldots,\ell$ so that
$$
	f(x) =
	\langle s,x\rangle +
	\sum_{j=0}^\ell g_j(x) \quad
	\mbox{for all }x\in \supp S_{B,\mathcal{P}}.
$$
\end{lem}

\begin{proof}
We use the notations of Corollary \ref{cor:critnogamma} 
throughout the proof.

\medskip

\textbf{Step 1.}
We begin by applying Lemma \ref{lem:propglue} with 
$(C_{k+1},\ldots,C_{n-1})\leftarrow (B,\mathcal{P}_{\backslash\gamma})$ 
and $E\leftarrow\mathcal{L}_\gamma$ to the function
$$
	h(x):=
	g(x) -
	\langle s,x\rangle -
	\sum_{j=1}^\ell g_j(x).
$$
Note that
\begin{equation}
\label{eq:vgamperppos}
	\V_{\mathcal{L}_\gamma}(P_\gamma,\mathcal{P}_\gamma)\,
	\V_{\mathcal{L}_\gamma^\perp}(
	\proj_{\mathcal{L}_\gamma^\perp} B,
	\proj_{\mathcal{L}_\gamma^\perp} \mathcal{P}_{\backslash\gamma})	
	=
	{n\choose |\gamma|+1}
	\,\V_n(B,P_\gamma,\mathcal{P})>0
\end{equation}
by Lemmas \ref{lem:proj} and \ref{lem:posmixcrit}, so 
the assumptions of Lemma \ref{lem:propglue} are satisfied. Denote by
$\varphi:\mathcal{L}_\gamma\to\mathbb{R}$ the $1$-homogeneous function 
constructed from $h$ by Lemma \ref{lem:propglue}.

\medskip

\textbf{Step 2.}
Now fix any $u\in U$, and choose
$$
	\tilde\varphi(x) := \langle t(u),x\rangle+g_{0,u}(x).
$$
As $g_{0,u}$ is a $\gamma$-degenerate function, it follows from
Lemma \ref{lem:deghomog} and $t(u)\in\mathcal{L}_\gamma$ that
$\tilde\varphi(x)=\tilde\varphi(\proj_{\mathcal{L}_\gamma}x)$.
Corollary \ref{cor:critnogamma} therefore states that
$$
	h(x)=\tilde\varphi(\proj_{\mathcal{L}_\gamma}x)
	\quad\mbox{for all }x\in \supp S_{[0,u],B,\mathcal{P}_{\backslash r}}.
$$
Recalling that $[0,u]\subset\mathcal{L}_\gamma$ by the
definition of $U$, we can apply the conclusion of Lemma 
\ref{lem:propglue} with $(K_1,\ldots,K_k)\leftarrow 
([0,u],\mathcal{P}_{\gamma\backslash\{r\}})$ to obtain
$$
	g(x) =
        \langle s,x\rangle +
        \sum_{j=1}^\ell g_j(x)
	+ \varphi(\proj_{\mathcal{L}_\gamma}x)
	\quad\mbox{for all }x\in\supp S_{[0,u],B,\mathcal{P}_{\backslash r}}
$$
for any $u\in U$. Moreover, as $U$ has full measure in 
$S^{n-1}\cap\mathcal{L}_r$ by Lemma 
\ref{lem:umeasure1}, we may further integrate over $u\in U$ as in the 
proof of Lemma \ref{lem:superglue} to conclude that the previous identity
remains valid for all $x\in\supp S_{B,B_r,\mathcal{P}_{\backslash r}}$.
It follows that
\begin{equation}
\label{eq:critpenult}
	f(x) = g(x)=
        \langle s,x\rangle +
        \sum_{j=1}^\ell g_j(x)
	+ \varphi(\proj_{\mathcal{L}_\gamma}x)
	\quad\mbox{for all }x\in\supp S_{B,\mathcal{P}}
\end{equation}
by Lemma \ref{lem:lowmaxsupp} and property $1$ of Lemma 
\ref{lem:critproj}.

\medskip

\textbf{Step 3.}
It remains to show that
$\varphi(\proj_{\mathcal{L}_\gamma}x)$ defines a $\gamma$-degenerate
function. To this end,
recall that $S_{f,\mathcal{P}}=0$ by assumption, and 
$S_{g_j,\mathcal{P}}=0$ for $j=1,\ldots,\ell$ by Lemmas \ref{lem:degeq} 
and \ref{lem:twodegs}. Thus
$\int f\,dS_{B,\mathcal{P}}=\int g_j\,dS_{B,\mathcal{P}}=
\int\langle s,\cdot\rangle\,dS_{B,\mathcal{P}}=0$
by \eqref{eq:mixvolarea}, the symmetry of mixed volumes, and
Lemma \ref{lem:maprop}. We therefore have
$$
	0 = 
	{n-1\choose |\gamma|}
	\int \varphi(\proj_{\mathcal{L}_\gamma}x)\,
	S_{B,\mathcal{P}}(dx)
	=
	\V_{\mathcal{L}_\gamma^\perp}(
	\proj_{\mathcal{L}_\gamma^\perp} B,
	\proj_{\mathcal{L}_\gamma^\perp} \mathcal{P}_{\backslash\gamma})
	\int \varphi\,dS_{\mathcal{P}_\gamma}
$$
using \eqref{eq:critpenult} and
Lemma \ref{lem:maproj}. But as
$\V_{\mathcal{L}_\gamma^\perp}(\proj_{\mathcal{L}_\gamma^\perp} B, 
\proj_{\mathcal{L}_\gamma^\perp}\mathcal{P}_{\backslash\gamma})>0$
by \eqref{eq:vgamperppos}, we can
apply Lemma \ref{lem:deghomog} to construct
a $\gamma$-degenerate function $g_0$ so that 
$$
	g_0(x)=\varphi(\proj_{\mathcal{L}_\gamma}x)
	\quad\mbox{for all }x\in\supp S_{B,\mathcal{P}}.
$$
Subsituting this identity into \eqref{eq:critpenult} completes the proof.
\end{proof}

The proof of Theorem \ref{thm:crit} is now readily completed.

\begin{proof}[Proof of Theorem \ref{thm:crit}]
The \emph{if} direction is an immediate
consequence of Lemmas \ref{lem:lineq}, \ref{lem:suppeq}, \ref{lem:degeq},
and \ref{lem:twodegs}, so it suffices to consider the \emph{only if}
direction.

To this end, note first that the case $n=2$ is always supercritical, so 
Theorem~\ref{thm:crit} is trivial in this case. For the induction 
step, it remains to prove that the validity of Theorem \ref{thm:crit} in 
dimension $n-1$ implies its validity in dimension $n$ for any $n\ge 3$. 
The latter is precisely the statement of Lemma \ref{lem:crit}.
\end{proof}

\section{Proof of the main result}
\label{sec:proofmain}

The aim of this section is to complete the proof of Theorem 
\ref{thm:main}. While the results that we have proved in the supercritical 
and critial cases provide a lot more information on the extremals than can 
be read off from Theorem \ref{thm:main} (which will be described in full 
detail in section \ref{sec:unique} below), the advantage of the 
formulation of Theorem~\ref{thm:main} is that it unifies all the different 
cases that arise in our analysis in one simple and universal statement. 
What remains is to verify that all possible cases are in fact captured by 
the formulation of Theorem \ref{thm:main}.

So far we have only considered the extremals under the criticality 
assumption on $\mathcal{P}$. The remaining cases turn out to be either 
trivial, or to reduce readily to a critical case in lower dimension.
We will first investigate the latter phenomenon, and then put everything 
together to complete the proof of Theorem \ref{thm:main}.

\subsection{Subcritical sets}

Let us begin by formally defining the cases that have not yet been 
considered in the previous sections.

\begin{defn}
\label{defn:subcrit}
A collection of convex bodies $\mathcal{C}=(C_1,\ldots,C_{n-2})$ 
is said to be \emph{subcritical} if $\dim(C_{i_1}+\cdots+C_{i_k})\ge k$
for all $k\in[n-2]$, $1\le i_1<\cdots<i_k\le n-2$.
A collection of convex bodies that is not subcritical is called
\emph{null}.
\end{defn}

In view of the following lemma, the null case is trivial.

\begin{lem}
\label{lem:null}
Let $\mathcal{C}=(C_1,\ldots,C_{n-2})$ be a null collection of convex 
bodies in $\mathbb{R}^n$. Then 
$\V_n(K,L,\mathcal{C})=0$ for any convex bodies $K,L$.
\end{lem}

\begin{proof}
This is an immediate consequence of Lemma \ref{lem:dim}.
\end{proof}

In the remainder of this section, we fix $n\ge 3$ and a collection 
$\mathcal{P}=(P_1,\ldots,P_{n-2})$ of convex bodies in $\mathbb{R}^n$ that 
is subcritical but not critical. We will assume that $P_i$ contains the 
origin in its relative interior for each $i\in[n-2]$, and define the 
spaces $\mathcal{L}_\alpha$ as in section \ref{sec:supercrit}. Thus the 
subcriticality assumption states that $\dim\mathcal{L}_\alpha\ge|\alpha|$ 
for every $\alpha$, and the lack of criticality implies that 
$\dim\mathcal{L}_\alpha=|\alpha|$ for at least one set $\alpha$.
In analogy with the critical case, we introduce the following 
terminology.

\begin{defn}
\label{defn:subcritset}
$\alpha\subseteq[n-2]$ is called a \emph{subcritical set}
if $\dim(\sum_{i\in\alpha}P_i)=|\alpha|$.
\end{defn}

The point of this definition is the following.

\begin{lem}
\label{lem:subcproj}
We have $\V_{\mathcal{L}_\alpha}(\mathcal{P}_\alpha)>0$ and
$$
	{n\choose|\alpha|}\,
	\V_n(K,L,\mathcal{P}) =
	\V_{\mathcal{L}_\alpha}(\mathcal{P}_\alpha)\,
	\V_{\mathcal{L}_\alpha^\perp}(
	\proj_{\mathcal{L}_\alpha^\perp}K,
	\proj_{\mathcal{L}_\alpha^\perp}L,
	\proj_{\mathcal{L}_\alpha^\perp}\mathcal{P}_{\backslash\alpha})
$$
for any convex bodies $K,L$ and subcritical set $\alpha\subseteq[n-2]$.
\end{lem}

\begin{proof}
That $\V_{\mathcal{L}_\alpha}(\mathcal{P}_\alpha)>0$ follows from 
Lemma \ref{lem:dim} and the assumption that $\mathcal{P}$ is subcritical.
The second statement is an immediate consequence of Lemma \ref{lem:proj}.
\end{proof}

In other words, any subcritical set will factor out of all mixed volumes 
that appear in the Alexandrov-Fenchel inequality, reducing it to a 
lower-dimensional case. However, there may \emph{a priori} be many 
subcritical sets, and it is also not clear what properties 
$\proj_{\mathcal{L}_\alpha^\perp}\mathcal{P}_{\backslash\alpha}$ may have. 
The main aim of this section is to show that there is a special choice of 
subcritical set $\eta$ so that 
$\proj_{\mathcal{L}_\eta^\perp}\mathcal{P}_{\backslash\eta}$ is a 
\emph{critical} collection of convex bodies, which reduces the study of 
extremals of the Alexandrov-Fenchel inequality in the subcritical case to 
the setting of Theorem \ref{thm:crit}.

To this end, we first prove a subcritical analogue of Lemma 
\ref{lem:panov}.

\begin{lem}
\label{lem:subpanov}
Let $\alpha,\alpha'$ be subcritical sets. 
Then $\alpha\cup\alpha'$ is also a subcritical set.
\end{lem}

\begin{proof}
As $\mathcal{P}$ is subcritical, we have
$\dim\mathcal{L}_{\alpha\cup\alpha'}\ge |\alpha\cup\alpha'|$
and $\dim\mathcal{L}_{\alpha\cap\alpha'}\ge |\alpha\cap\alpha'|$
(note that the latter holds even when $\alpha\cap\alpha'=\varnothing$,
unlike in Lemma \ref{lem:panov}). Thus
\begin{align*}
        |\alpha\cup\alpha'| &\le
        \dim\mathcal{L}_{\alpha\cup\alpha'}
        =
        \dim \mathcal{L}_\alpha + \dim\mathcal{L}_{\alpha'}
        -\dim(\mathcal{L}_\alpha\cap\mathcal{L}_{\alpha'}) \\
        &\le
        \dim \mathcal{L}_\alpha + \dim\mathcal{L}_{\alpha'}
        -\dim\mathcal{L}_{\alpha\cap\alpha'} \\
        &\le
        |\alpha| + |\alpha'|
        - |\alpha\cap\alpha'| = |\alpha\cup\alpha'|,
\end{align*}
where we used that $\dim \mathcal{L}_\alpha=|\alpha|$ and
$\dim\mathcal{L}_{\alpha'}=|\alpha'|$ as $\alpha,\alpha'$ are
subcritical.
It follows that $\dim\mathcal{L}_{\alpha\cup\alpha'}=
|\alpha\cup\alpha'|$, so $\alpha\cup\alpha'$ is subcritical.
\end{proof}

\begin{cor}
\label{cor:maxsubcrit}
There is a unique maximal subcritical set $\eta$, that is, a subcritical 
set $\eta$ such that $\alpha\subseteq\eta$ for every subcritical set
$\alpha$.
\end{cor}

\begin{proof}
By Lemma \ref{lem:subpanov}, we may choose $\eta$ to be the 
union of all subcritical sets.
\end{proof}

We now claim that applying Lemma \ref{lem:subcproj} to the set $\eta$ of 
Corollary \ref{cor:maxsubcrit} reduces the Alexandrov-Fenchel inequlity to 
the critical case.

\begin{lem}
\label{lem:subcritreduce}
Let $\eta$ be as in Corollary \ref{cor:maxsubcrit}. Then
$\proj_{\mathcal{L}_\eta^\perp}\mathcal{P}_{\backslash\eta}$ is
critical.
\end{lem}

\begin{proof}
Suppose the conclusion is false; then there must exist
$\alpha\subseteq [n-2]\backslash\eta$, $\alpha\ne\varnothing$ such that
$\dim \proj_{\mathcal{L}_\eta^\perp}\mathcal{L}_\alpha\le |\alpha|$.
Now note that
$$
	\dim \proj_{\mathcal{L}_\eta^\perp}\mathcal{L}_\alpha =
	\dim \mathcal{L}_\alpha -
	\dim(\mathcal{L}_\alpha\cap\ker \proj_{\mathcal{L}_\eta^\perp}) =
	\dim\mathcal{L}_\alpha - 
	\dim(\mathcal{L}_\alpha\cap \mathcal{L}_\eta).
$$
Therefore
$$
	\dim\mathcal{L}_{\eta\cup\alpha} =
	\dim\mathcal{L}_\eta + \dim\mathcal{L}_\alpha -
	\dim(\mathcal{L}_\alpha\cap \mathcal{L}_\eta)
	\le |\eta| + |\alpha| = |\eta\cup\alpha|,
$$
where we used that $\eta$ is subcritical and $\eta\cap\alpha=\varnothing$.
Thus we have shown that $\eta\cup\alpha$ is subcritical, which contradicts 
the maximality of $\eta$.
\end{proof}

\subsection{Proof of Theorem \ref{thm:main}}

We are now finally ready to conclude the proof.

\begin{proof}[Proof of Theorem \ref{thm:main}]
The \emph{if} direction of Theorem \ref{thm:main} follows from Lemmas
\ref{lem:simpleeq}, \ref{lem:lineq}, \ref{lem:suppeq}, and
\ref{lem:degeq}, so only the \emph{only if} part requires proof.

By translation-invariance, we may assume without loss of generality that 
each $P_i$ contains the origin in its relative interior. The assumption 
$\V_n(K,L,\mathcal{P})>0$ implies that $\mathcal{P}$ cannot be null by 
Lemma \ref{lem:null}. We now consider the remaining cases.

Suppose first that $\mathcal{P}$ is supercritical. Then the conclusion 
follows immediately from Corollary \ref{cor:schneider} (which was proved 
in section \ref{sec:supercrit}) and Lemma \ref{lem:supercrit}.

Now suppose that $\mathcal{P}$ is critical but not supercritical, and 
denote by $\alpha_0,\ldots,\alpha_\ell$ the $\mathcal{P}$-maximal sets.
By Lemma \ref{lem:simpleeq}, the equality condition in
Theorem \ref{thm:main} implies that there exists $a>0$ such
that $S_{f,\mathcal{P}}=0$ for $f=h_K-ah_L$. Thus
$$
	h_{K + N_0 + \cdots + N_\ell}(x) =
	h_{aL+s+M_0+\cdots+M_\ell}(x)\quad\mbox{for all }x\in
	\supp S_{B,\mathcal{P}}
$$
by Theorem \ref{thm:crit},
where $s\in\mathbb{R}^n$ and $(M_j,N_j)$ is an $\alpha_j$-degenerate pair
for every $j=0,\ldots,\ell$. The conclusion follows readily from 
Lemma \ref{lem:twodegs}.

Finally, suppose that $\mathcal{P}$ is subcritical but not critical, and 
let $\eta$ be the maximal subcritical set of Corollary 
\ref{cor:maxsubcrit}. Then the assumption and equality condition of
Theorem 
\ref{thm:main} imply that
$\V_{\mathcal{L}_\eta^\perp}(
\proj_{\mathcal{L}_\eta^\perp}K,
\proj_{\mathcal{L}_\eta^\perp}L,
\proj_{\mathcal{L}_\eta^\perp}\mathcal{P}_{\backslash\eta})>0$
and
\begin{align*}
	&\V_{\mathcal{L}_\eta^\perp}(
	\proj_{\mathcal{L}_\eta^\perp}K,
	\proj_{\mathcal{L}_\eta^\perp}L,
	\proj_{\mathcal{L}_\eta^\perp}\mathcal{P}_{\backslash\eta})^2 
	=
	\\&
	\V_{\mathcal{L}_\eta^\perp}(
	\proj_{\mathcal{L}_\eta^\perp}K,
	\proj_{\mathcal{L}_\eta^\perp}K,
	\proj_{\mathcal{L}_\eta^\perp}\mathcal{P}_{\backslash\eta})
	\,
	\V_{\mathcal{L}_\eta^\perp}(
	\proj_{\mathcal{L}_\eta^\perp}L,
	\proj_{\mathcal{L}_\eta^\perp}L,
	\proj_{\mathcal{L}_\eta^\perp}\mathcal{P}_{\backslash\eta})
\end{align*}
by Lemma \ref{lem:subcproj}. As 
$\proj_{\mathcal{L}_\eta^\perp}\mathcal{P}_{\backslash\eta}$ is critical
by Lemma \ref{lem:subcritreduce}, we can apply the critical case of
Theorem \ref{thm:main} in $\mathcal{L}_\eta^\perp$ to conclude that
we have
\begin{equation}
\label{eq:mainpfeq}
	h_{K+N_0+\cdots+N_\ell}(x)
	=
	h_{aL+s+M_0+\cdots+M_\ell}(x)
	\quad\mbox{for all }
	x\in\supp S_{\proj_{\mathcal{L}_\eta^\perp}B,
	\proj_{\mathcal{L}_\eta^\perp}\mathcal{P}_{\backslash\eta}}
\end{equation}
for some $a>0$, $s\in\mathcal{L}_\eta^\perp$, and 
$\proj_{\mathcal{L}_\eta^\perp}\mathcal{P}_{\backslash\eta}$-degenerate
pairs $(M_j,N_j)$ for $j=0,\ldots,\ell$, where
we used that 
$h_{\proj_{\mathcal{L}_\eta^\perp}K}(x)=h_K(x)$ 
and $h_{\proj_{\mathcal{L}_\eta^\perp}L}(x)=h_L(x)$ 
for $x\in\mathcal{L}_\eta^\perp$. But as
\begin{equation}
\label{eq:sbpsubcrit}
	{n-1\choose|\eta|}\,
	S_{B,\mathcal{P}}
	=
	\V_{\mathcal{L}_\eta}(\mathcal{P}_\eta)\,
	S_{\proj_{\mathcal{L}_\eta^\perp}B,
	\proj_{\mathcal{L}_\eta^\perp}\mathcal{P}_{\backslash\eta}}
\end{equation}
by applying Lemma \ref{lem:subcproj} as in Remark 
\ref{rem:lowdimma}, and as $\V_{\mathcal{L}_\eta}(\mathcal{P}_\eta)>0$ by 
Lemma \ref{lem:subcproj}, it follows that
\eqref{eq:mainpfeq} remains valid for
$x\in\supp S_{B,\mathcal{P}}$. Moreover, it follows readily from Definition
\ref{defn:deg} and Lemma \ref{lem:subcproj} that any
$\proj_{\mathcal{L}_\eta^\perp}\mathcal{P}_{\backslash\eta}$-degenerate
pair is also a $\mathcal{P}$-degenerate pair. Thus the conclusion 
of Theorem \ref{thm:main} is proved.
\end{proof}

\part{Complements and applications}

\section{The extremal decomposition}
\label{sec:unique}

Theorem \ref{thm:main} gives a very general description of the extremals 
of the Alexandrov-Fenchel inequality in terms of degenerate pairs. There 
is significant redundancy in this formulation, however: the same extremal 
bodies may be decomposed into degenerate pairs in different ways. A much 
more informative description of the extremals arises as a consequence of 
the theory developed in the previous sections. The aim of the present 
section is to extract from the proof of Theorem \ref{thm:main} a 
non-redundant characterization of the extremal decomposition. This 
formulation may be viewed as the definitive form of the main result of 
this paper.

\subsection{A unique extremal characterization}

Throughout this section, we fix polytopes 
$\mathcal{P}=(P_1,\ldots,P_{n-2})$ in $\mathbb{R}^n$. By 
translation-invariance, we may assume without loss of generality that each 
$P_i$ contains the origin in its relative interior. We also assume without 
loss of generality that $\mathcal{P}$ is subcritical 
(Definition~\ref{defn:subcrit}), as otherwise the extremal problem is 
vacuous by Lemma~\ref{lem:null}.

The following structural properties of $\mathcal{P}$ were introduced in 
the previous sections. Recall that subcritical and maximal sets are 
defined in Definitions \ref{defn:subcritset} and \ref{defn:critset}.
\begin{enumerate}[$\bullet$]
\itemsep\abovedisplayskip
\item $\mathcal{P}$ has a unique maximal subcritical set 
$\eta\subseteq[n-2]$ by Corollary \ref{cor:maxsubcrit}; we define
$$
	\mathcal{L}:=\sspan\sum_{i\in\eta}P_i.
$$
If $\mathcal{P}$ is 
critical (Definition \ref{defn:crit}), then $\eta =\varnothing$ and
$\mathcal{L}=\{0\}$.
\item  The collection 
$\proj_{\mathcal{L}^\perp}\mathcal{P}_{\backslash\eta}$ is critical by 
Lemma \ref{lem:subcritreduce}. Thus by Corollary \ref{cor:panov}, there 
are $c\ge 0$ disjoint  
$\proj_{\mathcal{L}^\perp}\mathcal{P}_{\backslash\eta}$-maximal sets
$\beta_1,\ldots,\beta_c\subseteq[n-2]\backslash\eta$; we define
$$
	\mathcal{L}_j := \sspan\sum_{i\in\beta_j}P_i.
$$
If $\proj_{\mathcal{L}^\perp}\mathcal{P}_{\backslash\eta}$ is 
supercritical (Definition \ref{defn:supercrit}), then $c=0$.
\end{enumerate}
Let us now introduce the spaces of degenerate functions
(Definition \ref{defn:maxdeg})
\begin{align*}
	\mathbb{D}_j :=
	\{  f:S^{n-1}\to\mathbb{R}~:~ &f\mbox{ is a }
	(\proj_{\mathcal{L}^\perp}
	\mathcal{P}_{\backslash\eta},\beta_j)\mbox{-degenerate function},
	\\
	& f\perp \langle v,\cdot\rangle\mbox{ in }
	L^2(S_{B,\mathcal{P}})\mbox{ for every }v\in
	\proj_{\mathcal{L}^\perp}\mathcal{L}_j
	\}
\end{align*}
for $j=1,\ldots,c$. By introducing the orthogonality condition with 
respect to linear functions, we eliminate one source of redundancy: that 
part of the linear term in Theorem \ref{thm:crit} may be absorbed in the 
definitions of the degenerate functions.
We will presently show that this is in fact the \emph{only} source of 
redundancy;
once it is eliminated, the extremal decomposition is uniquely determined 
on $\supp S_{B,\mathcal{P}}$. That is, we have the following unique 
characterization of the extremals of the Alexandrov-Fenchel inequality (in 
the formulation of part $b$ of Lemma \ref{lem:simpleeq}).

\begin{thm}
\label{thm:ultimatemain}
Let $f:S^{n-1}\to\mathbb{R}$ be any difference of support functions.
Then 
$$
	S_{f,\mathcal{P}}=0
$$
holds if and only if
$$
	f(x) = \langle s,x\rangle + \sum_{j=1}^c f_j(x)
	\quad\mbox{for all }x\in\supp S_{B,\mathcal{P}}
$$
holds for some $s\in\mathcal{L}^\perp$ and $f_j\in\mathbb{D}_j$,
$j=1,\ldots,c$. Moreover, in this representation, $s$ is unique 
and $f_1,\ldots,f_c$ are uniquely 
determined on $\supp S_{B,\mathcal{P}}$.
\end{thm}

Before we prove Theorem \ref{thm:ultimatemain}, let us first record a 
basic property. Analogous arguments appeared already several times in the 
proof of Theorem \ref{thm:main}.

\begin{lem}
\label{lem:sbpsupp}
$\sspan\supp S_{B,\mathcal{P}}=\mathcal{L}^\perp$.
\end{lem}

\begin{proof}
We first observe that $\sspan\supp S_{B,\mathcal{P}}\subseteq
\mathcal{L}^\perp$ (this is trivial if $\eta=\varnothing$, and follows 
directly from \eqref{eq:sbpsubcrit} otherwise).
Now suppose this inclusion is strict. Then we must have
$\supp S_{B,\mathcal{P}}\subseteq w^\perp$ for some 
$w\in\mathcal{L}^\perp$, so that
$$
	0 =
	{n\choose|\eta|}
	\frac{1}{n}
	\int \langle w,x\rangle_+\,S_{B,\mathcal{P}}(dx)
	=
	\V_{\mathcal{L}}(\mathcal{P}_\eta)
	\,
	\V_{\mathcal{L}^\perp}(
	[0,w],
	\proj_{\mathcal{L}^\perp}B,
	\proj_{\mathcal{L}^\perp}\mathcal{P}_{\backslash\eta})
$$
by \eqref{eq:mixvolarea} and 
Lemma \ref{lem:subcproj} (if $\eta=\varnothing$, this expression remains 
valid with $\V_{\mathcal{L}}(\mathcal{P}_\eta)\equiv 1$). This 
contradicts 
criticality of $\proj_{\mathcal{L}^\perp}\mathcal{P}_{\backslash\eta}$ 
by Lemma \ref{lem:dim}, establishing the claim.
\end{proof}

We now turn to the proof of Thereom \ref{thm:ultimatemain}. 

\begin{proof}[Proof of Theorem \ref{thm:ultimatemain}]
We first note that $S_{f,\mathcal{P}}=0$ if and only if 
$S_{\proj_{\mathcal{L}^\perp}f,\proj_{\mathcal{L}^\perp}\mathcal{P}_{\backslash\eta}}=0$ 
by applying Lemma \ref{lem:subcproj} as in Remark \ref{rem:lowdimma}. Thus 
as $\proj_{\mathcal{L}^\perp}\mathcal{P}_{\backslash\eta}$ is critical, we 
can use either Theorem \ref{thm:schneider} or Theorem \ref{thm:crit} in 
$\mathcal{L}^\perp$ to show that $S_{f,\mathcal{P}}=0$ if and only if
$$
	f(x)=\langle \bar s,x\rangle+\sum_{j=1}^c g_j(x)\quad
	\mbox{for all }x\in\supp S_{\proj_{\mathcal{L}^\perp}B,
	\proj_{\mathcal{L}^\perp}\mathcal{P}_{\backslash\eta}}
$$
for some $\bar s\in\mathcal{L}^\perp$ and 
$(\proj_{\mathcal{L}^\perp\mathcal{P}_{\backslash\eta}},\beta_j)$-degenerate
function $g_j$, $j=1,\ldots,c$. The conclusion remains valid 
for all $x\in\supp S_{B,\mathcal{P}}$ as $\supp S_{B,\mathcal{P}}=\supp
S_{\proj_{\mathcal{L}^\perp}B, 
\proj_{\mathcal{L}^\perp}\mathcal{P}_{\backslash\eta}}$
(this is trivial when $\eta=\varnothing$, and follows from
\eqref{eq:sbpsubcrit} otherwise).

Now note that for every $j=1,\ldots,c$, there exists
$s_j\in\proj_{\mathcal{L}^\perp}\mathcal{L}_j$ so that
$$
	f_j := g_j - \langle s_j,\cdot\rangle
	~\perp~ \sspan\{\langle v,\cdot\rangle:
	v\in\proj_{\mathcal{L}^\perp}\mathcal{L}_j\}
	\quad\mbox{in }L^2(S_{B,\mathcal{P}}).
$$
Then $f_j\in\mathbb{D}_j$ by construction. Moreover, let us write
$s := \bar s + s_1 + \cdots + s_c\in\mathcal{L}^\perp$.
Then we have shown that $S_{f,\mathcal{P}}=0$ holds if and only if
$$
	f(x)=\langle s,x\rangle+\sum_{j=1}^c f_j(x)\quad
	\mbox{for all }x\in\supp S_{B,\mathcal{P}}
$$
for
some $s\in\mathcal{L}^\perp$ and $f_j\in\mathbb{D}_j$,
$j=1,\ldots,c$. 
It remains to show uniqueness.

To this end, let us suppose that 
$$
	f(x)=
	\langle s,x\rangle+\sum_{j=1}^c f_j(x)=
	\langle s',x\rangle+\sum_{j=1}^c f_j'(x)
	\quad
	\mbox{for all }x\in\supp S_{B,\mathcal{P}}
$$
holds for $s,s'\in\mathcal{L}^\perp$ and 
$f_j,f_j'\in\mathbb{D}_j$, $j=1,\ldots,c$. Then we certainly have
$$
	\langle s-s',x\rangle + 
	\sum_{j=1}^c \{f_j(x)-f_j'(x)\}=0
	\quad
	\mbox{for all }
	x\in\supp S_{B,\mathcal{P}} =
	\supp S_{\proj_{\mathcal{L}^\perp}B,
	\proj_{\mathcal{L}^\perp}\mathcal{P}_{\backslash\eta}}.
$$
Suppose first that $c\ge 1$. 
Then by Lemma \ref{lem:maxsupp}, we may apply Proposition 
\ref{prop:decouplingmagic} in $\mathcal{L}^\perp$ to show that for every 
$j=1,\ldots,c$, there exists 
$w_j\in\proj_{\mathcal{L}^\perp}\mathcal{L}_j$ so that
$$
	f_j(x)-f_j'(x) = \langle w_j,x\rangle\quad
	\mbox{for all }
	x\in\supp S_{B,\mathcal{P}}.
$$
But the definition of $\mathbb{D}_j$ then implies that $w_j=0$, so
that each $f_j$ is uniquely determined on $\supp S_{B,\mathcal{P}}$. 
Moreover, for any $c\ge 0$, we now obtain
$\langle s-s',x\rangle=0$ for all $x\in\supp S_{B,\mathcal{P}}$.
Thus $s=s'$ by Lemma \ref{lem:sbpsupp}, so that $s$ is unique as 
well.
\end{proof}

\subsection{The space of extremals}

In addition to the spaces $\mathbb{D}_j$, let us denote by
$$
	\mathbb{X} := \{ f:S^{n-1}\to\mathbb{R}~:~
	f\mbox{ is a difference of support functions such that }S_{f,\mathcal{P}}=0\}
$$
the space of extremals of the Alexandrov-Fenchel inequality, and by
$$
	\mathbb{L} := \{ f:S^{n-1}\to\mathbb{R}~:~
	f=\langle v,\cdot\rangle\mbox{ for some }v\in\mathcal{L}^\perp\}
$$
the space of linear functions. In the present section, we will view
$\mathbb{X},\mathbb{L},\mathbb{D}_j$ as subspaces of 
$L^2(S_{B,\mathcal{P}})$; in particular, we identify the elements of these 
spaces that agree $S_{B,\mathcal{P}}$-a.e. In these terms, 
we may reformulate Theorem \ref{thm:ultimatemain} as 
$$
	\mathbb{X} = \mathbb{L}\oplus\mathbb{D}_1
	\oplus\cdots\oplus\mathbb{D}_c\quad\mbox{in }L^2(S_{B,\mathcal{P}}).
$$
That is, $\mathbb{L},\mathbb{D}_1,\ldots,\mathbb{D}_c$ are linearly 
independent subspaces of $L^2(S_{B,\mathcal{P}})$ that span $\mathbb{X}$.

Despite the definitive form of this result, its continuous formulation 
belies the essentially combinatorial nature of the extremals of the 
Alexandrov-Fenchel inequality for polytopes $\mathcal{P}$. For example, 
Proposition \ref{prop:qgraph} implies that the extremals are fully 
described by the kernel of a matrix, so that $\mathbb{X}$ must be 
finite-dimensional; this fact is not evident above. To illustrate that 
this kind of information is indeed contained in the above 
characterization, let us presently compute the dimensions of the subspaces 
$\mathbb{L},\mathbb{D}_j,\mathbb{X}$ of $L^2(S_{B,\mathcal{P}})$
in terms of the geometry of $\mathcal{P}$.

\begin{prop}
\label{prop:dimu}
Let $\Omega_j:=\supp S_{\proj_{\mathcal{L}^\perp}\mathcal{P}_{\beta_j}}$.
Then
$$
	\dim \mathbb{L} = n-|\eta|,\qquad\quad
	\dim \mathbb{D}_j = |\Omega_j|-|\beta_j|-2
$$
for $j=1,\ldots,c$. In particular,
$\dim\mathbb{X}=n-|\eta| + \sum_{j=1}^c \{|\Omega_j|-|\beta_j|-2\}$.
\end{prop}

\begin{rem}
Note that $|\Omega_j|<\infty$ and that $\Omega_j$ may be 
computed by Lemma \ref{lem:mapoly}.
\end{rem}

To prove Proposition \ref{prop:dimu},
we begin by characterizing the affine hull of $\Omega_j$.

\begin{lem}
\label{lem:affbetaj}
$\aff\Omega_j=\proj_{\mathcal{L}^\perp}\mathcal{L}_j$.
\end{lem}

\begin{proof}
That $\sspan\Omega_j = \sspan\supp 
S_{\proj_{\mathcal{L}^\perp}\mathcal{P}_{\beta_j}}= 
\proj_{\mathcal{L}^\perp}\mathcal{L}_j$ follows as 
$\proj_{\mathcal{L}^\perp}\mathcal{P}_{\backslash\eta}$ is critical by 
exactly the same argument as in the proof of Lemma \ref{lem:sbpsupp}.
On the other hand, by part $d$ of Lemma \ref{lem:maprop}, there is a 
vanishing linear combination of the elements of $\Omega_j$ with positive 
coefficients, so $0\in\aff\Omega_j$. Thus
$\aff\Omega_j = \sspan\Omega_j$.
\end{proof}

Next, we observe that the support functions of convex bodies in 
$\proj_{\mathcal{L}^\perp}\mathcal{L}_j$ are uniquely determined up to 
$S_{B,\mathcal{P}}$-a.e.\ equivalence by their values on $\Omega_j$.

\begin{lem}
\label{lem:finitebetaj}
Let $f=h_M-h_N$ for convex bodies $M,N\subset 
\proj_{\mathcal{L}^\perp}\mathcal{L}_j$. Then $f=0$ 
$S_{B,\mathcal{P}}$-a.e.\ if and only if $f(x)=0$ for all $x\in\Omega_j$. 
Moreover, for any $z\in\mathbb{R}^{\Omega_j}$, there exists a function $f$ 
of this form such that $f(x)=z_x$ for $x\in\Omega_j$.
\end{lem}

\begin{proof} We first apply \eqref{eq:sbpsubcrit} and Lemma 
\ref{lem:maproj} to write \begin{equation} \label{eq:betaj} 
\begin{aligned}
	& {n-1\choose |\eta|}
	{n-|\eta|-1\choose |\beta_j|}
	\int |f|\,dS_{B,\mathcal{P}}
\\&	=
	{n-|\eta|-1\choose |\beta_j|}
	\V_{\mathcal{L}}(\mathcal{P}_\eta)
	\int |f|\,dS_{\proj_{\mathcal{L}^\perp}B,
	\proj_{\mathcal{L}^\perp}\mathcal{P}_{\backslash\eta}}
\\&	=
	\V_{\mathcal{L}}(\mathcal{P}_\eta)\,
	\V_{\mathcal{L}^\perp\cap\mathcal{L}_j^\perp}(
	\proj_{\mathcal{L}^\perp\cap\mathcal{L}_j^\perp}B,
	\proj_{\mathcal{L}^\perp\cap\mathcal{L}_j^\perp}
	\mathcal{P}_{\backslash\{\eta,\beta_j\}})
	\int |f|\,dS_{\proj_{\mathcal{L}^\perp}\mathcal{P}_{\beta_j}}.
\end{aligned}
\end{equation}
Now note that Lemma
\ref{lem:subcproj} implies that $\V_{\mathcal{L}}(\mathcal{P}_\eta)>0$. On 
the other hand, we have
$\V_{\mathcal{L}^\perp}(\proj_{\mathcal{L}^\perp}B_j, 
\proj_{\mathcal{L}^\perp}B, 
\proj_{\mathcal{L}^\perp}\mathcal{P}_{\backslash\eta})>0$ by Lemma 
\ref{lem:dim} as $\proj_{\mathcal{L}^\perp}\mathcal{P}_{\backslash\eta}$ 
is critical, where $B_j$ denotes the unit ball in $\mathcal{L}_j$. Thus 
applying \eqref{eq:betaj} with $f=h_{\proj_{\mathcal{L}^\perp}B_j}$ 
shows that the mixed volumes on the last line of \eqref{eq:betaj} are 
positive. We have therefore 
shown that $\int |f|\,dS_{B,\mathcal{P}}>0$ if and 
only if $\int 
|f|\,dS_{\proj_{\mathcal{L}^\perp}\mathcal{P}_{\beta_j}}>0$.

To prove the second claim, it suffices to note that as $\Omega_j$ is a 
finite set, there exists for any $z\in\mathbb{R}^{\Omega_j}$ a $C^2$ 
function $g:S^{n-1}\cap \proj_{\mathcal{L}^\perp}\mathcal{L}_j\to\mathbb{R}$ 
such that $g(x)=z_x$ for all $x\in\Omega_j$; then 
$f(x):=g(\proj_{\proj_{\mathcal{L}^\perp}\mathcal{L}_j}x)$ has the
requisite form by Lemma \ref{lem:c2}.
\end{proof}

We can now conclude the proof of Proposition \ref{prop:dimu}.

\begin{proof}[Proof of Proposition \ref{prop:dimu}]
We first note that for any $v\in\mathcal{L}^\perp$, we have
$\langle v,\cdot\rangle=0$ $S_{B,\mathcal{P}}$-a.e.\ if and only if
$v=0$ by Lemma \ref{lem:sbpsupp}. Thus $\dim\mathbb{L}=
\dim\mathcal{L}^\perp = n-|\eta|$, where we used that $\dim\mathcal{L}=
|\eta|$ as $\eta$ is a subcritical set.

Now note that by Definition \ref{defn:maxdeg}, a
$(\proj_{\mathcal{L}^\perp}\mathcal{P}_{\backslash\eta},\beta_j)$-degenerate 
function $f$ is defined by $f=h_M-h_N$ for convex bodies $M,N\subset 
\proj_{\mathcal{L}^\perp}\mathcal{L}_j$ 
satisfying one linear constraint $\int 
f\,dS_{\proj_{\mathcal{L}^\perp}\mathcal{P}_{\beta_j}}=0$.
Thus by Lemma \ref{lem:finitebetaj}, the subspce of
$L^2(S_{B,\mathcal{P}})$ defined by
$$
	\mathbb{\tilde D}_j:=\{f:S^{n-1}\to\mathbb{R}
	~:~f\mbox{ is a }
	(\proj_{\mathcal{L}^\perp}\mathcal{P}_{\backslash\eta},\beta_j)\mbox{-degenerate function}\}
$$
has dimension $\dim\mathbb{\tilde D}_j=|\Omega_j|-1$.
Moreover, it follows from Lemmas \ref{lem:affbetaj} and 
\ref{lem:finitebetaj} that
$\sspan\{\langle 
v,\cdot\rangle:v\in\proj_{\mathcal{L}^\perp}\mathcal{L}_j\}$ is a
subspace of $\mathbb{\tilde D}_j$ of dimension $\dim 
\proj_{\mathcal{L}^\perp}\mathcal{L}_j=|\beta_j|+1$, where we used
that $\beta_j$ is a 
$\proj_{\mathcal{L}^\perp}\mathcal{P}_{\backslash\eta}$-critical set.
Thus $\dim\mathbb{D}_j=|\Omega_j|-|\beta_j|-2$ by the definition of
$\mathbb{D}_j$.
The remaining claim follows as
$\mathbb{X}=\mathbb{L}\oplus\mathbb{D}_1\oplus\cdots\oplus\mathbb{D}_c$.
\end{proof}

\begin{rem}
It has been emphasized throughout this paper that the distinction between 
the supercritical and critical cases of Theorem \ref{thm:main} is that 
nonlinear extremals appear in the latter case. It is sometimes possible, 
however, that no nonlinear extremals exist 
even in the critical case, because it may be the case that 
$|\Omega_j|=|\beta_j|+2$ and thus $\mathbb{D}_j=\{0\}$ as a subspace of 
$L^2(S_{B,\mathcal{P}})$ by Proposition \ref{prop:dimu}. For example, this 
is the case when $\mathcal{P}_{\beta_j}=(Q,\ldots,Q)$, where $Q$ is a 
simplex in $\mathcal{L}^\perp$ of dimension $|\beta_j|+1$.
By the same token, however, Proposition \ref{prop:dimu} shows that this
situation can occur only in very special cases: as soon as enough
normal directions are in play, nonlinear extremals will always appear in 
the critical case.
\end{rem}

\section{Extensions to quermassintegrals, smooth bodies, and zonoids}
\label{sec:extensions}

In the proof of Theorem \ref{thm:main}, the assumption that $\mathcal{P}$ 
are polytopes was used extensively in the proof of the local 
Alexandrov-Fenchel inequality. However, most other arguments of 
this paper are not specific to polytopes. As a byproduct of our methods, 
we will presently extend our main result to more general situations.
In particular, we will characterize the extremals of Theorem \ref{thm:af} 
in the following cases:
\begin{enumerate}[$\bullet$]
\itemsep\abovedisplayskip
\item $C_1=\cdots=C_m$ is arbitrary and $C_{m+1},\ldots,C_{n-2}$ 
are smooth (Theorem~\ref{thm:quer}). In particular, this settles the case of 
quermassintegrals (Corollary~\ref{cor:quer}).
\item $\mathcal{C}$ is any combination of polytopes, smooth 
bodies, and zonoids (Theorem \ref{thm:zonoid}).
\end{enumerate}
In the interest of space, we will consider in this section only the 
supercritical case. Results may also be obtained for the critical 
case with additional work.

It should be emphasized that the analysis of this section will rely 
on the special structure of smooth bodies and zonoids; it does not 
address the main missing ingredient for extending our main results to 
general convex bodies, which is a general form of the local 
Alexandrov-Fenchel inequality (see section \ref{sec:discussion}). However, 
the results of this section further illustrate the methods developed in 
this paper, and capture a number of cases that are important in 
applications.

\subsection{A smooth local Alexandrov-Fenchel principle}

A convex body $C$ in $\mathbb{R}^n$ is called \emph{smooth} if it has a 
unique normal vector at every point of its boundary. The main observation 
behind the extension of our results to cases that include smooth bodies is 
that an analogue of the local Alexandrov-Fenchel inequality may be 
obtained in this case by a direct argument. That this situation is rather 
special will be evident from the statement of the following result: in the 
smooth case, the naive inequality \eqref{eq:localaf} holds, and none 
of the subtleties of Theorem \ref{thm:localaf} arise.

\begin{prop}
\label{prop:smoothlocal}
Let $\mathcal{C}=(C_1,\ldots,C_{n-2})$ be any convex bodies in 
$\mathbb{R}^n$
such that $\V_n(B,B,\mathcal{C})>0$, and suppose that $C_r$ is smooth for 
a given $r\in[n-2]$. Then for any difference of support functions $f$ so 
that $S_{f,\mathcal{C}}=0$, we have $S_{f,f,\mathcal{C}_{\backslash 
r}}=0$.
\end{prop}

The proof is based on an idea due to Schneider, whose
basic step is the following.

\begin{lem}
\label{lem:parallelbody}
Let $C,M$ be smooth convex bodies in $\mathbb{R}^n$. Then there exist
$\varepsilon,\delta>0$ and a family 
$\{C^\tau\}_{\tau\in[-\varepsilon,\varepsilon]}$ of convex bodies in 
$\mathbb{R}^n$ so that $\|h_{C^\tau}-h_C\|_\infty\le \delta|\tau|$ and
$$
	\lim_{\tau\to 0}
	\frac{h_{C^\tau}(u)-h_C(u)}{\tau} = h_M(u)\quad\mbox{for all }u\in 
	S^{n-1}.
$$
\end{lem}

\begin{proof}
We define $C^\tau$ as
$C^\tau := C+\tau M$ for $\tau\ge 0$ and
$C^\tau := C\div (-\tau)M$ for $\tau<0$,
where $C\div A := \{x\in\mathbb{R}^n:x+A\subseteq C\}$ denotes Minkowski 
subtraction. That $\|h_{C^\tau}-h_C\|_\infty\le 
\delta|\tau|$ is trivial for $\tau\ge 0$, and is shown in
\cite[p.\ 425]{Sch14} for $\tau<0$. The remaining statement
follows from \cite[Lemma 7.5.4]{Sch14}. 
\end{proof}

We can now conclude the proof of Proposition \ref{prop:smoothlocal}.

\begin{proof}[Proof of Proposition \ref{prop:smoothlocal}]
Let $M$ be any smooth body, and define the family $C_r^\tau$ as in
Lemma \ref{lem:parallelbody} (with $C\leftarrow C_r$).
Define the function
$$
	\varphi(\tau) :=
	\V_n(f,B,\mathcal{C}_{\backslash r},C_r^\tau)^2 -
	\V_n(f,f,\mathcal{C}_{\backslash r},C_r^\tau)\,
	\V_n(B,B,\mathcal{C}_{\backslash r},C_r^\tau).
$$
Then $\varphi(\tau)\ge 0$ by Lemma \ref{lem:3af}, and
$\varphi(0)=0$ as $C_r^0=C_r$ and $S_{f,\mathcal{C}}=0$. Thus $\tau=0$ is 
a local minimum of $\varphi$. It follows that
$$
	0=\frac{d\varphi(\tau)}{d\tau}\bigg|_{\tau=0}
	=
	-\V_n(f,f,\mathcal{C}_{\backslash r},M)\,
	\V_n(B,B,\mathcal{C}),
$$
where we used Lemma \ref{lem:parallelbody}, \eqref{eq:mixvolarea} and 
$S_{f,\mathcal{C}}=0$ to compute the derivative. As $M$ was an arbitrary 
smooth body and $\V_n(B,B,\mathcal{C})>0$, we have shown that 
$$
	0 = \V_n(f,f,\mathcal{C}_{\backslash r},g)
	= \frac{1}{n}\int g\,dS_{f,f,\mathcal{C}_{\backslash r}}
$$
for every function $g$ that is a difference of support functions of smooth 
bodies. As any $g\in C^2$ may be written in this manner by Lemma 
\ref{lem:c2}, the conclusion follows.
\end{proof}

One of the consequences of Proposition \ref{prop:smoothlocal} is that for 
the characterization of extremals of the Alexandrov-Fenchel inequalities, 
all smooth bodies are indistinguishable. This conclusion is a variant of 
\cite[Theorem 7.6.7]{Sch14}.

\begin{cor}
\label{cor:smoothsilly}
Let $\mathcal{C}=(C_1,\ldots,C_n)$ 
be convex bodies
in $\mathbb{R}^n$ so that $C_1,\ldots,C_m$ are smooth
and
$\V_n(B,B,\mathcal{C})>0$. Let
$\mathcal{C}'=(C_1',\ldots,C_m',C_{m+1},\ldots,C_n)$
for smooth bodies $C_1',\ldots,C_m'$.
Then $\supp S_{B,\mathcal{C}}=\supp 
S_{B,\mathcal{C}'}$, and for any difference of support functions $f$ we 
have $S_{f,\mathcal{C}}=0$ if and only if $S_{f,\mathcal{C}'}=0$.
\end{cor}

\begin{proof}
It suffices to prove the case $m=1$, as the general case then 
follows by applying the result repeatedly.
Note also that as smooth bodies are 
full-dimensional, $\V_n(B,B,\mathcal{C})>0$ if and only if
$\V_n(K,K,\mathcal{C}')>0$ for any smooth body $K$.

Let $f$ be a difference of support functions such that 
$S_{f,\mathcal{C}}=0$. By integrating the mixed area measures in 
Proposition \ref{prop:smoothlocal} (with $r=1$) against $h_{C_1'}$, we 
obtain $\V_n(f,C_1,\mathcal{C}')=0$ and $\V_n(f,f,\mathcal{C}')=0$. Thus 
Lemma \ref{lem:afeq} implies $S_{f,\mathcal{C}'}=0$. The converse 
implication follows by reversing the roles of $C_1,C_1'$.

Now note that $x\not\in\supp S_{B,\mathcal{C}}$ holds if and only if
there is a nonnegative $C^2$ function $f$ such that $f(x)>0$ and
$f(u)=0$ for all $u\in \supp S_{B,\mathcal{C}}$. Suppose this is the case.
Then $S_{f,\mathcal{C}}=0$
by Lemma \ref{lem:suppeq}, so $S_{f,\mathcal{C}'}=0$ as well. Integrating 
against $h_B$ and using the symmetry of mixed volumes yields
$\int f \,dS_{B,\mathcal{C}'}=0$, and thus
$f(u)=0$ for $u\in\supp S_{B,\mathcal{C}'}$ as $f$ is nonnegative and 
continuous. It follows that $x\not\in\supp S_{B,\mathcal{C}'}$. 
The converse
implication follows again by reversing the roles of $C_1,C_1'$.
\end{proof}

\subsection{Quermassintegrals}

\emph{Quermassintegrals} of a convex body $K$, defined by
$$
	W_i(K) := \V_n(\underbrace{K,\ldots,K}_{n-i},
	\underbrace{B,\ldots,B}_i),
$$
play a special role in convexity and in integral geometry; see, e.g., 
\cite[\S 6.4]{Gru07}. The Alexandrov-Fenchel inequality implies that 
quermassintegrals form a log-concave sequence, that is, $W_i(K)^2\ge 
W_{i-1}(K)\,W_{i+1}(K)$. Even in this very special case, the extremal 
bodies $K$ have been characterized only in the presence of symmetry 
assumptions \cite[Theorem 7.6.20]{Sch14}. In this section, we will settle 
this problem as a special case of a much more general result.

In the remainder of this section, we fix the following setting. Let 
$m\in[n-2]$, and let $M$ be any convex 
body in $\mathbb{R}^n$ such that $\dim M \ge m+2$. We further let 
$C_{m+1},\ldots,C_{n-2}$ be any smooth convex bodies in $\mathbb{R}^n$, 
and denote by
\begin{equation}
\label{eq:querc}
	\mathcal{C} = (\underbrace{M,\ldots,M}_m,C_{m+1},\ldots,C_{n-2}).
\end{equation}
The main result of this section will characterize the extremal 
bodies $K,L$ such that $\V_n(K,L,\mathcal{C})^2=
\V_n(K,K,\mathcal{C})\,\V_n(L,L,\mathcal{C})$. The reason we are able to 
do this for a general convex body $M$ (rather than a polytope) relies
on two observations. First, note that the gluing argument of section
\ref{sec:supercrit} works \emph{verbatim} for general convex bodies, as 
long as a local Alexandrov-Fenchel inequality is available. We may 
therefore use the gluing argument together with Proposition 
\ref{prop:smoothlocal} to reduce to the case $m=n-2$. The latter case, 
known as \emph{Minkowski's quadratic inequality}, was settled in complete 
generality in \cite{SvH19}, which enables us to conclude the proof.

Before we formulate this result precisely, let us provide a geometric 
characterization of the support of $S_{B,\mathcal{C}}$ in the present 
setting.

\begin{defn}
A vector $u\in S^{n-1}$ is called an \emph{$r$-extreme normal vector} of a 
convex body $M$ in $\mathbb{R}^n$ if there do not exist linearly 
independent normal vectors $u_1,\ldots,u_{r+2}$ at a boundary point of $M$ 
such that $u=u_1+\cdots+u_{r+2}$.
\end{defn}

For example, if $M$ is a polytope, then $u$ is an $r$-extreme normal 
vector of $M$ if and only if it is an outer normal of a face of $K$ of 
dimension at least $n-1-r$.

\begin{lem}
\label{lem:quasupp}
Let $m\in[n-2]$, let $M$ be a convex body in $\mathbb{R}^n$ with $\dim
M\ge m$, let $C_{m+1},\ldots,C_{n-2}$ be smooth convex bodies in
$\mathbb{R}^n$, and let
$\mathcal{C}$ be as in \eqref{eq:querc}. Then
$$
	\supp S_{B,\mathcal{C}}=\mathrm{cl}\{u\in S^{n-1}:u\mbox{ is an }
	(n-1-m)\mbox{-extreme normal vector of }M\}.
$$
\end{lem}

\begin{proof}
That $\V_n(B,B,\mathcal{C})>0$ follows from $\dim M\ge m$ and
Lemma \ref{lem:dim}. Therefore, by Corollary \ref{cor:smoothsilly}, we may 
assume without loss of 
generality that $C_{m+1}=\cdots=C_{n-2}=B$. In the latter case, the result
was proved in \cite[Satz 4]{Sch75}.
\end{proof}

We can now formulate the main result of this section.

\begin{thm}
\label{thm:quer}
Let $m\in[n-2]$, let $M$ be a convex body in $\mathbb{R}^n$ with $\dim 
M\ge m+2$, let $C_{m+1},\ldots,C_{n-2}$ be smooth convex bodies in
$\mathbb{R}^n$, and let 
$\mathcal{C}$ be as defined in \eqref{eq:querc}.
Then for any convex bodies $K,L$ in $\mathbb{R}^n$ so
that $\V_n(K,L,\mathcal{C})>0$, we have
$$
	\V_n(K,L,\mathcal{C})^2 =
	\V_n(K,K,\mathcal{C})\,
	\V_n(L,L,\mathcal{C})
$$
if and only if there exist $a>0$ and $v\in\mathbb{R}^n$ so that
$K$ and $aL+v$ have the same supporting hyperplanes in all
$(n-1-m)$-extreme normal directions of $M$.
\end{thm}

\begin{proof}
By Lemmas \ref{lem:simpleeq} and \ref{lem:quasupp}, the conclusion is 
equivalent to the statement that for any differences of support 
functions $f$, we have $S_{f,\mathcal{C}}=0$ if and only if
there exists $s\in\mathbb{R}^n$ so that $f(x)=\langle s,x\rangle$
for all $s\in\supp S_{B,\mathcal{C}}$. The \emph{if} direction follows 
from Lemmas \ref{lem:lineq} and \ref{lem:suppeq}, so it remains
to prove the \emph{only if} direction.

We will fix $m\ge 1$, and prove this statement by induction on $n$.
The base case of the induction, $n=m+2$, is the main result of 
\cite{SvH19}. We now suppose $n>m+2$, and assume the induction hypothesis 
that the conclusion has been proved in dimension $n-1$. We aim to show 
that the conclusion then also holds in dimension $n$.

To this end, let $f$ be a difference of support functions such that 
$S_{f,\mathcal{C}}=0$, and let $r=n-2$. Then $C_r$ is smooth, so 
$S_{f,f,\mathcal{C}_{\backslash r}}=0$ by Proposition 
\ref{prop:smoothlocal}. Moreover, $\mathcal{C}$ is supercritical
as $\dim M\ge m+2$. The argument of section 
\ref{sec:supercrit} now applies \emph{verbatim} in the present setting 
with $\mathcal{P}\leftarrow\mathcal{C}$ and $g\equiv f$ (indeed, that 
$\mathcal{P}$ are polytopes was used in section \ref{sec:supercrit} only 
to apply the local Alexandrov-Fenchel inequality). In particular, it 
follows from Lemma \ref{lem:superglue} that there exists 
$s\in\mathbb{R}^n$ so that $f(x)=\langle s,x\rangle$ for all $x\in\supp 
S_{B,B,\mathcal{C}_{\backslash r}}$, and the conclusion now follows from 
Lemma \ref{lem:maxsupp}.
\end{proof}

When specialized to quermassintegrals, we obtain the following. In the 
case that $K$ is centrally symmetric, this result was proved in 
\cite[Theorem 7.6.20]{Sch14}.

\begin{defn}
A convex body $K$ in $\mathbb{R}^n$ is a \emph{$(n-1-i)$-tangential body 
of a ball} if there exist $a>0$, $v\in\mathbb{R}^n$ so that
$aB+v\subseteq K$, and such that $aB+v$ and $K$ have the same supporting
hyperplanes in all $i$-extreme normal directions of $K$.
\end{defn}

\begin{cor}
\label{cor:quer}
Let $K$ be any convex body in $\mathbb{R}^n$, and let $i\in[n-1]$. Then we 
have equality $W_i(K)^2=W_{i-1}(K)\,W_{i+1}(K)$ if and only if either
$\dim K<n-i$, or $K$ is an $(n-1-i)$-tangential body of a   
ball.
\end{cor}

\begin{proof}
By Lemma \ref{lem:dim}, we have $W_i(K)=0$ if and only if $\dim K<n-i$;
in this case equality always holds. On the other hand,
if $\dim K=n-i$, then $W_i(K)>0$ and $W_{i-1}(K)=0$, so equality cannot 
hold; and $K$ cannot be a tangential body, as $aB+v\subseteq 
K$ implies that any tangential body satisfies $\dim K=n$.
Finally, if $\dim K \ge n-i+1$, Theorem \ref{thm:quer} implies that 
equality $W_i(K)^2=W_{i-1}(K)\,W_{i+1}(K)$ holds if and only if there 
exist $a>0$ and $v\in\mathbb{R}^n$ so that $aB+v$ and $K$ have the same 
supporting hyperplanes in all $i$-extreme normal directions of $K$.

It remains to show that the latter condition implies
\emph{a fortiori} that $aB+v\subseteq K$. Indeed, if $K\subset w^\perp$ 
for some $w\in S^{n-1}$, then both $w$ 
and $-w$ are $i$-extreme, so $aB+v\subset w^\perp$ as well. As that cannot 
be, we must have $\dim K=n$. But then \cite[Theorem 2.2.6]{Sch14} implies 
that $K$ is the intersection of its regular (and thus $i$-extreme)
supporting halfspaces. As these also support $aB+v$, it follows 
that $aB+v\subseteq K$. 
\end{proof}

\subsection{Zonoids}

A \emph{zonotope} is a polytope that is the Minkowski sum of a finite 
number of segments. A convex body $Z$ is called a \emph{zonoid} if it is a 
limit of zonotopes. For simplicity, we assume by convention that all 
zonoids are symmetric $Z=-Z$ (this entails no loss of generality for 
our purposes, as any zonoid is symmetric up to translation).
Then $Z$ is a zonoid if and only if \cite[Theorem 3.5.3]{Sch14}
$$
	h_Z(x) = \int |\langle u,x\rangle|\,\rho(du)
$$
for some even finite measure on $\rho$ on $S^{n-1}$, called the 
\emph{generating measure} of $Z$. 

When the reference bodies in the 
Alexandrov-Fenchel inequality are zonoids, their additive structure 
enables an inductive approach to the analysis of the extremals 
that is very special to this case. Such 
arguments were exploited by Schneider \cite{Sch88} to characterize the 
extremals for full-dimensional zonoids under additional symmetry 
assumptions. In this section, we will fully characterize the extremals for 
any supercritical collection of zonoids (the analysis of the critical case 
is more delicate, and is omitted here in the interest of space). In fact, 
the proof of the following more general result will present no additional
difficulties.

\begin{thm}
\label{thm:zonoid}
Let $m\in[n-2]$, let $C_1,\ldots,C_m$ be convex bodies in $\mathbb{R}^n$ 
such that each $C_i$ is either a zonoid or a smooth body, and let 
$P_{m+1},\ldots,P_{n-2}$ be any polytopes in $\mathbb{R}^n$. Assume that
$\mathcal{C} := (C_1,\ldots,C_m,P_{m+1},\ldots,P_{n-2})$
is supercritical. Then for any convex bodies 
$K,L$ in $\mathbb{R}^n$ so that $\V_n(K,L,\mathcal{C})>0$, we have
$$
	\V_n(K,L,\mathcal{C})^2 =
	\V_n(K,K,\mathcal{C})\,
	\V_n(L,L,\mathcal{C})
$$
if and only if there exist $a>0$ and $v\in\mathbb{R}^n$ so that
$K$ and $aL+v$ have the same supporting hyperplanes in all
directions in $\supp S_{B,\mathcal{C}}$.
\end{thm}

It must be emphasized that in contrast to the settings of Theorems 
\ref{thm:main} and \ref{thm:quer}, we have not given a geometric 
characterization of $\supp S_{B,\mathcal{C}}$ in the general setting of 
Theorem \ref{thm:zonoid}. This problem is in fact not yet fully settled 
\cite[Conjecture 7.6.14]{Sch14}, and its analysis is outside the scope of 
this paper. However, we will provide such a characterization in 
Proposition \ref{prop:zonsupp}
for the case $m=n-2$ that all reference bodies are zonoids 
(or smooth), completing an analysis due to Schneider \cite{Sch88}.

The basis for the proof of Theorem \ref{thm:zonoid} is a kind of analogue 
of Corollary \ref{cor:smoothsilly}.
Analogous arguments may be found in \cite[Lemma 7.4.7]{Sch14} and
in \cite[\S 4]{Sch88}.

\begin{lem}
\label{lem:subzonoid}
Let $\mathcal{K}=(K_1,\ldots,K_{n-3})$ be
convex bodies in $\mathbb{R}^n$ and $Z$ be a zonoid 
with generating measure $\rho$. Assume $(Z,\mathcal{K})$ is 
supercritical. Then any difference of support functions $f$ so that
$S_{f,Z,\mathcal{K}}=0$ satisfies $S_{f,[-u,u],\mathcal{K}}=0$
for all $u\in\supp\rho$.
\end{lem}

\begin{proof}
Let $f$ be a difference of support functions such that
$S_{f,Z,\mathcal{K}}=0$. Then 
$$
	\int \V_n(f,f,[-u,u],\mathcal{K})\,\rho(du) =
	\V_n(f,f,Z,\mathcal{K})=0,
$$
where we used $h_{[-u,u]}(x)=|\langle u,x\rangle|$ and the definition
of $\rho$. On the other hand, for any $u\in S^{n-1}$, we have 
$\V_n(f,Z,[-u,u],\mathcal{K})=0$ by \eqref{eq:mixvolarea} and
$\V_n(Z,Z,[-u,u],\mathcal{K})>0$ by supercriticality and Lemma 
\ref{lem:dim}. Thus $\V_n(f,f,[-u,u],\mathcal{K})\le 0$ for any $u\in 
S^{n-1}$ by Lemma \ref{lem:3af}. We may therefore conclude that 
$$
	\V_n(f,f,[-u,u],\mathcal{K}) = 0
	\quad\mbox{for all }u\in\supp\rho,
$$
where we used that $u\mapsto \V_n(f,f,[-u,u],\mathcal{K})$ is continuous 
by Lemma \ref{lem:cont}. In particular, it follows from Lemma 
\ref{lem:afeq} that $S_{f,[-u,u],\mathcal{K}}=0$ for every 
$u\in\supp\rho$.
\end{proof}

We can now complete the proof of Theorem \ref{thm:zonoid}.

\begin{proof}[Proof of Theorem \ref{thm:zonoid}]
By Corollary 
\ref{cor:smoothsilly}, we may assume without loss of generality that each 
$C_i$ that is a smooth body satisfies $C_i=B$. But as $B$ is a zonoid, we 
can assume in the remainder of the proof that $C_1,\ldots,C_m$ are zonoids 
with generating measures $\rho_1,\ldots,\rho_m$, respectively. We also fix 
a difference of support functions $f$ so that $S_{f,\mathcal{C}}=0$; by 
Lemmas \ref{lem:simpleeq}, \ref{lem:lineq}, and \ref{lem:suppeq}, it 
suffices to prove that there exists $s\in\mathbb{R}^n$ so that 
$f(x)=\langle s,x\rangle$ for all $s\in\supp S_{B,\mathcal{C}}$.

For each $i\in[m]$, we may choose (for example, using the law of large 
numbers) a sequence $\{u_i^j\}_{j\ge 1}\subseteq\supp\rho_i$ so that the 
zonotopes
$$
	C_i^j := j^{-1}\{[-u_i^1,u_i^1]+\cdots+[-u_i^j,u_i^j]\}
$$
satisfy $C_i^j\to C_i$ as $j\to\infty$ in Hausdorff distance.
In particular, $\aff C_i^j=\aff C_i$ for sufficiently large $j$, so 
that $\mathcal{C}^j:=(C_1^j,\ldots,C_m^j,P_{m+1},\ldots,P_{n-2})$ is 
supercritical for sufficiently large $j$. Using linearity of mixed 
area measures and applying Lemma~\ref{lem:subzonoid} repeatedly, it 
follows that $S_{f,\mathcal{C}^j}=0$ for sufficiently large $j$. 
Therefore, as $\mathcal{C}^j$ consists entirely of polytopes,
Theorem \ref{thm:schneider} yields for every sufficiently 
large $j$ a vector $s_j\in\mathbb{R}^n$ so that
$f(x)=\langle s_j,x\rangle$ for all $x\in\supp S_{B,\mathcal{C}^j}$.

Now note that by linearity of mixed area measures, $\supp S_{B,\mathcal{C}^j}$ 
is increasing in $j$. In particular, we have
$\langle s_j,x\rangle=f(x)=\langle s_{j+1},x\rangle$ for all
$x\in\supp S_{B,\mathcal{C}^j}$ and sufficiently large $j$. It follows 
from Lemma \ref{lem:sbpsupp} and supercriticality that $s_j=s_{j+1}$ for 
all sufficiently large $j$. Thus we have shown that there exists
$s\in\mathbb{R}^n$ such that
$$
	f(x)=\langle s,x\rangle\quad\mbox{for all }x\in\supp S_{B,\mathcal{C}^j}
$$
holds for all sufficiently large $j$. To conclude, note that
$$
	\int |f(x)-\langle s,x\rangle|\,S_{B,\mathcal{C}}(dx) =
	\lim_{j\to\infty}
	\int |f(x)-\langle s,x\rangle|\,S_{B,\mathcal{C}^j}(dx) = 0
$$
by Lemma \ref{lem:cont}, so that $f(x)=\langle s,x\rangle$ for all
$x\in\supp S_{B,\mathcal{C}}$ as well.
\end{proof}

\begin{rem}
We exploited Lemma \ref{lem:subzonoid} above to approximate zonoids 
$C_i$ by zonotopes $C_i^j$. However, one could also attempt to use Lemma 
\ref{lem:subzonoid} as a replacement for the local Alexandrov-Fenchel 
inequality: by Remark \ref{rem:lowdimma}, it implies that for any extremal $f$ of the 
Alexandrov-Fenchel inequality with reference bodies $(Z,\mathcal{K})$ and 
$u\in\supp\rho$, the projection $\proj_{u^\perp}f$ is extremal in 
$u^\perp$ with reference bodies $\proj_{u^\perp}\mathcal{K}$. Such an 
argument was used by Schneider in \cite{Sch88}. The difficulty with this 
approach is that $\proj_{u^\perp}\mathcal{K}$ need not be supercritical 
for $\rho$-a.e.\ $u$ if we only assume that $\mathcal{K}$ is supercritical. 
On the other hand, this method works well when all the bodies in 
$\mathcal{K}$ are full-dimensional, and yields some more general results 
in this case (for example, an analogue of Theorem \ref{thm:quer} where 
some of the bodies $C_i$ are zonoids).
\end{rem}

We now revisit the problem of characterizing $\supp S_{B,\mathcal{C}}$ 
geometrically. When $\mathcal{C}$ are polytopes, such a characterization 
is given in Lemma \ref{lem:supp} in terms of their faces. It has been 
conjectured by Schneider that $\supp S_{B,\mathcal{C}}$ is characterized 
in general by a local analogue of Lemma \ref{lem:supp}, in which the faces 
are replaced by certain ``tangent spaces'' of the convex bodies 
$\mathcal{C}$. Let us recall the relevant notions.

A convex body $C$ in $\mathbb{R}^n$ associates to each of its boundary 
points a cone of outer normal vectors. These cones generate a 
partition of $\mathbb{R}^n$ into relatively open convex cones, which are 
the \emph{touching cones} of $C$. We denote by $T(C,u)$ the unique 
touching cone of $C$ that contains $u\in\mathbb{R}^n$. One may think of 
$T(C,u)^\perp$ as the ``tangent space'' of $C$ with outer normal vector 
$u$. In analogy with Lemma \ref{lem:supp}, we define:

\begin{defn}
\label{defn:bcextreme}
Let $\mathcal{C}=(C_1,\ldots,C_{n-1})$ be convex bodies in $\mathbb{R}^n$. 
Then $u\in S^{n-1}$ is a \emph{$\mathcal{C}$-extreme normal 
direction} if there are segments $I_i\subset T(C_i,u)^\perp$ for
$i\in[n-1]$ with linearly independent directions.
\end{defn}

The above notions are due to Schneider (see \cite[\S 2.2]{Sch14} for 
equivalent definitions). In particular, Schneider has conjectured 
\cite[Conjecture 7.6.14]{Sch14} that $\supp S_{\mathcal{C}}$ always 
coincides with the closure of the set of $\mathcal{C}$-extreme normal 
directions. It is readily verified that this conjecture agrees with the 
special cases of Lemmas \ref{lem:supp} and \ref{lem:quasupp}. We will 
presently verify this conjecture for the case $m=n-2$ of Theorem 
\ref{thm:zonoid}. This is essentially proved in \cite{Sch88},
up to a minor observation.

\begin{prop}
\label{prop:zonsupp}
Let $\mathcal{C}=(C_1,\ldots,C_{n-2})$ be convex bodies in $\mathbb{R}^n$
such that each $C_i$ is either a zonoid or a smooth body. Then
$$
	\supp S_{B,\mathcal{C}} =
	\mathrm{cl}\{u\in S^{n-1}:
	u\mbox{ is a }(B,\mathcal{C})\mbox{-extreme normal direction}\}.
$$
\end{prop}

In the proof we will need the following simple measure-theoretic fact.

\begin{lem}
\label{lem:mixture}
Let $X,Y$ be Polish spaces, $\rho_x$ be a finite measure on $Y$ 
for each $x\in X$, and $\eta$ be a finite measure on $X$.
Assume $x\mapsto\rho_x$ is weakly continuous. Then the measure
$\mu := \int \rho_x\,\eta(dx)$ on $Y$ satisfies
$\supp\rho_x\subseteq\supp\mu$ for every $x\in\supp\eta$.
\end{lem}

\begin{proof}
Let $x\in\supp\eta$ and $y\in\supp\rho_x$. Let
$f:Y\to[0,1]$ be a continuous function so that $f(y)>0$. Then 
$\varepsilon:=\int f\,d\rho_x>0$. As $x\mapsto \int f\,d\rho_x$ is 
continuous, there is an open neighborhood $W\ni x$ so that 
$\int f\,d\rho_z\ge\frac{\varepsilon}{2}$ for $z\in W$. Thus $\int f\,d\mu 
\ge \frac{\varepsilon}{2}\eta(W) >0$.
As this holds for any function $f$ as above, it follows that $y\in 
\supp\mu$. 
\end{proof}

We can now complete the proof of Proposition \ref{prop:zonsupp}.

\begin{proof}[Proof of Proposition \ref{prop:zonsupp}]
When $\V_n(B,B,\mathcal{C})=0$, there exist no 
$(B,\mathcal{C})$-extreme directions by Lemma \ref{lem:dim}, and the 
conclusion is trivial. When $\V_n(B,B,\mathcal{C})>0$, 
Corollary \ref{cor:smoothsilly} shows that $\supp 
S_{B,\mathcal{C}}$ is unchanged if each smooth $C_i$  is 
replaced by $B$. On the other hand, for smooth $C_i$ we have 
$T(C_i,u)=T(B,u)=\mathop{\mathrm{pos}}u$ for all $u$, so the 
$(B,\mathcal{C})$-extreme directions are also unchanged by this 
replacement. Thus we may assume that $C_1,\ldots,C_{n-2}$ are 
zonoids with generating measures $\rho_1,\ldots,\rho_{n-2}$. 

That $\supp S_{B,\mathcal{C}}\subseteq 
\mathrm{cl}\{(B,\mathcal{C})\mbox{-extreme directions}\}$ is shown in 
\cite[Proposition 3.8]{Sch88}. We will prove the converse inclusion by 
induction on $n$. For $n=3$, the conclusion is a special case of Lemma 
\ref{lem:quasupp}. From now on, we assume that the claim has been proved 
in dimension $n-1$, and show that the result follows in dimension $n$.

Let $v\in S^{n-1}$ be a $(B,\mathcal{C})$-extreme normal direction.
Then it is shown in \cite[p.\ 125]{Sch88} that there exists
$j\in[n-2]$ and $u\in v^\perp\cap\supp\rho_j$ so that $v$ is also
a $(\proj_{u^\perp}B,\proj_{u^\perp}\mathcal{C}_{\backslash j})$-extreme 
normal direction (as defined in $u^\perp$). It remains to show that
the latter are included in $\supp S_{B,\mathcal{C}}$. To this end, note 
that
$$
        S_{B,\mathcal{C}} =
        \int S_{B,[-u,u],\mathcal{C}_{\backslash j}}\,
        \rho_j(du),
$$
and that $u\mapsto S_{B,[-u,u],\mathcal{C}_{\backslash j}}$ is continuous
by Lemma \ref{lem:cont}. Thus
$\supp S_{B,[-u,u],\mathcal{C}_{\backslash j}}\subseteq
\supp S_{B,\mathcal{C}}$ for every $u\in\supp\rho_j$ by Lemma 
\ref{lem:mixture}. But the induction hypothesis and Remark 
\ref{rem:lowdimma} imply that $v\in \supp 
S_{\proj_{u^\perp}B,\proj_{u^\perp}\mathcal{C}_{\backslash j}} =
\supp S_{B,[-u,u],\mathcal{C}_{\backslash j}}$. Thus we have shown that 
any $(B,\mathcal{C})$-extreme normal direction is contained in
$\supp S_{B,\mathcal{C}}$, and the conclusion follows as the latter is 
a closed set.
\end{proof}

\section{Application to combinatorics of partially ordered sets}
\label{sec:stanley}

A sequence $N_1,\ldots,N_n$ of positive numbers is \emph{log-concave} if 
$N_i^2\ge N_{i-1}N_{i+1}$ for all $i\in\{2,\ldots,n-1\}$. It was noticed 
long ago that log-concave sequences arise in a surprisingly broad range of 
combinatorial problems \cite{Sta89}. One explanation for this phenomenon 
appears in the work of Stanley \cite{Sta81}, who observed that if one can 
represent the relevant combinatorial quantities in terms of mixed volumes, 
log-concavity arises as a consequence of the Alexandrov-Fenchel 
inequality. This provides a common mechanism for the emergence of 
log-concavity in several combinatorial problems that appear to be 
otherwise unrelated. In recent years, it has been realized that this idea 
extends to a much broader setting: even for combinatorial problems that 
may not be represented in terms of classical convexity, one may often 
develop algebraic analogues of the Alexandrov-Fenchel inequality that 
explain the emergence of log-concavity. Such ideas have led to a series of 
recent breakthroughs in combinatorics due to Huh et al.\ \cite{Huh18}.

As was explained in the introduction, one may view the Alexandrov-Fenchel 
inequality as a generalized isoperimetric inequality. In particular, 
associated to any instance of the Alexandrov-Fenchel inequality is a 
corresponding extremal problem: what bodies minimize the left-hand side in 
Theorem \ref{thm:af} when the right-hand side is fixed? One may 
analogously associate to any log-concave sequence of combinatorial 
quantities $(N_i)$ a corresponding extremal problem: what combinatorial 
objects achieve equality $N_i=N_{i+1}N_{i-1}$ for a given $i$? This 
question was already 
posed by Stanley in \cite{Sta81}. Despite deep advances in understanding 
log-concavity through Alexandrov-Fenchel type inequalities, the analysis 
of the associated extremal problems appears to be inaccessible by 
currently known methods.

In this section, we will show how the theory developed in this paper makes 
it possible to settle such extremal problems in Stanley's original 
setting. For sake of illustration, we focus on one particular example from 
\cite{Sta81} that arises in the combinatorics of partially ordered sets; 
other combinatorial applications of the Alexandrov-Fenchel inequality may 
be investigated analogously. Whether the theory of this paper has 
analogues outside convexity is an intriguing question (cf.\ section 
\ref{sec:discussion}).

\subsection{Linear extensions and extremal posets}

Let $\alpha:=\{x,y_1,\ldots,y_{n-1}\}$ be a partially ordered set 
(poset) that will be fixed throughout this section. We denote by $N_i$ the 
number of order-preserving bijections $\sigma:\alpha\to[n]$ such that
$\sigma(x)=i$; that is, $N_i$ is the number of linear extensions 
of the partial order of $\alpha$ for which $x$ has rank $i$. The 
following was conjectured by Chung, Fishburn, and Graham \cite{CFG80}.

\begin{thm}[Stanley \cite{Sta81}]
\label{thm:stanley}
The sequence $N_1,\ldots,N_n$ is log-concave.
\end{thm}

\begin{proof}
The poset $\alpha$ defines polytopes $K,L$ in $\mathbb{R}^{n-1}$ by
\begin{align*}
	K &:= \{
	t\in[0,1]^{n-1}:t_j\le t_k\mbox{ if }y_j\le y_k,~
	t_j=1\mbox{ if }y_j>x\},\\
	L &:= \{
	t\in[0,1]^{n-1}:t_j\le t_k\mbox{ if }y_j\le y_k,~
	t_j=0\mbox{ if }y_j<x\}.
\end{align*}
Let us denote
$$
	\mathcal{K}_l :=
	(\underbrace{K,\ldots,K}_{l}),\qquad\quad
	\mathcal{L}_m :=
	(\underbrace{L,\ldots,L}_{m}).
$$
Then it is shown in \cite[Theorem 3.2]{Sta81} that
$$
	N_i = (n-1)!\,\V_{n-1}(\mathcal{K}_{i-1},\mathcal{L}_{n-i}).
$$
The conclusion is now immediate by the Alexandrov-Fenchel inequality.
\end{proof}

The extremal question associated to Theorem \ref{thm:stanley} is: given 
$i\in\{2,\ldots,n-1\}$, which posets $\alpha$ attain equality 
$N_i^2=N_{i+1}N_{i-1}$? The proof of Theorem \ref{thm:stanley} 
reduces this question to a special case of Theorem \ref{thm:main}. To 
obtain a result of combinatorial interest, however, one must deduce from 
the geometric conditions of Theorem \ref{thm:main} a combinatorial 
characterization of the corresponding poset $\alpha$. We presently state 
the resulting theorem, whose proof will occupy the remainder of this 
section.

As in Theorem \ref{thm:main}, we must distinguish between trivial and 
nontrivial extremals. When $N_i=0$, log-concavity implies that we always 
have $N_i^2=N_{i+1}N_{i-1}$ for trivial reasons. Let us first characterize 
when this happens. In the following, we denote by 
$\alpha_{\mathrm{R}z}:=\{y\in\alpha:y\mathrm{R}z\}$ for any relation 
$\mathrm{R}\in\{<,\le,>,\ge\}$ of $\alpha$.

\begin{lem}[Trivial extremals]
\label{lem:trivexst}
For any $i\in[n]$, we have
$$
	N_i=0\qquad\mbox{if and only if}\qquad
	|\alpha_{<x}|>i-1
	\quad\mbox{or}\quad 
	|\alpha_{>x}|>n-i.
$$
\end{lem}

This simple lemma, which will be proved in section \ref{sec:order}, is 
intuitively obvious: it states that no linear extension of $\alpha$ can 
give $x$ rank $i$ if there are $i$ elements of $\alpha$ that are smaller 
than $x$ (as these elements must have smaller rank than $x$), or 
analogously if there are $n-i+1$ elements larger than $x$. This statement 
has an easy direct proof. In contrast, 
the characterization of the nontrivial extremals is not obvious. The 
following theorem is the main result of this section.

\begin{thm}[Nontrivial extremals]
\label{thm:exst}
Let $i\in\{2,\ldots,n-1\}$ be such that $N_i>0$. Then the following are
equivalent:
\begin{enumerate}[a.]
\item $N_i^2=N_{i+1}N_{i-1}$.
\item $N_i=N_{i+1}=N_{i-1}$.
\item Every linear extension
$\sigma:\alpha\to[n]$ with $\sigma(x)=i$
assigns ranks $i-1$ and $i+1$
to elements of $\alpha$ that are incomparable to $x$.
\item $|\alpha_{<y}|>i$ for all $y\in\alpha_{>x}$, and
$|\alpha_{>y}|>n-i+1$ for all $y\in\alpha_{<x}$.
\end{enumerate}
\end{thm}

The formulation of condition $d$ of Theorem \ref{thm:exst} was derived by 
the authors from the analysis of the associated Alexandrov-Fenchel 
inequality (section \ref{sec:stanleypf}). Once the correct statement has 
been realized, however, it is straightforward to find a direct proof of 
the easy directions $d\Rightarrow c\Rightarrow b\Rightarrow a$ of Theorem 
\ref{thm:exst}. 

\begin{proof}[Proof of Theorem \ref{thm:exst}, $d\Rightarrow c\Rightarrow
b\Rightarrow a$]
Fix $i\in\{2,\ldots,n-1\}$ such that $N_i>0$, and assume that 
condition $d$ holds. We first show that this implies condition $c$. 
Indeed, suppose to the contrary that $\sigma(y)=i-1$ for some 
$y\in\alpha_{<x}$; 
then we have $|\alpha_{>y}|\le
|\{z\in\alpha:\sigma(z)>\sigma(y)\}|=n-i+1$, contradicting
condition $d$. We can analogously rule out that $\sigma(y)=i+1$ for
some $y\in\alpha_{>x}$.

We now show that condition $c$ implies $b$.
Denote by $\mathcal{N}_i$ the set of linear extensions 
$\sigma:\alpha\to[n]$ with $\sigma(x)=i$ (so that $N_i=|\mathcal{N}_i|$). 
We further denote by $\mathcal{N}_i^\pm:=\{\Pi_{i,i\pm 1}\sigma:
\sigma\in\mathcal{N}_i\}$, where $\Pi_{i,j}$ denotes the permutation of 
$[n]$ that exchanges $i,j$. Condition $c$ now implies 
$\mathcal{N}_i^\pm\subseteq\mathcal{N}_{i\pm 1}$, so that
$N_i=|\mathcal{N}_i^\pm|\le |\mathcal{N}_{i\pm 1}|=N_{i\pm 1}$. Thus
$$
	N_i^2\ge N_{i+1}N_{i-1} \ge N_iN_{i-1} \ge N_i^2,
$$
where the first inequality follows from Theorem \ref{thm:stanley}.
As $N_i>0$, condition $b$ follows readily. This completes 
the proof, as the implication $b\Rightarrow a$ is trivial.
\end{proof}

Condition $c$ of Theorem \ref{thm:exst} identifies one particular 
combinatorial mechanism that gives rise to equality in Theorem 
\ref{thm:stanley}. It is far from obvious, however, why this suffices to 
yield a complete solution to the extremal problem. The hard part of 
Theorem \ref{thm:exst} is to show that this is the \emph{only} mechanism 
that gives rise to equality, which will be proved in section 
\ref{sec:stanleypf} below.

Let us emphasize that the combinatorial characterization of the equality 
cases of Stanley's inequality has strong consequences beyond the solution 
of the extrmal problem itself: it provides detailed information on the 
shape of the log-concave sequences $N_1,\ldots,N_n$ that can arise in 
Theorem \ref{thm:stanley}. For example, the implication $a\Rightarrow b$ 
of Theorem \ref{thm:exst} shows that such sequences cannot contain any 
3-term geometric progressions. As any positive log-concave sequence is 
unimodal, it follows from Lemma \ref{lem:trivexst} and Theorem 
\ref{thm:exst} that one can decompose $[n]=I_1\cup\cdots\cup I_5$ into 
consecutive (possibly empty) intervals $I_k$ so that the sequence 
$N_1,\ldots,N_n$ has the form that is illustrated in Figure 
\ref{fig:stshape}. Furthermore, Lemma \ref{lem:trivexst} and Theorem 
\ref{thm:exst} enable us to compute the length of each interval 
explicitly for any given poset.
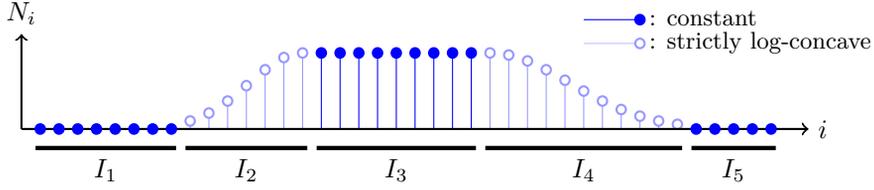
\begin{figure}
\centering
\begin{tikzpicture}[scale=.25]

\foreach \i in {9,...,15}
{
	\draw[color=blue!40!white] (\i,{4*exp(-(\i-15)^2/16)}) -- (\i,0);
	\draw[color=blue!40!white,fill=white,thick] (\i,{4*exp(-(\i-15)^2/16)}) circle (0.25);
}

\foreach \i in {16,...,24}
{
	\draw[color=blue,fill=blue,thick] (\i,4) circle (0.25);
	\draw[color=blue] (\i,4) -- (\i,0);
}

\foreach \i in {25,...,35}
{
	\draw[color=blue!40!white] (\i,{4*exp(-(25-\i)^2/36)}) -- (\i,0);
	\draw[color=blue!40!white,fill=white,thick] (\i,{4*exp(-(25-\i)^2/36)}) circle (0.25);
}

\draw[thick,->] (0,0) -- (0,5) node[above] {$N_i$};
\draw[thick,->] (0,0) -- (42,0) node[right] {$i$};

\draw[ultra thick] (.75,-1) -- (8.25,-1) node[midway,below] {$I_1$};
\draw[ultra thick] (8.75,-1) -- (15.25,-1) node[midway,below] {$I_2$};
\draw[ultra thick] (15.75,-1) -- (24.25,-1) node[midway,below] {$I_3$};
\draw[ultra thick] (24.75,-1) -- (35.25,-1) node[midway,below] {$I_4$};
\draw[ultra thick] (35.75,-1) -- (40.25,-1) node[midway,below] {$I_5$};

\foreach \i in {1,...,8}
{
	\draw[color=blue,fill=blue,thick] (\i,0) circle (0.25);
}

\foreach \i in {36,...,40}
{
	\draw[color=blue,fill=blue,thick] (\i,0) circle (0.25);
}

\draw[color=blue] (30,5.75) -- (33,5.75);
\draw[color=blue,fill=blue,thick] (33,5.75) circle (0.25);
\draw (33,5.85) node[right] {\small : constant};

\draw[color=blue!25!white] (30,4.5) -- (33,4.5);
\draw[color=blue!40!white,fill=white,thick] (33,4.5) circle (0.25);
\draw (33,4.55) node[right] {\small : strictly log-concave};

\end{tikzpicture}
\caption{Structure of the sequence $N_1,\ldots,N_n$.\label{fig:stshape}}
\end{figure}

\begin{example}
For any $k,l,r,s,t\ge 1$, consider the poset 
$$
	\alpha=\{x,(y_i)_{i\in[k]},(z_i)_{i\in[l]},
	(u_i)_{i\in[r-1]},(v_i)_{i\in[s-1]},(w_i)_{i\in[t+1]}\}
$$
defined by the following relations:
\begin{align*}
	y_1<\cdots<\mbox{}&y_k<x<z_1<\cdots<z_l, &
	y_k<w_1<\cdots<w_{t+1}<\mbox{}&z_1, \\
	&y_k<u_1<\cdots<u_{r-1}, &
	v_1<\cdots<v_{s-1}<\mbox{}&z_1.
\end{align*}
Then $|\alpha_{<x}|=k$, $|\alpha_{>x}|=l$, 
$\min_{y\in\alpha_{<x}}|\alpha_{>y}|=l+r+t+1$, and
$\min_{y\in\alpha_{>x}}|\alpha_{<y}|=k+s+t+1$.
Therefore, Lemma \ref{lem:trivexst} and
Theorem \ref{thm:exst} imply that the sequence $N_1,\ldots,N_{|\alpha|}$
has the form that is illustrated in Figure \ref{fig:stshape} with
$$
	|I_1|=k,\qquad
	|I_2|=s,\qquad
	|I_3|=t,\qquad
	|I_4|=r,\qquad
	|I_5|=l.
$$
Thus any decomposition $[n]=I_1\cup\cdots\cup I_5$ with 
$I_1,\ldots,I_5\ne\varnothing$ is achievable.
This example is readily modified to construct situations where some $I_k$
may be empty.
\end{example}

\subsection{Order polytopes}
\label{sec:order}

The convex bodies $K,L$ that appear in the proof of Theorem 
\ref{thm:stanley} are examples of order polytopes. Before we proceed to 
the proof of Theorem \ref{thm:exst}, we must recall some basic properties of 
such polytopes.

We fix once and for all the poset $\alpha:=\{x,y_1,\ldots,y_{n-1}\}$, and 
define $\bar\alpha:=\alpha\backslash\{x\}$. For any $\beta\subseteq\alpha$ 
(with the induced partial order) and $z\in\alpha$, we will denote by 
$\beta_{\mathrm{R}z}:=\{y\in\beta: y\mathrm{R}z\}$ for 
$\mathrm{R}\in\{<,\le,>,\ge\}$, by $\beta_{\not\sim z}$ the set of 
elements of $\beta$ that are not comparable to $z$, and by 
$\beta^\uparrow$ ($\beta^\downarrow$) the set of maximal (minimal) 
elements of $\beta$. 

For $y,z\in\beta$, we say that $z$ \emph{covers}
$y$ in $\beta$ if $z\in(\beta_{>y})^\downarrow$. Moreover,
$\beta\subseteq\alpha$ is called an \emph{upper} (\emph{lower}) set in 
$\alpha$ if
$\alpha_{>y}\subseteq\beta$ ($\alpha_{<y}\subseteq\beta$) for every
$y\in\beta$.

Now consider $\beta\subseteq\bar\alpha$. By a slight abuse of notation, we 
will define the subspace $\mathbb{R}^\beta := 
\{t\in\mathbb{R}^{n-1}:t_i=0\mbox{ for } y_i\not\in\beta\}$.
The \emph{order polytope} $O_\beta$ is defined as
$$
	O_\beta :=
	\{t\in\mathbb{R}^\beta:t_j\in[0,1],\mbox{ and }
	t_j\le t_k\mbox{ if }y_j\le y_k,\mbox{ for all }
	y_j,y_k\in\beta\}.
$$
The following basic facts may be found in \cite[\S 1]{Sta86}.

\begin{lem}
\label{lem:order}
For any $\beta\subseteq\bar\alpha$, we have
$\dim O_\beta = |\beta|$. The $(|\beta|-1)$-dimensional faces of 
$O_\beta$ are precisely the following subsets of $O_\beta$:
\begin{enumerate}[1.]
\item $O_\beta\cap\{t_j=0\}$ for $y_j\in\beta^\downarrow$.
\item $O_\beta\cap\{t_j=1\}$ for $y_j\in\beta^\uparrow$.
\item $O_\beta\cap\{t_j=t_k\}$ for $y_j,y_k\in\beta$ such that
$y_k$ covers $y_j$ in $\beta$.
\end{enumerate}
\end{lem}

In the sequel, we denote by $e_1,\ldots,e_{n-1}$ the coordinate
basis of $\mathbb{R}^{n-1}$, and we let
$1_\beta:=\sum_{y_j\in\beta}e_j$ for any $\beta\subseteq\bar\alpha$.
We can now formulate the following basic fact.

\begin{lem}
\label{lem:upperlower}
If $\beta\subseteq\bar\alpha$ is a lower set in $\bar\alpha$, then
$$
	O_{\bar\alpha\backslash\beta} = \{ t\in[0,1]^{n-1}:
	t_j\le t_k\mbox{ if }y_j\le y_k,~t_j=0\mbox{ if }y_j\in\beta\}.
$$
Analogously, if $\beta\subseteq\bar\alpha$ is an upper set in 
$\bar\alpha$, then
$$
	O_{\bar\alpha\backslash\beta} + 1_\beta = \{ t\in[0,1]^{n-1}:
	t_j\le t_k\mbox{ if }y_j\le y_k,~t_j=1\mbox{ if }y_j\in\beta\}.
$$
\end{lem}

\begin{proof}
Suppose $\beta$ is a lower set and $t\in[0,1]^{n-1}$ satisfies $t_j=0$ for 
all $y_j\in\beta$. Then $t_j\le t_k$ holds trivially for $y_j\le y_k$ 
with $y_j\in\beta$. On the other hand, as $\beta$ is a lower set, $y_j\le 
y_k$ with $y_j\not\in\beta$ implies $y_k\not\in\beta$. Thus the only 
nontrivial constraints $t_j\le t_k$ for $y_j\le y_k$ are those that appear 
in the definition of $O_{\bar\alpha\backslash\beta}$, which concludes the 
proof when $\beta$ is a lower set. The proof when
$\beta$ is an upper set is analogous. 
\end{proof}

In the remainder of this section, we will define the polytopes $K,L$ as in 
the proof of Theorem \ref{thm:stanley}. By Lemma \ref{lem:upperlower}, we 
can write equivalently
$$
	K = O_{\bar\alpha\backslash\alpha_{>x}} + 1_{\alpha_{>x}},
	\qquad\quad
	L = O_{\bar\alpha\backslash\alpha_{<x}},
$$
where we used that $\alpha_{>x}$ is an upper set and
$\alpha_{<x}$ is a lower set in $\bar\alpha$.

\begin{lem}
\label{lem:dimkl}
$\dim K = n-1-|\alpha_{>x}|$, $\dim L=n-1-|\alpha_{<x}|$, and
$\dim(K+L)=n-1$.
\end{lem}

\begin{proof}
Lemma \ref{lem:order} shows that $K-1_{\alpha_{>x}}$ is a full-dimensional
polytope in $\mathbb{R}^{\bar\alpha\backslash\alpha_{>x}}$, and that
$L$ is a full-dimensional polytope in 
$\mathbb{R}^{\bar\alpha\backslash\alpha_{<x}}$. It follows that
$\dim K = |\bar\alpha\backslash\alpha_{>x}|$,
$\dim L=|\bar\alpha\backslash\alpha_{<x}|$, and
$\dim(K+L) = |(\bar\alpha\backslash\alpha_{>x})\cup
(\bar\alpha\backslash\alpha_{<x})|=|\bar\alpha|$.
\end{proof}

Let us finally prove Lemma \ref{lem:trivexst}. While a direct 
combinatorial proof is a simple exercise, we find it 
instructive to show how it arises from the mixed volumes.

\begin{proof}[Proof of Lemma \ref{lem:trivexst}]
The representation in the proof of Theorem \ref{thm:stanley} 
shows that $N_i=0$ if and only if 
$\V_{n-1}(\mathcal{K}_{i-1},\mathcal{L}_{n-i})=0$.
By Lemma \ref{lem:dim}, this is the case if and only if either 
$\dim(K+L)<n-1$, $\dim(K)<i-1$, or $\dim(L)<n-i$. The conclusion now 
follows immediately from Lemma \ref{lem:dimkl}.
\end{proof}

\subsection{Combinatorial characterization of the extremals}
\label{sec:stanleypf}

We now turn to the proof of Theorem \ref{thm:exst}. We have already proved 
the implications $d\Rightarrow c\Rightarrow b\Rightarrow a$; the remainder 
of this section is devoted to the proof of the implication $a\Rightarrow d$.

We begin by reducing the problem to an extremal case of the 
Alexandrov-Fenchel inequality in $\mathbb{R}^{n-1}$. While the polytopes 
$K,L\subset\mathbb{R}^{n-1}$ generally have empty interior, it turns out 
that nontrivial extremals can only arise in the present setting from a 
supercritical case of the Alexandrov-Fenchel inequality.

\begin{lem}
\label{lem:staf}
Let $i\in\{2,\ldots,n-1\}$ be such that $N_i>0$ and $N_i^2=N_{i+1}N_{i-1}$.
Then $|\alpha_{<x}|+1 < i < n-|\alpha_{>x}|$, and
there exist $a>0$ and $v\in\mathbb{R}^{n-1}$ so that
$$
	h_K(x)=h_{aL+v}(x) \quad\mbox{for all}\quad 
	x\in\supp S_{B,\mathcal{K}_{i-2},\mathcal{L}_{n-i-1}}.
$$
\end{lem}

\begin{proof}
We note first that $N_i>0$ and $N_i^2=N_{i+1}N_{i-1}$ imply that
$N_{i-1}>0$ and $N_{i+1}>0$. It therefore follows from Lemma 
\ref{lem:trivexst} that $|\alpha_{<x}|+1 < i < n-|\alpha_{>x}|$.
Moreover, this implies by Lemma \ref{lem:dimkl} that
$\dim K \ge i$ and $\dim L \ge n-i+1$, and $\dim(K+L)=n-1$. Thus the 
collection $(\mathcal{K}_{i-2},\mathcal{L}_{n-i-1})$ is supercritical.

Now note that by the mixed volume representation in the proof of Theorem 
\ref{thm:stanley}, $N_i>0$ and $N_i^2=N_{i+1}N_{i-1}$ imply
$\V_{n-1}(K,L,\mathcal{K}_{i-2},\mathcal{L}_{n-i-1})>0$ and
$$
	\V_{n-1}(K,L,\mathcal{K}_{i-2},\mathcal{L}_{n-i-1})^2 =
	\V_{n-1}(K,K,\mathcal{K}_{i-2},\mathcal{L}_{n-i-1})\,
	\V_{n-1}(L,L,\mathcal{K}_{i-2},\mathcal{L}_{n-i-1}).
$$
The conclusion therefore follows from Corollary \ref{cor:schneider}.
\end{proof}

To exploit Lemma \ref{lem:staf}, there are two distinct difficulties: we 
must gain some understanding of which vectors lie in $\supp 
S_{B,\mathcal{K}_{i-2},\mathcal{L}_{n-i-1}}$, and we must understand how 
to exploit the fact that the supporting hyperplanes of $K$ and $aL+v$ 
coincide in these directions. We begin by addressing the first issue.

\begin{lem}
\label{lem:stsuppkl}
Let $i\in\{2,\ldots,n-1\}$ be such that $N_i>0$. Then the following hold:
\begin{enumerate}[a.]
\item $-e_j\in \supp S_{B,\mathcal{K}_{i-2},\mathcal{L}_{n-i-1}}$
for $y_j\in\bar\alpha^\downarrow$.
\item $-e_j\in \supp S_{B,\mathcal{K}_{i-2},\mathcal{L}_{n-i-1}}$
for $y_j\in(\alpha_{>x})^\downarrow$ such that 
$|\alpha_{<y_j}| \le i$.
\item $e_j\in \supp S_{B,\mathcal{K}_{i-2},\mathcal{L}_{n-i-1}}$
for $y_j\in\bar\alpha^\uparrow$.
\item $e_j\in \supp S_{B,\mathcal{K}_{i-2},\mathcal{L}_{n-i-1}}$
for $y_j\in(\alpha_{<x})^\uparrow$ such that 
$|\alpha_{>y_j}| \le n-i+1$.
\item $e_{jk}:=\frac{e_j-e_k}{\sqrt{2}}\in \supp 
S_{B,\mathcal{K}_{i-2},\mathcal{L}_{n-i-1}}$
when the following conditions are all satisfied:
\begin{enumerate}[\rm (i)]
\item $y_k$ covers $y_j$ in $\bar\alpha$; 
\item if $y_j\in\alpha_{<x}$, $y_k\in\bar\alpha\backslash\alpha_{<x}$ then
$y_k\in(\bar\alpha\backslash\alpha_{<x})^\downarrow$;
\item if $y_k\in\alpha_{>x}$, $y_j\in\bar\alpha\backslash\alpha_{>x}$, 
then $y_j\in(\bar\alpha\backslash\alpha_{>x})^\uparrow$.
\end{enumerate}
\end{enumerate}
\end{lem}

\begin{proof}
As $N_i>0$, Lemma \ref{lem:trivexst} implies $|\alpha_{<x}|+1 \le i \le 
n-|\alpha_{>x}|$. Moreover, by Lemma \ref{lem:supp}, we have
$u\in \supp S_{B,\mathcal{K}_{i-2},\mathcal{L}_{n-i-1}}$ if and only if
$$
	\dim F(K,u) \ge i-2,\quad
	\dim F(L,u) \ge n-i-1,\quad
	\dim F(K+L,u) \ge n-3.
$$
The latter condition will be verified in each part of the lemma.

\medskip

\textbf{Part $\bm a$.} We begin by noting that for any $j\in[n-1]$,
$$
	F(K,-e_j) = K\cap\{t_j=1_{y_j\in\alpha_{>x}}\},\qquad
	F(L,-e_j) = L\cap\{t_j=0\}.
$$
Indeed, it is readily seen by the definitions of $K,L$ that $h_K(-e_j) = 
-\inf_{t\in K}t_j=-1_{y_j\in\alpha_{>x}}$ and $h_L(-e_j)=-\inf_{t\in 
L}t_j=0$, so the claim follows from \eqref{eq:facedef}.

Now let $y_j\in\bar\alpha^\downarrow$. We claim that the following hold:
\begin{enumerate}[1.]
\item $F(K,-e_j)$ is a full-dimensional polytope in 
$\mathbb{R}^{\bar\alpha\backslash(\alpha_{>x}\cup\{y_j\})}+1_{\alpha_{>x}}$; 
and
\item 
$F(L,-e_j)$ is a full-dimensional polytope in 
$\mathbb{R}^{\bar\alpha\backslash(\alpha_{<x}\cup\{y_j\})}$.
\end{enumerate}
Indeed, recall first that 
$K=O_{\bar\alpha\backslash\alpha_{>x}}+1_{\alpha_{>x}}$ by Lemma 
\ref{lem:upperlower}. If $y_j\in\alpha_{>x}$, then $F(K,-e_j)=K$ and
the first claim follows directly from Lemma \ref{lem:order}.
On the other hand, if $y_j\in\bar\alpha\backslash\alpha_{>x}$, then 
$F(K,-e_j) = K\cap\{t_j=0\} = O_{\bar\alpha\backslash\alpha_{>x}}\cap
\{t_j=0\} + 1_{\alpha_{>x}}$. But then $y_j\in\bar\alpha^\downarrow$ 
implies $y_j\in(\bar\alpha\backslash\alpha_{>x})^\downarrow$, and 
the first claim follows again from Lemma \ref{lem:order}. The proof of the 
second claim is completely analogous.

To conclude the proof of part $a$, it suffices to note that the above 
claims imply
\begin{align*}
	\dim F(K,-e_j) &= |\bar\alpha\backslash(\alpha_{>x}\cup\{y_j\})|
	\ge n-2-|\alpha_{>x}|\ge i-2, \\
	\dim F(L,-e_j) &= |\bar\alpha\backslash(\alpha_{<x}\cup\{y_j\})|
	\ge n-2-|\alpha_{<x}|\ge n-i-1, \\
	\dim F(K+L,-e_j) &=
	|\bar\alpha\backslash(\alpha_{>x}\cup\{y_j\})
	\cup
	\bar\alpha\backslash(\alpha_{<x}\cup\{y_j\})|
	=
	|\bar\alpha\backslash\{y_j\}| = n-2,
\end{align*}
where we used $|\alpha_{<x}|+1 \le i \le
n-|\alpha_{>x}|$ and $\alpha_{>x}\cap \alpha_{<x}=\varnothing$.

\medskip

\textbf{Part $\bm b$.} 
Let $y_j\in (\alpha_{>x})^\downarrow$ with $|\alpha_{<y_j}|\le i$.
We already showed in part $a$ that
$$
	F(K,-e_j) = K,\qquad\quad
	F(L,-e_j) = L\cap\{t_j=0\}
$$
and that 
$F(K,-e_j)$ is a full-dimensional polytope in
$\mathbb{R}^{\bar\alpha\backslash\alpha_{>x}}+1_{\alpha_{>x}}$
(the proofs of these facts for $y_j\in\alpha_{>x}$ did not use the
assumption of part $a$).

Now note that as $y_j\in\alpha_{>x}$, we have
$y_i\le y_j$ for all $y_i\in\alpha_{<x}$. It follows that
$$
	L\cap\{t_j=0\} =
	\{ t\in[0,1]^{n-1}:
        t_i\le t_k\mbox{ if }y_i\le y_k,~t_i=0\mbox{ if }y_i\le y_j\}
	=
	O_{\bar\alpha\backslash \bar\alpha_{\le y_j}},
$$
where we used Lemma \ref{lem:upperlower} and
that $\bar\alpha_{\le y_j}$ is a lower set in $\bar\alpha$. It therefore
follows from Lemma \ref{lem:order} that
$F(L,-e_j)$ is a full-dimensional polytope in
$\mathbb{R}^{\bar\alpha\backslash \bar\alpha_{\le y_j}}$.

To conclude the proof of part $b$, note that
as $i\le n-|\alpha_{>x}|$, we have
$$
	\dim F(K,-e_j) = |\bar\alpha\backslash\alpha_{>x}|
	= n-1-|\alpha_{>x}| \ge i-1.
$$
On the other hand, as $y_j\in\alpha_{>x}$, we have
$|\bar\alpha_{\le y_j}|=|\alpha_{<y_j}|\le i$, so that
$$
	\dim F(L,-e_j) = |\bar\alpha\backslash \bar\alpha_{\le y_j}|
	= n-1 - |\bar\alpha_{\le y_j}| \ge n-i-1.
$$
Finally, we note that 
$$
	\dim F(K+L,-e_j) =
	|\bar\alpha\backslash\alpha_{>x}\cup
	\bar\alpha\backslash \bar\alpha_{\le y_j}|
	=
	|\bar\alpha\backslash\{y_j\}|=n-2,
$$
where we used that $\alpha_{>x}\cap \bar\alpha_{\le y_j}=\{y_j\}$ as 
$y_j\in(\alpha_{>x})^\downarrow$.

\medskip

\textbf{Parts $\bm c$ and $\bm d$.} The proofs are completely
analogous to those of parts $a$ and $b$.

\medskip

\textbf{Part $\bm e$.} We begin by noting that 
$$
	F(K,e_{jk}) = K\cap\{t_j=t_k\},\qquad\quad
	F(L,e_{jk}) = L\cap\{t_j=t_k\}
$$
whenever $y_j<y_k$. Indeed, as $y_j<y_k$ implies $t_j\le t_k$ for any 
$t\in K$, it follows readily from the definition of $K$ that
$h_K(e_{jk}) = 2^{-1/2}\sup_{t\in K} (t_j-t_k) = 0$. That
$h_L(e_{jk})=0$ follows analogously, and the claim now follows from 
\eqref{eq:facedef}.

Now suppose $y_j,y_k$ satisfy conditions (i)--(iii) of part $e$.
We claim the following.
\begin{enumerate}[1.]
\item $F(K,e_{jk})\subset \mathbb{R}^{\bar\alpha\backslash\alpha_{>x}}+
1_{\alpha_{>x}}$ with $\dim F(K,e_{jk})\ge 
|\bar\alpha\backslash\alpha_{>x}|-1$.
\item $F(L,e_{jk})\subset \mathbb{R}^{\bar\alpha\backslash\alpha_{<x}}$ 
with $\dim F(L,e_{jk})\ge |\bar\alpha\backslash\alpha_{<x}|-1$.
\end{enumerate}
Indeed, as
$K=O_{\bar\alpha\backslash\alpha_{>x}}+1_{\alpha_{>x}}$ by Lemma 
\ref{lem:upperlower}, the first
part of the first claim is immediate. For the second part of the first 
claim, we consider three cases.
\begin{enumerate}[$\bullet$]
\itemsep\abovedisplayskip
\item If $y_j,y_k\in\alpha_{>x}$, then $F(K,e_{jk})=K$, so
$\dim F(K,e_{jk})=|\bar\alpha\backslash\alpha_{>x}|$ by Lemma 
\ref{lem:order}.
\item If $y_j\in\bar\alpha\backslash\alpha_{>x}$ and
$y_k\in\alpha_{>x}$, then 
$F(K,e_{jk})=
O_{\bar\alpha\backslash\alpha_{>x}}\cap\{t_j=1\}+1_{\alpha_{>x}}$, and 
(iii) states that
$y_j\in (\bar\alpha\backslash\alpha_{>x})^\uparrow$. Thus 
$\dim F(K,e_{jk}) = |\bar\alpha\backslash\alpha_{>x}|-1$
by Lemma \ref{lem:order}.
\item If $y_j,y_k\in \bar\alpha\backslash\alpha_{>x}$, then
$F(K,e_{jk})=O_{\bar\alpha\backslash\alpha_{>x}}\cap\{t_j=t_k\}+1_{\alpha_{>x}}$,
and (i) implies that
$y_k$ covers $y_j$ in $\bar\alpha\backslash\alpha_{>x}$. Thus
$\dim F(K,e_{jk}) = |\bar\alpha\backslash\alpha_{>x}|-1$
by Lemma \ref{lem:order}.
\end{enumerate}
This proves the first claim. The proof of the second claim is 
completely analogous (using condition (ii) rather than condition (iii)).

To conclude the proof of part $e$, note that the above claims imply
\begin{align*}
	\dim F(K,e_{jk}) &\ge 
	n-2-|\alpha_{>x}| \ge i-2,
\\	\dim F(L,e_{jk}) &\ge 
	n-2-|\alpha_{<x}| \ge n-i-1
\end{align*}
as $|\alpha_{<x}|+1 \le i \le n-|\alpha_{>x}|$. On the other hand, we have
\begin{align*}
	\dim F(K+L,e_{jk}) &\ge
	\dim F(K,e_{jk}) +
	\dim F(L,e_{jk}) -
	|\bar\alpha\backslash\alpha_{>x}\cap
	\bar\alpha\backslash\alpha_{<x}|
	\\ &\ge 
	|\bar\alpha\backslash\alpha_{>x}
	\cup \bar\alpha\backslash\alpha_{<x}| - 2
	= n-3,
\end{align*}
where we used that $\alpha_{>x}\cap\alpha_{<x}=\varnothing$.
The proof is complete.
\end{proof}

From this point onwards we place ourselves in the setting of Lemma 
\ref{lem:staf}. In particular, we will assume without further comment that
$i\in\{2,\ldots,n-1\}$ with $N_i>0$ and
$N_i^2=N_{i+1}N_{i-1}$, and that $a>0$, 
$v\in\mathbb{R}^{n-1}$ have been fixed so that
$$
	h_K(x)=h_{aL+v}(x) \quad\mbox{for all}\quad
        x\in\supp S_{B,\mathcal{K}_{i-2},\mathcal{L}_{n-i-1}}.
$$
Let us begin by formulating a first consequence of Lemma \ref{lem:stsuppkl}.

\begin{lem}
\label{lem:sthelper}
The following hold.
\begin{enumerate}[a.]
\item $v_j=0$ for $y_j\in \bar\alpha^\downarrow\backslash\alpha_{>x}$.
\item $v_j=1-a$ for $y_j\in\bar\alpha^\uparrow\backslash\alpha_{<x}$.
\item $v_j=v_k$ whenever conditions {\rm (i)--(iii)} of Lemma 
\ref{lem:stsuppkl} are all satisfied.
\end{enumerate}
\end{lem}

\begin{proof}
For part $a$, note that $-e_j\in \supp 
S_{B,\mathcal{K}_{i-2},\mathcal{L}_{n-i-1}}$
by Lemma \ref{lem:stsuppkl}, so that
$h_K(-e_j) = ah_L(-e_j)-v_j$. But as $y_j\not\in\alpha_{>x}$,
we have  $h_K(-e_j)=h_L(-e_j)=0$ as in the proof of part $a$ of
Lemma \ref{lem:stsuppkl}, so the conclusion follows.

The argument for part $b$ is analogous: we have
$e_j\in \supp  
S_{B,\mathcal{K}_{i-2},\mathcal{L}_{n-i-1}}$
by Lemma \ref{lem:stsuppkl}, so that
$h_K(e_j) = ah_L(e_j)+v_j$. But as $y_j\not\in\alpha_{<x}$,
it follows readily that $h_K(e_j)=h_L(e_j)=1$, and the conclusion follows.

Finally, for part $c$, we have $e_{jk}\in \supp
S_{B,\mathcal{K}_{i-2},\mathcal{L}_{n-i-1}}$
by Lemma \ref{lem:stsuppkl}, so that
$h_K(e_{jk}) = ah_L(e_{jk})+2^{-1/2}(v_j-v_k)$. But condition (i) implies
$h_K(e_{jk}) =h_L(e_{jk})=0$ as in the proof of part $e$ of
Lemma \ref{lem:stsuppkl}, so the conclusion follows.
\end{proof}

We can now use Lemma \ref{lem:sthelper} as a basic step to
compute $a$ and $v$.

\begin{cor}
\label{cor:sta1}
$a=1$.
\end{cor}

\begin{proof}
We first note that $\alpha_{\not\sim x}\ne\varnothing$. Indeed, if every 
element of $\alpha$ were comparable to $x$, then 
$|\alpha_{<x}|+|\alpha_{>x}|=n-1$, which implies by Lemma 
\ref{lem:trivexst} that $N_i=0$ whenever $i\ne|\alpha_{<x}|+1$.
The latter contradicts $N_i^2=N_{i+1}N_{i-1}>0$.

Fix any $y_{j_0}\in \alpha_{\not\sim x}$. We construct a chain 
$y_{j_{-s}}<\cdots < y_{j_{i}}<y_{j_{i+1}}<\cdots<y_{j_t}$ according 
to the following algorithm. For the upper part of the chain, we 
iteratively choose $y_{j_{i+1}}\in (\alpha_{>y_{j_i}})^\downarrow$ for 
$i\ge 0$ under the constraint that we select $y_{j_{i+1}}\in 
\alpha_{\not\sim x}$ whenever possible. The chain is extended until 
a maximal element of $\alpha$ is reached. The lower part of the chain is 
constructed by iteratively choosing $y_{j_{i-1}}\in 
(\alpha_{<y_{j_i}})^\uparrow$ for $i\le 0$ under the constraint that we 
select $y_{j_{i-1}}\in \alpha_{\not\sim x}$ whenever possible. The chain 
is extended until a minimal element of $\alpha$ is reached.

We claim that this construction ensures the following properties:
\begin{enumerate}[1.]
\item $y_{j_{-s}}\in\bar\alpha^\downarrow\backslash\alpha_{>x}$
and $y_{j_{t}}\in\bar\alpha^\uparrow\backslash\alpha_{<x}$.
\item 
Conditions
(i)--(iii) of Lemma \ref{lem:stsuppkl} 
hold with $j=j_i$, $k=j_{i+1}$ for all $i$.
\end{enumerate}
Indeed, it is impossible that $y_{j_i}\in\alpha_{\le x}$ for some
$i\ge 0$, or that $y_{j_i}\in\alpha_{\ge x}$ for some
$i\le 0$, as that would violate $y_{j_0}\in\alpha_{\not\sim x}$.
Thus the first claim follows as $y_{j_{-s}}$ is minimal and
$y_{j_t}$ is maximal by construction.
To prove the second claim, consider first $i\ge 0$, so that
$y_{j_i},y_{j_{i+1}}\not\in\alpha_{\le x}$. Then (i) holds as
$y_{j_{i+1}}$ covers $y_{j_i}$ in $\bar\alpha$ by construction, 
and (ii) holds automatically. Finally, as we chose 
$y_{j_{i+1}}\in\alpha_{\not\sim x}$ whenever possible, we can only have
$y_{j_i}\in\bar\alpha\backslash\alpha_{>x}$ and 
$y_{j_{i+1}}\in\alpha_{>x}$ if
$y_{j_i}\in(\bar\alpha\backslash\alpha_{>x})^\uparrow$, which establishes 
condition (iii). The proof of the second claim for $i\le 0$ is completely 
analogous.

To conclude the proof, we observe that the above claims and
Lemma \ref{lem:sthelper} imply that
$v_{j_{-s}}=0$, $v_{j_t}=1-a$, and $v_{j_i}=v_{j_{i+1}}$ for all $i$.
Thus $a=1$.
\end{proof}

\begin{cor}
\label{cor:stv0}
$v_j=0$ for all $y_j\in\alpha_{<x}\cup\alpha_{>x}$.
\end{cor}

\begin{proof}
Fix any $y_{j_0}\in\alpha_{>x}$, and construct an increasing chain 
$y_{j_0}<\cdots<y_{j_t}$ by iteratively choosing $y_{j_{i+1}}\in 
(\alpha_{>y_{j_i}})^\downarrow$ until a maximal element is reached. Then 
$y_{j_i}\in\alpha_{>x}$ for all $i$, so 
$y_{j_t}\in\bar\alpha^\uparrow\backslash\alpha_{<x}$ and 
conditions (i)---(iii) of Lemma \ref{lem:stsuppkl} hold with $j=j_i$, 
$k=j_{i+1}$ for all $i$. Thus Lemma \ref{lem:sthelper} and Corollary 
\ref{cor:sta1} imply that $v_{j_t}=0$ and $v_{j_i}=v_{j_{i+1}}$ for all 
$i$. We have therefore shown that $v_j=0$ for any $y_j\in\alpha_{>x}$. The 
proof for the case $y_j\in\alpha_{<x}$ follows in a completely analogous 
manner by constructing a decreasing chain from any 
$y_{j_0}\in\alpha_{<x}$.
\end{proof}

We are now ready to complete the proof of Theorem \ref{thm:exst}.

\begin{proof}[Proof of Theorem \ref{thm:exst}, $a\Rightarrow d$]
Fix $i\in\{2,\ldots,n-1\}$ such that $N_i>0$ and $N_i^2=N_{i+1}N_{i-1}$.
Then Lemma \ref{lem:staf} and Corollaries \ref{cor:sta1} and 
\ref{cor:stv0} imply that
$$
	h_K(x)=h_{L}(x)+\langle v,x\rangle \quad\mbox{for all}\quad
        x\in\supp S_{B,\mathcal{K}_{i-2},\mathcal{L}_{n-i-1}}
$$
holds for a vector $v\in\mathbb{R}^{n-1}$ with
$v_j=0$ for all $y_j\in\alpha_{<x}\cup\alpha_{>x}$.

Now consider any $y_j\in\alpha_{>x}$. Then we claim that 
$-e_j\not\in\supp S_{B,\mathcal{K}_{i-2},\mathcal{L}_{n-i-1}}$. Indeed, if 
$-e_j$ did lie in the support, then we would have $h_K(-e_j)=h_L(-e_j)$.
But this entails a contradiction, as we showed in the
proof of part $a$ of Lemma
\ref{lem:stsuppkl} that $h_K(-e_j)=-1$ and $h_L(-e_j)=0$. It follows by a 
completely analogous argument that
$e_j\not\in\supp S_{B,\mathcal{K}_{i-2},\mathcal{L}_{n-i-1}}$ when
$y_j\in\alpha_{<x}$.

On the other hand, Lemma \ref{lem:stsuppkl} implies that
$-e_j\in\supp S_{B,\mathcal{K}_{i-2},\mathcal{L}_{n-i-1}}$ when
$y_j\in(\alpha_{>x})^\downarrow$ such that $|\alpha_{<y_j}|\le i$.
Consequently, we have shown that $y_j$ satisfying the latter condition 
cannot exist, that is, $|\alpha_{<y}|>i$ whenever
$y\in(\alpha_{>x})^\downarrow$. As $\alpha_{<y}$ can only increase if
we increase $y$, it follows that
$|\alpha_{<y}|>i$ for all $y\in\alpha_{>x}$. By applying the same
reasoning to $e_j$ for $y_j\in(\alpha_{<x})^\uparrow$, it 
follows in a completely analogous manner that 
$|\alpha_{>y}|>n-i+1$ for all $y\in\alpha_{<x}$.
\end{proof}

\section{Discussion and open questions}
\label{sec:discussion}

The main results of this paper completely settle the extremals of the 
Alexandrov-Fenchel inequality for convex polytopes. Analogous extremal 
problems also arise, however, in other situations where Alexandrov-Fenchel 
type inequalities appear. The aim of the final section of this paper is to 
briefly discuss a number of basic open questions in this direction that 
arise from our results.

\subsection{General convex bodies}

The Alexandrov-Fenchel inequality (Theorem \ref{thm:af}) applies to any 
reference bodies $C_1,\ldots,C_{n-2}$. While the main results of this 
paper require that the reference bodies are polytopes, the statements of 
our main results (Theorems \ref{thm:main} and \ref{thm:ultimatemain}) make 
sense for general convex bodies. One may therefore conjecture that the 
statements of our main results extend \emph{verbatim} to the general 
setting. This conjecture is due to Schneider \cite{Sch85} for 
full-dimensional bodies, to which our results add detailed predictions on 
the lower-dimensional cases.

While we have made essential use of the polytope assumption in this paper, 
it should be emphasized that many of our arguments are already completely 
general: neither the gluing arguments nor the statement of the local 
Alexandrov-Fenchel inequality (Theorem \ref{thm:localaf}) rely 
fundamentally on the polytope assumption, a fact that we already exploited 
in section \ref{sec:extensions}. The main obstacle to extending the theory 
of this paper to general convex bodies therefore lies in the proof of 
Theorem \ref{thm:localaf}: if such a result could be proved in the general 
setting, this would essentially complete the extremal characterization for 
general convex bodies.

\subsubsection{The local Alexandrov-Fenchel inequality}

Let us now briefly recall how the polytope structure was used in the proof 
of Theorem \ref{thm:localaf}.

In sections \ref{sec:matrix}--\ref{sec:augment}, we used the combinatorial 
structure of polytopes to reduce Theorem \ref{thm:localaf} to a 
finite-dimensional problem. However, the actual proof of the local 
Alexandrov-Fenchel inequality in section \ref{sec:localaf} does not make 
direct use of the geometry of polytopes. What is really exploited here is 
that the mixed area measures $S_{g,\mathcal{P}}$ and 
$S_{g,g,\mathcal{P}_{\backslash r}}$ are supported in a finite set 
$\{u_i\}_{i\in[N]}$, which has two key consequences: the existence problem 
of Theorem \ref{thm:localaf} can be formulated as the solution of a system 
of linear equations; and the masses $S_{g,\mathcal{P}}(\{u_i\})$ satisfy 
Alexandrov-Fenchel inequalities $S_{g,\mathcal{P}}(\{u_i\})^2\ge 
S_{g,g,\mathcal{P}_{\backslash 
r}}(\{u_i\})\,S_{P_r,P_r,\mathcal{P}_{\backslash r}}(\{u_i\})$ by Lemma 
\ref{lem:mapoly}.

In principle, however, one may conjecture that similar objects could be 
defined directly for general convex bodies in suitable functional-analytic 
framework. For example, it was shown in \cite{SvH19} that for any convex 
bodies $\mathcal{C}=(C_1,\ldots,C_{n-2})$ in $\mathbb{R}^n$, there exists 
a self-adjoint operator $\mathscr{A}_{\mathcal{C}}$ on the 
Hilbert space $L^2(S_{B,\mathcal{C}})$ so that
$$
	\V_n(K,L,\mathcal{C}) = 
	\langle h_K,\mathscr{A}_{\mathcal{C}}h_L
	\rangle_{L^2(S_{B,\mathcal{C}})}
$$
for any convex bodies $K,L$ in $\mathbb{R}^n$. The operator 
$\mathscr{A}_{\mathcal{C}}$ may be viewed as an infinite-dimensional 
analogue of the Alexandrov matrix $\mathrm{A}$ (in the sense of Corollary 
\ref{cor:augprobrep}), and one might therefore attempt to use such objects 
as a replacement for the finite-dimensional computations in the proof of 
Theorem \ref{thm:localaf}.
The problem with this construction, however, is that not only 
$\mathscr{A}_{\mathcal{C}}$ but also the underlying space 
$L^2(S_{B,\mathcal{C}})$ depends on $\mathcal{C}$, while the proof of 
Theorem \ref{thm:localaf} requires us to consider several such objects 
simultaneously. For example, an analogue 
``$(\mathscr{A}_{\mathcal{C}}f)^2\ge 
\mathscr{A}_{f,\mathcal{C}_{\backslash r}}f\cdot 
\mathscr{A}_\mathcal{C}h_{C_r}$'' of the Alexandrov-Fenchel inequality for 
$S_{g,\mathcal{P}}(\{u_i\})$ does not make sense in this form, as the 
operators $\mathscr{A}_{\mathcal{C}}$ and 
$\mathscr{A}_{f,\mathcal{C}_{\backslash r}}$ are defined on different spaces.

A basic question in this context is therefore whether one may construct 
analogues of the self-adjoint operators $\mathscr{A}_{\mathcal{C}}$ for 
different choices of $\mathcal{C}$ on the same space, and whether these 
operators satisfy an appropriate analogue of the Alexandrov-Fenchel 
inequality. Such analytic questions on the structure of mixed volumes are 
in principle completely independent from the study of the extremals. The 
development of such a functional-analytic framework could, however, 
provide a foundation for the extension of the proof of Theorem 
\ref{thm:localaf} to the general setting.

\subsubsection{Supports of mixed area measures}

The main results of this paper characterize the extremal functions $f$ 
such that $S_{f,\mathcal{C}}=0$ on the support of $S_{B,\mathcal{C}}$. To 
obtain a fully geometric interpretation of these results, however, they 
must be combined with a geometric characterization of $\supp 
S_{B,\mathcal{C}}$. The latter is elementary in the setting of polytopes 
(Lemma \ref{lem:supp}), but remains open in general.

The fundamental conjecture in this direction, due to Schneider 
\cite{Sch85}, is that a local analogue of Lemma \ref{lem:supp} (as 
formulated after Definition \ref{defn:bcextreme}) remains valid in the 
general setting. Let us emphasize, however, that the characterization of 
$\supp S_{B,\mathcal{C}}$ played essentially no role in the proofs of our 
main results: this problem appears to be essentially orthogonal to the 
theory developed in this paper.

\subsection{Algebraic analogues}

Beyond its fundamental role in convex geometry, the Alexandrov-Fenchel 
inequality has deep connections with other areas of mathematics. It was 
realized in the 1970s by Teissier and Khovanskii (cf.\ 
\cite{BZ88,Ful93,Gro90}) 
that the Alexandrov-Fenchel inequality has natural analogues in algebraic 
and complex geometry. More recently, it has been realized that such 
connections extend even further: various combinatorial problems, which 
cannot be expressed directly in terms of convex or algebraic geometry, 
nonetheless fit within a general algebraic framework in which analogues of 
the Alexandrov-Fenchel inequality hold \cite{Huh18}.

The rich algebraic theory surrounding the Alexandrov-Fenchel inequality 
raises the intriguing question whether our results might extend to a 
broader context. This question arises, for example, if we aim to develop 
combinatorial applications as in section \ref{sec:stanley} in situations 
that cannot be formulated in convex geometric terms. Related questions in 
algebraic geometry date back to Teissier \cite{Tei82,BFJ09}.

It may not be entirely obvious, however, how to even formulate algebraic 
analogues of the main results of this paper. The aim of this section is to 
sketch how our main results may be expressed in algebraic terms, which 
could (conjecturally) carry over to analogues of the Alexandrov-Fenchel 
inequality outside convexity. To the best of our knowledge, such problems 
are at present almost entirely open.

\subsubsection{Polytope algebra}

In order to describe our main results algebraically, we must first recall 
some aspects of the polytope algebra due to McMullen \cite{McM93,Tim99}.

Let us fix as in section \ref{sec:matrix} a simple polytope $P$ in 
$\mathbb{R}^n$. Define the linear space
$$
	\mathrm{D}:=\{h_Q-h_R:Q,R\mbox{ strongly isomorphic to }P\}.
$$
For $f,g\in\mathrm{D}$, we will write $f\sim g$ if $f-g$ is a linear 
function, and denote by $[f]$ the equivalence class of $f$ with respect to 
this equivalence relation.

The polytope algebra generated by $P$ is a graded algebra
$$
	\mathrm{A}(P) = \bigoplus_{k=0}^n \mathrm{A}^k
$$
with $\mathrm{A}^1\simeq\mathrm{D}/\mathord\sim$ and
$\mathrm{A}^n\simeq\mathbb{R}$. Moreover, the 
(commutative) multiplication of $\mathrm{A}(P)$ has the property that
$F\cdot G\in\mathrm{A}^{k+l}$ for $F\in\mathrm{A}^k$, $G\in\mathrm{A}^l$, 
and 
$$
	[f_1]\cdot\ldots\cdot[f_n] = 
	n!\,\V_n(f_1,\ldots,f_n)
$$
for any $f_1,\ldots,f_n\in\mathrm{D}$. By \eqref{eq:mixvolarea}, one may 
therefore view $[f_1]\cdot\ldots\cdot[f_{n-1}]\in\mathrm{A}^{n-1}$ as 
the algebraic formulation of the mixed area measure 
$(n-1)!\,S_{f_1,\ldots,f_{n-1}}$.

Several different 
notions of positivity are defined by convex cones in $\mathrm{A}^1$:
\begin{enumerate}[$\bullet$]
\item $\mathrm{Amp}:=\{[h_Q]:Q\mbox{ is strongly isomorphic to }P\}$
(the ample cone).
\item $\mathrm{Nef}:=\mathop{\mathrm{cl}}(\mathrm{Amp})=\{[h_Q]:
Q\mbox{ is homothetic to a summand of }P\}$ (the nef cone).
\item $\mathrm{Big}:=\{[f]:f\in\mathrm{D},~f>0\}$ (the big cone).
\item $\overline{\mathrm{Eff}}:=\mathop{\mathrm{cl}}(\mathrm{Big}) =
\{[f]:f\in\mathrm{D},~f\ge 0\}$ (the pseudoeffective cone).
\end{enumerate}
The terminology used here is borrowed from algebraic geometry \cite{Laz04}.
In these terms, two classical algebraic properties admit a familiar 
interpretation in convexity \cite{Huh18}. The \emph{Hodge-Riemann 
relation} of degree one states that
$$
	(\eta\cdot L_0\cdot\ldots\cdot L_{n-2})^2
	\ge
	(\eta\cdot\eta\cdot L_1\cdot\ldots\cdot L_{n-2})\,
	(L_0\cdot L_0\cdot\ldots\cdot L_{n-2})
$$
for any $\eta\in\mathrm{A}^1$ and $L_0,\ldots,L_{n-2}\in\mathrm{Amp}$. 
This is nothing other than the Alexandrov-Fenchel inequality for polytopes 
strongly isomorphic to $P$. On the other hand, the \emph{hard Lefschetz 
theorem} of degree one states that
$$
	\eta\cdot L_1\cdot \ldots\cdot L_{n-2}=0
	\quad\mbox{if and only if}\quad\eta=0
$$
for any $\eta\in\mathrm{A}^1$ and $L_0,\ldots,L_{n-2}\in\mathrm{Amp}$. 
This statement is equivalent to the fact (which was proved in the original 
work of Alexandrov \cite{Ale37}) that $S_{f,\mathcal{P}}=0$ if and only if 
$f$ is a linear function when $f=h_K-h_L$ and $K,L,P_1,\ldots,P_{n-2}$ are 
strongly isomorphic polytopes; in other words, it states that the 
Alexandrov-Fenchel inequality has no nontrivial extremals in the strongly 
isomorphic setting.

With this algebraic language in hand, one may describe various notions of 
convexity in algebraic terms. To give one further example, if we associate 
to any $\eta\in\mathrm{Nef}$ a numerical dimension $\dim\eta := 
\max\{k:\eta^{\cdot k}\ne 0\}$ (cf.\ \cite{Leh13}), then it follows from 
Lemma \ref{lem:dim} that $[h_Q]\in\mathrm{Nef}$ satisfies $\dim[h_Q]=\dim 
Q$. One can therefore readily describe notions such as critical sets, 
supercriticality, etc.\ algebraically.

\subsubsection{Extremals}

While the Hodge-Riemann inequality extends readily to the setting that 
$L_1,\ldots,L_{n-2}\in\mathrm{Nef}$ by continuity, this is not the case 
for the hard Lefschetz theorem. Indeed, the main results of this paper are 
concerned with the case that $\mathcal{P}=(P_1,\ldots,P_{n-2})$ are 
summands of $P$ (see section \ref{sec:matrix}), and it is precisely in 
this case that nontrivial extremals appear. Our results may therefore 
be viewed as refined forms of the hard Lefschetz theorem for nef 
classes.

To express our results algebraically, we must understand the algebraic 
meaning of the condition $f(x)=0$ for all $x\in\supp S_{B,\mathcal{P}}$. 
To this end, let us make the following observation (we freely use the
notation of sections \ref{sec:matrix}--\ref{sec:augment} in the proof).

\begin{lem}
$f\in\mathrm{D}$ satisfies $f(x)=0$ for all $x\in\supp S_{B,\mathcal{P}}$
if and only if there exist $f_1,f_2\in\mathrm{D}$, $f_1,f_2\ge 0$ so that
$f=f_1-f_2$ and $S_{f_1,\mathcal{P}}=S_{f_2,\mathcal{P}}=0$.
\end{lem}

\begin{proof}
For the \emph{if} direction, note that
$S_{f_i,\mathcal{P}}=0$ implies $\int f_i\,dS_{B,\mathcal{P}}=
\int h_B\,dS_{f_i,\mathcal{P}}=0$.
As $f_i\ge 0$, it follows that $f_i(x)=0$ for all $x\in\supp 
S_{B,\mathcal{P}}$ and $i=1,2$.

For the \emph{only if} direction, suppose $f(x)=0$ for all $x\in\supp 
S_{B,\mathcal{P}}$. Applying Lemma \ref{lem:generate} to the vectors 
$z_1=(f(u_i)_+)_{i\in[N]}$ and $z_2=(f(u_i)_-)_{i\in[N]}$, we obtain 
$f_1,f_2\in\mathrm{D}$ so that $f_1(u_i),f_2(u_i)\ge 0$ and 
$f(u_i)=f_1(u_i)-f_2(u_i)$ for all $i\in[N]$, and $f_1(u_i)=f_2(u_i)=0$ 
for $i\in V$. As any function in $\mathrm{D}$ is linear on the normal 
cones of $P$, it follows that $f_1,f_2\ge 0$ and $f=f_1-f_2$ everywhere, 
and that $f_1=f_2=0$ on $\supp S_{B,\mathcal{P}}$ (by Lemma 
\ref{lem:suppsbp}). The conclusion follows from Lemma \ref{lem:suppeq}.
\end{proof}

For sake of illustration, let us consider the simplest setting where 
$L_1,\ldots,L_{n-2}$ are big and nef, that is, $L_i=[h_{P_i}]$ for 
full-dimensional polytopes $P_1,\ldots,P_{n-2}$. The conclusion of Theorem 
\ref{thm:schneider} may then be reformulated as follows.

\begin{cor}
\label{cor:nef}
Let $L_1,\ldots,L_{n-2}\in\mathrm{Nef}\cap\mathrm{Big}$ and
$\eta\in\mathrm{A}^1$. Then $\eta\cdot L_1\cdot\ldots\cdot L_{n-2}=0$
if and only if $\eta=\eta_1-\eta_2$ for some
$\eta_1,\eta_2\in\overline{\mathrm{Eff}}$ so
that $\eta_i\cdot L_1\cdot\ldots\cdot L_{n-2}=0$, $i=1,2$.
\end{cor}

One may analogously reformulate the result of Theorem \ref{thm:main} 
in algebraic terms to characterize any $\eta\in\mathrm{A}^1$ and 
$L_1,\ldots,L_{n-2}\in\mathrm{Nef}$ such that $\eta\cdot 
L_1\cdot\ldots\cdot L_{n-2}=0$ (in Definition \ref{defn:deg}, one may then 
replace $B$ by any ample class).

From the perspective of convex geometry, there is of course nothing new in 
the present formulation. The point of the algebraic formulation is, 
however, that the same algebraic structures carry over to other 
mathematical problems \cite{Huh18}. The statement of Corollary 
\ref{cor:nef} (for example) therefore gives rise to natural conjectures on 
what analogues of the results of this paper might look like in other 
contexts.

\begin{example}
A structure similar to $\mathrm{A}(P)$ arises in algebraic geometry: here 
$\mathrm{D}$ is the space of divisors, $\mathrm{A}^k$ is the space of 
$k$-cycles modulo numerical equivalence, and $\,\cdot\,$ is the 
intersection product on a projective variety \cite{Ful98}. The ample, nef, 
big, and pseudoeffective cones are described in \cite{Laz04}. One might 
therefore ask whether a result such as Corollary \ref{cor:nef} carries 
over to this setting, at least in sufficiently nice situations. To the 
best of our knowledge this question is entirely open, except for toric 
varieties which admit a precise correspondence with convex geometry 
\cite{Ful93,Ew96} (for which such a conclusion follows from the results of 
this paper).
\end{example}

\begin{example}
A structure similar to $\mathrm{A}(P)$ arises in the theory of 
mixed discriminants \cite{Tim98}. This is a much simpler 
setting: for example, here $\mathrm{Amp}=\mathrm{Big}$ is the cone 
of positive definite matrices, so Corollary \ref{cor:nef} reduces to the 
hard Lefschetz theorem. On the other hand, it was shown by Panov 
\cite{Pan85} that this setting admits degenerate extremals 
in complete analogy to Theorem \ref{thm:main}. This setting 
therefore provides an example outside convexity in which 
analogous structures appear.
\end{example}

\SkipTocEntry\subsection*{Acknowledgments}

This work was supported in part by NSF grant DMS-1811735 and by the Simons 
Collaboration on Algorithms \& Geometry. We thank Karim Adiprasito, Swee 
Hong Chan, June Huh, J\'anos Koll\'ar, Igor Pak, Greta Panova, Rolf 
Schneider, and Amir Yehudayoff for helpful comments.

\SkipTocEntry


\begin{thebibliography}{10}

\bibitem{Ale37}
A.~D. Alexandrov.
\newblock Zur {T}heorie der gemischten {V}olumina von konvexen {K}\"orpern
  {II}.
\newblock {\em Mat. Sbornik N.S.}, 2:1205--1238, 1937.

\bibitem{Ale96}
A.~D. Alexandrov.
\newblock {\em Selected works. {P}art {I}}.
\newblock Gordon and Breach Publishers, Amsterdam, 1996.

\bibitem{BF87}
T.~Bonnesen and W.~Fenchel.
\newblock {\em Theory of convex bodies}.
\newblock BCS Associates, Moscow, ID, 1987.

\bibitem{BFJ09}
S.~Boucksom, C.~Favre, and M.~Jonsson.
\newblock Differentiability of volumes of divisors and a problem of {T}eissier.
\newblock {\em J. Algebraic Geom.}, 18(2):279--308, 2009.

\bibitem{BZ88}
Y.~D. Burago and V.~A. Zalgaller.
\newblock {\em Geometric inequalities}.
\newblock Springer-Verlag, Berlin, 1988.

\bibitem{CFG80}
F.~R.~K. Chung, P.~C. Fishburn, and R.~L. Graham.
\newblock On unimodality for linear extensions of partial orders.
\newblock {\em SIAM J. Algebraic Discrete Methods}, 1(4):405--410, 1980.

\bibitem{Con90}
J.~B. Conway.
\newblock {\em A course in functional analysis}, volume~96 of {\em Graduate
  Texts in Mathematics}.
\newblock Springer-Verlag, New York, second edition, 1990.

\bibitem{Ew88}
G.~Ewald.
\newblock On the equality case in {A}lexandrov-{F}enchel's inequality for
  convex bodies.
\newblock {\em Geom. Dedicata}, 28(2):213--220, 1988.

\bibitem{Ew96}
G.~Ewald.
\newblock {\em Combinatorial convexity and algebraic geometry}, volume 168 of
  {\em Graduate Texts in Mathematics}.
\newblock Springer-Verlag, New York, 1996.

\bibitem{ET94}
G.~Ewald and E.~Tondorf.
\newblock A contribution to equality in {A}lexandrov-{F}enchel's inequality.
\newblock {\em Geom. Dedicata}, 50(3):217--233, 1994.

\bibitem{Fav33}
J.~{Favard}.
\newblock {Sur les corps convexes.}
\newblock {\em {J. Math. Pures Appl. (9)}}, 12:219--282, 1933.

\bibitem{Fen36}
W.~Fenchel.
\newblock In\'egalit\'es quadratiques entre les volumes mixtes des corps
  convexes.
\newblock {\em C. R. Acad. Sci. Paris}, 203:647--650, 1936.

\bibitem{Ful93}
W.~Fulton.
\newblock {\em Introduction to toric varieties}, volume 131 of {\em Annals of
  Mathematics Studies}.
\newblock Princeton University Press, Princeton, NJ, 1993.

\bibitem{Ful98}
W.~Fulton.
\newblock {\em Intersection theory}.
\newblock Springer-Verlag, Berlin, second edition, 1998.

\bibitem{Gro90}
M.~Gromov.
\newblock Convex sets and {K}\"{a}hler manifolds.
\newblock In {\em Advances in differential geometry and topology}, pages 1--38.
  World Sci. Publ., Teaneck, NJ, 1990.

\bibitem{Gru07}
P.~M. Gruber.
\newblock {\em Convex and discrete geometry}, volume 336 of {\em Grundlehren
  der Mathematischen Wissenschaften [Fundamental Principles of Mathematical
  Sciences]}.
\newblock Springer, Berlin, 2007.

\bibitem{Huh18}
J.~Huh.
\newblock Combinatorial applications of the {H}odge-{R}iemann relations.
\newblock In {\em Proceedings of the {I}nternational {C}ongress of
  {M}athematicians---{R}io de {J}aneiro 2018. {V}ol. {IV}. {I}nvited lectures},
  pages 3093--3111. World Sci. Publ., Hackensack, NJ, 2018.

\bibitem{Kal02}
O.~Kallenberg.
\newblock {\em Foundations of modern probability}.
\newblock Probability and its Applications (New York). Springer-Verlag, New
  York, second edition, 2002.

\bibitem{Laz04}
R.~Lazarsfeld.
\newblock {\em Positivity in algebraic geometry. {I}}.
\newblock Springer-Verlag, Berlin, 2004.
\newblock Classical setting: line bundles and linear series.

\bibitem{Leh13}
B.~Lehmann.
\newblock Comparing numerical dimensions.
\newblock {\em Algebra Number Theory}, 7(5):1065--1100, 2013.

\bibitem{McM93}
P.~McMullen.
\newblock On simple polytopes.
\newblock {\em Invent. Math.}, 113(2):419--444, 1993.

\bibitem{Min03}
H.~Minkowski.
\newblock Volumen und {O}berfl\"{a}che.
\newblock {\em Math. Ann.}, 57(4):447--495, 1903.

\bibitem{Pan85}
A.~A. Panov.
\newblock Some properties of mixed discriminants.
\newblock {\em Mat. Sb. (N.S.)}, 128(170)(3):291--305, 446, 1985.

\bibitem{San93}
J.~R. Sangwine-Yager.
\newblock Mixed volumes.
\newblock In {\em Handbook of convex geometry, {V}ol. {A}, {B}}, pages 43--71.
  North-Holland, Amsterdam, 1993.

\bibitem{Sch75}
R.~Schneider.
\newblock Kinematische {B}er\"{u}hrma\ss e f\"{u}r konvexe {K}\"{o}rper und
  {I}ntegralrelationen f\"{u}r {O}berfl\"{a}chenma\ss e.
\newblock {\em Math. Ann.}, 218(3):253--267, 1975.

\bibitem{Sch85}
R.~Schneider.
\newblock On the {A}leksandrov-{F}enchel inequality.
\newblock In {\em Discrete geometry and convexity}, volume 440 of {\em Ann. New
  York Acad. Sci.}, pages 132--141. New York Acad. Sci., 1985.

\bibitem{Sch88}
R.~Schneider.
\newblock On the {A}leksandrov-{F}enchel inequality involving zonoids.
\newblock {\em Geom. Dedicata}, 27(1):113--126, 1988.

\bibitem{Sch94b}
R.~Schneider.
\newblock Equality in the {A}leksandrov-{F}enchel inequality---present state
  and new results.
\newblock In {\em Intuitive geometry ({S}zeged, 1991)}, volume~63 of {\em
  Colloq. Math. Soc. J\'{a}nos Bolyai}, pages 425--438. North-Holland,
  Amsterdam, 1994.

\bibitem{Sch94}
R.~Schneider.
\newblock Polytopes and {B}runn-{M}inkowski theory.
\newblock In {\em Polytopes: abstract, convex and computational}, volume 440 of
  {\em NATO Adv. Sci. Inst. Ser. C}, pages 273--299. Kluwer, 1994.

\bibitem{Sch14}
R.~Schneider.
\newblock {\em Convex bodies: the {B}runn-{M}inkowski theory}.
\newblock Cambridge University Press, expanded edition, 2014.

\bibitem{SvH18}
Y.~Shenfeld and R.~van Handel.
\newblock Mixed volumes and the {B}ochner method.
\newblock {\em Proc. Amer. Math. Soc.}, 147(12):5385--5402, 2019.

\bibitem{SvH19}
Y.~Shenfeld and R.~{Van Handel}.
\newblock The extremals of {M}inkowski's quadratic inequality.
\newblock {\em Duke Math. J.}, 2022.
\newblock To appear.

\bibitem{Sta81}
R.~P. Stanley.
\newblock Two combinatorial applications of the {A}leksandrov-{F}enchel
  inequalities.
\newblock {\em J. Combin. Theory Ser. A}, 31(1):56--65, 1981.

\bibitem{Sta86}
R.~P. Stanley.
\newblock Two poset polytopes.
\newblock {\em Discrete Comput. Geom.}, 1(1):9--23, 1986.

\bibitem{Sta89}
R.~P. Stanley.
\newblock Log-concave and unimodal sequences in algebra, combinatorics, and
  geometry.
\newblock In {\em Graph theory and its applications: {E}ast and {W}est
  ({J}inan, 1986)}, volume 576 of {\em Ann. New York Acad. Sci.}, pages
  500--535. New York Acad. Sci., New York, 1989.

\bibitem{Sus32}
W.~{S\"uss}.
\newblock Zusammensetzung von eik\"orpern und homothetische eifl\"achen.
\newblock {\em {T\^ohoku Math. J.}}, 35:47--50, 1932.

\bibitem{Tei82}
B.~Teissier.
\newblock Bonnesen-type inequalities in algebraic geometry. {I}. {I}ntroduction
  to the problem.
\newblock In {\em Seminar on {D}ifferential {G}eometry}, volume 102 of {\em
  Ann. of Math. Stud.}, pages 85--105. Princeton Univ. Press, Princeton, N.J.,
  1982.

\bibitem{Tim98}
V.~A. Timorin.
\newblock The mixed {H}odge-{R}iemann bilinear relations in the linear
  situation.
\newblock {\em Funct. Anal. Appl.}, 32(4):268--272, 1998.

\bibitem{Tim99}
V.~A. Timorin.
\newblock An analogue of the {H}odge-{R}iemann relations for simple convex
  polytopes.
\newblock {\em Russ. Math. Surv.}, 54(2):381--426, 1999.

\bibitem{Wey17}
H.~Weyl.
\newblock {\"U}ber die {S}tarrheit der {E}ifl\"achen und konvexen {P}olyeder.
\newblock {\em Sitzungsber. Preuss. Akad. Wiss. Berlin}, pages 250--266, 1917.

\end{thebibliography}

\end{document}